\documentclass[a4paper,fleqn]{cas-dc}

\usepackage[numbers]{natbib}
\usepackage{graphicx}
\usepackage{subcaption}
\usepackage{dcolumn}
\usepackage{bm}

\usepackage{amsthm}  
\newtheorem{theorem}{Theorem}[section]
\newtheorem{lemma}[theorem]{Lemma}

\theoremstyle{definition} 
\newtheorem{remark}{Remark}

\usepackage{etoolbox}
\usepackage{booktabs}   
\usepackage{tabularx}
\usepackage{hyperref}

\def\tsc#1{\csdef{#1}{\textsc{\lowercase{#1}}\xspace}}
\tsc{WGM}
\tsc{QE}

\begin{document}
\let\WriteBookmarks\relax
\def\floatpagepagefraction{1}
\def\textpagefraction{.001}

\shorttitle{}    

\shortauthors{}  

\title [mode = title]{Higher-Order Multiscale Computational Method for Multi-Continuum Problems in Highly Heterogeneous Media}  



%

\author[1]{Hao Dong}
\cormark[1]
\ead{hao.dong@xidian.edu.cn}

\author[1]{Jiayuan Peng}

\author[2]{Jian Huang}

\affiliation[1]{organization={School of Mathematics and Statistics, Xidian University},
            addressline={No.2 South Taibai Road}, 
            city={Xi'an},
            postcode={710071}, 
            state={Shaanxi},
            country={China}}

\affiliation[2]{organization={School of Mathematics and Computational Science, Xiangtan University},
            addressline={}, 
            city={Xiangtan},
            postcode={411105}, 
            state={Hunan},
            country={China}}

\cortext[1]{Corresponding author.}


\begin{abstract}
This paper presents a high-accuracy higher-order multiscale method for solving multi-continuum problems in in highly heterogeneous media. First, microscopic unit cell functions are defined, leading to the derivation of macroscopic homogenized equations and formulas for calculating effective parameters, which yield a higher-order multi-scale (HOMS) asymptotic solution. Subsequently, the pointwise approximation properties of this solution to the original equations are analyzed, and its convergence rate in the integral norm is rigorously established under certain assumptions. Furthermore, a multiscale numerical algorithm is developed by integrating the finite element method (FEM), finite difference method, and interpolation technique. Finally, numerical experiments demonstrate the high accuracy, efficiency, and stability of the proposed HOMS numerical algorithm.
\end{abstract}




\begin{keywords}
porous media \sep multi-continuum \sep multiscale method
\end{keywords}

\maketitle

\section{Introduction}
Porous media are functional materials composed of a solid skeleton and interconnected or isolated pores, which can be filled with gases, liquids, or other fluid phases. Such materials are widely found in both natural and engineered systems, with typical examples including soil, rocks, sand layers, wood, biological tissues, ceramic filters, and catalyst supports, as shown in Fig~\ref{fig:gctpic}. Due to their internal structure featuring a vast number of typically interconnected pores, porous media provide physical pathways for fluid storage, transport, and reactions, making them a central focus of interdisciplinary research across fields such as fluid mechanics, geosciences, environmental engineering, biomedical engineering, and materials science. Accurately understanding and modeling fluid flow behavior in porous media is essential for numerous critical technological applications. These include the efficient extraction of oil and gas resources, prediction and remediation of groundwater contamination, optimization of mass transport in porous electrodes of fuel cells, and long-term stability assessments in geological-scale carbon capture and storage processes.
\begin{figure}[pos=htbp]
\centering
{\includegraphics[width=0.3\textwidth]{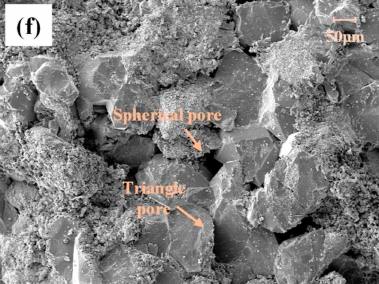}}
\caption{Structure of sandstone (from Zhao et al.~\cite{ZHAO2022110987})}
\label{fig:gctpic}
\end{figure}

Some of the first studies on porous media treated the porous medium as a single continuum. Fluid flow in porous media is often described at the macroscopic scale by Darcy's Law, which assigns effective properties to each representative volume through local simulations to construct coarse-grid equations\cite{henning_heterogeneous_2009,hillairet_homogenization_2018}. Allaire et al. investigated the influence of different obstacles on fluid behavior by homogenizing the microscopic Stokes and Navier-Stokes equations, deriving macroscopic models including Darcy's law, Stokes equations, and Brinkman-type equations~\cite{Gr1991Homogenization1,Gr1991Homogenization2,Gr1991Homogenization3}. In terms of numerical methods, Efendiev et al. developed the Generalized Multiscale Finite Element Method (GMsFEM) and extended it to perforated domains, enabling efficient coarse-grid simulations through the offline construction of local snapshot spaces and spectral decomposition\cite{2013Generalized,chung_conservative_2017,chung_generalized_2016,chung_mixed_2016,chung_online_2017,xie_time_2025}.

In contrast to the single-continuum problem, the multi-continuum problem refers to the coexistence of two or more overlapping and interacting continua within the same spatial region. Each continuum possesses its own distinct physical properties and field variables, and they are coupled through specific interaction terms. Park et al. proposed a multi-continuum model that partitions complex porous media into multiple interacting continua (e.g., matrix and fractures) and employs coupling terms to describe mass exchange between them, thereby capturing non-equilibrium flow phenomena beyond the capability of traditional single-continuum models. \cite{park_hierarchical_2020,park_homogenization_2022,park_multiscale_2020,park_multiscale_2021}. G.I. Barenblatt et al., addressing the issue of seepage in fractured rocks, introduced the pressure in pores and the pressure of liquid in fractures at each point in space. Considering liquid transfer between fractures and pores, they derived the fundamental equations for liquid seepage in fractured rocks and the general equations for liquid seepage in porous media with dual porosity \cite{1960Basic}. Eric T. Chun systematically studied the intrinsic connections and coupling mechanisms between multiscale methods and multi-continuum approaches, establishing a generalized multiscale computational framework capable of efficiently simulating flow-heat transfer processes in complex fractured media \cite{2017Coupling}. K. Pruess et al. proposed the MINC method, which accurately captures transient thermal-fluid coupling between fractures and matrix by subdividing matrix blocks into multiple interacting continua, thereby addressing strongly non-equilibrium physical phenomena that traditional models fail to describe \cite{K1982On,1985Practical,1988A}. Xie et al. divided the perforated region into multiple continua based on the size of the holes, established macroscopic equations for each medium separately, and captured the coupling between them by imposing constraints on the local problems~\cite{xie_multicontinuum_2025}. Ammosov et al. derived a multi-continuum model for coupled flow and transport equations by applying the multi-continuum homogenization method, performing multi-continuum expansions for both flow and transport solutions and formulating novel coupled constraint cell problems to capture multiscale characteristics, ultimately constructing a macroscopic system composed of homogenized elliptic equations and convection-diffusion-reaction equations~\cite{ammosov_multicontinuum_2026,2024Generalized}. Wu et al. proposed a triple-continuum model that treats fractures, matrix, and vugs as independent continua and established the mass exchange mechanisms among them. This model is capable of characterizing different connection patterns between vugs and fractures. In terms of numerical implementation, it accommodates non-Darcy flow and multiphase flow simulation~\cite{wu_multiple-continuum_2011,li_multi-continuum_2015}. 

In previous studies on multiphase flow in porous media, the matrix system was commonly treated as a single continuum with locally uniform pressure and fluid saturation distributions. Using traditional single-continuum models to describe flow in fractured rocks may lead to incorrect or even opposite conclusions, as such models fail to account for the non-equilibrium mass transfer between the matrix and fractures. Therefore, it is essential to establish multi-continuum coupling models that incorporate the interaction between the matrix and fractures to more accurately characterize flow behavior in complex fractured media. Furthermore, due to strong material heterogeneity and the complexity of micro- and meso-scale structures, the multi-continuum model is inherently a complex multiscale problem. Conventional numerical methods, such as finite element and finite volume methods, typically require extremely fine mesh discretization to achieve sufficiently accurate solutions, resulting in high computational cost and low efficiency—particularly in large-scale engineering applications. Moreover, for multi-continuum problems with strong nonlinearity and high-contrast parameters, the convergence and stability of traditional methods face significant challenges. To address these issues, this paper develops the HOMS computational method tailored for multi-continuum problems in microscopic highly heterogeneous media. This approach achieves high-precision asymptotic solutions while substantially reducing computational resource consumption, striking a balance between efficiency and accuracy. It provides a new theoretical framework and numerical tool for the efficient simulation of multiphase flow in complex fractured media.

This paper primarily investigates multi-continuum problems in in highly heterogeneous media and establishes a relatively comprehensive HOMS computational framework for addressing these challenges. The main contents are organized as follows: In Section~\ref{sec:2}, a HOMS method is developed for multi-continuum problems in periodic composite structures. This includes the definition of microscopic unit cell functions and the derivation of macroscopic equivalent parameter calculation formulas along with the HOMS asymptotic solution. In Section~\ref{sec:level3}, the pointwise approximation properties of the HOMS solution to the original equations are analyzed. Furthermore, under certain assumptions, the convergence order of the solution in the integral norm is rigorously established. In Section ~\ref{sec:4} and Section ~\ref{sec:5} , the corresponding HOMS algorithm is presented. Numerical experiments are conducted to validate the high accuracy, efficiency, and stability of the proposed method, thereby demonstrating the necessity of incorporating second-order correction terms. In Section~\ref{sec:6}, the research findings of the paper are systematically summarized, and potential directions for future investigation are proposed.

\section{\label{sec:2}Higher-Order Multiscale Computational Method for Multi-Continuum Problems in in Highly Heterogeneous Media}

\subsection{\label{sec:2.1}Mathematical Model of the Multiscale Multi-Continuum Problem}
For multi-continuum problems, we denote the solution of the \(l\)-th continuum as \({u}_{\scriptscriptstyle l}^{\scriptscriptstyle \varepsilon}\), and assume each continuum interacts with the others. The governing equations for the continua can be formulated as:
\begin{equation}
\begin{aligned}
&c_{\scriptscriptstyle l}^{\scriptscriptstyle \varepsilon}(\boldsymbol{x})
\frac{\partial u_{\scriptscriptstyle l}^{\scriptscriptstyle \varepsilon}(\boldsymbol{x},t)}{\partial t} 
= \frac{\partial}{\partial x_{\scriptscriptstyle i}}
\left(\kappa_{\scriptscriptstyle l,ij}^{\scriptscriptstyle \varepsilon}(\boldsymbol{x})
\frac{\partial u_{\scriptscriptstyle l}^{\scriptscriptstyle \varepsilon}(\boldsymbol{x},t)}{\partial x_{\scriptscriptstyle j}}\right) \\
&+ Q_{\scriptscriptstyle l}^{\scriptscriptstyle \varepsilon}
\left(u_{\scriptscriptstyle 1}^{\scriptscriptstyle \varepsilon}(\boldsymbol{x},t),\ldots,
u_{\scriptscriptstyle N}^{\scriptscriptstyle \varepsilon}(\boldsymbol{x},t)\right), 
\quad (\boldsymbol{x},t)\in\Omega\times(0,T^{*}),
\end{aligned}
\end{equation}
where \(c_{\scriptscriptstyle l}^{\scriptscriptstyle\varepsilon}(\boldsymbol{x})\ (l=1,2,...N)\) represents the multiscale porosity, \(\kappa_{\scriptscriptstyle l,ij}^{\scriptscriptstyle\varepsilon}(\boldsymbol{x})\) denotes the multiscale permeability, and \({Q}_{\scriptscriptstyle l}^{\scriptscriptstyle\varepsilon}\left({u}_{\scriptscriptstyle 1}^{\scriptscriptstyle\varepsilon}(\boldsymbol{x},t),\ldots,{u}_{\scriptscriptstyle N}^{\scriptscriptstyle\varepsilon}(\boldsymbol{x},t)\right)\) is the exchange term describing the interactions between continua.

In the context of seepage flow in fissured rocks\cite{1960Basic,2017Coupling,K1982On,1988A,1976Numerical,1963The}, if only the interactions between the first continuum \({u}_{\scriptscriptstyle 1}^{\scriptscriptstyle \varepsilon}\) and the second continuum \({u}_{\scriptscriptstyle 2}^{\scriptscriptstyle \varepsilon}\) are considered, the governing equations for the multi-continuum problem can be written as follows:
\begin{equation}
\begin{cases}
\begin{aligned}
&c_{\scriptscriptstyle 1}^{\scriptscriptstyle \varepsilon}(\boldsymbol{x})
\frac{\partial{u}_{\scriptscriptstyle 1}^{\scriptscriptstyle \varepsilon}(\boldsymbol{x},t)}{\partial t} \!
= \! \frac{\partial}{\partial x_{\scriptscriptstyle i}} \!
\Bigl(\kappa_{\scriptscriptstyle 1,ij}^{\scriptscriptstyle \varepsilon}(\boldsymbol{x})
\frac{\partial{u}_{\scriptscriptstyle 1}^{\scriptscriptstyle \varepsilon}(\boldsymbol{x},t)}{\partial x_{\scriptscriptstyle j}}\Bigr) \! + q(\boldsymbol{x},t)\\ 
&+ \alpha\big({u}_{\scriptscriptstyle 1}^{\scriptscriptstyle \varepsilon}(\boldsymbol{x},t) 
- {u}_{\scriptscriptstyle 2}^{\scriptscriptstyle \varepsilon}(\boldsymbol{x},t)\big) ,
\ (\boldsymbol{x},t)\in\Omega\times(0,T^{*}),  \\
&c_{\scriptscriptstyle 2}^{\scriptscriptstyle \varepsilon}(\boldsymbol{x})
\frac{\partial{u}_{\scriptscriptstyle 2}^{\scriptscriptstyle \varepsilon}(\boldsymbol{x},t)}{\partial t} \!
= \! \frac{\partial}{\partial x_{\scriptscriptstyle i}} \!
\Bigl(\kappa_{\scriptscriptstyle 2,ij}^{\scriptscriptstyle \varepsilon}(\boldsymbol{x})
\frac{\partial{u}_{\scriptscriptstyle 2}^{\scriptscriptstyle \varepsilon}(\boldsymbol{x},t)}{\partial x_{\scriptscriptstyle j}}\Bigr) \! + q(\boldsymbol{x},t) \\
&+ \alpha\big({u}_{\scriptscriptstyle 2}^{\scriptscriptstyle \varepsilon}(\boldsymbol{x},t) 
- {u}_{\scriptscriptstyle 1}^{\scriptscriptstyle \varepsilon}(\boldsymbol{x},t)\big) ,
\ (\boldsymbol{x},t)\in\Omega\times(0,T^{*}), 
\label{eq:1}
\end{aligned}
\end{cases}
\end{equation}
where \(q(\boldsymbol{x},t)\) is the source term, and \(\alpha\) represents the intensity of liquid transfer between the two media, which is inversely proportional to \(\varepsilon^{\scriptscriptstyle 2}\), i.e., \(\alpha=O(1/\varepsilon^{\scriptscriptstyle 2})\). 

The multi-continuum approach establishes an averaged model that differs from classical seepage theory. Instead of introducing a single liquid pressure at each point in space, it introduces two distinct pressures, \({u}_{\scriptscriptstyle 1}^{\scriptscriptstyle\varepsilon}(\boldsymbol{x},t)\) and \({u}_{\scriptscriptstyle 2}^{\scriptscriptstyle\varepsilon}(\boldsymbol{x},t)\) . Here, \({u}_{\scriptscriptstyle 1}^{\scriptscriptstyle\varepsilon}\) denotes the average pressure of the liquid in the fractures near a given point, while \({u}_{\scriptscriptstyle 2}^{\scriptscriptstyle\varepsilon}\) represents the average pressure of the liquid in the porous matrix in the vicinity of the same point. In this paper, we investigate multi-continuum problems in highly heterogeneous media. Considering the interactions between the two media, we establish the following model:
\begin{equation}
\begin{cases} 
\displaystyle c_{\scriptscriptstyle 1}^{\varepsilon}(\boldsymbol{x}) \frac{\partial u_{\scriptscriptstyle 1}^{\varepsilon}(\boldsymbol{x},t)}{\partial t} \! =  \!
\frac{\partial}{\partial x_{i}} \Bigl( \kappa_{\scriptscriptstyle 1,ij}^{\scriptscriptstyle \varepsilon}(\boldsymbol{x}) \frac{\partial u_{\scriptscriptstyle 1}^{\scriptscriptstyle\varepsilon}(\boldsymbol{x},t)}{\partial x_{j}} \! \Bigr)\!+ \! q(\boldsymbol{x},t)\\ 
\displaystyle + \frac{1}{\varepsilon} Q_{\scriptscriptstyle 1}^{\scriptscriptstyle \varepsilon}(\boldsymbol{x}) \bigl( \! u_{\scriptscriptstyle 2}^{\scriptscriptstyle \varepsilon}(\boldsymbol{x},t) \! - \! u_{\scriptscriptstyle 1}^{\scriptscriptstyle \varepsilon}(\boldsymbol{x},t) \! \bigr)
 , \ (\boldsymbol{x},t) \in \Omega \times (0,T^{*} \! ), \\
\displaystyle c_{\scriptscriptstyle 2}^{\varepsilon}(\boldsymbol{x}) \frac{\partial u_{\scriptscriptstyle 2}^{\varepsilon}(\boldsymbol{x},t)}{\partial t} \! = \!
\frac{\partial}{\partial x_{i}} \Bigl(\! \kappa_{\scriptscriptstyle 2,ij}^{\varepsilon}(\boldsymbol{x}) \frac{\partial u_{\scriptscriptstyle 2}^{\varepsilon}(\boldsymbol{x},t)}{\partial x_{j}} \! \Bigr) \!+\!q(\boldsymbol{x},t)\\
\displaystyle + \frac{1}{\varepsilon} Q_{\scriptscriptstyle 2}^{\scriptscriptstyle \varepsilon}(\boldsymbol{x}) \bigl( \! u_{\scriptscriptstyle 1}^{\scriptscriptstyle \varepsilon}(\boldsymbol{x},t) \!-\! u_{\scriptscriptstyle 2}^{\scriptscriptstyle \varepsilon}(\boldsymbol{x},t) \! \bigr)  , \ (\boldsymbol{x},t) \in \Omega \times (0,T^{*}\!),\\
\displaystyle u_{\scriptscriptstyle 1}^{\varepsilon}(\boldsymbol{x},t) = 0, \ u_{\scriptscriptstyle 2}^{\scriptscriptstyle\varepsilon}(\boldsymbol{x},t) = 0,\ (\boldsymbol{x},t) \in \partial \Omega \times (0,T^*), \\
\displaystyle u_{\scriptscriptstyle 1}^{\scriptscriptstyle\varepsilon}(\boldsymbol{x},0) = g_{\scriptscriptstyle 1}(\boldsymbol{x}),\ u_{\scriptscriptstyle 2}^{\scriptscriptstyle\varepsilon}(\boldsymbol{x},0) = g_{\scriptscriptstyle 2}(\boldsymbol{x}),\ \boldsymbol{x} \in \Omega. 
\label{eq:3}
\end{cases}
\end{equation}
Here, \(\Omega\) is a bounded convex domain in \(\mathbb{R}^{d}(d=2,3)\), with \(\partial\Omega\) as its boundary; \(\varepsilon\) is a small periodic parameter, \(u_{\scriptscriptstyle l}^{\scriptscriptstyle\varepsilon}(\boldsymbol{x},t) (l=1,2)\) is the unknown pressure field, \(c_{\scriptscriptstyle l}^{\scriptscriptstyle\varepsilon}(\boldsymbol{x})\) represents the multiscale porosity, \(\kappa_{\scriptscriptstyle l,ij}^{\scriptscriptstyle\varepsilon}(\boldsymbol{x})\) represents the multiscale permeability, \(Q_{\scriptscriptstyle l}^{\scriptscriptstyle\varepsilon}(\boldsymbol{x})\) is the exchange term and \(Q_{\scriptscriptstyle l}^{\scriptscriptstyle\varepsilon}(\boldsymbol{x})=O(1/\varepsilon)\); \(q(\boldsymbol{x},t)\) denotes the heat source function, and \(g_{\scriptscriptstyle l}(\boldsymbol{x})\) denotes the pressure at the initial moment.

First, the following assumptions are made regarding the coefficients of problem~\eqref{eq:3}:
\begin{description}
\item[\((A_{1})\)] Let \(\boldsymbol{y}=\displaystyle\frac{\boldsymbol{x}}{\varepsilon}\) denote the local coordinate in the microscale unit cell \(Y=[0,1]^{d}\). Then we have:
\end{description}
\begin{align*}
Q_{\scriptscriptstyle l}^{\varepsilon}(\boldsymbol{x}) &= Q_{\scriptscriptstyle l}\! \left(\frac{\boldsymbol{x}}{\varepsilon}\right) = Q_{\scriptscriptstyle l}(\boldsymbol{y}), \
c_{l}^{\scriptscriptstyle\varepsilon}(\boldsymbol{x}) = c_{\scriptscriptstyle l}\left(\frac{\boldsymbol{x}}{\varepsilon}\right) = c_{\scriptscriptstyle l}(\boldsymbol{y}), \\
&\kappa_{\scriptscriptstyle l,ij}^{\scriptscriptstyle\varepsilon}(\boldsymbol{x}) = \kappa_{\scriptscriptstyle l,ij}\left(\frac{\boldsymbol{x}}{\varepsilon}\right) = \kappa_{\scriptscriptstyle l,ij}(\boldsymbol{y}),
\end{align*}
where \(Q_{\scriptscriptstyle l}(\boldsymbol{y})\), \(c_{\scriptscriptstyle l}(\boldsymbol{y})\) and \(\kappa_{\scriptscriptstyle l,ij}(\boldsymbol{y})\) are all 1-periodic functions with respect to the variable \(\boldsymbol{y}\).
\begin{description}
\item[\((A_{2})\)] \(\kappa_{\scriptscriptstyle l,ij}(\boldsymbol{y})\) satisfies symmetry, i.e., \(\kappa_{\scriptscriptstyle l,ij}(\boldsymbol{y})=\kappa_{\scriptscriptstyle l,ji}(\boldsymbol{y})\); and there exist \(\gamma_{0},\gamma_{1}>0\) such that the following holds:
\end{description}
\begin{equation*}
\gamma_{0}\left|\xi\right|^{2}\leq\kappa_{\scriptscriptstyle l,ij}(\boldsymbol{y})\xi_{i}\xi_{j}\leq\gamma_{1}\left|\xi\right|^{2},
\end{equation*}
where \(\boldsymbol{\xi}=(\xi_{1},\cdots,\xi_{n})^{\top}\) is any real vector in \(\mathbb{R}^{d}\).
\begin{description}
\item[\((A_{3})\)] \(c_{\scriptscriptstyle l}(\boldsymbol{y})\in L^{\infty}(\Omega)\) and \(\kappa_{\scriptscriptstyle l,ij}(\boldsymbol{y})\in L^{\infty}(\Omega)\). There exist positive constants \(\underline{C}\) and \(\underline{\kappa}\) such that
\end{description}
\[0 \leq \underline{C} \leq c_{\scriptscriptstyle l}(\boldsymbol{y}), \quad 0 \leq \underline{\kappa} \leq \kappa_{\scriptscriptstyle l, ij}(\boldsymbol{y}) .
\]

\subsection{\label{sec:2.2}High-Order Multiscale Computational Model}
For problem~\eqref{eq:3} , assume that the solution has the following asymptotic expansion form:
\begin{equation}
\begin{cases}
\begin{aligned}
u_{\scriptscriptstyle 1}^{\varepsilon}(\boldsymbol{x},t) &= u_{\scriptscriptstyle 1}^{\scriptscriptstyle(0)}(\boldsymbol{x},\boldsymbol{y},t) 
+ \varepsilon u_{\scriptscriptstyle 1}^{\scriptscriptstyle(1)}(\boldsymbol{x},\boldsymbol{y},t) \\
&+ \varepsilon^{\scriptscriptstyle 2} u_{\scriptscriptstyle 1}^{\scriptscriptstyle(2)}(\boldsymbol{x},\boldsymbol{y},t) 
+ O(\varepsilon^{\scriptscriptstyle 3}), \\
u_{\scriptscriptstyle 2}^{\varepsilon}(\boldsymbol{x},t) &= u_{\scriptscriptstyle 2}^{\scriptscriptstyle(0)}(\boldsymbol{x},\boldsymbol{y},t) 
+ \varepsilon u_{\scriptscriptstyle 2}^{\scriptscriptstyle(1)}(\boldsymbol{x},\boldsymbol{y},t) \\
& + \varepsilon^{\scriptscriptstyle 2} u_{\scriptscriptstyle 2}^{\scriptscriptstyle(2)}(\boldsymbol{x},\boldsymbol{y},t) 
+ O(\varepsilon^{\scriptscriptstyle 3}).
\end{aligned}
\end{cases}
\label{eq:expansion}
\end{equation}

Substitute Eqs.~\eqref{eq:expansion} into problem~\eqref{eq:3} and apply the chain rule
\begin{equation}
\frac{\partial \Phi^\varepsilon(\boldsymbol{x},t)}{\partial x_i} = \frac{\partial \Phi(\boldsymbol{x},\boldsymbol{y},t)}{\partial x_i} + \frac{1}{\varepsilon} \frac{\partial \Phi(\boldsymbol{x},\boldsymbol{y},t)}{\partial y_i},
\label{eq:chainrule}
\end{equation}
to expand the partial derivatives in problem~\eqref{eq:3}. Then, consolidate terms with like powers of \(\varepsilon\) on both sides. Through comparison, a sequence of equations can be obtained:
\begin{equation}
O(\frac{1}{\varepsilon^{\scriptscriptstyle 2}}):
\begin{cases}
\displaystyle \frac{\partial}{\partial y_{i}} \Bigl( \kappa_{\scriptscriptstyle 1,ij}(\boldsymbol{y}) 
\frac{\partial u_{\scriptscriptstyle 1}^{\scriptscriptstyle(0)}}{\partial y_{j}} \Bigr) = 0, \\[8pt]
\displaystyle \frac{\partial}{\partial y_{i}} \Bigl( \kappa_{\scriptscriptstyle 2,ij}(\boldsymbol{y}) 
\frac{\partial u_{\scriptscriptstyle 2}^{\scriptscriptstyle(0)}}{\partial y_{j}} \Bigr) = 0,
\end{cases}
\label{eq:order-2}
\end{equation}
\begin{equation}
O(\frac{1}{\varepsilon}):
\begin{cases}
\displaystyle \frac{\partial}{\partial y_{i}} \Bigl( \kappa_{\scriptscriptstyle 1,ij}(\boldsymbol{y}) 
\frac{\partial u_{\scriptscriptstyle 1}^{\scriptscriptstyle(0)}}{\partial x_{j}} \Bigr) 
+ \frac{\partial}{\partial y_{i}} \Bigl( \kappa_{\scriptscriptstyle 1,ij}(\boldsymbol{y}) 
\frac{\partial u_{\scriptscriptstyle 1}^{\scriptscriptstyle(1)}}{\partial y_{j}} \Bigr) \\[8pt]
+ {Q}(\boldsymbol{y}) \bigl( u_{\scriptscriptstyle 2}^{\scriptscriptstyle(0)} - u_{\scriptscriptstyle 1}^{\scriptscriptstyle(0)} \bigr) = 0, \\[2pt]
\displaystyle \frac{\partial}{\partial y_{i}} \Bigl( \kappa_{\scriptscriptstyle 2,ij}(\boldsymbol{y}) 
\frac{\partial u_{\scriptscriptstyle 2}^{\scriptscriptstyle(0)}}{\partial x_{j}} \Bigr) 
+ \frac{\partial}{\partial y_{i}} \Bigl( \kappa_{\scriptscriptstyle 2,ij}(\boldsymbol{y}) 
\frac{\partial u_{\scriptscriptstyle 2}^{\scriptscriptstyle(1)}}{\partial y_{j}} \Bigr) \\[8pt]
 + {Q}(\boldsymbol{y}) \bigl( u_{\scriptscriptstyle 1}^{\scriptscriptstyle(0)} - u_{\scriptscriptstyle 2}^{\scriptscriptstyle(0)} \bigr) = 0,
\end{cases}
\label{eq:order-1}
\end{equation}
\begin{equation}
O(\varepsilon^{0}):
\begin{cases}
\displaystyle c_{\scriptscriptstyle 1}(\boldsymbol{y}) \frac{\partial u_{\scriptscriptstyle 1}^{\scriptscriptstyle(0)}}{\partial t} 
= \frac{\partial}{\partial x_{i}} \Bigl( \kappa_{\scriptscriptstyle 1,ij}(\boldsymbol{y}) 
\frac{\partial u_{\scriptscriptstyle1}^{\scriptscriptstyle(0)}}{\partial x_{j}} \Bigr) +q\\
\displaystyle+ \frac{\partial}{\partial y_{i}} \Bigl( \kappa_{\scriptscriptstyle 1,ij}(\boldsymbol{y}) 
\frac{\partial u_{\scriptscriptstyle 1}^{\scriptscriptstyle(1)}}{\partial x_{j}} \Bigr) 
\! + \! \frac{\partial}{\partial x_{i}} \Bigl( \kappa_{\scriptscriptstyle 1,ij}(\boldsymbol{y}) 
\frac{\partial u_{\scriptscriptstyle 1}^{\scriptscriptstyle(1)}}{\partial y_{j}} \Bigr) \\[8pt]
\displaystyle + \frac{\partial}{\partial y_{i}} \Bigl( \kappa_{\scriptscriptstyle 1,ij}(\boldsymbol{y}) 
\frac{\partial u_{\scriptscriptstyle 1}^{\scriptscriptstyle(2)}}{\partial y_{j}} \Bigr) 
\displaystyle + Q_{\scriptscriptstyle 1}(\boldsymbol{y}) \bigl( u_{\scriptscriptstyle 2}^{\scriptscriptstyle(1)} - u_{\scriptscriptstyle 1}^{\scriptscriptstyle(1)} \bigr), \\[8pt]
\displaystyle c_{\scriptscriptstyle 2}(\boldsymbol{y}) \frac{\partial u_{\scriptscriptstyle 2}^{\scriptscriptstyle(0)}}{\partial t} 
= \frac{\partial}{\partial x_{i}} \Bigl( \kappa_{\scriptscriptstyle 2,ij}(\boldsymbol{y}) 
\frac{\partial u_{\scriptscriptstyle 2}^{\scriptscriptstyle(0)}}{\partial x_{j}} \Bigr)+q \\
\displaystyle + \frac{\partial}{\partial y_{i}} \Bigl( \kappa_{\scriptscriptstyle 2,ij}(\boldsymbol{y}) 
\frac{\partial u_{\scriptscriptstyle 2}^{\scriptscriptstyle(1)}}{\partial x_{j}} \Bigr) 
\! + \! \frac{\partial}{\partial x_{i}} \Bigl( \kappa_{\scriptscriptstyle 2,ij}(\boldsymbol{y}) 
\frac{\partial u_{\scriptscriptstyle 2}^{\scriptscriptstyle(1)}}{\partial y_{j}} \Bigr) \\[8pt]
\displaystyle + \frac{\partial}{\partial y_{i}} \Bigl( \kappa_{\scriptscriptstyle 2,ij}(\boldsymbol{y}) 
\frac{\partial u_{\scriptscriptstyle 2}^{\scriptscriptstyle(2)}}{\partial y_{j}} \Bigr)
\displaystyle + Q_{\scriptscriptstyle 2}(\boldsymbol{y}) \bigl( u_{\scriptscriptstyle 1}^{\scriptscriptstyle(1)} - u_{\scriptscriptstyle 2}^{\scriptscriptstyle(1)} \bigr).
\end{cases}
\label{eq:order0}
\end{equation}

Next, the aforementioned three sets of equations are analyzed sequentially. The specific expressions of each term in \eqref{eq:expansion} are solved recursively, and finally, the second-order two-scale expanded solution of problem~\eqref{eq:3} is assembled.

According to the classical asymptotic homogenization theory\cite{Bensoussan1978Asymptotic,2012Mathematical,1999An}, from Eqs.~\eqref{eq:order-2}, it is known that \(u_{\scriptscriptstyle l}^{\scriptscriptstyle(0)}\) is independent of the microscale \(\boldsymbol{y}\), i.e.,
\begin{equation}
u_{\scriptscriptstyle l}^{\scriptscriptstyle(0)}(\boldsymbol{x},\boldsymbol{y},t) = u_{\scriptscriptstyle l}^{\scriptscriptstyle(0)}(\boldsymbol{x},t).
\label{eq:u0_independent}
\end{equation}

Substituting Eq.~\eqref{eq:u0_independent} into Eqs.~\eqref{eq:order-1}, the expression for \(u_{\scriptscriptstyle l}^{\scriptscriptstyle(1)}(\boldsymbol{x},\boldsymbol{y},t)\) can be constructively given as follows:
\begin{equation}
\begin{cases}
\displaystyle u_{\scriptscriptstyle 1}^{\scriptscriptstyle(1)}(\boldsymbol{x},\boldsymbol{y},t) = N_{\scriptscriptstyle 1}^{\scriptscriptstyle\alpha_{1}}(\boldsymbol{y}) 
\frac{\partial u_{\scriptscriptstyle 1}^{\scriptscriptstyle(0)}}{\partial x_{\scriptscriptstyle\alpha_{1}}} + M_{\scriptscriptstyle 1}(\boldsymbol{y}) 
\bigl( u_{\scriptscriptstyle 2}^{\scriptscriptstyle(0)} - u_{\scriptscriptstyle 1}^{\scriptscriptstyle(0)} \bigr), \\[8pt]
\displaystyle u_{\scriptscriptstyle 2}^{\scriptscriptstyle(1)}(\boldsymbol{x},\boldsymbol{y},t) = N_{\scriptscriptstyle2}^{\scriptscriptstyle\alpha_{1}}(\boldsymbol{y}) 
\frac{\partial u_{\scriptscriptstyle2}^{\scriptscriptstyle(0)}}{\partial x_{\scriptscriptstyle\alpha_{1}}} + M_{\scriptscriptstyle 2}(\boldsymbol{y}) 
\bigl( u_{\scriptscriptstyle 1}^{\scriptscriptstyle(0)} - u_{\scriptscriptstyle 2}^{\scriptscriptstyle(0)} \bigr),
\end{cases}
\label{eq:u1_expression}
\end{equation}
where \(N_{\scriptscriptstyle l}^{\scriptscriptstyle\alpha_{1}}(\boldsymbol{y})\) and \(M_{\scriptscriptstyle l}(\boldsymbol{y})\) are first-order auxiliary cell functions. Substituting Eqs.~\eqref{eq:u0_independent} and~\eqref{eq:u1_expression} into Eqs.~\eqref{eq:order-1}, it can be obtained that \(N_{\scriptscriptstyle l}^{\scriptscriptstyle\alpha_{1}}(\boldsymbol{y})\) and \(M_{\scriptscriptstyle l}(\boldsymbol{y})\) satisfy the following cell problem:
\begin{equation}
\begin{cases}
\displaystyle \frac{\partial}{\partial y_{i}} \left( \kappa_{\scriptscriptstyle l,ij} 
\frac{\partial N_{\scriptscriptstyle l}^{\scriptscriptstyle\alpha_{1}}(\boldsymbol{y})}{\partial y_{j}} \right) 
= -\frac{\partial \kappa_{\scriptscriptstyle l,i\alpha_{1}}}{\partial y_{i}}, \quad \boldsymbol{y} \in Y, \\[8pt]
\displaystyle N_{\scriptscriptstyle l}^{\scriptscriptstyle\alpha_{1}}(\boldsymbol{y}) = 0, \quad \boldsymbol{y} \in \partial Y, \\
\end{cases}
\label{eq:N1_cell_problem}
\end{equation}
\begin{equation}
\begin{cases}
\displaystyle \frac{\partial}{\partial y_{i}} \Bigl( \kappa_{\scriptscriptstyle l,ij} 
\frac{\partial M_{\scriptscriptstyle l}(\boldsymbol{y})}{\partial y_{j}} \Bigr) 
= Q_{\scriptscriptstyle l}(\boldsymbol{y}), \quad \boldsymbol{y} \in Y, \\[8pt]
\displaystyle M_{\scriptscriptstyle l}(\boldsymbol{y}) = 0, \quad \boldsymbol{y} \in \partial Y.
\end{cases}
\label{eq:M_cell_problem2}
\end{equation}

Substituting Eqs.~\eqref{eq:u0_independent} and~\eqref{eq:u1_expression} into Eqs.~\eqref{eq:order0} and integrating over the unit cell \(Y\) yields the following homogenized problem:
\begin{equation}
\begin{cases}
\begin{aligned}
&c_{\scriptscriptstyle 1}^{\scriptscriptstyle\ast} \frac{\partial u_{\scriptscriptstyle 1}^{\scriptscriptstyle(0)}(\boldsymbol{x},t)}{\partial t} 
= \frac{\partial}{\partial x_{i}} \Bigl( \kappa_{\scriptscriptstyle 1,ij}^{\ast} 
\frac{\partial u_{\scriptscriptstyle 1}^{\scriptscriptstyle(0)}(\boldsymbol{x},t)}{\partial x_{j}} \Bigr) \\
&+ Q_{\scriptscriptstyle 1}^{\scriptscriptstyle*} \bigl( u_{\scriptscriptstyle 2}^{\scriptscriptstyle(0)}(\boldsymbol{x},t) - u_{\scriptscriptstyle 1}^{\scriptscriptstyle(0)}(\boldsymbol{x},t) \bigr) 
+ \bar{K}_{\scriptscriptstyle 1i}^{\scriptscriptstyle1\ast} \frac{\partial u_{\scriptscriptstyle 2}^{\scriptscriptstyle(0)}(\boldsymbol{x},t)}{\partial x_{i}} \\
&- \bar{K}_{\scriptscriptstyle 2i}^{\scriptscriptstyle1\ast} \frac{\partial u_{\scriptscriptstyle 1}^{\scriptscriptstyle(0)}(\boldsymbol{x},t)}{\partial x_{i}} 
+ q(\boldsymbol{x},t), \ (\boldsymbol{x},t) \in \Omega \times (0,T^{*}), \\
&c_{\scriptscriptstyle 2}^{\scriptscriptstyle\ast} \frac{\partial u_{\scriptscriptstyle 2}^{\scriptscriptstyle(0)}(\boldsymbol{x},t)}{\partial t} 
= \frac{\partial}{\partial x_{i}} \Bigl( \kappa_{\scriptscriptstyle 2,ij}^{\ast} 
\frac{\partial u_{\scriptscriptstyle 2}^{\scriptscriptstyle(0)}(\boldsymbol{x},t)}{\partial x_{j}} \Bigr) \\
&+ Q_{\scriptscriptstyle 2}^{\scriptscriptstyle\ast} \bigl( u_{\scriptscriptstyle 1}^{\scriptscriptstyle(0)}(\boldsymbol{x},t) - u_{\scriptscriptstyle 2}^{\scriptscriptstyle(0)}(\boldsymbol{x},t) \bigr) 
+ \bar{K}_{\scriptscriptstyle 2i}^{\scriptscriptstyle2\ast} \frac{\partial u_{\scriptscriptstyle 1}^{\scriptscriptstyle(0)}(\boldsymbol{x},t)}{\partial x_{i}} \\
&- \bar{K}_{\scriptscriptstyle 1i}^{\scriptscriptstyle 2\ast} \frac{\partial u_{\scriptscriptstyle 2}^{\scriptscriptstyle(0)}(\boldsymbol{x},t)}{\partial x_{i}} 
+ q(\boldsymbol{x},t), \ (\boldsymbol{x},t) \in \Omega \times (0,T^{\scriptscriptstyle *}), \\
&u_{\scriptscriptstyle 1}^{\scriptscriptstyle(0)}(\boldsymbol{x},t) \!= 0, \ u_{\scriptscriptstyle 2}^{\scriptscriptstyle(0)}(\boldsymbol{x},t) \! = 0, \
(\! \boldsymbol{x},t \!) \in \partial\Omega \times (0,T^{\scriptscriptstyle *}), \\
&u_{\scriptscriptstyle 1}^{\scriptscriptstyle(0)}(\boldsymbol{x},0) = g_{1}(\boldsymbol{x}), \ u_{\scriptscriptstyle 2}^{\scriptscriptstyle(0)}(\boldsymbol{x},0) = g_{2}(\boldsymbol{x}), \ \boldsymbol{x} \in \Omega,
\end{aligned}
\end{cases}
\label{eq:homogenized}
\end{equation}
where the homogenized material coefficients \(c_{\scriptscriptstyle l}^{\scriptscriptstyle*}\), \(\kappa_{\scriptscriptstyle l,ij}^{\scriptscriptstyle*}\), 
\(\bar{K}_{\scriptscriptstyle 1i}^{\scriptscriptstyle l\ast}\), \(\bar{K}_{\scriptscriptstyle 2i}^{\scriptscriptstyle l\ast}\) and \(Q_{\scriptscriptstyle l}^{\scriptscriptstyle\ast}\) are given as follows:

\begin{equation}
\begin{cases}
\begin{aligned}
&c_{\scriptscriptstyle l}^{\scriptscriptstyle\ast} = \int_{Y} c_{\scriptscriptstyle l}(\boldsymbol{y}) \, \mathrm{d}Y, \\
&\kappa_{\scriptscriptstyle l,ij}^{\scriptscriptstyle\ast} = \int_{Y} \left( \kappa_{\scriptscriptstyle l,ij}(\boldsymbol{y}) 
+ \kappa_{\scriptscriptstyle l,i\alpha_{1}}(\boldsymbol{y}) \frac{\partial N_{\scriptscriptstyle l}^{\scriptscriptstyle j}(\boldsymbol{y})}{\partial y_{\scriptscriptstyle\alpha_1}} \right) \mathrm{d}Y, \\
&\bar{K}_{\scriptscriptstyle 1i}^{\scriptscriptstyle l\ast} = \int_{Y} \left( \kappa_{\scriptscriptstyle l,ij}(\boldsymbol{y}) 
\frac{\partial M_{\scriptscriptstyle 1}(\boldsymbol{y})}{\partial y_{j}} 
+ Q_{\scriptscriptstyle l}(\boldsymbol{y}) N_{\scriptscriptstyle 2}^{\scriptscriptstyle i}(\boldsymbol{y}) \right) \mathrm{d}Y, \\
&\bar{K}_{\scriptscriptstyle 2i}^{\scriptscriptstyle l\ast} = \int_{Y} \left( \kappa_{\scriptscriptstyle l,ij}(\boldsymbol{y}) 
\frac{\partial M_{\scriptscriptstyle 1}(\boldsymbol{y})}{\partial y_{j}} 
+ Q_{\scriptscriptstyle l}(\boldsymbol{y}) N_{\scriptscriptstyle 1}^{\scriptscriptstyle i}(\boldsymbol{y}) \right) \mathrm{d}Y, \\
&Q_{\scriptscriptstyle l}^{\scriptscriptstyle\ast} = - \int_{Y} Q_{\scriptscriptstyle l}(\boldsymbol{y}) \left( M_{\scriptscriptstyle 1}(\boldsymbol{y}) 
+ M_{\scriptscriptstyle 2}(\boldsymbol{y}) \right) \mathrm{d}Y.
\end{aligned}
\end{cases}
\label{eq:coefficients}
\end{equation}

To further solve for \(u_{\scriptscriptstyle l}^{\scriptscriptstyle(2)}(\boldsymbol{x},\boldsymbol{y},t)\), subtract the homogenized equation in the homogenized problem~\eqref{eq:homogenized} from Eqs.~\eqref{eq:order0}, and then apply Eqs.~\eqref{eq:u0_independent} and~\eqref{eq:u1_expression} to obtain:
\begin{equation}
\begin{cases}
\begin{aligned}
&\displaystyle \frac{\partial}{\partial y_i} \Bigl(\kappa_{\scriptscriptstyle 1,ij}(\boldsymbol{y}) \frac{\partial u^{\scriptscriptstyle (2)}_{\scriptscriptstyle 1}}{\partial y_j} \Bigr) 
= \bigl( c_{\scriptscriptstyle 1}(\boldsymbol{y}) - c^{\scriptscriptstyle *}_{\scriptscriptstyle 1} \bigr) \frac{\partial u^{\scriptscriptstyle (0)}_{\scriptscriptstyle 1}}{\partial t} \\
&\displaystyle  + \bigl( Q_{\scriptscriptstyle 1}^{\scriptscriptstyle *} + Q_{\scriptscriptstyle 1}(\boldsymbol{y}) \left( M_{\scriptscriptstyle 1}(\boldsymbol{y}) + M_{\scriptscriptstyle 2}(\boldsymbol{y}) \right) \bigr) 
\bigl( u^{\scriptscriptstyle (0)}_{\scriptscriptstyle 2} - u^{\scriptscriptstyle (0)}_{\scriptscriptstyle 1} \bigr) \\
&\displaystyle  + \Bigl( \kappa^{\scriptscriptstyle *}_{\scriptscriptstyle 1,\alpha_1\alpha_2} - \kappa_{\scriptscriptstyle 1,\alpha_1\alpha_2}(\boldsymbol{y}) -
\frac{\partial \bigl( \kappa_{\scriptscriptstyle 1,i\alpha_1}(\boldsymbol{y}) N^{\scriptscriptstyle \alpha_{2}}_{\scriptscriptstyle 1}(\boldsymbol{y}) \bigr)}{\partial y_i} \\
&\displaystyle  - \kappa_{\scriptscriptstyle 1,a_{\scriptscriptstyle 1}j}(\boldsymbol{y}) \frac{\partial N^{\scriptscriptstyle \alpha_{2}}_{\scriptscriptstyle 1}(\boldsymbol{y})}{\partial y_j} \Bigr) 
\frac{\partial^2 u^{\scriptscriptstyle (0)}_{\scriptscriptstyle 1}}{\partial x_{\alpha_{\scriptscriptstyle 1}} \partial x_{\alpha_{\scriptscriptstyle 2}}} \\
&+ \Bigl(  \kappa_{\scriptscriptstyle 1,a_{\scriptscriptstyle 1}j}(\boldsymbol{y}) \frac{\partial M_{\scriptscriptstyle 1}(\boldsymbol{y})}{\partial y_j} 
+ \frac{\partial \bigl( \kappa_{\scriptscriptstyle 1,i\alpha_{\scriptscriptstyle 1}}(\boldsymbol{y}) M_{\scriptscriptstyle 1}(\boldsymbol{y}) \bigr)}{\partial y_i} \\
&+ Q_{\scriptscriptstyle 1}(\boldsymbol{y}) N^{\scriptscriptstyle \alpha_{\scriptscriptstyle 1}}_{\scriptscriptstyle 1}(\boldsymbol{y}) - \bar{K}^{\scriptscriptstyle 1*}_{\scriptscriptstyle 2\alpha_{\scriptscriptstyle 1}} \Bigr)
\frac{\partial u^{\scriptscriptstyle (0)}_{\scriptscriptstyle 1}}{\partial x_{\alpha_{\scriptscriptstyle 1}}} 
+\Bigl( \bar{K}^{\scriptscriptstyle 1*}_{\scriptscriptstyle 1\alpha_{\scriptscriptstyle i}} \\
&- \frac{\partial \bigl( \kappa_{\scriptscriptstyle 1,i\alpha_{\scriptscriptstyle 1}}(\boldsymbol{y}) M_{\scriptscriptstyle 1}(\boldsymbol{y}) \bigr)}{\partial y_i} 
- \kappa_{\scriptscriptstyle 1,\alpha_{\scriptscriptstyle 1}j}(\boldsymbol{y}) \frac{\partial M_{\scriptscriptstyle 1}(\boldsymbol{y})}{\partial y_j} \\
&- Q_{\scriptscriptstyle 1}(\boldsymbol{y}) N^{\scriptscriptstyle \alpha_{\scriptscriptstyle 1}}_{\scriptscriptstyle 2}(\boldsymbol{y}) \Bigr)
\frac{\partial u^{\scriptscriptstyle (0)}_{\scriptscriptstyle 2}}{\partial x_{a_{\scriptscriptstyle 1}}}, \\

&\displaystyle \frac{\partial}{\partial y_i} \Bigl( \kappa_{\scriptscriptstyle 2,ij}(\boldsymbol{y}) \frac{\partial u^{\scriptscriptstyle (2)}_{\scriptscriptstyle 2}}{\partial y_j} \Bigr) 
= \bigl( c_{\scriptscriptstyle 2}(\boldsymbol{y}) - c^{\scriptscriptstyle *}_{\scriptscriptstyle 2} \bigr) \frac{\partial u^{\scriptscriptstyle (0)}_{\scriptscriptstyle 2}}{\partial t} \\
&\displaystyle  + \bigl( Q_{\scriptscriptstyle 2}^{\scriptscriptstyle *} + Q_{\scriptscriptstyle 2}(\boldsymbol{y}) \left( M_{\scriptscriptstyle 1}(\boldsymbol{y}) + M_{\scriptscriptstyle 2}(\boldsymbol{y}) \right) \bigr) 
\bigl( u^{\scriptscriptstyle (0)}_{\scriptscriptstyle 1} - u^{\scriptscriptstyle (0)}_{\scriptscriptstyle 2} \bigr) \\
&\displaystyle  + \Bigl( \kappa^{\scriptscriptstyle *}_{\scriptscriptstyle 2,\alpha_1\alpha_2} - \kappa_{\scriptscriptstyle 2,\alpha_1\alpha_2}(\boldsymbol{y}) -
\frac{\partial \bigl( \kappa_{\scriptscriptstyle 2,i\alpha_1}(\boldsymbol{y}) N^{\scriptscriptstyle \alpha_{2}}_{\scriptscriptstyle 2}(\boldsymbol{y}) \bigr)}{\partial y_i} \\
&\displaystyle  - \kappa_{\scriptscriptstyle 2,\alpha_1j}(\boldsymbol{y})  \frac{\partial N^{\scriptscriptstyle \alpha_{2}}_{\scriptscriptstyle 2}(\boldsymbol{y})}{\partial y_j} \Bigr) 
\frac{\partial^2 u^{\scriptscriptstyle (0)}_{\scriptscriptstyle 2}}{\partial x_{\alpha_{\scriptscriptstyle 1}} \partial x_{\alpha_{\scriptscriptstyle 2}}} \\
& + \Bigl( \frac{\partial \bigl( \kappa_{\scriptscriptstyle 2,i\alpha_{\scriptscriptstyle 1}}(\boldsymbol{y}) M_{\scriptscriptstyle 2}(\boldsymbol{y}) \bigr)}{\partial y_i} 
+ \kappa_{\scriptscriptstyle 2,\alpha_{\scriptscriptstyle 1}j}(\boldsymbol{y}) \frac{\partial M_{\scriptscriptstyle 2}(\boldsymbol{y})}{\partial y_j} \\
&  + Q_{\scriptscriptstyle 2}(\boldsymbol{y}) N^{\scriptscriptstyle \alpha_{\scriptscriptstyle 1}}_{\scriptscriptstyle 2}(\boldsymbol{y}) - \bar{K}^{\scriptscriptstyle 2*}_{\scriptscriptstyle 1\alpha_{\scriptscriptstyle 1}} \Bigr) 
\frac{\partial u^{\scriptscriptstyle (0)}_{\scriptscriptstyle 2}}{\partial x_{\alpha_{\scriptscriptstyle 1}}} 
+ \Bigl( \bar{K}^{\scriptscriptstyle 2*}_{\scriptscriptstyle 2\alpha_{\scriptscriptstyle 1}} \\
&  - \frac{\partial \bigl( \kappa_{\scriptscriptstyle 2,i\alpha_{\scriptscriptstyle 1}}(\boldsymbol{y}) M_{\scriptscriptstyle 2}(\boldsymbol{y}) \bigr)}{\partial y_i} 
- \kappa_{\scriptscriptstyle 2,\alpha_{\scriptscriptstyle 1}j}(\boldsymbol{y}) \frac{\partial M_{\scriptscriptstyle 2}(\boldsymbol{y})}{\partial y_j} \\
&  - Q_{\scriptscriptstyle 2}(\boldsymbol{y}) N^{\scriptscriptstyle \alpha_{\scriptscriptstyle 1}}_{\scriptscriptstyle 1}(\boldsymbol{y}) \Bigr)
\frac{\partial u^{\scriptscriptstyle (0)}_{\scriptscriptstyle 1}}{\partial x_{\alpha_{\scriptscriptstyle 1}}},
\end{aligned}
\end{cases}
\label{eq:u2_equation}
\end{equation}

According to Eqs.~\eqref{eq:u2_equation}, \( u^{\scriptscriptstyle(2)}_{\scriptscriptstyle i}(\boldsymbol{x},\boldsymbol{y},t) \) is defined to have the following specific form:
\begin{equation}
\begin{cases}
\begin{aligned}
u_{\scriptscriptstyle 1}^{\scriptscriptstyle (2)}(\boldsymbol{x},\boldsymbol{y}, t) &= G_{\scriptscriptstyle 1}(\boldsymbol{y}) \frac{\partial u_{\scriptscriptstyle 1}^{\scriptscriptstyle (0)}}{\partial t} 
+ N_{\scriptscriptstyle 1}^{\scriptscriptstyle \alpha_{1}\alpha_{2}} (\boldsymbol{y}) \frac{\partial^2 u_{\scriptscriptstyle 1}^{\scriptscriptstyle (0)}}{\partial x_{\scriptscriptstyle \alpha_1} \partial x_{\scriptscriptstyle \alpha_2}} \\
& + C_{\scriptscriptstyle 1}^{\scriptscriptstyle \alpha_1} (\boldsymbol{y}) \frac{\partial u_{\scriptscriptstyle 1}^{\scriptscriptstyle (0)}}{\partial x_{\scriptscriptstyle \alpha_1}} 
+ F_{\scriptscriptstyle 1}^{\scriptscriptstyle \alpha_1} (\boldsymbol{y}) \frac{\partial u_{\scriptscriptstyle 2}^{\scriptscriptstyle (0)}}{\partial x_{\scriptscriptstyle \alpha_1}}\\
& + K_{\scriptscriptstyle 1} (\boldsymbol{y}) \bigl( u_{\scriptscriptstyle 2}^{\scriptscriptstyle (0)} - u_{\scriptscriptstyle 1}^{\scriptscriptstyle (0)} \bigr), \\

u_{\scriptscriptstyle 2}^{\scriptscriptstyle (2)}(\boldsymbol{x}, \boldsymbol{y}, t) &= G_{\scriptscriptstyle 2}(\boldsymbol{y}) \frac{\partial u_{\scriptscriptstyle 2}^{\scriptscriptstyle (0)}}{\partial t} 
+ N_{\scriptscriptstyle 2}^{\scriptscriptstyle \alpha_{1}\alpha_{2}} (\boldsymbol{y}) \frac{\partial^2 u_{\scriptscriptstyle 2}^{\scriptscriptstyle (0)}}{\partial x_{\scriptscriptstyle \alpha_1} \partial x_{\scriptscriptstyle \alpha_2}} \\
& + C_{\scriptscriptstyle 2}^{\scriptscriptstyle \alpha_2} (\boldsymbol{y}) \frac{\partial u_{\scriptscriptstyle 2}^{\scriptscriptstyle (0)}}{\partial x_{\scriptscriptstyle \alpha_1}} 
+ F_{\scriptscriptstyle 2}^{\scriptscriptstyle \alpha_2} (\boldsymbol{y}) \frac{\partial u_{\scriptscriptstyle 1}^{\scriptscriptstyle (0)}}{\partial x_{\scriptscriptstyle \alpha_1}} \\
& + K_{\scriptscriptstyle 2} (\boldsymbol{y}) \bigl( u_{\scriptscriptstyle 1}^{\scriptscriptstyle (0)} - u_{\scriptscriptstyle 2}^{\scriptscriptstyle (0)} \bigr),
\end{aligned}
\end{cases}
\label{eq:u2_form}
\end{equation}
where \( G_{\scriptscriptstyle l} (\boldsymbol{y}) \), \( N_{\scriptscriptstyle l}^{\scriptscriptstyle \alpha_{1}\alpha_{2}} (\boldsymbol{y}) \), \( C_{\scriptscriptstyle l}^{\scriptscriptstyle \alpha_1} (\boldsymbol{y}) \), \( F_{\scriptscriptstyle l}^{\scriptscriptstyle \alpha_1} (\boldsymbol{y}) \) and \( K_{\scriptscriptstyle l} (\boldsymbol{y}) \) are second-order auxiliary cell functions.

Substituting Eqs.~\eqref{eq:u2_form} into Eqs.~\eqref{eq:u2_equation}, it is obtained that \( G_{\scriptscriptstyle l} (\boldsymbol{y}) \), \( N_{\scriptscriptstyle l}^{\scriptscriptstyle \alpha_{1}\alpha_{2}} (\boldsymbol{y}) \), \( C_{\scriptscriptstyle l}^{\scriptscriptstyle \alpha_1} (\boldsymbol{y}) \), \( F_{\scriptscriptstyle l}^{\scriptscriptstyle \alpha_1} (\boldsymbol{y}) \) and \( K_{\scriptscriptstyle l} (\boldsymbol{y}) \) satisfy the following cell problems, respectively:
\begin{equation}
\begin{cases}
\displaystyle \frac{\partial}{\partial y_i} \Bigl( \kappa_{\scriptscriptstyle l,ij} (\boldsymbol{y}) \frac{\partial G_{\scriptscriptstyle l} (\boldsymbol{y})}{\partial y_j} \Bigr) 
= c_{\scriptscriptstyle l} (\boldsymbol{y}) - c_{\scriptscriptstyle l}^{\scriptscriptstyle *}, \quad \boldsymbol{y} \in Y, \\[8pt]
\displaystyle G_{\scriptscriptstyle l} (\boldsymbol{y}) = 0, \quad \boldsymbol{y} \in \partial Y,
\end{cases}
\label{eq:cell_problem_G}
\end{equation}
\begin{equation}
\begin{cases}
\displaystyle \frac{\partial}{\partial y_i} \Bigl( \kappa_{\scriptscriptstyle l,ij}(\boldsymbol{y}) \frac{\partial N_{\scriptscriptstyle l}^{\scriptscriptstyle \alpha_1 \alpha_2}(\boldsymbol{y})}{\partial y_j} \Bigr) 
= \kappa_{\scriptscriptstyle l,\alpha_1\alpha_2}^{\scriptscriptstyle *} 
- \kappa_{\scriptscriptstyle l,\alpha_1\alpha_2}(\boldsymbol{y}) \\
\displaystyle - \frac{\partial}{\partial y_i} \! \left( \! \kappa_{\scriptscriptstyle l,i\alpha_1} \! ( \boldsymbol{y}) N_{\scriptscriptstyle l}^{\scriptscriptstyle \alpha_2}\!(\boldsymbol{y})  \right) \! - \! \kappa_{\scriptscriptstyle l,\alpha_1j}(\boldsymbol{y}) \! \frac{\partial N_{\scriptscriptstyle l}^{\scriptscriptstyle \alpha_2} \!(\boldsymbol{y})}{\partial y_j}\! , \ \boldsymbol{y} \in Y, \\
\displaystyle N_{\scriptscriptstyle l}^{\scriptscriptstyle \alpha_1\alpha_2}(\boldsymbol{y}) = 0, \quad \boldsymbol{y} \in \partial Y,
\end{cases}
\label{eq:N_cell_problem}
\end{equation}
\begin{equation}
\begin{cases}
\displaystyle \frac{\partial}{\partial y_i} \Bigl( \kappa_{\scriptscriptstyle 1,ij}(\boldsymbol{y}) \frac{\partial C_{\scriptscriptstyle 1}^{\scriptscriptstyle \alpha_1}(\boldsymbol{y})}{\partial y_j} \Bigr) 
= \frac{\partial}{\partial y_i} \bigl( \kappa_{\scriptscriptstyle 1,i\alpha_1}(\boldsymbol{y}) M_{\scriptscriptstyle 1}(\boldsymbol{y}) \bigr) \\[8pt]
\displaystyle   
+ Q_{\scriptscriptstyle 1}(\boldsymbol{y}) N_{\scriptscriptstyle 1}^{\scriptscriptstyle \alpha_1}\!(\boldsymbol{y}) \!-\! \bar{K}_{\scriptscriptstyle 1\alpha_1}^{\scriptscriptstyle 1*}
 \!+\! \kappa_{\scriptscriptstyle 1,\alpha_1j}(\boldsymbol{y}) \frac{\partial M_{\scriptscriptstyle 1}(\boldsymbol{y})}{\partial y_j}, \ \boldsymbol{y} \in Y, \\
\displaystyle C_{\scriptscriptstyle 1}^{\scriptscriptstyle \alpha_1}(\boldsymbol{y}) = 0, \quad \boldsymbol{y} \in \partial Y,
\end{cases}
\label{eq:C1_cell_problem}
\end{equation}
\begin{equation}
\begin{cases}
\displaystyle \frac{\partial}{\partial y_i} \Bigl( \kappa_{\scriptscriptstyle 2,ij}(\boldsymbol{y}) \frac{\partial C_{\scriptscriptstyle 2}^{\scriptscriptstyle \alpha_1}(\boldsymbol{y})}{\partial y_j} \Bigr) 
= \frac{\partial}{\partial y_i} \bigl( \kappa_{\scriptscriptstyle 2,i\alpha_1}(\boldsymbol{y}) M_{\scriptscriptstyle 2}(\boldsymbol{y}) \bigr) \\[8pt]
\displaystyle  
+ Q_{\scriptscriptstyle 2}(\boldsymbol{y}) N_{\scriptscriptstyle 2}^{\scriptscriptstyle \alpha_1}\!(\boldsymbol{y}) \!-\! \bar{K}_{\scriptscriptstyle 1a_1}^{\scriptscriptstyle 2*} \!+\! \kappa_{\scriptscriptstyle 2,\alpha_1j}(\boldsymbol{y}) \frac{\partial M_{\scriptscriptstyle 2}(\boldsymbol{y})}{\partial y_j}, \ \boldsymbol{y} \in Y, \\
\displaystyle C_{\scriptscriptstyle 2}^{\scriptscriptstyle \alpha_1}(\boldsymbol{y}) = 0, \quad \boldsymbol{y} \in \partial Y,
\end{cases}
\label{eq:C2_cell_problem}
\end{equation}
\begin{equation}
\begin{cases}
\displaystyle \frac{\partial}{\partial y_i} \Bigl( \kappa_{\scriptscriptstyle {1,ij}}(\boldsymbol{y}) \frac{\partial F_{\scriptscriptstyle 1}^{\scriptscriptstyle \alpha_1}\!(\boldsymbol{y})}{\partial y_j} \Bigr) 
=-\frac{\partial}{\partial y_i} \left( \kappa_{\scriptscriptstyle 1,i\alpha_1}(\boldsymbol{y}) M_{\scriptscriptstyle 1}(\boldsymbol{y}) \right)\\
+\bar{K}_{\scriptscriptstyle {1\alpha_1}}^{\scriptscriptstyle 1*} 
\displaystyle \!-\! \kappa_{\scriptscriptstyle {1,\alpha_1j}}(\boldsymbol{y}) \frac{\partial M_{\scriptscriptstyle1} \!(\boldsymbol{y})}{\partial y_j}
   -  Q_{\scriptscriptstyle 1}(\boldsymbol{y}) N_{\scriptscriptstyle 2}^{\scriptscriptstyle{\alpha_1}}(\boldsymbol{y}),\ \boldsymbol{y} \in Y, \\[8pt]
\displaystyle F_{\scriptscriptstyle 1}^{\scriptscriptstyle \alpha_1}(\boldsymbol{y}) = 0, \quad \boldsymbol{y} \in \partial Y,
\end{cases}
\label{eq:F1_cell_problem}
\end{equation}
\begin{equation}
\begin{cases}
\displaystyle \frac{\partial}{\partial y_i} \Bigl( \kappa_{\scriptscriptstyle 2,ij}(\boldsymbol{y}) \frac{\partial F_{\scriptscriptstyle 2}^{\scriptscriptstyle \alpha_1}(\boldsymbol{y})}{\partial y_j} \Bigr) 
= -\frac{\partial}{\partial y_i} \bigl( \! \kappa_{\scriptscriptstyle 2,i\alpha_1}\!(\boldsymbol{y}) M_{\scriptscriptstyle 2}(\boldsymbol{y}) \bigr) \\
\displaystyle+\bar{K}_{\scriptscriptstyle 2\alpha_1}^{\scriptscriptstyle 2*} \! -\!  Q_{\scriptscriptstyle 2}(\boldsymbol{y}) N_{\scriptscriptstyle 1}^{\scriptscriptstyle \alpha_1}(\boldsymbol{y})
\displaystyle \!-\! \kappa_{\scriptscriptstyle 2,\alpha_1j} (\boldsymbol{y}) \frac{\partial M_{\scriptscriptstyle 2}(\boldsymbol{y})}{\partial y_j} \!
, \ \boldsymbol{y} \in Y, \\[8pt]
\displaystyle F_{\scriptscriptstyle 2}^{\scriptscriptstyle \alpha_1}(\boldsymbol{y}) = 0, \quad \boldsymbol{y} \in \partial Y,
\end{cases}
\label{eq:F2_cell_problem}
\end{equation}
\begin{equation}
\begin{cases} 
\displaystyle \frac{\partial}{\partial y_i} \left( \kappa_{\scriptscriptstyle l,ij}(\boldsymbol{y}) \frac{\partial K_{\scriptscriptstyle l}(\boldsymbol{y})}{\partial y_j} \right) 
= Q_{\scriptscriptstyle l}^{\scriptscriptstyle *}\\[8pt]
 + Q_{\scriptscriptstyle l}(\boldsymbol{y}) \left( M_{\scriptscriptstyle 1}(\boldsymbol{y}) + M_{\scriptscriptstyle 2}(\boldsymbol{y}) \right), \quad \boldsymbol{y} \in Y, \\[4pt]
\displaystyle K_{\scriptscriptstyle l}(\boldsymbol{y}) = 0, \quad \boldsymbol{y} \in \partial Y. 
\end{cases}
\label{eq:K_cell_problem}
\end{equation}

\begin{remark}
Classical auxiliary cell functions are defined with periodic boundary conditions on the boundary of the unit cell Y. However, it should be emphasized that all auxiliary cell functions in this paper adopt homogeneous Dirichlet boundary conditions. There are two main reasons for replacing the periodic boundary condition with the homogeneous Dirichlet boundary condition.
First, existing studies have proved that the homogeneous Dirichlet boundary condition is a suitable boundary condition to replace the periodic boundary condition; second, the homogeneous Dirichlet boundary condition makes the practical numerical computation of the auxiliary cell problems more convenient~\cite{wang_multiscale_2015,dong_multiscale_2018,cao_multiscale_2006,dong_second-order_2018,dong_second-order_2018-1,dong_multiscale_2018-1}.
\label{rem:existence_uniqueness}
\end{remark}

In summary, the expressions for the FOMS asymptotic solution \( u_{\scriptscriptstyle l}^{\scriptscriptstyle(1\varepsilon)}(\boldsymbol{x},t) \) and the HOMS asymptotic solution \( u_{\scriptscriptstyle l}^{\scriptscriptstyle(2\varepsilon)}(\boldsymbol{x},t) \) of problem ~\eqref{eq:3} are given as follows, respectively:
\begin{equation}
\begin{cases} 
\begin{aligned}
u_{\scriptscriptstyle 1}^{\scriptscriptstyle(1\varepsilon)} (\boldsymbol{x},t) &=  u_{\scriptscriptstyle 1}^{\scriptscriptstyle(0)} + \varepsilon \! \Bigl[  N_{\scriptscriptstyle 1}^{\scriptscriptstyle\alpha_{\scriptscriptstyle 1}} \!(\boldsymbol{y})  \frac{\partial u_{\scriptscriptstyle 1}^{\scriptscriptstyle(0)}}{\partial x_{\scriptscriptstyle \alpha_{\scriptscriptstyle 1}}} \\
 &+   M_{\scriptscriptstyle {\scriptscriptstyle 1}}(\boldsymbol{y}) ( u_{\scriptscriptstyle 2}^{\scriptscriptstyle(0)} - u_{\scriptscriptstyle 1}^{\scriptscriptstyle (0)}  )  \Bigr]  ,\\

u_{\scriptscriptstyle 2}^{\scriptscriptstyle(1\varepsilon)} \!(\boldsymbol{x},t) &= u_{\scriptscriptstyle 2}^{\scriptscriptstyle(0)} + \varepsilon  \Bigl[  N_{\scriptscriptstyle 2}^{\scriptscriptstyle\alpha_{\scriptscriptstyle 1}}(\boldsymbol{y})  \frac{\partial u_{\scriptscriptstyle 2}^{\scriptscriptstyle(0)}}{\partial x_{\scriptscriptstyle\alpha_{\scriptscriptstyle 1}}} \\
\displaystyle &+ M_{\scriptscriptstyle 2}(\boldsymbol{y}) (u_{\scriptscriptstyle 1}^{\scriptscriptstyle(0)} - u_{\scriptscriptstyle 2}^{\scriptscriptstyle(0)}    ) \Bigr],
\end{aligned}
\end{cases}
\label{eq:first_order_solution}
\end{equation}
\begin{equation}
\begin{cases} 
\begin{aligned}
&u_{\scriptscriptstyle 1}^{\scriptscriptstyle (2\varepsilon)}(\boldsymbol{x},t) \!= \!u_{\scriptscriptstyle 1}^{\!\scriptscriptstyle (0)}
\!+\! \varepsilon \Bigl[ N_{\scriptscriptstyle 1}^{\scriptscriptstyle \alpha_1}\!(\boldsymbol{y}) \frac{\partial u_{\scriptscriptstyle 1}^{\! \scriptscriptstyle (0)}}{\partial x_{\scriptscriptstyle \alpha}} \!+\! M_{\scriptscriptstyle 1}\!(\boldsymbol{y}) (u_{\scriptscriptstyle 2}^{\scriptscriptstyle (0)} \!-\! u_{\scriptscriptstyle 1}^{\scriptscriptstyle (0)})  \Bigr] \\
&\!+\! \varepsilon^{\scriptscriptstyle 2}\! \Bigl[ G_{\scriptscriptstyle 1}\!(\boldsymbol{y}) \!\frac{\partial u_{\scriptscriptstyle 1}^{\! \scriptscriptstyle (0)}}{\partial t} \! + \! N_{\scriptscriptstyle 1}^{\scriptscriptstyle \alpha_{\scriptscriptstyle 1}\!\alpha_{\scriptscriptstyle 2}}\! (\boldsymbol{y}) \!\frac{\partial^{\scriptscriptstyle 2} u_{\scriptscriptstyle 1}^{\!\scriptscriptstyle (0)}}{\partial x_{\scriptscriptstyle \alpha_{\scriptscriptstyle 1}} \partial x_{\scriptscriptstyle \alpha_{\scriptscriptstyle 2}}} \!+\! C_{\scriptscriptstyle 1}^{\scriptscriptstyle \alpha_{\scriptscriptstyle 1}}\!(\boldsymbol{y}) \!\frac{\partial u_{\scriptscriptstyle 1}^{\! \scriptscriptstyle (0)}}{\partial x_{\scriptscriptstyle \alpha_{\scriptscriptstyle 1}}}  \\
&\!  + F_{\scriptscriptstyle 1}^{\scriptscriptstyle \alpha_1}(\boldsymbol{y}) \frac{\partial u_{\scriptscriptstyle 2}^{\scriptscriptstyle (0)}}{\partial x_{\scriptscriptstyle \alpha_{1}}} + K_{\scriptscriptstyle 1}(\boldsymbol{y}) (u_{\scriptscriptstyle 2}^{\scriptscriptstyle (0)} - u_{\scriptscriptstyle 1}^{\scriptscriptstyle (0)}) \Bigr], 
\end{aligned} \\
\begin{aligned}
&u_{\scriptscriptstyle 2}^{\scriptscriptstyle (2\varepsilon)}(\boldsymbol{x},t) \! =\!u_{\scriptscriptstyle 2}^{\!\scriptscriptstyle (0)} \!+\! \varepsilon \!\Bigl[  N_{\scriptscriptstyle 2}^{\scriptscriptstyle \alpha_{\scriptscriptstyle 1}}\!(\boldsymbol{y}) \! \frac{\partial u_{\scriptscriptstyle 2}^{\scriptscriptstyle (0)}}{\partial x_{\scriptscriptstyle \alpha_1}} \!+\! M_{\scriptscriptstyle 2}\!(\boldsymbol{y}) (u_{\scriptscriptstyle 1}^{\scriptscriptstyle (0)} \!-\! u_{\!\scriptscriptstyle 2}^{\!\scriptscriptstyle (0)}) \Bigr] \\
 &\!+ \!\varepsilon^{\scriptscriptstyle 2}\! \Bigl[ G_{\scriptscriptstyle 2}\!(\boldsymbol{y})\! \frac{\partial u_{\scriptscriptstyle 2}^{\!\scriptscriptstyle (0)}}{\partial t} \!+\! N_{\scriptscriptstyle 2}^{\scriptscriptstyle \alpha_{\scriptscriptstyle 1}\!\alpha_{\scriptscriptstyle 2}}\!(\boldsymbol{y}) \!\frac{\partial^{\scriptscriptstyle 2} u_{\scriptscriptstyle 2}^{\!\scriptscriptstyle (0)}}{\partial x_{\scriptscriptstyle \alpha_{\scriptscriptstyle 1}} \partial x_{\scriptscriptstyle \alpha_{\scriptscriptstyle 2}}} \!+\! C_{\scriptscriptstyle 2}^{\scriptscriptstyle \alpha_1}\!(\boldsymbol{y}) \!\frac{\partial u_{\scriptscriptstyle 2}^{\!\scriptscriptstyle (0)}}{\partial x_{\scriptscriptstyle \alpha_{\scriptscriptstyle 1}}} \\
 & \! +\! F_{\scriptscriptstyle 2}^{\scriptscriptstyle \alpha_1}(\boldsymbol{y}) \frac{\partial u_{\scriptscriptstyle 1}^{\scriptscriptstyle (0)}}{\partial x_{\scriptscriptstyle \alpha_{1}}} + K_{\scriptscriptstyle 2}(\boldsymbol{y}) (u_{\scriptscriptstyle 1}^{\scriptscriptstyle (0)} - u_{\scriptscriptstyle 2}^{\scriptscriptstyle (0)}) \Bigr].
\end{aligned}
\end{cases}
\label{eq:second_order_solution}
\end{equation}

\begin{theorem}
The multi-continuum problem~\eqref{eq:3} in periodic composite structures possesses the following HOMS asymptotic solution:
\end{theorem}
\begin{equation}
\begin{cases} 
\begin{aligned}
&u_{\scriptscriptstyle 1}^{\scriptscriptstyle \varepsilon}(\boldsymbol{x},t) \!\approx \!u_{\scriptscriptstyle 1}^{\!\scriptscriptstyle (0)}
\!+\! \varepsilon \Bigl[ N_{\scriptscriptstyle 1}^{\scriptscriptstyle \alpha_1}\!(\boldsymbol{y}) \frac{\partial u_{\scriptscriptstyle 1}^{\! \scriptscriptstyle (0)}}{\partial x_{\scriptscriptstyle \alpha}} \!+\! M_{\scriptscriptstyle 1}\!(\boldsymbol{y}) (u_{\scriptscriptstyle 2}^{\scriptscriptstyle (0)} \!-\! u_{\scriptscriptstyle 1}^{\scriptscriptstyle (0)})  \Bigr] \\
&\!+\! \varepsilon^{\scriptscriptstyle 2}\! \Bigl[ G_{\scriptscriptstyle 1}\!(\boldsymbol{y}) \!\frac{\partial u_{\scriptscriptstyle 1}^{\! \scriptscriptstyle (0)}}{\partial t} \! + \! N_{\scriptscriptstyle 1}^{\scriptscriptstyle \alpha_{\scriptscriptstyle 1}\!\alpha_{\scriptscriptstyle 2}}\! (\boldsymbol{y}) \!\frac{\partial^{\scriptscriptstyle 2} u_{\scriptscriptstyle 1}^{\!\scriptscriptstyle (0)}}{\partial x_{\scriptscriptstyle \alpha_{\scriptscriptstyle 1}} \partial x_{\scriptscriptstyle \alpha_{\scriptscriptstyle 2}}} \!+\! C_{\scriptscriptstyle 1}^{\scriptscriptstyle \alpha_{\scriptscriptstyle 1}}\!(\boldsymbol{y}) \!\frac{\partial u_{\scriptscriptstyle 1}^{\! \scriptscriptstyle (0)}}{\partial x_{\scriptscriptstyle \alpha_{\scriptscriptstyle 1}}}  \\
&\!  + F_{\scriptscriptstyle 1}^{\scriptscriptstyle \alpha_1}(\boldsymbol{y}) \frac{\partial u_{\scriptscriptstyle 2}^{\scriptscriptstyle (0)}}{\partial x_{\scriptscriptstyle \alpha_{1}}} + K_{\scriptscriptstyle 1}(\boldsymbol{y}) (u_{\scriptscriptstyle 2}^{\scriptscriptstyle (0)} - u_{\scriptscriptstyle 1}^{\scriptscriptstyle (0)}) \Bigr], 
\end{aligned} \\
\begin{aligned}
&u_{\scriptscriptstyle 2}^{\scriptscriptstyle \varepsilon}(\boldsymbol{x},t) \! \approx \!u_{\scriptscriptstyle 2}^{\!\scriptscriptstyle (0)} \!+\! \varepsilon \!\Bigl[  N_{\scriptscriptstyle 2}^{\scriptscriptstyle \alpha_{\scriptscriptstyle 1}}\!(\boldsymbol{y}) \! \frac{\partial u_{\scriptscriptstyle 2}^{\scriptscriptstyle (0)}}{\partial x_{\scriptscriptstyle \alpha_1}} \!+\! M_{\scriptscriptstyle 2}\!(\boldsymbol{y}) (u_{\scriptscriptstyle 1}^{\scriptscriptstyle (0)} \!-\! u_{\!\scriptscriptstyle 2}^{\!\scriptscriptstyle (0)}) \Bigr] \\
 &\!+ \!\varepsilon^{\scriptscriptstyle 2}\! \Bigl[ G_{\scriptscriptstyle 2}\!(\boldsymbol{y})\! \frac{\partial u_{\scriptscriptstyle 2}^{\!\scriptscriptstyle (0)}}{\partial t} \!+\! N_{\scriptscriptstyle 2}^{\scriptscriptstyle \alpha_{\scriptscriptstyle 1}\!\alpha_{\scriptscriptstyle 2}}\!(\boldsymbol{y}) \!\frac{\partial^{\scriptscriptstyle 2} u_{\scriptscriptstyle 2}^{\!\scriptscriptstyle (0)}}{\partial x_{\scriptscriptstyle \alpha_{\scriptscriptstyle 1}} \partial x_{\scriptscriptstyle \alpha_{\scriptscriptstyle 2}}} \!+\! C_{\scriptscriptstyle 2}^{\scriptscriptstyle \alpha_1}\!(\boldsymbol{y}) \!\frac{\partial u_{\scriptscriptstyle 2}^{\!\scriptscriptstyle (0)}}{\partial x_{\scriptscriptstyle \alpha_{\scriptscriptstyle 1}}} \\
 & \! +\! F_{\scriptscriptstyle 2}^{\scriptscriptstyle \alpha_1}(\boldsymbol{y}) \frac{\partial u_{\scriptscriptstyle 1}^{\scriptscriptstyle (0)}}{\partial x_{\scriptscriptstyle \alpha_{1}}} + K_{\scriptscriptstyle 2}(\boldsymbol{y}) (u_{\scriptscriptstyle 1}^{\scriptscriptstyle (0)} - u_{\scriptscriptstyle 2}^{\scriptscriptstyle (0)}) \Bigr].
\end{aligned}
\end{cases}
\label{eq:final_solution}
\end{equation}

\section{\label{sec:level3}Error Analysis of High-Order Multiscale Solution}
This chapter analyzes the pointwise approximation properties of the first-order multi-scale (FOMS) and HOMS asymptotic solutions derived in Section~\ref{sec:2.1} with respect to the original equation. Under certain assumptions, it also proves the convergence order of the second-order two-scale asymptotic solution in the integral sense.
\subsection{Pointwise Error Analysis}
To compare the FOMS and HOMS solutions with the true solution of problem~\eqref{eq:3}, we assume the domain \(\Omega\) is composed of perfectly periodic unit cells, i.e., 
\(\overline{\Omega} = \bigcup_{\boldsymbol{z} \in T_{\varepsilon}} \varepsilon (\boldsymbol{z} + \overline{Y})\), 
where \(T_{\varepsilon} = \{ \boldsymbol{z} = (z_{1}, \cdots, z_{n}) \in \mathbb{Z}^{n}, \; \varepsilon (\boldsymbol{z} + \overline{Y}) \subset \overline{\Omega} \}\). 
Then, the error functions for the FOMS solution \(u_{\scriptscriptstyle l}^{\scriptscriptstyle(1\varepsilon)}(\boldsymbol{x},t)\) and the HOMS solution \(u_{\scriptscriptstyle l}^{\scriptscriptstyle(2\varepsilon)}(\boldsymbol{x},t)\) are defined as follows:
\begin{equation}
\begin{aligned}
u_{\scriptscriptstyle 1\Delta}^{\scriptscriptstyle (1\varepsilon)}(\boldsymbol{x},t) &= u_{\scriptscriptstyle 1}^{\scriptscriptstyle \varepsilon}(\boldsymbol{x},t) - u_{\scriptscriptstyle 1}^{\scriptscriptstyle (1\varepsilon)}(\boldsymbol{x},t), \\
u_{\scriptscriptstyle 2\Delta}^{\scriptscriptstyle (1\varepsilon)}(\boldsymbol{x},t) &= u_{\scriptscriptstyle 2}^{\scriptscriptstyle \varepsilon}(\boldsymbol{x},t) - u_{\scriptscriptstyle 2}^{\scriptscriptstyle (1\varepsilon)}(\boldsymbol{x},t),
\end{aligned}
\label{eq:error_first_order}
\end{equation}
\begin{equation}
\begin{aligned}
u_{\scriptscriptstyle 1\Delta}^{\scriptscriptstyle (2\varepsilon)}(\boldsymbol{x},t) &= u_{\scriptscriptstyle 1}^{\scriptscriptstyle \varepsilon}(\boldsymbol{x},t) - u_{\scriptscriptstyle 1}^{\scriptscriptstyle (2\varepsilon)}(\boldsymbol{x},t), \\
u_{\scriptscriptstyle 2\Delta}^{\scriptscriptstyle (2\varepsilon)}(\boldsymbol{x},t) &= u_{\scriptscriptstyle 2}^{\scriptscriptstyle \varepsilon}(\boldsymbol{x},t) - u_{\scriptscriptstyle 2}^{\scriptscriptstyle (2\varepsilon)}(\boldsymbol{x},t).
\end{aligned}
\label{eq:error_second_order}
\end{equation}

Substituting the error function \(u_{\scriptscriptstyle l\Delta}^{\scriptscriptstyle(1\varepsilon)}(\boldsymbol{x},t)\) into problem~\eqref{eq:3}, it can be obtained after calculation and simplification that \(u_{\scriptscriptstyle l\Delta}^{\scriptscriptstyle(1\varepsilon)}(\boldsymbol{x},t)\) satisfies the following residual problem:
\begin{equation}
\begin{cases}
\begin{aligned}
&c_{\scriptscriptstyle 1}^{\scriptscriptstyle \varepsilon}(\boldsymbol{x}) \frac{\partial u_{\scriptscriptstyle 1\Delta}^{\scriptscriptstyle (1\varepsilon)}(\boldsymbol{x},t)}{\partial t} 
- \frac{\partial}{\partial x_{i}} \Bigl( \kappa_{\scriptscriptstyle 1,ij}^{\scriptscriptstyle \varepsilon}(\boldsymbol{x}) 
\frac{\partial u_{\scriptscriptstyle 1\Delta}^{\scriptscriptstyle (1\varepsilon)}(\boldsymbol{x},t)}{\partial x_{j}} \Bigr) \\
&\displaystyle = f_{\scriptscriptstyle 10}(\boldsymbol{x},\boldsymbol{y},t) \!+\! \varepsilon \tilde{f}_{\scriptscriptstyle 1}(\boldsymbol{x},\boldsymbol{y},t), \ (\boldsymbol{x},t) \in \Omega \times (0,T^{\scriptscriptstyle *}\!),
\end{aligned} \\[18pt]
\begin{aligned}
&c_{\scriptscriptstyle 2}^{\scriptscriptstyle \varepsilon}(\boldsymbol{x}) \frac{\partial u_{\scriptscriptstyle 2\Delta}^{\scriptscriptstyle (1\varepsilon)}(\boldsymbol{x},t)}{\partial t} 
- \frac{\partial}{\partial x_{i}} \Bigl( \kappa_{\scriptscriptstyle 2,ij}^{\scriptscriptstyle \varepsilon}(\boldsymbol{x}) 
\frac{\partial u_{2\Delta}^{\scriptscriptstyle (1\varepsilon)}(\boldsymbol{x},t)}{\partial x_{j}} \Bigr) \\
&= f_{\scriptscriptstyle 20}(\boldsymbol{x},\boldsymbol{y},t) + \varepsilon \tilde{f}_{\scriptscriptstyle 2}(\boldsymbol{x},\boldsymbol{y},t), \ (\boldsymbol{x},t) \in \Omega \times (0,T^{\scriptscriptstyle *}),
\end{aligned} \\[16pt]
u_{\scriptscriptstyle 1\Delta}^{\scriptscriptstyle (1\varepsilon)}\!(\boldsymbol{x},t) \!=\! 0, \ u_{\scriptscriptstyle 2\Delta}^{\scriptscriptstyle (1\varepsilon)}\!(\boldsymbol{x},t) \!=\! 0, \ (\boldsymbol{x},t) \! \in \! \partial\Omega \times (0,T^{\scriptscriptstyle *}\!)\!, \\
u_{\scriptscriptstyle 1\Delta}^{\scriptscriptstyle (1\varepsilon)} \!(\boldsymbol{x},0) \!=\! \varepsilon \bar{\psi}_{\scriptscriptstyle 1}(\boldsymbol{x}), \ u_{\scriptscriptstyle 2\Delta}^{\scriptscriptstyle (1\varepsilon)} \!(\boldsymbol{x},0) \!=\! \varepsilon \bar{\psi}_{\scriptscriptstyle 2}(\boldsymbol{x}), \ \boldsymbol{x} \in \Omega,
\end{cases}
\label{eq:residual_first_order}
\end{equation}
Among them, the interaction term in the residual problem
\(\frac{1}{\displaystyle \varepsilon} Q^{\scriptscriptstyle \varepsilon}(\boldsymbol{x})(u_{\scriptscriptstyle 2}^{\scriptscriptstyle \varepsilon}-u_{\scriptscriptstyle 1}^{\scriptscriptstyle \varepsilon})\) contains terms of order $O(\varepsilon^{\scriptscriptstyle 2})$ after performing the multiscale expansion of $u_{\scriptscriptstyle l}^{\scriptscriptstyle \varepsilon}$. Therefore, these terms can be incorporated into $\tilde{f}_{\scriptscriptstyle l}(\boldsymbol{x}, \boldsymbol{y}, t)$. The specific forms of \(f_{\scriptscriptstyle l0}(\boldsymbol{x},t)\), \(\tilde{f}_{\scriptscriptstyle l}(\boldsymbol{x},t)\) and \(\bar{\psi}_{\scriptscriptstyle l}(\boldsymbol{x})\)in Problem~\eqref{eq:residual_first_order} are provided in Appendix~\ref{app}.

Similarly, substituting the error function \(u_{\scriptscriptstyle l\Delta}^{\scriptscriptstyle (2\varepsilon)}(\boldsymbol{x},t)\) into problem~\eqref{eq:3}, it can be obtained after calculation and simplification that \(u_{\scriptscriptstyle l\Delta}^{\scriptscriptstyle (2\varepsilon)}(\boldsymbol{x},t)\) satisfies the following residual problem:
\begin{equation}
\begin{cases}
\begin{aligned}
&c_{\scriptscriptstyle 1}^{\scriptscriptstyle \varepsilon}(\boldsymbol{x}) \frac{\partial u_{\scriptscriptstyle 1\Delta}^{\scriptscriptstyle (2\varepsilon)}(\boldsymbol{x},t)}{\partial t} 
- \frac{\partial}{\partial x_i} \Bigl( \kappa_{\scriptscriptstyle 1,ij}^{\scriptscriptstyle \varepsilon}(\boldsymbol{x}) 
\frac{\partial u_{\scriptscriptstyle 1\Delta}^{\scriptscriptstyle (2\varepsilon)}(\boldsymbol{x},t)}{\partial x_j} \Bigr) \\
&= \varepsilon f_{\scriptscriptstyle 1}(\boldsymbol{x}, \boldsymbol{y}, t), \quad (\boldsymbol{x}, t) \in \Omega \times (0, T^{\scriptscriptstyle *}),
\end{aligned} \\[18pt]
\begin{aligned}
&c_{\scriptscriptstyle 2}^{\scriptscriptstyle \varepsilon}(\boldsymbol{x}) \frac{\partial u_{\scriptscriptstyle 2\Delta}^{\scriptscriptstyle (2\varepsilon)}(\boldsymbol{x},t)}{\partial t} 
- \frac{\partial}{\partial x_i} \Bigl( \kappa_{\scriptscriptstyle 2,ij}^{\scriptscriptstyle \varepsilon}(\boldsymbol{x}) 
\frac{\partial u_{\scriptscriptstyle 2\Delta}^{\scriptscriptstyle (2\varepsilon)}(\boldsymbol{x},t)}{\partial x_j} \Bigr) \\
&= \varepsilon f_{\scriptscriptstyle 2}(\boldsymbol{x}, \boldsymbol{y}, t), \quad (\boldsymbol{x}, t) \in \Omega \times (0, T^*),
\end{aligned} \\[16pt]
u_{\scriptscriptstyle 1\Delta}^{\scriptscriptstyle (2\varepsilon)} \!(\boldsymbol{x}, t) \!=\! 0, 
\ u_{\scriptscriptstyle 2\Delta}^{\scriptscriptstyle (2\varepsilon)} \!(\boldsymbol{x}, t) \!=\! 0, \ (\boldsymbol{x}, t) \!\in\! \partial \Omega \times (0, T^{\scriptscriptstyle *} \!)\!, \\
u_{\scriptscriptstyle 1\Delta}^{\scriptscriptstyle (2\varepsilon)} \!(\boldsymbol{x}, 0) \!=\! \varepsilon \psi_{\scriptscriptstyle 1}(\boldsymbol{x}), \ u_{\scriptscriptstyle 2\Delta}^{\scriptscriptstyle (2\varepsilon)} \!(\boldsymbol{x}, 0) \!=\! \varepsilon \psi_{\scriptscriptstyle 2}(\boldsymbol{x}), \ \boldsymbol{x} \in \Omega,
\end{cases}
\label{eq:residual_second_order}
\end{equation}
Among them, the interaction term in the residual problem
\(\frac{1}{\displaystyle \varepsilon} Q^{\scriptscriptstyle \varepsilon}(\boldsymbol{x})(u_{\scriptscriptstyle 2}^{\scriptscriptstyle \varepsilon}-u_{\scriptscriptstyle 1}^{\scriptscriptstyle \varepsilon})\) contains terms of order $O(\varepsilon^{\scriptscriptstyle 3})$ after performing the multiscale expansion of $u_{\scriptscriptstyle l}^{\scriptscriptstyle \varepsilon}$. Therefore, these terms can be incorporated into ${f}_{\scriptscriptstyle l}(\boldsymbol{x}, \boldsymbol{y}, t)$. The specific forms of \({f}_{\scriptscriptstyle l}(\boldsymbol{x},t)\) and \(\psi_{\scriptscriptstyle l}(\boldsymbol{x})\)in Problem~\eqref{eq:residual_second_order} are provided in Appendix~\ref{app}.
%


In practical engineering applications, the small periodic parameter \(\varepsilon\) is often a fixed constant. Through the above analysis, it can be concluded that, the approximation accuracy of the FOMS solution \(u_{\scriptscriptstyle l}^{\scriptscriptstyle (1\varepsilon)}(\boldsymbol{x}, t)\) to the true solution \(u_{\scriptscriptstyle l}^{\varepsilon}(\boldsymbol{x}, t)\) of the original problem~\eqref{eq:3} is of order \(O(1)\), which is clearly insufficient. In contrast, the approximation accuracy of the HOMS solution \(u_{\scriptscriptstyle l}^{\scriptscriptstyle (2\varepsilon)}(\boldsymbol{x}, t)\) is of order \(O(\varepsilon)\), which not only provides higher numerical accuracy but also precisely captures the oscillatory information at the microscale of highly heterogeneous media.

\subsection{Error Analysis of High-Order Multiscale Solutions in the Integral Sense}
Next, we proceed with the error analysis of HOMS solutions in the integral sense. First, the following assumptions are made:
\begin{description}
    \item[$(B_{1})$] $c_{\scriptscriptstyle l}(\boldsymbol{y}) \in L^{\infty}(\overline{\Omega}), \quad \kappa_{\scriptscriptstyle l,ij}(\boldsymbol{y}) \in L^{\infty}(\overline{\Omega})$.   
    \item[$(B_{2})$] Both $c_{\scriptscriptstyle l}(\boldsymbol{y})$ and $\kappa_{\scriptscriptstyle l,ij}(\boldsymbol{y})$ are symmetric with respect to all mid-planes of the reference unit cell $Y$, while $\kappa_{\scriptscriptstyle l,ij}(\boldsymbol{y})$ is skew-symmetric with respect to all mid-planes of the reference unit cell $Y$\cite{cao_asymptotic_2004,feng_multi-scale_2004,wang_multiscale_2015} .
\end{description}
\begin{lemma}\label{lemma:cell_functions_continuity}
    If the multi-continuum problem~\eqref{eq:3} satisfies assumptions $(A_1)$-$(A_3)$ and $(B_1)$-$(B_2)$, then the normal derivatives of all cell functions $\sigma_{\scriptscriptstyle CY}^{\scriptscriptstyle l}(N_{\scriptscriptstyle l}^{\scriptscriptstyle \alpha_1})$, $\sigma_{\scriptscriptstyle CY}^{\scriptscriptstyle l}(M_{\scriptscriptstyle l})$, $\sigma_{\scriptscriptstyle CY}^{\scriptscriptstyle l}(G_{\scriptscriptstyle l})$, $\sigma_{\scriptscriptstyle CY}^{\scriptscriptstyle l}(N_{\scriptscriptstyle l}^{\scriptscriptstyle \alpha_1\alpha_2})$, $\sigma_{\scriptscriptstyle CY}^{\scriptscriptstyle l}(C_{\scriptscriptstyle l}^{\scriptscriptstyle \alpha_1})$, $\sigma_{\scriptscriptstyle CY}^{\scriptscriptstyle l}(F_{\scriptscriptstyle l}^{\scriptscriptstyle \alpha_1})$, $\sigma_{\scriptscriptstyle CY}^{\scriptscriptstyle l}(K_{\scriptscriptstyle l})$ are continuous on the boundary $\partial Y$ of the microscale unit cell domain $Y$, where $\sigma_{\scriptscriptstyle CY}^{\scriptscriptstyle l} = n_{\scriptscriptstyle i} \kappa_{\scriptscriptstyle l,ij} \frac{\partial \Phi}{\partial y_j}$.
\end{lemma}
\begin{theorem}\label{thm:error_estimate}
    Let \( u_{\scriptscriptstyle l}^{\varepsilon}(x,t) \) be the solution of problem \eqref{eq:3}, and \( u_{\scriptscriptstyle l}^{\scriptscriptstyle(0)}(x,t) \) be the solution of the homogenized problem~\eqref{eq:homogenized}. Suppose assumptions \((A_1)\)-\((A_3)\) and \((B_1)\)-\((B_2)\) hold, and that \( u_{\scriptscriptstyle l}^{\scriptscriptstyle (0)}(\boldsymbol{x},t) \in L^\infty\left(0,T^*;H^{\scriptscriptstyle 4}(\Omega)\right) \) and \( \frac{\partial u_{\scriptscriptstyle l}^{\scriptscriptstyle(0)}}{\partial t} \in L^\infty\left(0,T^*;H^{\scriptscriptstyle 2}(\Omega)\right) \) are satisfied. Then the following error estimate holds:
\begin{equation}
\begin{aligned}
&\bigl\|u_{\scriptscriptstyle 1\Delta}^{\scriptscriptstyle (2\varepsilon)}(x,t)\bigr\|_{\scriptscriptstyle L^\infty\left(0,T^*,L^2(\Omega) \right)} 
+ \bigl\|u_{\scriptscriptstyle 1\Delta}^{\scriptscriptstyle (2\varepsilon)}(x,t)\bigr\|_{\scriptscriptstyle L^\infty\left(0,T^*,H^1(\Omega) \right)} \\
&+ \bigl\|u_{\scriptscriptstyle 2\Delta}^{\scriptscriptstyle (2\varepsilon)}(x,t)\bigr\|_{\scriptscriptstyle L^\infty\left(0,T^*,L^2(\Omega) \right)} 
+ \bigl\|u_{\scriptscriptstyle 2\Delta}^{\scriptscriptstyle (2\varepsilon)}(x,t)\bigr\|_{\scriptscriptstyle L^\infty\left(0,T^*,H^1(\Omega) \right)} \\
&\leq C(T^*) \varepsilon,
\end{aligned}
\label{eq:error_estimate}
\end{equation}
where \( C(T^{*}) \) is a constant independent of the small periodic parameter \( \varepsilon \).
\end{theorem}
\begin{proof}
First, by combining the chain rule~\eqref{eq:chainrule} with Eqs.~\eqref{eq:homogenized}, we obtain:
\begin{equation}
\begin{cases}
\begin{aligned}
&\sigma_{\scriptscriptstyle Y}^{\scriptscriptstyle 1}(u^{\scriptscriptstyle (2\varepsilon)}_{\scriptscriptstyle 1}) 
= n_{\scriptscriptstyle i}\kappa_{\scriptscriptstyle 1,ij}(\boldsymbol{y})\frac{\partial u^{\scriptscriptstyle (2\varepsilon)}_{\scriptscriptstyle 1}}{\partial x_{\scriptscriptstyle j}} \\
& = n_{\scriptscriptstyle i}\kappa_{\scriptscriptstyle 1,ij}\Big{(}\frac{\partial}{\partial x_{\scriptscriptstyle j}} + \frac{1}{\varepsilon}\frac{\partial}{\partial y_{\scriptscriptstyle j}}\Big{)} 
\bigg{[}u^{\scriptscriptstyle (0)}_{\scriptscriptstyle 1} + \varepsilon \Big{(}N^{\scriptscriptstyle \alpha_{1}}_{\scriptscriptstyle 1}\frac{\partial u^{\scriptscriptstyle (0)}_{\scriptscriptstyle 1}}{\partial x_{\scriptscriptstyle \alpha_{1}}} \\
& + M_{\scriptscriptstyle 1}\big{(}u^{\scriptscriptstyle (0)}_{\scriptscriptstyle 2} \! - u^{\scriptscriptstyle (0)}_{\scriptscriptstyle 1}\big{)} \! \Big{)} 
+ \varepsilon^{2}\Big{(} \! G_{\scriptscriptstyle 1}\frac{\partial u^{\scriptscriptstyle (0)}_{\scriptscriptstyle 1}}{\partial t}+ K_{\scriptscriptstyle 1}\big{(}u^{\scriptscriptstyle (0)}_{\scriptscriptstyle 2} - u^{\scriptscriptstyle (0)}_{\scriptscriptstyle 1}\big{)} \\
& + N^{\scriptscriptstyle \alpha_{1}\alpha_{2}}_{\scriptscriptstyle 1}\frac{\partial^{2}u^{\scriptscriptstyle (0)}_{\scriptscriptstyle 1}}{\partial x_{\scriptscriptstyle \alpha_{1}}\partial x_{\scriptscriptstyle \alpha_{2}}} 
+ C^{\scriptscriptstyle \alpha_{1}}_{\scriptscriptstyle 1}\frac{\partial u^{\scriptscriptstyle (0)}_{\scriptscriptstyle 1}}{\partial x_{\scriptscriptstyle \alpha_{1}}} 
+ F^{\scriptscriptstyle \alpha_{1}}_{\scriptscriptstyle 1}\frac{\partial u^{\scriptscriptstyle (0)}_{\scriptscriptstyle 2}}{\partial x_{\scriptscriptstyle \alpha_{1}}} 
\Big{)}\bigg{]} \\
& = n_{\scriptscriptstyle i}\kappa_{\scriptscriptstyle 1,ij} \! \bigg{[}\frac{\partial u^{\scriptscriptstyle(0)}_{\scriptscriptstyle 1}}{\partial x_{\scriptscriptstyle j}} \!
+\! \varepsilon\frac{\partial}{\partial x_{\scriptscriptstyle j}} \Big{(} \! N^{\scriptscriptstyle \alpha_{1}}_{\scriptscriptstyle 1}\frac{\partial u^{\scriptscriptstyle (0)}_{\scriptscriptstyle 1}}{\partial x_{\scriptscriptstyle \alpha_{1}}} \!
+ \! M_{\scriptscriptstyle 1} \! \big{(}  \! u^{\scriptscriptstyle(0)}_{\scriptscriptstyle 2} \! -  \! u^{\scriptscriptstyle(0)}_{\scriptscriptstyle 1} \! \big{)} \! \Big{)} \\
& + \varepsilon^{2}\frac{\partial}{\partial x_{\scriptscriptstyle j}}\Big{(}G_{\scriptscriptstyle 1}\frac{\partial u^{\scriptscriptstyle(0)}_{\scriptscriptstyle 1}}{\partial t} 
+ N^{\scriptscriptstyle \alpha_{1}\alpha_{2}}_{\scriptscriptstyle 1}\frac{\partial^{2}u^{\scriptscriptstyle (0)}_{\scriptscriptstyle 1}}{\partial x_{\scriptscriptstyle\alpha_{1}}\partial x_{\scriptscriptstyle \alpha_{2}}} 
+ C^{\scriptscriptstyle \alpha_{1}}_{\scriptscriptstyle 1}\frac{\partial u^{\scriptscriptstyle (0)}_{\scriptscriptstyle 1}}{\partial x_{\scriptscriptstyle\alpha_{1}}} \\
& + F^{\scriptscriptstyle\alpha_{1}}_{\scriptscriptstyle 1}\frac{\partial u^{\scriptscriptstyle (0)}_{\scriptscriptstyle 1}}{\partial x_{\scriptscriptstyle\alpha_{1}}} 
+ K_{\scriptscriptstyle 1}\big{(}u^{\scriptscriptstyle (0)}_{\scriptscriptstyle 2} - u^{\scriptscriptstyle (0)}_{\scriptscriptstyle 1}\big{)}\Big{)}\bigg{]} 
 + \sigma^{\scriptscriptstyle 1}_{\scriptscriptstyle CY}(N^{\scriptscriptstyle \alpha_{1}}_{\scriptscriptstyle 1})\frac{\partial u^{\scriptscriptstyle (0)}_{\scriptscriptstyle 1}}{\partial x_{\scriptscriptstyle\alpha_{1}}} \\
& + \sigma^{\scriptscriptstyle 1}_{\scriptscriptstyle CY}(M_{\scriptscriptstyle 1})\big{(}u^{\scriptscriptstyle (0)}_{\scriptscriptstyle 2} - u^{\scriptscriptstyle (0)}_{\scriptscriptstyle 1}\big{)} 
 + \varepsilon\Big{(}\sigma^{\scriptscriptstyle 1}_{\scriptscriptstyle CY}(N^{\scriptscriptstyle\alpha_{1}\alpha_{2}}_{\scriptscriptstyle 1})\frac{\partial^{2}u^{\scriptscriptstyle(0)}_{\scriptscriptstyle 1}}{\partial x_{\scriptscriptstyle\alpha_{1}}\partial x_{\scriptscriptstyle\alpha_{2}}} \\
&+ \sigma^{\scriptscriptstyle 1}_{\scriptscriptstyle CY}(C^{\scriptscriptstyle\alpha_{1}}_{\scriptscriptstyle 1})\frac{\partial u^{\scriptscriptstyle(0)}_{\scriptscriptstyle 1}}{\partial x_{\scriptscriptstyle\alpha_{1}}} 
+ \sigma^{\scriptscriptstyle 1}_{\scriptscriptstyle CY}(F^{\scriptscriptstyle\alpha_{1}}_{\scriptscriptstyle 1})\frac{\partial u^{\scriptscriptstyle(0)}_{\scriptscriptstyle 1}}{\partial x_{\scriptscriptstyle\alpha_{1}}} \\
&\ + \sigma^{\scriptscriptstyle 1}_{\scriptscriptstyle CY}(G_{\scriptscriptstyle 1})\frac{\partial u^{\scriptscriptstyle(0)}_{\scriptscriptstyle 1}}{\partial t} 
+ \sigma^{\scriptscriptstyle 1}_{\scriptscriptstyle CY}(K_{\scriptscriptstyle 1})\big{(}u^{\scriptscriptstyle(0)}_{\scriptscriptstyle 2} - u^{\scriptscriptstyle(0)}_{\scriptscriptstyle 1}\big{)}\Big{)}, \\

&\sigma_{\scriptscriptstyle Y}^{\scriptscriptstyle 2}(u^{\scriptscriptstyle (2\varepsilon)}_{\scriptscriptstyle 2}) 
= n_{\scriptscriptstyle i}\kappa_{\scriptscriptstyle 2,ij}(\boldsymbol{y})\frac{\partial u^{\scriptscriptstyle (2\varepsilon)}_{\scriptscriptstyle 2}}{\partial x_{\scriptscriptstyle j}} \\
& = n_{\scriptscriptstyle i}\kappa_{\scriptscriptstyle 2,ij}\Big{(}\frac{\partial}{\partial x_{\scriptscriptstyle j}} + \frac{1}{\varepsilon}\frac{\partial}{\partial y_{\scriptscriptstyle j}}\Big{)} 
\bigg{[}u^{\scriptscriptstyle (0)}_{\scriptscriptstyle 2} + \varepsilon\Big{(}N^{\scriptscriptstyle \alpha_{1}}_{\scriptscriptstyle 1}\frac{\partial u^{\scriptscriptstyle (0)}_{\scriptscriptstyle 2}}{\partial x_{\scriptscriptstyle \alpha_{1}}} \\
& + M_{\scriptscriptstyle 2}\big{(}u^{\scriptscriptstyle(0)}_{\scriptscriptstyle 1} - u^{\scriptscriptstyle (0)}_{\scriptscriptstyle 2}\big{)} \! \Big{)} 
+ \varepsilon^{2}\Big{(} \! G_{\scriptscriptstyle 2}\frac{\partial u^{\scriptscriptstyle (0)}_{\scriptscriptstyle 1}}{\partial t}+\! K_{\scriptscriptstyle 1}\big{(}u^{\scriptscriptstyle(0)}_{\scriptscriptstyle 1} - u^{\scriptscriptstyle (0)}_{\scriptscriptstyle 2}\big{)} \\
& + N^{\scriptscriptstyle \alpha_{1}\alpha_{2}}_{\scriptscriptstyle 1}\frac{\partial^{2}u^{\scriptscriptstyle (0)}_{\scriptscriptstyle 2}}{\partial x_{\scriptscriptstyle\alpha_{1}}\partial x_{\scriptscriptstyle\alpha_{2}}} 
+ C^{\scriptscriptstyle \alpha_{1}}_{\scriptscriptstyle 2}\frac{\partial u^{\scriptscriptstyle (0)}_{\scriptscriptstyle 2}}{\partial x_{\scriptscriptstyle\alpha_{1}}} 
+ F^{\scriptscriptstyle \alpha_{1}}_{\scriptscriptstyle 2}\frac{\partial u^{\scriptscriptstyle (0)}_{\scriptscriptstyle 1}}{\partial x_{\scriptscriptstyle\alpha_{1}}} 
\Big{)}\bigg{]} \\
& = n_{\scriptscriptstyle i}\kappa_{\scriptscriptstyle 2,ij} \! \bigg{[} \! \frac{\partial u^{\scriptscriptstyle (0)}_{\scriptscriptstyle 2}}{\partial x_{\scriptscriptstyle j}} \!
+ \varepsilon\frac{\partial}{\partial x_{\scriptscriptstyle j}} \! \Big{(} \! N^{\scriptscriptstyle \alpha_{1}}_{\scriptscriptstyle 2}\frac{\partial u^{\scriptscriptstyle (0)}_{\scriptscriptstyle 2}}{\partial x_{\scriptscriptstyle\alpha_{1}}} \!
+ \! M_{\scriptscriptstyle 2}\big{(} \! u^{\scriptscriptstyle (0)}_{\scriptscriptstyle 1} \! - \! u^{\scriptscriptstyle (0)}_{\scriptscriptstyle 2} \! \big{)} \! \Big{)} \\
& + \varepsilon^{2}\frac{\partial}{\partial x_{\scriptscriptstyle j}}\Big{(}G_{\scriptscriptstyle 2}\frac{\partial u^{\scriptscriptstyle(0)}_{\scriptscriptstyle 2}}{\partial t} 
+ N^{\scriptscriptstyle \alpha_{1}\alpha_{2}}_{\scriptscriptstyle 1}\frac{\partial^{2}u^{\scriptscriptstyle (0)}_{\scriptscriptstyle 2}}{\partial x_{\scriptscriptstyle\alpha_{1}}\partial x_{\scriptscriptstyle\alpha_{2}}} 
+ C^{\scriptscriptstyle \alpha_{1}}_{\scriptscriptstyle 2}\frac{\partial u^{\scriptscriptstyle (0)}_{\scriptscriptstyle 2}}{\partial x_{\scriptscriptstyle\alpha_{1}}} \\
& + F^{\scriptscriptstyle \alpha_{1}}_{\scriptscriptstyle 2}\frac{\partial u^{\scriptscriptstyle (0)}_{\scriptscriptstyle 1}}{\partial x_{\scriptscriptstyle\alpha_{1}}} 
+ K_{\scriptscriptstyle 2}\big{(}u^{\scriptscriptstyle (0)}_{\scriptscriptstyle 1} - u^{\scriptscriptstyle (0)}_{\scriptscriptstyle 2}\big{)}\Big{)}\bigg{]} 
 + \sigma^{\scriptscriptstyle 2}_{\scriptscriptstyle CY}(N^{\scriptscriptstyle\alpha_{1}}_{\scriptscriptstyle 2})\frac{\partial u^{\scriptscriptstyle (0)}_{\scriptscriptstyle 2}}{\partial x_{\scriptscriptstyle\alpha_{1}}} \\
& + \sigma^{\scriptscriptstyle 2}_{\scriptscriptstyle CY}(M_{\scriptscriptstyle 2})\big{(}u^{\scriptscriptstyle (0)}_{\scriptscriptstyle 1} - u^{\scriptscriptstyle (0)}_{\scriptscriptstyle 2}\big{)} 
 + \varepsilon\Big{(}\sigma^{\scriptscriptstyle 2}_{\scriptscriptstyle CY}(N^{\scriptscriptstyle\alpha_{1}\alpha_{2}}_{\scriptscriptstyle 2})\frac{\partial^{2}u^{\scriptscriptstyle (0)}_{\scriptscriptstyle 2}}{\partial x_{\scriptscriptstyle\alpha_{1}}\partial x_{\scriptscriptstyle\alpha_{2}}} \\
&+ \sigma^{\scriptscriptstyle 2}_{\scriptscriptstyle CY}(C^{\scriptscriptstyle\alpha_{1}}_{\scriptscriptstyle 2})\frac{\partial u^{\scriptscriptstyle (0)}_{\scriptscriptstyle 2}}{\partial x_{\scriptscriptstyle\alpha_{1}}} 
+ \sigma^{\scriptscriptstyle 2}_{\scriptscriptstyle CY}(F^{\scriptscriptstyle\alpha_{1}}_{\scriptscriptstyle 2})\frac{\partial u^{\scriptscriptstyle (0)}_{\scriptscriptstyle 1}}{\partial x_{\scriptscriptstyle\alpha_{1}}} \\
& + \sigma^{\scriptscriptstyle 2}_{\scriptscriptstyle CY}(G_{\scriptscriptstyle 2})\frac{\partial u^{\scriptscriptstyle (0)}_{\scriptscriptstyle 2}}{\partial t} 
+ \sigma^{\scriptscriptstyle 2}_{\scriptscriptstyle CY}(K_{\scriptscriptstyle 2})\big{(}u^{\scriptscriptstyle (0)}_{\scriptscriptstyle 1} - u^{\scriptscriptstyle (0)}_{\scriptscriptstyle 2}\big{)}\Big{)}.
\end{aligned}
\end{cases}
\label{eq:sigma_expressions}
\end{equation}

Next, multiply both sides of the residual equation by \(u_{\scriptscriptstyle l\Delta}^{\scriptscriptstyle(2\varepsilon)}\) and integrate over \(\Omega\) to obtain:
\begin{equation}
\begin{cases}
\begin{aligned}
&- \int_{\Omega} 
\frac{\partial}{\partial x_{i}}\Bigl(\kappa_{\scriptscriptstyle 1,ij}(\boldsymbol{y})\frac{\partial u_{\scriptscriptstyle 1\Delta}^{\scriptscriptstyle(2\varepsilon)}}{\partial x_{j}}\Bigr) 
u_{\scriptscriptstyle 1\Delta}^{\scriptscriptstyle(2\varepsilon)} \mathrm{d}\Omega \\
&+\int_{\Omega} c_{\scriptscriptstyle 1}(\boldsymbol{y})\frac{\partial u_{\scriptscriptstyle 1\Delta}^{\scriptscriptstyle(2\varepsilon)}}{\partial t} u_{\scriptscriptstyle 1\Delta}^{\scriptscriptstyle(2\varepsilon)} \mathrm{d}\Omega
= \int_{\Omega} \varepsilon f_{1} u_{\scriptscriptstyle 1\Delta}^{\scriptscriptstyle(2\varepsilon)} \mathrm{d}\Omega, \\
&-\int_{\Omega}  \frac{\partial}{\partial x_{i}}\Bigl(\kappa_{\scriptscriptstyle 2,ij}(\boldsymbol{y})\frac{\partial u_{\scriptscriptstyle 2\Delta}^{\scriptscriptstyle(2\varepsilon)}}{\partial x_{j}}\Bigr) 
u_{\scriptscriptstyle 2\Delta}^{\scriptscriptstyle(2\varepsilon)} \mathrm{d}\Omega \\
&+\int_{\Omega} c_{\scriptscriptstyle 2}(\boldsymbol{y})\frac{\partial u_{\scriptscriptstyle 2\Delta}^{\scriptscriptstyle(2\varepsilon)}}{\partial t} 
u_{\scriptscriptstyle 2\Delta}^{\scriptscriptstyle(2\varepsilon)} \mathrm{d}\Omega 
= \int_{\Omega} \varepsilon f_{2} u_{\scriptscriptstyle 2\Delta}^{\scriptscriptstyle(2\varepsilon)} \mathrm{d}\Omega,
\end{aligned}
\end{cases}
\label{eq:integral_form}
\end{equation}

Using the formula to simplify Eqs.~\eqref{eq:integral_form}, we obtain:
\begin{equation}
\begin{cases}
\begin{aligned}
&\int_{\Omega} \! c_{\scriptscriptstyle 1} \! (\boldsymbol{y}) \! \frac{\partial u_{\scriptscriptstyle 1\Delta}^{\scriptscriptstyle (2\varepsilon)}}{\partial t} u_{\scriptscriptstyle 1\Delta}^{\scriptscriptstyle (2\varepsilon)} \mathrm{d}\Omega 
+ \! \int_{\Omega} \kappa_{\scriptscriptstyle 1,ij}(\! \boldsymbol{y}\!) \! \frac{\partial u_{\scriptscriptstyle 1\Delta}^{\scriptscriptstyle (2\varepsilon)}}{\partial x_{\scriptscriptstyle j}} \frac{\partial u_{\scriptscriptstyle 1\Delta}^{\scriptscriptstyle (2\varepsilon)}}{\partial x_{\scriptscriptstyle i}} \mathrm{d}\Omega \\
&= \int_{\Omega} \varepsilon f_{\scriptscriptstyle 1} u_{\scriptscriptstyle 1\Delta}^{\scriptscriptstyle (2\varepsilon)} \mathrm{d}\Omega 
+ \int_{\bigcup_{\boldsymbol{z}\in T_{\boldsymbol{z}}}\partial E_{\boldsymbol{z}}} r_{\scriptscriptstyle 1} u_{\scriptscriptstyle 1\Delta}^{\scriptscriptstyle (2\varepsilon)} \mathrm{d}\Gamma_{\scriptscriptstyle\boldsymbol{y}}, \\

&\int_{\Omega} \! c_{\scriptscriptstyle 2} \!(\boldsymbol{y}) \! \frac{\partial u_{\scriptscriptstyle 2\Delta}^{\scriptscriptstyle (2\varepsilon)}}{\partial t} u_{\scriptscriptstyle 2\Delta}^{\scriptscriptstyle (2\varepsilon)} \mathrm{d}\Omega 
+ \! \int_{\Omega} \kappa_{\scriptscriptstyle 2,ij}(\! \boldsymbol{y}\!) \! \frac{\partial u_{\scriptscriptstyle 2\Delta}^{\scriptscriptstyle (2\varepsilon)}}{\partial x_{\scriptscriptstyle j}} \frac{\partial u_{\scriptscriptstyle 2\Delta}^{\scriptscriptstyle (2\varepsilon)}}{\partial x_{\scriptscriptstyle i}} \mathrm{d}\Omega \\
&= \int_{\Omega} \varepsilon f_{\scriptscriptstyle 2} u_{\scriptscriptstyle 2\Delta}^{\scriptscriptstyle (2\varepsilon)} \mathrm{d}\Omega 
+ \int_{\bigcup_{\boldsymbol{z}\in T_{\boldsymbol{z}}}\partial E_{\boldsymbol{z}}} r_{\scriptscriptstyle 2} u_{\scriptscriptstyle 2\Delta}^{\scriptscriptstyle (2\varepsilon)} \mathrm{d}\Gamma_{\scriptscriptstyle\boldsymbol{y}},
\end{aligned}
\end{cases}
\label{eq:simplified_form}
\end{equation}
where \(r_{\scriptscriptstyle l} = n_{\scriptscriptstyle i} \kappa_{\scriptscriptstyle l,ij}(y) \frac{\partial u_{\scriptscriptstyle l\Delta}^{\scriptscriptstyle (2\varepsilon)}}{\partial x_{j}} = \sigma_{\scriptscriptstyle Y}(u_{\scriptscriptstyle l\Delta}^{\scriptscriptstyle (2\varepsilon)})\).

Then, combining Eqs.~\eqref{eq:simplified_form} and the H\"older inequality, we calculate to obtain:
\begin{equation}
\begin{cases}
\begin{aligned}
\left\langle r_{\scriptscriptstyle 1}, u^{\scriptscriptstyle (2\varepsilon)}_{\scriptscriptstyle 1\Delta} \right\rangle 
&= \int_{\bigcup_{\scriptscriptstyle \boldsymbol{z} \in T_{\scriptscriptstyle \varepsilon}} \partial E_{\scriptscriptstyle \varepsilon}} r_{\scriptscriptstyle 1} u^{\scriptscriptstyle (2\varepsilon)}_{\scriptscriptstyle 1\Delta} \mathrm{d}\Gamma_{\scriptscriptstyle \boldsymbol{y}} \\
&= \sum_{\scriptscriptstyle z \in T_{\scriptscriptstyle \varepsilon}} \int_{\partial E_{\scriptscriptstyle \varepsilon}} \sigma_{\scriptscriptstyle Y} \bigl( u^{\scriptscriptstyle \varepsilon}_{\scriptscriptstyle 1} - u^{\scriptscriptstyle (2\varepsilon)}_{\scriptscriptstyle 1} \bigr) u^{\scriptscriptstyle (2\varepsilon)}_{\scriptscriptstyle 1\Delta} \mathrm{d}\Gamma_{\scriptscriptstyle y} \\
&= -\sum_{\scriptscriptstyle z \in T_{\scriptscriptstyle \varepsilon}} \int_{\partial E_{\scriptscriptstyle \varepsilon}} \sigma_{\scriptscriptstyle Y} \bigl( u^{\scriptscriptstyle (2\varepsilon)}_{\scriptscriptstyle 1} \bigr) u^{\scriptscriptstyle (2\varepsilon)}_{\scriptscriptstyle 1\Delta} \mathrm{d}\Gamma_{\scriptscriptstyle y} = 0, \\

\left\langle r_{\scriptscriptstyle 2}, u^{\scriptscriptstyle (2\varepsilon)}_{\scriptscriptstyle 2\Delta} \right\rangle 
&= \int_{\bigcup_{\scriptscriptstyle \boldsymbol{z} \in T_{\scriptscriptstyle \varepsilon}} \partial E_{\scriptscriptstyle \varepsilon}} r_{\scriptscriptstyle 2} u^{\scriptscriptstyle (2\varepsilon)}_{\scriptscriptstyle 2\Delta} \mathrm{d}\Gamma_{\scriptscriptstyle \boldsymbol{y}} \\
&= \sum_{\scriptscriptstyle z \in T_{\scriptscriptstyle \varepsilon}} \int_{\partial E_{\scriptscriptstyle \varepsilon}} \sigma_{\scriptscriptstyle Y} \bigl( u^{\scriptscriptstyle \varepsilon}_{\scriptscriptstyle 2} - u^{\scriptscriptstyle (2\varepsilon)}_{\scriptscriptstyle 2} \bigr) u^{\scriptscriptstyle (2\varepsilon)}_{\scriptscriptstyle 2\Delta} \mathrm{d}\Gamma_{\scriptscriptstyle y} \\
&= -\sum_{\scriptscriptstyle z \in T_{\scriptscriptstyle \varepsilon}} \int_{\partial E_{\scriptscriptstyle \varepsilon}} \sigma_{\scriptscriptstyle Y} \bigl( u^{\scriptscriptstyle (2\varepsilon)}_{\scriptscriptstyle 2} \bigr) u^{\scriptscriptstyle (2\varepsilon)}_{\scriptscriptstyle 2\Delta} \mathrm{d}\Gamma_{\scriptscriptstyle y} = 0.
\end{aligned}
\end{cases}
\label{eq:I_results}
\end{equation}

Further, combining Eqs.~\eqref{eq:simplified_form} and Eqs.~\eqref{eq:I_results}, it is straightforward to verify that the following equation holds:
\begin{equation}
\begin{cases}
\begin{aligned}
&\int_{\Omega} c_{\scriptscriptstyle 1}(\boldsymbol{y}) \frac{\partial u^{\scriptscriptstyle (2\varepsilon)}_{\scriptscriptstyle 1\Delta}}{\partial t} u^{\scriptscriptstyle (2\varepsilon)}_{\scriptscriptstyle 1\Delta} \mathrm{d}\Omega \\
&+\int_{\Omega} \kappa_{\scriptscriptstyle 1,ij}(\boldsymbol{y}) \frac{\partial u^{\scriptscriptstyle (2\varepsilon)}_{\scriptscriptstyle 1\Delta}}{\partial x_{\scriptscriptstyle j}} \frac{\partial u^{\scriptscriptstyle (2\varepsilon)}_{\scriptscriptstyle 1\Delta}}{\partial x_{\scriptscriptstyle i}} \mathrm{d}\Omega 
= \int_{\Omega} \varepsilon f_{\scriptscriptstyle 1} u^{\scriptscriptstyle (2\varepsilon)}_{\scriptscriptstyle 1\Delta} \mathrm{d}\Omega, \\[6pt]
& \int_{\Omega} c_{\scriptscriptstyle 2}(\boldsymbol{y}) \frac{\partial u^{\scriptscriptstyle (2\varepsilon)}_{\scriptscriptstyle 2\Delta}}{\partial t} u^{\scriptscriptstyle (2\varepsilon)}_{\scriptscriptstyle 2\Delta} \mathrm{d}\Omega \\
&+\int_{\Omega} \kappa_{\scriptscriptstyle 2,ij}(\boldsymbol{y}) \frac{\partial u^{\scriptscriptstyle (2\varepsilon)}_{\scriptscriptstyle 2\Delta}}{\partial x_{\scriptscriptstyle j}} \frac{\partial u^{\scriptscriptstyle (2\varepsilon)}_{\scriptscriptstyle 2\Delta}}{\partial x_{\scriptscriptstyle i}} \mathrm{d}\Omega = \int_{\Omega} \varepsilon f_{\scriptscriptstyle 2} u^{\scriptscriptstyle (2\varepsilon)}_{\scriptscriptstyle 2\Delta} \mathrm{d}\Omega.
\end{aligned}
\end{cases}
\label{eq:combined_equation}
\end{equation}

From Eqs.~\eqref{eq:combined_equation}, it is easily verified that the following two equations hold:
\begin{equation}
\begin{aligned}
&\frac{1}{2} \frac{\partial}{\partial t} \int_{\Omega} c_{\scriptscriptstyle 1}(\boldsymbol{y}) u^{\scriptscriptstyle (2\varepsilon)}_{\scriptscriptstyle 1\Delta} u^{\scriptscriptstyle (2\varepsilon)}_{\scriptscriptstyle 1\Delta} \mathrm{d}\Omega \\
&+ \int_{\Omega} \kappa_{\scriptscriptstyle 1,ij}(\boldsymbol{y}) \frac{\partial u^{\scriptscriptstyle (2\varepsilon)}_{\scriptscriptstyle 1\Delta}}{\partial x_{\scriptscriptstyle j}} \frac{\partial u^{\scriptscriptstyle (2\varepsilon)}_{\scriptscriptstyle 1\Delta}}{\partial x_{\scriptscriptstyle i}} \mathrm{d}\Omega 
= \int_{\Omega} \varepsilon f_{\scriptscriptstyle 1} u^{\scriptscriptstyle (2\varepsilon)}_{\scriptscriptstyle 1\Delta} \mathrm{d}\Omega,
\end{aligned}
\label{eq:energy_1}
\end{equation}
\begin{equation}
\begin{aligned}    
&\frac{1}{2} \frac{\partial}{\partial t} \int_{\Omega} c_{\scriptscriptstyle 2}(\boldsymbol{y}) u^{\scriptscriptstyle (2\varepsilon)}_{\scriptscriptstyle 2\Delta} u^{\scriptscriptstyle (2\varepsilon)}_{\scriptscriptstyle 2\Delta} \mathrm{d}\Omega \\
&+ \int_{\Omega} \kappa_{\scriptscriptstyle 2,ij}(\boldsymbol{y}) \frac{\partial u^{\scriptscriptstyle (2\varepsilon)}_{\scriptscriptstyle 2\Delta}}{\partial x_{\scriptscriptstyle j}} \frac{\partial u^{\scriptscriptstyle (2\varepsilon)}_{\scriptscriptstyle 2\Delta}}{\partial x_{\scriptscriptstyle i}} \mathrm{d}\Omega 
= \int_{\Omega} \varepsilon f_{\scriptscriptstyle 2} u^{\scriptscriptstyle (2\varepsilon)}_{\scriptscriptstyle 2\Delta} \mathrm{d}\Omega.
\end{aligned}
\label{eq:energy_2}
\end{equation}

Subsequently, adding Eq.~\eqref{eq:energy_1} to Eq.~\eqref{eq:energy_2} and simplifying the calculation yields:
\begin{align}
\frac{1}{2}\frac{\partial}{\partial t}&\int_{\Omega}c_{\scriptscriptstyle 1}(y)u^{\scriptscriptstyle (2\varepsilon)}_{\scriptscriptstyle 1\Delta}u^{\scriptscriptstyle (2\varepsilon)}_{\scriptscriptstyle 1\Delta}\mathrm{d}\Omega
+\int_{\Omega}\kappa_{\scriptscriptstyle 1,ij}(y)\frac{\partial u^{\scriptscriptstyle (2\varepsilon)}_{\scriptscriptstyle 1\Delta}}{\partial x_{\scriptscriptstyle j}}\frac{\partial u^{\scriptscriptstyle (2\varepsilon)}_{\scriptscriptstyle 1\Delta}}{\partial x_{\scriptscriptstyle i}}\mathrm{d}\Omega \notag\\
+\frac{1}{2}\frac{\partial}{\partial t}&\int_{\Omega}c_{\scriptscriptstyle 2}(y)u^{\scriptscriptstyle (2\varepsilon)}_{\scriptscriptstyle 2\Delta}u^{\scriptscriptstyle (2\varepsilon)}_{\scriptscriptstyle 2\Delta}\mathrm{d}\Omega
+\int_{\Omega}\kappa_{\scriptscriptstyle 2,ij}(y)\frac{\partial u^{\scriptscriptstyle (2\varepsilon)}_{\scriptscriptstyle 2\Delta}}{\partial x_{\scriptscriptstyle j}}\frac{\partial u^{\scriptscriptstyle (2\varepsilon)}_{\scriptscriptstyle 2\Delta}}{\partial x_{\scriptscriptstyle i}}\mathrm{d}\Omega \notag\\
&=\int_{\Omega}\varepsilon f_{\scriptscriptstyle 1}u^{\scriptscriptstyle (2\varepsilon)}_{\scriptscriptstyle 1\Delta}\mathrm{d}\Omega
+\int_{\Omega}\varepsilon f_{\scriptscriptstyle 2}u^{\scriptscriptstyle (2\varepsilon)}_{\scriptscriptstyle 2\Delta}\mathrm{d}\Omega.
\label{eq:combined_derivative}
\end{align}


Then, integrating both sides over the time interval \([0,t]\) (\(0<t<T^{*}\)), and simplifying the calculation, we obtain:
\begin{equation}
\begin{aligned}
&\int_{0}^{t}\frac{\partial}{\partial\tau}\bigg{[}\int_{\Omega}c_{\scriptscriptstyle 1}(\boldsymbol{y})u^{\scriptscriptstyle (2\varepsilon)}_{\scriptscriptstyle 1\Delta}u^{\scriptscriptstyle (2\varepsilon)}_{\scriptscriptstyle 1\Delta}\mathrm{d}\Omega\bigg{]}\mathrm{d}\tau\\
&+2\int_{0}^{t}\int_{\Omega}\kappa_{\scriptscriptstyle 1,ij}(\boldsymbol{y})\frac{\partial u^{\scriptscriptstyle (2\varepsilon)}_{\scriptscriptstyle 1\Delta}}{\partial x_{\scriptscriptstyle j}}\frac{\partial u^{\scriptscriptstyle (2\varepsilon)}_{\scriptscriptstyle 1\Delta}}{\partial x_{\scriptscriptstyle i}}\mathrm{d}\Omega\mathrm{d}\tau \\
&+\int_{0}^{t}\frac{\partial}{\partial\tau}\bigg{[}\int_{\Omega}c_{\scriptscriptstyle 2}(\boldsymbol{y})u^{\scriptscriptstyle (2\varepsilon)}_{\scriptscriptstyle 2\Delta}u^{\scriptscriptstyle (2\varepsilon)}_{\scriptscriptstyle 2\Delta}\mathrm{d}\Omega\bigg{]}\mathrm{d}\tau\\
&+2\int_{0}^{t}\int_{\Omega}\kappa_{\scriptscriptstyle 2,ij}(\boldsymbol{y})\frac{\partial u^{\scriptscriptstyle (2\varepsilon)}_{\scriptscriptstyle 2\Delta}}{\partial x_{\scriptscriptstyle j}}\frac{\partial u^{\scriptscriptstyle (2\varepsilon)}_{\scriptscriptstyle 2\Delta}}{\partial x_{\scriptscriptstyle i}}\mathrm{d}\Omega\mathrm{d}\tau \\
&=2\int_{0}^{t}\int_{\Omega}\varepsilon f_{\scriptscriptstyle 1}u^{\scriptscriptstyle (2\varepsilon)}_{\scriptscriptstyle 1\Delta}\mathrm{d}\Omega\mathrm{d}\tau
+2\int_{0}^{t}\int_{\Omega}\varepsilon f_{\scriptscriptstyle 2}u^{\scriptscriptstyle (2\varepsilon)}_{\scriptscriptstyle 2\Delta}\mathrm{d}\Omega\mathrm{d}\tau.
\end{aligned}
\label{eq:time_integrated}
\end{equation}

After further calculation, it can be rearranged into the following equation:
\begin{equation}
\begin{aligned}
&\int_{\Omega}c_{\scriptscriptstyle 1}(\boldsymbol{y})\bigl(u^{\scriptscriptstyle (2\varepsilon)}_{\scriptscriptstyle 1\Delta}(\boldsymbol{x},t)\bigr)^{2}\mathrm{d}\Omega
+\int_{\Omega}c_{\scriptscriptstyle 2}(\boldsymbol{y})\bigl(u^{\scriptscriptstyle (2\varepsilon)}_{\scriptscriptstyle 2\Delta}(\boldsymbol{x},t)\bigr)^{2}\mathrm{d}\Omega \\
&+2\int_{0}^{t}\!\!\int_{\Omega}\kappa_{\scriptscriptstyle 1,ij}(\boldsymbol{y})\frac{\partial u^{\scriptscriptstyle (2\varepsilon)}_{\scriptscriptstyle 1\Delta}}{\partial x_{\scriptscriptstyle j}}\frac{\partial u^{\scriptscriptstyle (2\varepsilon)}_{\scriptscriptstyle 1\Delta}}{\partial x_{\scriptscriptstyle i}}\mathrm{d}\Omega\mathrm{d}\tau \\
&+2\int_{0}^{t}\!\!\int_{\Omega}\kappa_{\scriptscriptstyle 2,ij}(\boldsymbol{y})\frac{\partial u^{\scriptscriptstyle (2\varepsilon)}_{\scriptscriptstyle 2\Delta}}{\partial x_{\scriptscriptstyle j}}\frac{\partial u^{\scriptscriptstyle (2\varepsilon)}_{\scriptscriptstyle 2\Delta}}{\partial x_{\scriptscriptstyle i}}\mathrm{d}\Omega\mathrm{d}\tau \\
&=2\int_{0}^{t}\!\!\int_{\Omega}\varepsilon f_{\scriptscriptstyle 1}u^{\scriptscriptstyle (2\varepsilon)}_{\scriptscriptstyle 1\Delta}\mathrm{d}\Omega\mathrm{d}\tau
+2\int_{0}^{t}\!\!\int_{\Omega}\varepsilon f_{\scriptscriptstyle 2}u^{\scriptscriptstyle (2\varepsilon)}_{\scriptscriptstyle 2\Delta}\mathrm{d}\Omega\mathrm{d}\tau \\
&+\int_{\Omega}c_{\scriptscriptstyle 1}(\boldsymbol{y})\bigl(u^{\scriptscriptstyle (2\varepsilon)}_{\scriptscriptstyle 1\Delta}(\boldsymbol{x},0)\bigr)^{2}\mathrm{d}\Omega\\
&+\int_{\Omega}c_{\scriptscriptstyle 2}(\boldsymbol{y})\bigl(u^{\scriptscriptstyle (2\varepsilon)}_{\scriptscriptstyle 2\Delta}(\boldsymbol{x},0)\bigr)^{2}\mathrm{d}\Omega.
\end{aligned}
\label{eq:rearranged}
\end{equation}

Next, substituting the initial condition of the residual problem~\eqref{eq:residual_second_order} into Eq.~\eqref{eq:rearranged} yields:
\begin{equation}
\begin{aligned}
&\int_{\Omega}c_{\scriptscriptstyle 1}(\boldsymbol{y})\bigl(u^{\scriptscriptstyle (2\varepsilon)}_{\scriptscriptstyle 1\Delta}(\boldsymbol{x},t)\bigr)^{2}\mathrm{d}\Omega
+\int_{\Omega}c_{\scriptscriptstyle 2}(\boldsymbol{y})\bigl(u^{\scriptscriptstyle (2\varepsilon)}_{\scriptscriptstyle 2\Delta}(\boldsymbol{x},t)\bigr)^{2}\mathrm{d}\Omega \\
&+2\int_{0}^{t}\!\!\int_{\Omega}\kappa_{\scriptscriptstyle 1,ij}(\boldsymbol{y})\frac{\partial u^{\scriptscriptstyle (2\varepsilon)}_{\scriptscriptstyle 1\Delta}}{\partial x_{\scriptscriptstyle j}}\frac{\partial u^{\scriptscriptstyle (2\varepsilon)}_{\scriptscriptstyle 1\Delta}}{\partial x_{\scriptscriptstyle i}}\mathrm{d}\Omega\mathrm{d}\tau \\
&+2\int_{0}^{t}\!\!\int_{\Omega}\kappa_{\scriptscriptstyle 2,ij}(\boldsymbol{y})\frac{\partial u^{\scriptscriptstyle (2\varepsilon)}_{\scriptscriptstyle 2\Delta}}{\partial x_{\scriptscriptstyle j}}\frac{\partial u^{\scriptscriptstyle (2\varepsilon)}_{\scriptscriptstyle 2\Delta}}{\partial x_{\scriptscriptstyle i}}\mathrm{d}\Omega\mathrm{d}\tau \\
&=2\int_{0}^{t}\!\!\int_{\Omega}\varepsilon f_{\scriptscriptstyle 1}u^{\scriptscriptstyle (2\varepsilon)}_{\scriptscriptstyle 1\Delta}\mathrm{d}\Omega\mathrm{d}\tau
+2\int_{0}^{t}\!\!\int_{\Omega}\varepsilon f_{\scriptscriptstyle 2}u^{\scriptscriptstyle (2\varepsilon)}_{\scriptscriptstyle 2\Delta}\mathrm{d}\Omega\mathrm{d}\tau \\
&+\int_{\Omega}c_{\scriptscriptstyle 1}(\boldsymbol{y})\bigl(\varepsilon\psi_{\scriptscriptstyle 1}(\boldsymbol{x})\bigr)^{2}\mathrm{d}\Omega
+\int_{\Omega}c_{\scriptscriptstyle 2}(\boldsymbol{y})\bigl(\varepsilon\psi_{\scriptscriptstyle 2}(\boldsymbol{x})\bigr)^{2}\mathrm{d}\Omega.
\end{aligned}
\label{eq:with_initial}
\end{equation}

In the following, we estimate both sides of Eq.~\eqref{eq:with_initial} separately. First, for the left-hand side, using Assumption $(A_2)$, $(A_3)$ and Poincar\'{e}-Friedrichs inequality, we obtain the following inequality:
\begin{equation}
\begin{aligned}
&\int_{\Omega}c_{\scriptscriptstyle 1}(\boldsymbol{y})\big{(}u^{\scriptscriptstyle (2\varepsilon)}_{\scriptscriptstyle 1\Delta}(\boldsymbol{x},t)\big{)}^{2}\,\mathrm{d}\Omega
+\int_{\Omega}c_{\scriptscriptstyle 2}(\boldsymbol{y})\big{(}u^{\scriptscriptstyle (2\varepsilon)}_{\scriptscriptstyle 2\Delta}(\boldsymbol{x},t)\big{)}^{2}\,\mathrm{d}\Omega\\
&+2\int_{0}^{t}\int_{\Omega}\kappa_{\scriptscriptstyle 1,ij}(\boldsymbol{y})\,\frac{\partial u^{\scriptscriptstyle (2\varepsilon)}_{\scriptscriptstyle 1\Delta}(\boldsymbol{x},t)}{\partial x_{\scriptscriptstyle j}}\,\frac{\partial u^{\scriptscriptstyle (2\varepsilon)}_{\scriptscriptstyle 1\Delta}(\boldsymbol{x},t)}{\partial x_{\scriptscriptstyle i}}\,\mathrm{d}\Omega\mathrm{d}\tau \\
&+2\int_{0}^{t}\int_{\Omega}\kappa_{\scriptscriptstyle 2,ij}(\boldsymbol{y})\,\frac{\partial u^{\scriptscriptstyle (2\varepsilon)}_{\scriptscriptstyle 2\Delta}(\boldsymbol{x},t)}{\partial x_{\scriptscriptstyle j}}\,\frac{\partial u^{\scriptscriptstyle (2\varepsilon)}_{\scriptscriptstyle 2\Delta}(\boldsymbol{x},t)}{\partial x_{\scriptscriptstyle i}}\,\mathrm{d}\Omega\mathrm{d}\tau \\
&\geq C\big{\|}u^{\scriptscriptstyle (2\varepsilon)}_{\scriptscriptstyle 1\Delta}(\boldsymbol{x},t)\big{\|}_{\scriptscriptstyle L^{2}(\Omega)}^{2}
+2K\int_{0}^{t}\big{\|}u^{\scriptscriptstyle (2\varepsilon)}_{\scriptscriptstyle 1\Delta}(\boldsymbol{x},t)\big{\|}_{\scriptscriptstyle H^{1}(\Omega)}^{2}\,\mathrm{d}\tau \\
&+C\big{\|}u^{\scriptscriptstyle (2\varepsilon)}_{\scriptscriptstyle 2\Delta}(\boldsymbol{x},t)\big{\|}_{\scriptscriptstyle L^{2}(\Omega)}^{2}
+2K\int_{0}^{t}\big{\|}u^{\scriptscriptstyle (2\varepsilon)}_{\scriptscriptstyle 2\Delta}(\boldsymbol{x},t)\big{\|}_{\scriptscriptstyle H^{1}_{0}(\Omega)}^{2}\,\mathrm{d}\tau.
\end{aligned}
\label{eq:left_hand_inequality}
\end{equation}

Then, using Assumption $(A_2)$, $(A_3)$ and Young's inequality, and selecting the same parameter $\lambda$ for any arbitrary $\tau\in[0,t]$, the estimation of the right-hand side of Eq.~\eqref{eq:with_initial} yields the following inequality:
\begin{equation}
\begin{aligned}
&\int_{\Omega}c_{\scriptscriptstyle 1}\left(\varepsilon\psi_{\scriptscriptstyle 1}(\boldsymbol{x})\right)^{2}\mathrm{d}\Omega+2\int_{0}^{t}\int_{\Omega}\varepsilon f_{\scriptscriptstyle 1}u_{\scriptscriptstyle 1\Delta}^{\scriptscriptstyle (2\varepsilon)}(\boldsymbol{x},t)\mathrm{d}\Omega\mathrm{d}\tau\\
&+\int_{\Omega}c_{\scriptscriptstyle 2}\left(\varepsilon\psi_{\scriptscriptstyle 2}(\boldsymbol{x})\right)^{2}\mathrm{d}\Omega+2\int_{0}^{t}\int_{\Omega}\varepsilon f_{\scriptscriptstyle 2}u_{\scriptscriptstyle 2\Delta}^{\scriptscriptstyle (2\varepsilon)}(\boldsymbol{x},t)\mathrm{d}\Omega\mathrm{d}\tau\\
&\leq C\varepsilon^{2}+2\left[\frac{1}{\lambda}\int_{0}^{t}\left(\|\varepsilon f_{\scriptscriptstyle 1}\|_{\scriptscriptstyle L^{2}(\Omega)}^{2}
+\|\varepsilon f_{\scriptscriptstyle 2}\|_{\scriptscriptstyle L^{2}(\Omega)}^{2}\right)\mathrm{d}\tau\right. \\
&\left.+\lambda\int_{0}^{t}\left(\|u_{\scriptscriptstyle 1\Delta}^{\scriptscriptstyle (2\varepsilon)}(\boldsymbol{x},t)\|_{\scriptscriptstyle L^{2}(\Omega)}^{2}
+\|u_{\scriptscriptstyle 2\Delta}^{\scriptscriptstyle (2\varepsilon)}(\boldsymbol{x},t)\|_{\scriptscriptstyle L^{2}(\Omega)}^{2}\right)\mathrm{d}\tau\right] \\  
&\leq C\varepsilon^{2}+\lambda\int_{0}^{t}\|u_{\scriptscriptstyle 1\Delta}^{\scriptscriptstyle (2\varepsilon)}(\boldsymbol{x},t)\|_{\scriptscriptstyle L^{2}(\Omega)}^{2}\mathrm{d}\tau\\
&+\lambda\int_{0}^{t}\|u_{\scriptscriptstyle 2\Delta}^{\scriptscriptstyle (2\varepsilon)}(\boldsymbol{x},t)\|_{\scriptscriptstyle L^{2}(\Omega)}^{2}\mathrm{d}\tau \\
&+M\int_{0}^{t}\int_{0}^{\tau}\|u_{\scriptscriptstyle 1\Delta}^{\scriptscriptstyle (2\varepsilon)}(\boldsymbol{x},t)\|_{\scriptscriptstyle H_{0}^{1}(\Omega)}^{2}\mathrm{d}s\mathrm{d}\tau \\
&+M\int_{0}^{t}\int_{0}^{\tau}\|u_{\scriptscriptstyle 2\Delta}^{\scriptscriptstyle (2\varepsilon)}(\boldsymbol{x},t)\|_{\scriptscriptstyle H_{0}^{1}(\Omega)}^{2}\mathrm{d}s\mathrm{d}\tau.
\end{aligned}
\label{eq:right_hand_inequality}
\end{equation}

Combining Eq.~\eqref{eq:left_hand_inequality} and Eq.~\eqref{eq:right_hand_inequality}, we obtain the following inequality:
\begin{equation}
\begin{aligned}
&C\bigl\|u_{\scriptscriptstyle 1\Delta}^{\scriptscriptstyle (2\varepsilon)}(\boldsymbol{x},t)\bigr\|_{\scriptscriptstyle L^{2}(\Omega)}^{2} 
+2K\int_{0}^{t}\bigl\|u_{\scriptscriptstyle 1\Delta}^{\scriptscriptstyle (2\varepsilon)}(\boldsymbol{x},t)\bigr\|_{\scriptscriptstyle H_{0}^{1}(\Omega)}^{2}\mathrm{d}\tau\\
&+C\bigl\|u_{\scriptscriptstyle 2\Delta}^{\scriptscriptstyle (2\varepsilon)}(\boldsymbol{x},t)\bigr\|_{\scriptscriptstyle L^{2}(\Omega)}^{2} 
+2K\int_{0}^{t}\bigl\|u_{\scriptscriptstyle 2\Delta}^{\scriptscriptstyle (2\varepsilon)}(\boldsymbol{x},t)\bigr\|_{\scriptscriptstyle H_{0}^{1}(\Omega)}^{2}\mathrm{d}\tau \\
&\leq C\varepsilon^{2}+\lambda\int_{0}^{t}\bigl\|u_{\scriptscriptstyle 1\Delta}^{\scriptscriptstyle (2\varepsilon)}(\boldsymbol{x},t)\bigr\|_{\scriptscriptstyle L^{2}(\Omega)}^{2}\mathrm{d}\tau\\
&+\lambda\int_{0}^{t}\bigl\|u_{\scriptscriptstyle 2\Delta}^{\scriptscriptstyle (2\varepsilon)}(\boldsymbol{x},t)\bigr\|_{\scriptscriptstyle L^{2}(\Omega)}^{2}\mathrm{d}\tau \\
&+M\int_{0}^{t}\int_{0}^{\tau}\bigl\|u_{\scriptscriptstyle 1\Delta}^{\scriptscriptstyle (2\varepsilon)}(\boldsymbol{x},t)\bigr\|_{\scriptscriptstyle H_{0}^{1}(\Omega)}^{2}\mathrm{d}s\mathrm{d}\tau\\
&+M\int_{0}^{t}\int_{0}^{\tau}\bigl\|u_{\scriptscriptstyle 2\Delta}^{\scriptscriptstyle (2\varepsilon)}(\boldsymbol{x},t)\bigr\|_{\scriptscriptstyle H_{0}^{1}(\Omega)}^{2}\mathrm{d}s\mathrm{d}\tau.
\end{aligned}
\label{eq:combined_inequality}
\end{equation}

Defining \(\delta_{1}=\min(\underline{C},2K)\) and \(\delta_{2}=\max(\lambda,M)\), Eq.~\eqref{eq:combined_inequality} can be further simplified as:
\begin{equation}
\begin{aligned}
&\delta_{1}\Bigl(\bigl\|u_{\scriptscriptstyle 1\Delta}^{\scriptscriptstyle (2\varepsilon)}(\boldsymbol{x},t)\bigr\|_{\scriptscriptstyle L^{2}(\Omega)}^{2}
+\bigl\|u_{\scriptscriptstyle 2\Delta}^{\scriptscriptstyle (2\varepsilon)}(\boldsymbol{x},t)\bigr\|_{\scriptscriptstyle L^{2}(\Omega)}^{2} \\
&+\int_{0}^{t}\bigl\|u_{\scriptscriptstyle 1\Delta}^{\scriptscriptstyle (2\varepsilon)}(\boldsymbol{x},t)\bigr\|_{\scriptscriptstyle H_{0}^{1}(\Omega)}^{2}\mathrm{d}\tau
+\int_{0}^{t}\bigl\|u_{\scriptscriptstyle 2\Delta}^{\scriptscriptstyle (2\varepsilon)}(\boldsymbol{x},t)\bigr\|_{\scriptscriptstyle H_{0}^{1}(\Omega)}^{2}\mathrm{d}\tau\Bigr) \\
&\leq C\varepsilon^{2}+\delta_{2}\Bigl(\int_{0}^{t}\bigl\|u_{\scriptscriptstyle 1\Delta}^{\scriptscriptstyle (2\varepsilon)}(\boldsymbol{x},t)\bigr\|_{\scriptscriptstyle L^{2}(\Omega)}^{2}\mathrm{d}\tau\\
&+\int_{0}^{t}\bigl\|u_{\scriptscriptstyle 2\Delta}^{\scriptscriptstyle (2\varepsilon)}(\boldsymbol{x},t)\bigr\|_{\scriptscriptstyle L^{2}(\Omega)}^{2}\mathrm{d}\tau \\
&+\int_{0}^{t}\int_{0}^{\tau}\bigl\|u_{\scriptscriptstyle 1\Delta}^{\scriptscriptstyle (2\varepsilon)}(\boldsymbol{x},t)\bigr\|_{\scriptscriptstyle H_{0}^{1}(\Omega)}^{2}\mathrm{d}s\mathrm{d}\tau\\
&+\int_{0}^{t}\int_{0}^{\tau}\bigl\|u_{\scriptscriptstyle 2\Delta}^{\scriptscriptstyle (2\varepsilon)}(\boldsymbol{x},t)\bigr\|_{\scriptscriptstyle H_{0}^{1}(\Omega)}^{2}\mathrm{d}s\mathrm{d}\tau\Bigr).
\end{aligned}
\label{eq:simplified_inequality}
\end{equation}

Without loss of generality, defining \(C=C/\delta_{1}\) and \(\delta=\delta_{2}/\delta_{1}\), Eq.~\eqref{eq:simplified_inequality} can be further simplified as:
\begin{equation}
\begin{aligned}
&\bigl\|u_{\scriptscriptstyle 1\Delta}^{\scriptscriptstyle (2\varepsilon)}(\boldsymbol{x},t)\bigr\|_{\scriptscriptstyle L^{2}(\Omega)}^{2}
+\bigl\|u_{\scriptscriptstyle 2\Delta}^{\scriptscriptstyle (2\varepsilon)}(\boldsymbol{x},t)\bigr\|_{\scriptscriptstyle L^{2}(\Omega)}^{2} \\
&+\int_{0}^{t}\bigl\|u_{\scriptscriptstyle 1\Delta}^{\scriptscriptstyle (2\varepsilon)}(\boldsymbol{x},t)\bigr\|_{\scriptscriptstyle H_{0}^{1}(\Omega)}^{2}\,\mathrm{d}\tau
+\int_{0}^{t}\bigl\|u_{\scriptscriptstyle 2\Delta}^{\scriptscriptstyle (2\varepsilon)}(\boldsymbol{x},t)\bigr\|_{\scriptscriptstyle H_{0}^{1}(\Omega)}^{2}\,\mathrm{d}\tau \\
&\leq C\varepsilon^{2}+\delta\Bigl(\int_{0}^{t}\bigl\|u_{\scriptscriptstyle 1\Delta}^{\scriptscriptstyle (2\varepsilon)}(\boldsymbol{x},t)\bigr\|_{\scriptscriptstyle L^{2}(\Omega)}^{2}\,\mathrm{d}\tau \\
& +\int_{0}^{t}\bigl\|u_{\scriptscriptstyle 2\Delta}^{\scriptscriptstyle (2\varepsilon)}(\boldsymbol{x},t)\bigr\|_{\scriptscriptstyle L^{2}(\Omega)}^{2}\,\mathrm{d}\tau \\
& +\int_{0}^{t}\int_{0}^{\tau}\bigl\|u_{\scriptscriptstyle 1\Delta}^{\scriptscriptstyle (2\varepsilon)}(\boldsymbol{x},t)\bigr\|_{\scriptscriptstyle H_{0}^{1}(\Omega)}^{2}\,\mathrm{d}\varsigma\,\mathrm{d}\tau \\
&+\int_{0}^{t}\int_{0}^{\tau}\bigl\|u_{\scriptscriptstyle 2\Delta}^{\scriptscriptstyle (2\varepsilon)}(\boldsymbol{x},t)\bigr\|_{\scriptscriptstyle H_{0}^{1}(\Omega)}^{2}\,\mathrm{d}\varsigma\,\mathrm{d}\tau\Bigr).
\end{aligned}
\label{eq:final_inequality}
\end{equation}

Subsequently, applying Gronwall's inequality to Eq.~\eqref{eq:final_inequality} yields:
\begin{equation}
\begin{aligned}
&\bigl\|u_{\scriptscriptstyle 1\Delta}^{\scriptscriptstyle (2\varepsilon)}(\boldsymbol{x},t)\bigr\|_{L^{\scriptscriptstyle 2}(\Omega)}^{2}
+\bigl\|u_{\scriptscriptstyle 2\Delta}^{\scriptscriptstyle (2\varepsilon)}(\boldsymbol{x},t)\bigr\|_{L^{\scriptscriptstyle 2}(\Omega)}^{2} \\
&+\int_{0}^{t}\big{\|}u_{\scriptscriptstyle 1\Delta}^{\scriptscriptstyle (2\varepsilon)}(\boldsymbol{x},t)\big{\|}_{H_{\scriptscriptstyle 0}^{\scriptscriptstyle 1}(\Omega)}^{2}\,\mathrm{d}\tau
+\int_{0}^{t}\big{\|}u_{\scriptscriptstyle 2\Delta}^{\scriptscriptstyle (2\varepsilon)}(\boldsymbol{x},t)\big{\|}_{H_{\scriptscriptstyle 0}^{\scriptscriptstyle 1}(\Omega)}^{2}\,\mathrm{d}\tau \\
&\leq C(T^{*})\varepsilon^{\scriptscriptstyle 2},\quad \forall\,t\in[0,T^{*}].
\end{aligned}
\label{eq:after_gronwall}
\end{equation}

Then, applying the mean value inequality to Eq.~\eqref{eq:after_gronwall}, we obtain the following inequality:
\begin{equation}
\begin{aligned}
&\bigl\|u_{\scriptscriptstyle 1\Delta}^{\scriptscriptstyle (2\varepsilon)}(\boldsymbol{x},t)\bigr\|_{\scriptscriptstyle L^{\scriptscriptstyle 2}(\Omega)}
+\bigl\|u_{\scriptscriptstyle 2\Delta}^{\scriptscriptstyle (2\varepsilon)}(\boldsymbol{x},t)\bigr\|_{\scriptscriptstyle L^{\scriptscriptstyle 2}(\Omega)} \\
&+\bigl\|u_{\scriptscriptstyle 1\Delta}^{\scriptscriptstyle (2\varepsilon)}(\boldsymbol{x},t)\bigr\|_{\scriptscriptstyle L^{\scriptscriptstyle 2}(0,t;H_{\scriptscriptstyle 0}^{\scriptscriptstyle 1}(\Omega))}
+\bigl\|u_{\scriptscriptstyle 2\Delta}^{\scriptscriptstyle (2\varepsilon)}(\boldsymbol{x},t)\bigr\|_{\scriptscriptstyle L^{\scriptscriptstyle 2}(0,t;H_{\scriptscriptstyle 0}^{\scriptscriptstyle 1}(\Omega))} \\
&\leq C(T^{*})\varepsilon,\quad \forall t\in[0,T^{*}].
\end{aligned}
\label{eq:final_norm_estimate}
\end{equation}

Finally, utilizing the arbitrariness of the time variable \(t\), the following inequality holds:
\begin{equation}
\begin{aligned}
&\bigl\|u_{\scriptscriptstyle 1\Delta}^{\scriptscriptstyle (2\varepsilon)}\bigr\|_{L^{\infty}(0,T^{*};L^{2}(\Omega))}
+\bigl\|u_{\scriptscriptstyle 1\Delta}^{\scriptscriptstyle (2\varepsilon)}\bigr\|_{L^{2}(0,T^{*};H_{1}^{0}(\Omega))} \\
&+\bigl\|u_{\scriptscriptstyle 2\Delta}^{\scriptscriptstyle (2\varepsilon)}\bigr\|_{L^{\infty}(0,T^{*};L^{2}(\Omega))}
+\bigl\|u_{\scriptscriptstyle 2\Delta}^{\scriptscriptstyle (2\varepsilon)}\bigr\|_{L^{2}(0,T^{*};H_{1}^{0}(\Omega))} \\
&\leq C(T^{*})\varepsilon.
\end{aligned}
\label{eq:norm_estimate}
\end{equation}

In summary, Theorem~\ref{thm:error_estimate} is proved.
\end{proof}

\section{\label{sec:4}Numerical Algorithms and Experiments}
This chapter presents the corresponding algorithm for the HOMS computational model established in Section~\ref{sec:2.2}, and conducts numerical experiments to validate the effectiveness of the method and the necessity of introducing the second-order correction term.

Based on the HOMS asymptotic solution for the multiscale problem~\eqref{eq:3} provided in Section~\ref{sec:2.2}, the HOMS algorithm for solving multi-continuum problems in highly heterogeneous media is given as follows:
\begin{enumerate}
    \item Determine the geometric configuration of the reference unit cell \( Y \) and the homogenized macroscopic region \( \Omega \), as well as the material parameters of each constituent phase. Then, perform finite element discretization of the unit cell \( Y \) and the homogenized macroscopic region \( \Omega \) using triangular meshes in two dimensions or tetrahedral meshes in three dimensions. Let \(J^{h_{\scriptscriptstyle 1}} = \{K\}\) and \(J^{h_{\scriptscriptstyle 0}} = \{e\}\) denote a family of regular tetrahedral mesh partitions of the unit cell \(Y\) and the homogenized macroscopic domain \(\Omega\), respectively, where \(h_{\scriptscriptstyle 1} = \max_{\scriptscriptstyle K} \{h_{\scriptscriptstyle K}\}\) and \(h_{\scriptscriptstyle 0} = \max_e \{h_e\}\). Then define the conforming finite element space on the reference unit cell \(Y\) as 
    \( V_{h_{\scriptscriptstyle 1}}(Y) = \left\{ v \in C^{\scriptscriptstyle 0}(\overline{Y}) : v|_{\partial Y} = 0,\ v|_K \in P(K) \right\} \subset H^{\scriptscriptstyle 1}_{\scriptscriptstyle 0}(Y),\)
    and the conforming finite element space on the homogenized macroscopic domain \(\Omega\) as  \(V_{h_{\scriptscriptstyle 0}}(\Omega) = \left\{ v \in C^{\scriptscriptstyle 0}(\overline{\Omega}) : v|_e \in P(e) \right\} \subset H^{\scriptscriptstyle 1}_{\scriptscriptstyle 0}(\Omega).\)

    \item Solve for the first-order cell functions \(N^{\scriptscriptstyle\alpha_1}_l(\boldsymbol{y})\) and \(M_{\scriptscriptstyle l}(\boldsymbol{y})\) using the variational formulation in the finite element space \(V_{h_0}(Y)\).
    \begin{equation}
    \begin{aligned}
        &\int_Y \kappa_{\scriptscriptstyle l,ij}(y) \frac{\partial N^{\alpha_{\scriptscriptstyle 1}}_l(y)}{\partial y_j} \frac{\partial \varphi^{h_{\scriptscriptstyle 1}}}{\partial y_i} \, \mathrm{d}Y \\
        &+ \int_Y \kappa_{\scriptscriptstyle l,i\alpha_1}(y) \frac{\partial \varphi^{h_{\scriptscriptstyle 1}}}{\partial y_i} \, \mathrm{d}Y = 0,
         \ \forall \varphi^{h_{\scriptscriptstyle 1}} \in V_{h_{\scriptscriptstyle 1}}(Y). \label{eq:cell_eq1} 
    \end{aligned}
    \end{equation}
    \item Then, using the obtained first-order cell functions and Eqs.~\eqref{eq:homogenized}, solve for the homogenized coefficients \(c^*_{\scriptscriptstyle l}\), \(\kappa^*_{\scriptscriptstyle l,ij}\), \(\bar{K}^{\scriptscriptstyle l*}_{\scriptscriptstyle 1i}\), \(\bar{K}^{\scriptscriptstyle l*}_{\scriptscriptstyle 2i}\) and \(Q_{\scriptscriptstyle l}^*\). Subsequently, establish the following coupled finite difference-finite element discretization scheme for solving the homogenized problem~\eqref{eq:3}:
\begin{equation}
\begin{cases}
\begin{aligned}
 & \int_{\Omega}c_{\scriptscriptstyle 1}^{*}\frac{u_{\scriptscriptstyle 1}^{\scriptscriptstyle (0),n+1}-u_{\scriptscriptstyle 1}^{\scriptscriptstyle (0),n}}{\Delta t}\nu^{\scriptscriptstyle h_{\scriptscriptstyle 0}}\mathrm{d}\Omega\\
 & =-\frac{1}{2} \!\int_{\Omega}\kappa_{\scriptscriptstyle 1,ij}^{*} \left( \frac{\partial u_{\scriptscriptstyle 1}^{\scriptscriptstyle (0),n+1}}{\partial x_{j}} \! + \!\frac{\partial u_{\scriptscriptstyle 1}^{\scriptscriptstyle (0),n}}{\partial x_{j}} \right)\frac{\partial\nu^{\scriptscriptstyle h_{\scriptscriptstyle 0}}}{\partial x_{i}} \mathrm{d}\Omega  \\
 & +\frac{1}2{}\int_{\Omega} Q_{\scriptscriptstyle 1}^{\scriptscriptstyle *}\left(u_{\scriptscriptstyle 2}^{\scriptscriptstyle (0),n+1}-u_{\scriptscriptstyle 1}^{\scriptscriptstyle (0),n+1}\right)\nu^{\scriptscriptstyle h_{\scriptscriptstyle 0}}\mathrm{d}\Omega\\
 & +\frac{1}{2}\int_{\Omega} Q_{\scriptscriptstyle 1}^{\scriptscriptstyle *}\left(u_{\scriptscriptstyle 2}^{\scriptscriptstyle (0),n}-u_{\scriptscriptstyle 1}^{\scriptscriptstyle (0),n}\right)\nu^{\scriptscriptstyle h_{\scriptscriptstyle 0}}\mathrm{d}\Omega\\
&+\frac{1}{2}\int_{\Omega}\bar{K}_{\scriptscriptstyle1i}^{\scriptscriptstyle1*}\left( \frac{\partial u_{2}^{\scriptscriptstyle (0),n+1}}{\partial x_{j}}+\frac{\partial u_{2}^{\scriptscriptstyle (0),n}}{\partial x_{j}} \right) \nu^{\scriptscriptstyle h_{\scriptscriptstyle 0}}\mathrm{d}\Omega\\
 &- \frac{1}{2}\! \int_{\Omega}\bar{K}_{\scriptscriptstyle 2i}^{\scriptscriptstyle1*}\left(\frac{\partial u_{1}^{\scriptscriptstyle (0),n+1}}{\partial x_{i}}+\frac{\partial u_{1}^{\scriptscriptstyle (0),n}}{\partial x_{i}} \right) \nu^{\scriptscriptstyle h_{\scriptscriptstyle 0}}\mathrm{d}\Omega \\
 & +\int_{\Omega}q\nu^{
  h_{\scriptscriptstyle 0}}\mathrm{d}\Omega,\ \forall\nu^{\scriptscriptstyle h_{\scriptscriptstyle 0}}\in V_{h_{\scriptscriptstyle 0}}(\Omega),t\in(0,T^{*}), \\
 & \int_{\Omega}c_{\scriptscriptstyle 2}^{*}\frac{u_{\scriptscriptstyle 2}^{\scriptscriptstyle (0),n+1}-u_{\scriptscriptstyle 2}^{\scriptscriptstyle (0),n}}{\Delta t}\nu^{\scriptscriptstyle h_{\scriptscriptstyle 0}}\mathrm{d}\Omega\\
 &=-\frac{1}{2}\int_{\Omega}\kappa_{\scriptscriptstyle 2,ij}^{\scriptscriptstyle *}\left(\frac{\partial u_{\scriptscriptstyle 2}^{\scriptscriptstyle (0),n+1}}{\partial x_{j}}+\frac{\partial u_{\scriptscriptstyle 2}^{\scriptscriptstyle (0),n}}{\partial x_{j}}\right)\frac{\partial\nu^{\scriptscriptstyle h_{\scriptscriptstyle 0}}}{\partial x_{i}}\mathrm{d}\Omega \\
 & +\frac{1}{2}\int_{\Omega} Q_{\scriptscriptstyle 2}^{*}\left(u_{\scriptscriptstyle 1}^{\scriptscriptstyle (0),n+1}-u_{\scriptscriptstyle 2}^{\scriptscriptstyle (0),n+1}\right)\nu^{\scriptscriptstyle h_{\scriptscriptstyle 0}}\mathrm{d}\Omega\\
 & +\frac{1}{2}\int_{\Omega} Q_{\scriptscriptstyle 2}^{*}\left(u_{\scriptscriptstyle 1}^{\scriptscriptstyle (0),n}-u_{\scriptscriptstyle 2}^{\scriptscriptstyle (0),n}\right)\nu^{\scriptscriptstyle h_{\scriptscriptstyle 0}}\mathrm{d}\Omega\\
 & +\frac{1}{2}\int_{\Omega}\bar{K}_{\scriptscriptstyle 1i}^{\scriptscriptstyle 2*}\left(\frac{\partial u_{\scriptscriptstyle 1}^{\scriptscriptstyle (0),n+1}}{\partial x_{i}}+\frac{\partial u_{\scriptscriptstyle 1}^{\scriptscriptstyle (0),n}}{\partial x_{i}}\right)\nu^{\scriptscriptstyle h_{\scriptscriptstyle 0}}\mathrm{d}\Omega\\
 &- \frac{1}{2}\! \int_{\Omega}\bar{K}_{\scriptscriptstyle 2i}^{\scriptscriptstyle 2*}\left(\frac{\partial u_{\scriptscriptstyle 2}^{\scriptscriptstyle (0),n+1}}{\partial x_{i}}+\frac{\partial u_{\scriptscriptstyle 2}^{\scriptscriptstyle (0),n}}{\partial x_{i}}\right)\nu^{\scriptscriptstyle h_{\scriptscriptstyle 0}}\mathrm{d}\Omega \\
 & +\int_{\Omega}q\nu^{\scriptscriptstyle h_{\scriptscriptstyle 0}}\mathrm{d}\Omega, \ \forall\nu^{\scriptscriptstyle h_{\scriptscriptstyle 0}}\in V_{h_{\scriptscriptstyle 0}}(\Omega),t\in(0,T^{*}).
\end{aligned}
\end{cases}
\label{eq:yuan68}
\end{equation}

    \item Using the mesh \(J^{\scriptscriptstyle h_{\scriptscriptstyle 1}}=\{K\}\) and the finite element space \(V_{\scriptscriptstyle h_{\scriptscriptstyle 0}}(Y)\) employed for solving the first-order cell functions, determine the second-order cell functions \(G_{\scriptscriptstyle l}(\boldsymbol{y})\), \(N_{\scriptscriptstyle l}^{\scriptscriptstyle\alpha_1\alpha_2}(\boldsymbol{y})\), \(C_{\scriptscriptstyle l}^{\scriptscriptstyle\alpha_1}(\boldsymbol{y})\), \(F_{\scriptscriptstyle l}^{\scriptscriptstyle\alpha_1}(\boldsymbol{y})\) and \(K_{\scriptscriptstyle l}(\boldsymbol{y})\) according to the cell problems \eqref{eq:cell_problem_G}-\eqref{eq:K_cell_problem}.

    \item For any point \((\boldsymbol{x}, t) \in \Omega \times (0, T^*)\), the values of the first-order and second-order cell functions and the homogenized solutions at that point are obtained using interpolation methods. The partial derivatives of the homogenized solution \(u_{\scriptscriptstyle l}^{\scriptscriptstyle(0)}\) with respect to the spatial variables, \(\frac{\partial u_{\scriptscriptstyle l}^{\scriptscriptstyle(0)}}{\partial x_{\scriptscriptstyle\alpha_1}}\) and \(\frac{\partial^2 u_{\scriptscriptstyle l}^{\scriptscriptstyle(0)}}{\partial x_{\scriptscriptstyle\alpha_1}\partial x_{\scriptscriptstyle\alpha_2}}\), are computed using the element average method\cite{dong_multiscale_2009,dong_multiscale_2014}. The partial derivative with respect to time, \(\frac{\partial u_{\scriptscriptstyle l}^{\scriptscriptstyle(0)}}{\partial t}\), can be obtained from \eqref{eq:yuan68}. Finally, the HOMS solution of the multiscale problem \eqref{eq:3} at any point \((\boldsymbol{x}, t) \in \Omega \times (0, T^*)\) is calculated according to formula \eqref{eq:second_order_solution}.

    \item The reference solution \(u_{\scriptscriptstyle
     l}^{\varepsilon}\) for problem \eqref{eq:3} is computed using the refined FEM over the highly heterogeneous media domain \(\Omega\).
\end{enumerate}

\section{\label{sec:5}Numerical Experiments}
In this section, we present four numerical examples to verify the effectiveness and stability of the HOMS method. For simplicity, we assume each medium is isotropic in this paper (and the anisotropic case is handled similarly). All numerical experiments are conducted on an HP desktop workstation equipped with an 3th Gen Intel(R) Core(TM)17-13700 processor (2.10 GHz) and 32.0 GB of internal memory, and all numerical simulations are performed based on Freefem++ software.  The relative errors of $u_{\scriptscriptstyle l}^{\scriptscriptstyle(0)}$, $u_{\scriptscriptstyle l}^{\scriptscriptstyle(1\varepsilon)}$ and $u_{\scriptscriptstyle l}^{\scriptscriptstyle(2\varepsilon)}$ under the corresponding norms are denoted as:
\begin{flalign}    
&&\text{Lerr10}(t)=\frac{\left\|u_{\scriptscriptstyle 1e}-u_{\scriptscriptstyle 1}^{(\scriptscriptstyle 0)}\right\|_{L^{\scriptscriptstyle 2}}}{\left\|u_{\scriptscriptstyle 1e}\right\|_{L^{\scriptscriptstyle 2}}},\ \text{Lerr11}(t)=\frac{\left\|u_{\scriptscriptstyle 1e}-u_{\scriptscriptstyle 1}^{(\scriptscriptstyle 1\varepsilon)}\right\|_{L^{\scriptscriptstyle 2}}}{\left\|u_{\scriptscriptstyle 1e}\right\|_{L^{\scriptscriptstyle 2}}},\notag\\
&&\text{Lerr12}(t)=\frac{\left\|u_{\scriptscriptstyle 1e}-u_{\scriptscriptstyle 1}^{(\scriptscriptstyle 2\varepsilon)}\right\|_{L^{\scriptscriptstyle 2}}}{\left\|u_{\scriptscriptstyle 1e}\right\|_{L^{\scriptscriptstyle 2}}},\
\text{Herr10}(t)=\frac{\left|u_{\scriptscriptstyle 1e}-u_{\scriptscriptstyle 1}^{(\scriptscriptstyle 0)}\right|_{H^{\scriptscriptstyle 1}}}{\left|u_{\scriptscriptstyle 1e}\right|_{H^{\scriptscriptstyle 1}}},\notag \\
&&\text{Herr11}(t)=\frac{\left|u_{\scriptscriptstyle 1e}-u_{\scriptscriptstyle 1}^{(\scriptscriptstyle 1\varepsilon)}\right|_{H^{\scriptscriptstyle 1}}}{\left|u_{\scriptscriptstyle 1e}\right|_{H^{\scriptscriptstyle 1}}},\ \text{Herr12}(t)=\frac{\left|u_{\scriptscriptstyle 1e}-u_{\scriptscriptstyle 1}^{(\scriptscriptstyle 2\varepsilon)}\right|_{H^{\scriptscriptstyle 1}}}{\left|u_{\scriptscriptstyle 1e}\right|_{H^{\scriptscriptstyle 1}}},\notag \\
&&\text{Lerr20}(t)=\frac{\left\|u_{\scriptscriptstyle 2e}-u_{\scriptscriptstyle 2}^{(\scriptscriptstyle 0)}\right\|_{L^{\scriptscriptstyle 2}}}{\left\|u_{\scriptscriptstyle 2e}\right\|_{L^{\scriptscriptstyle 2}}},\ \text{Lerr21}(t)=\frac{\left\|u_{\scriptscriptstyle 2e}-u_{\scriptscriptstyle 2}^{(\scriptscriptstyle 1\varepsilon)}\right\|_{L^{\scriptscriptstyle 2}}}{\left\|u_{\scriptscriptstyle 2e}\right\|_{L^{\scriptscriptstyle 2}}},\notag \\
&&\text{Lerr22}(t)=\frac{\left\|u_{\scriptscriptstyle 2e}-u_{\scriptscriptstyle 2}^{(\scriptscriptstyle 2\varepsilon)}\right\|_{L^{\scriptscriptstyle 2}}}{\left\|u_{\scriptscriptstyle 2e}\right\|_{L^{\scriptscriptstyle 2}}}, \ \text{Herr20}(t)=\frac{\left|u_{\scriptscriptstyle 2e}-u_{\scriptscriptstyle 2}^{(\scriptscriptstyle 0)}\right|_{H^{\scriptscriptstyle 1}}}{\left|u_{\scriptscriptstyle 2e}\right|_{H^{\scriptscriptstyle 1}}},\notag \\
&&\text{Herr21}(t)=\frac{\left|u_{\scriptscriptstyle 2e}-u_{\scriptscriptstyle 2}^{(\scriptscriptstyle 1\varepsilon)}\right|_{H^{\scriptscriptstyle 1}}}{\left|u_{\scriptscriptstyle 2e}\right|_{H^{\scriptscriptstyle 1}}},\ \text{Herr22}(t)=\frac{\left|u_{\scriptscriptstyle 2e}-u_{\scriptscriptstyle 2}^{(\scriptscriptstyle 2\varepsilon)}\right|_{H^{\scriptscriptstyle 1}}}{\left|u_{\scriptscriptstyle 2e}\right|_{H^{\scriptscriptstyle 1}}}.\notag
\end{flalign}

\subsection{Example 1. 2D porous media}
For the two-dimensional case of problem \eqref{eq:3}, consider the macroscopic domain \(\Omega = (x_{\scriptscriptstyle 1}, x_{\scriptscriptstyle 2}) = [0, 1]^2\) and the microscopic unit cell \(Y = (y_{\scriptscriptstyle 1}, y_{\scriptscriptstyle 2}) = [0, 1]^2\) as shown in Fig.~\ref{fig:yuan2d}, where \(\varepsilon = 1/8\).

\begin{figure}[pos=htbp]
\centering
{\includegraphics[width=0.483\textwidth]{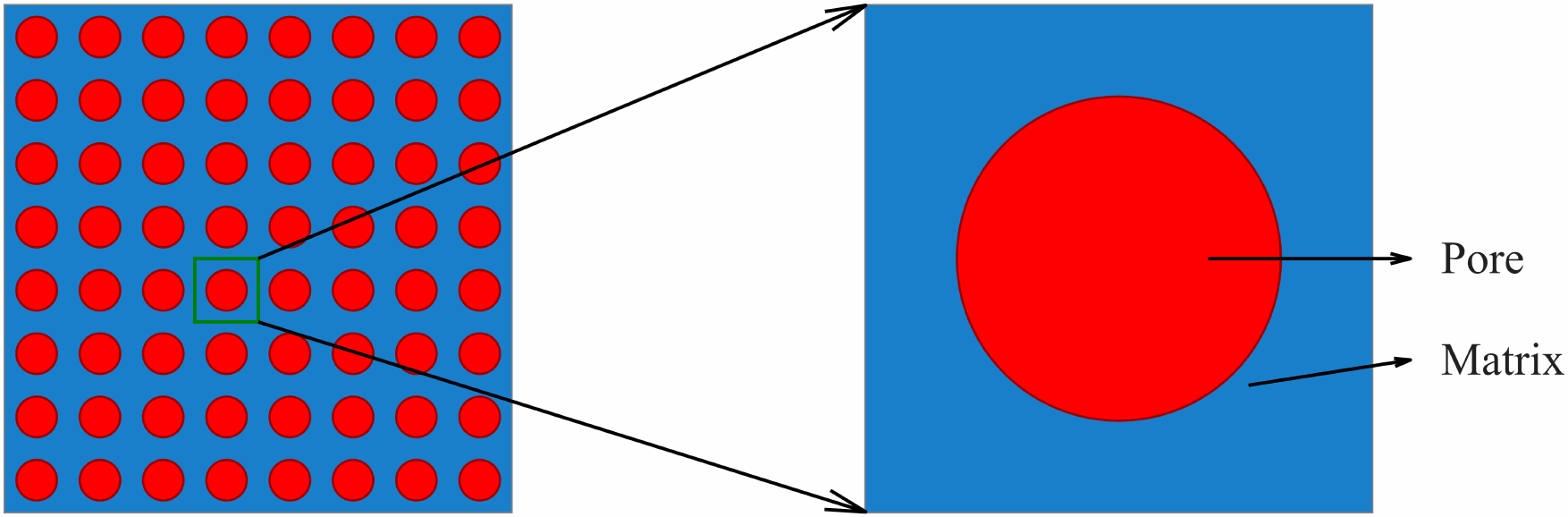}}
\caption{The schematic of 2D periodic media}\label{fig:yuan2d}
\end{figure}

The input parameters for the validation example are given in Table~\ref{tab:materialy2d}. The source term, initial pressure, and boundary conditions for this problem are as follows:
\[
\begin{gathered}
q(\boldsymbol{x},t)=1\times10^{\scriptscriptstyle 5},\ g_{\scriptscriptstyle 1}(\boldsymbol{x})=10,\ g_{\scriptscriptstyle 2}(\boldsymbol{x})=10, \\
u_{\scriptscriptstyle 1}^{\scriptscriptstyle \epsilon}(\boldsymbol{x},t)=10,\ u_{\scriptscriptstyle 2}^{\scriptscriptstyle \epsilon}(\boldsymbol{x},t)=10.
\end{gathered}
\]

\begin{table}[htbp]
\centering
\caption{Input parameters}
\label{tab:materialy2d}
\begin{tabular}{@{} >{\centering\arraybackslash}p{3.9cm} >{\centering\arraybackslash}p{3.9cm} @{}}
\toprule
Material parameters    & Matrix / Porous  \\
\midrule
\( c_{\scriptscriptstyle 1}\)    
    & \( 5 \ / \ 2 \) \\
\( c_{\scriptscriptstyle 2} \)    
    & \( 4.5 \ / \ 1.5 \) \\
\( \kappa_{\scriptscriptstyle 1,ij} \)    
    & \( 100 \ / \ 1 \) \\
\( \kappa_{\scriptscriptstyle 2,ij}\)    
    & \( 50 \ /\ 1 \) \\
\( Q_{\scriptscriptstyle 1} \)    
    & 20 \ / \ 0.25 \\
\( Q_{\scriptscriptstyle 2} \)    
    & 20 \ / \ 0.25 \\
\bottomrule
\end{tabular}
\end{table}

Given that the analytical solution \( u_{\scriptscriptstyle l}^{\varepsilon} \) for problem \eqref{eq:3} is difficult to obtain directly, this paper uses the finite element reference solution \( u_{\scriptscriptstyle l,Fe}^{\scriptscriptstyle\varepsilon} \) computed on an extremely fine mesh as a substitute. By comparing and analyzing this finite element reference solution with the asymptotic solutions of various orders obtained through the two-scale method, the accuracy of the HOMS solution is validated. To this end, tetrahedral mesh partitions are performed for the macroscopic domain, the unit cell domain, and the homogenized domain. Table~\ref{tab:y2d} lists the element and node information for the three sets of meshes, as well as a comparison of computation times between the refined FEM and the HOMS method. 

\begin{table}[htbp]
\centering
\setlength{\tabcolsep}{2pt}
\caption{Summary of computational cost}
\label{tab:y2d}
\begin{tabular}{lccc}
\toprule
 & \textbf{Multiscale eqs.} & \textbf{Cell eqs.} & \textbf{Homogenized eqs.} \\
\midrule
\textbf{FEM Elements}   & 56,064    & 858     & 8,590 \\
\textbf{FEM Nodes}      & 28,353    & 470     & 4,416 \\
\midrule
 & \textbf{FEM}    & \multicolumn{2}{c}{\textbf{HOMS}} \\
\midrule
\textbf{Time}   & 168.284\,s   & \multicolumn{2}{c}{61.164\,s} \\
\bottomrule
\end{tabular}
\end{table}

Set the time step as $\Delta t = 0.02$, then compute the solutions $u_{\scriptscriptstyle l}^{\scriptscriptstyle\varepsilon}$, $u_{\scriptscriptstyle l}^{\scriptscriptstyle(0)}$, $u_{\scriptscriptstyle l}^{\scriptscriptstyle(1\varepsilon)}$ and $u_{\scriptscriptstyle l}^{\scriptscriptstyle(2\varepsilon)}$ of problem~\eqref{eq:3} over the time interval $[0,1]$. Denote the $L^{\scriptscriptstyle 2}$ norm and $H^1$ seminorm as $\|\cdot\|_{\scriptscriptstyle{L^{\scriptscriptstyle 2}(\Omega)}}$ and $|\cdot|_{\scriptscriptstyle{H^{\scriptscriptstyle 1}(\Omega)}}$, respectively.

Fig.~\ref{fig:yuan2du1} and Fig.~\ref{fig:yuan2du2} display the distribution profiles of $u_{\scriptscriptstyle l}^{\scriptscriptstyle(0)}$, $u_{\scriptscriptstyle l}^{\scriptscriptstyle(1\varepsilon)}$, $u_{\scriptscriptstyle l}^{\scriptscriptstyle(2\varepsilon)}$ and $u_{\scriptscriptstyle l}^{\scriptscriptstyle\varepsilon}$ at time $t = 1.0$. Fig.~\ref{fig:Ery2d} shows the evolution of the relative errors in the $L^{\scriptscriptstyle 2}$ norm and $H^{\scriptscriptstyle 1}$ seminorm for $u_{\scriptscriptstyle l}^{\scriptscriptstyle(0)}$, $u_{\scriptscriptstyle l}^{\scriptscriptstyle(1\varepsilon)}$, and $u_{\scriptscriptstyle l}^{\scriptscriptstyle(2\varepsilon)}$ over the time interval $[0,1]$.

\begin{figure}[pos=htbp]
\centering
\subcaptionbox{$u_{\scriptscriptstyle 1}^{\scriptscriptstyle(0)}$\label{fig:y2d_u1A}}
{\includegraphics[width=0.235\textwidth]{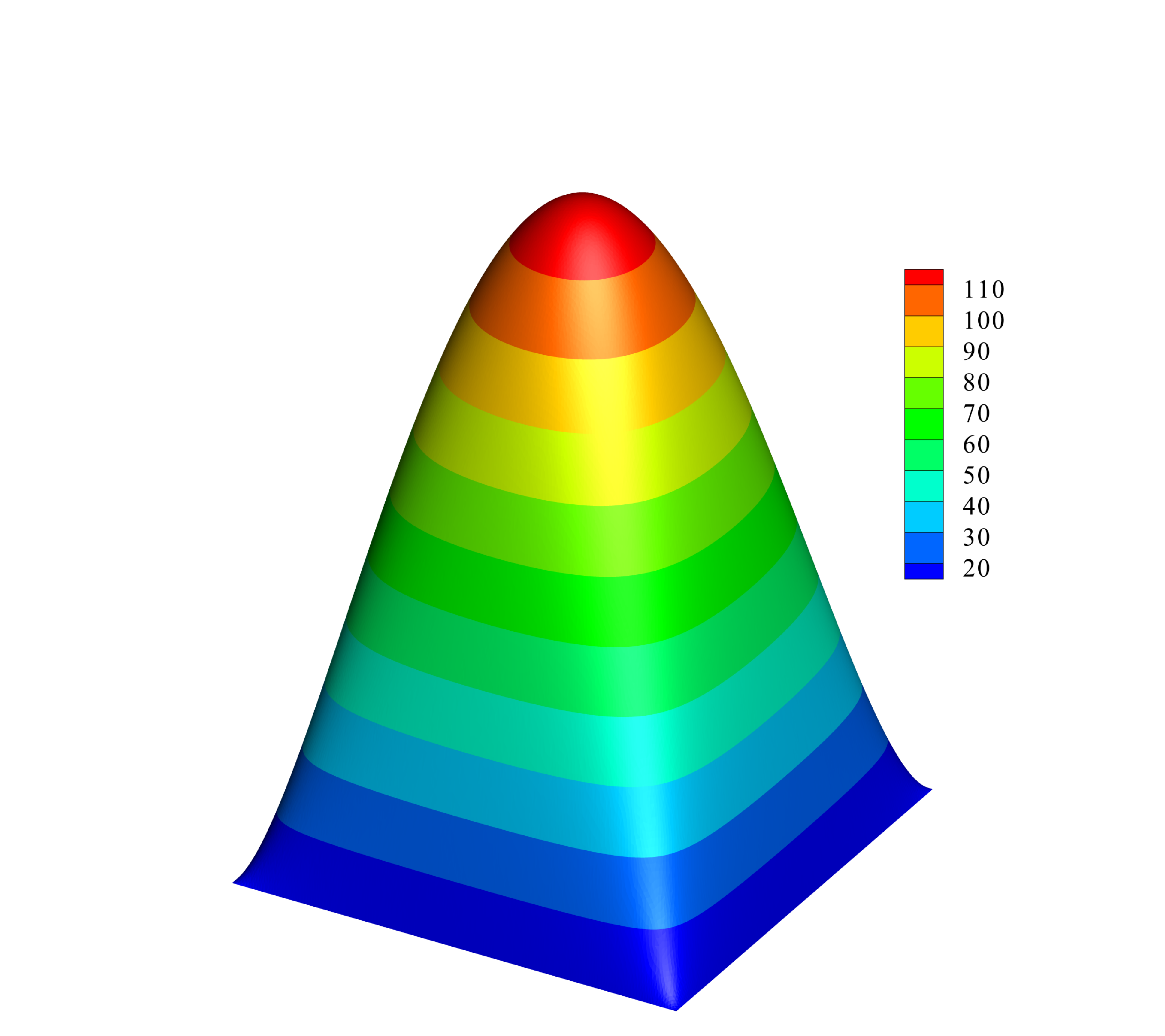}}
\subcaptionbox{$u_{\scriptscriptstyle 1}^{\scriptscriptstyle(1\varepsilon)}$\label{fig:y2d_u1B}}
{\includegraphics[width=0.235\textwidth]{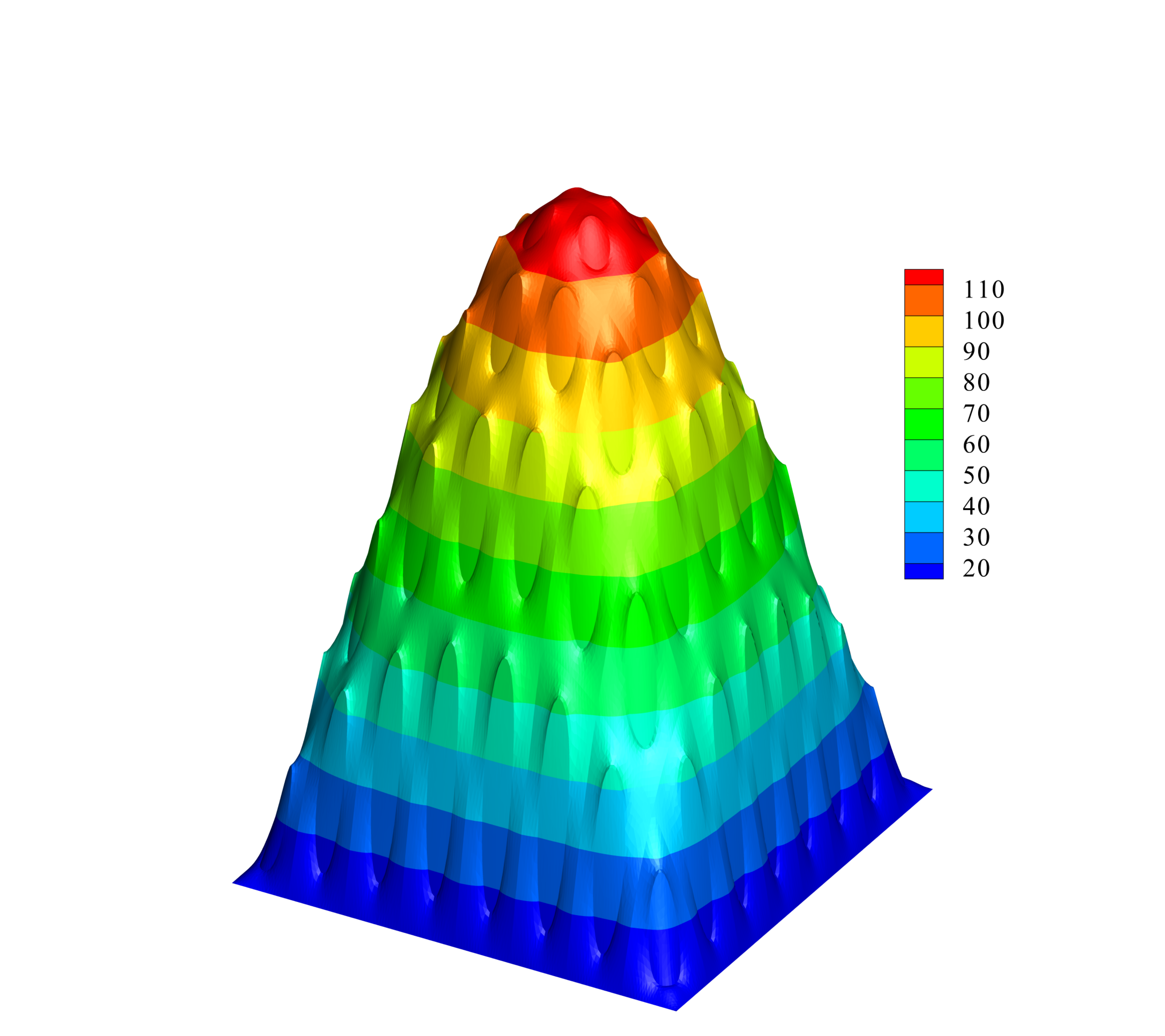}}
\subcaptionbox{$u_{\scriptscriptstyle 1}^{\scriptscriptstyle(2\varepsilon)}$\label{fig:y2d_u1C}}
{\includegraphics[width=0.235\textwidth]{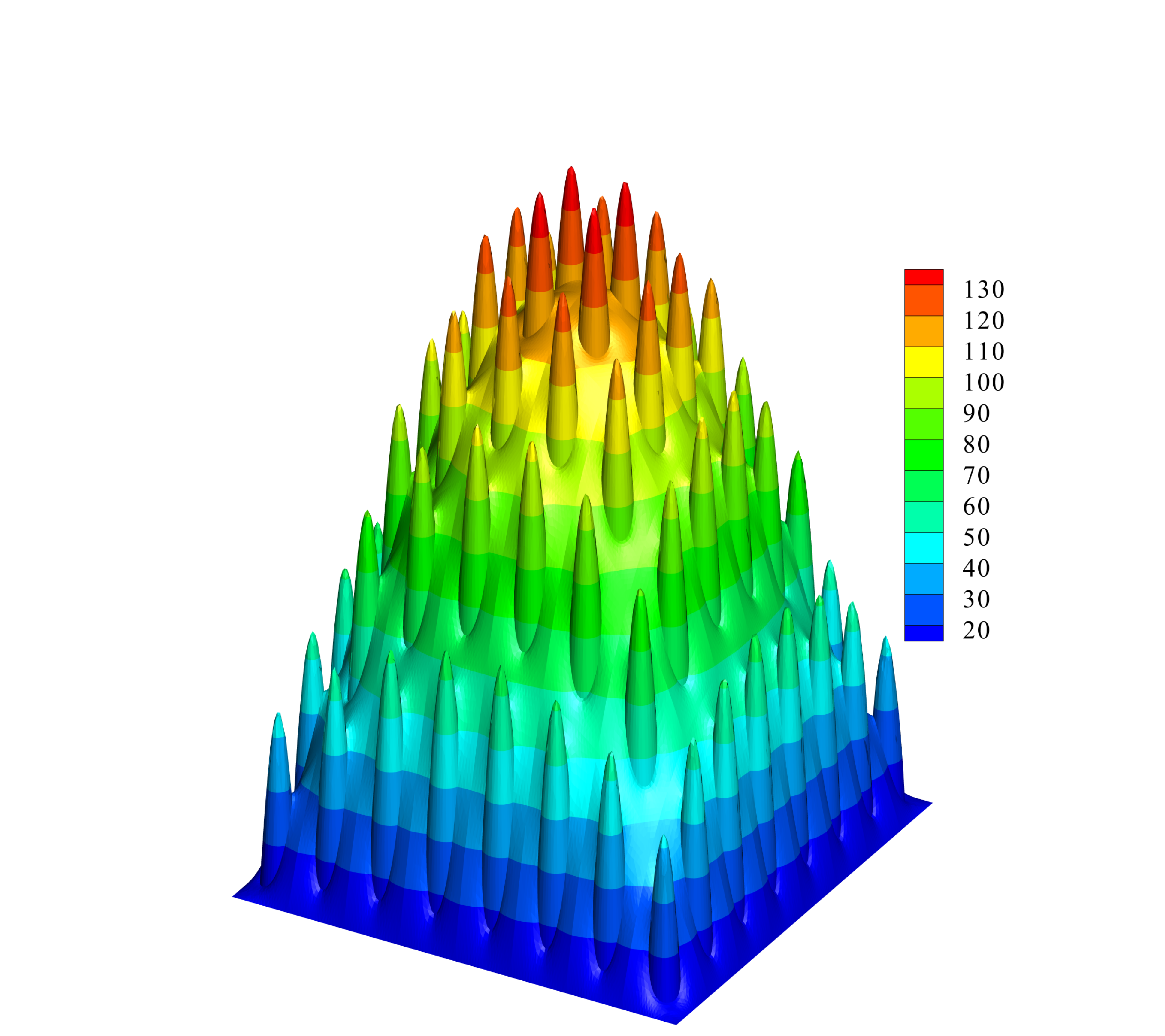}}
\subcaptionbox{$u_{\scriptscriptstyle 1}^{\scriptscriptstyle\varepsilon}$\label{fig:y2d_u1D}}
{\includegraphics[width=0.235\textwidth]{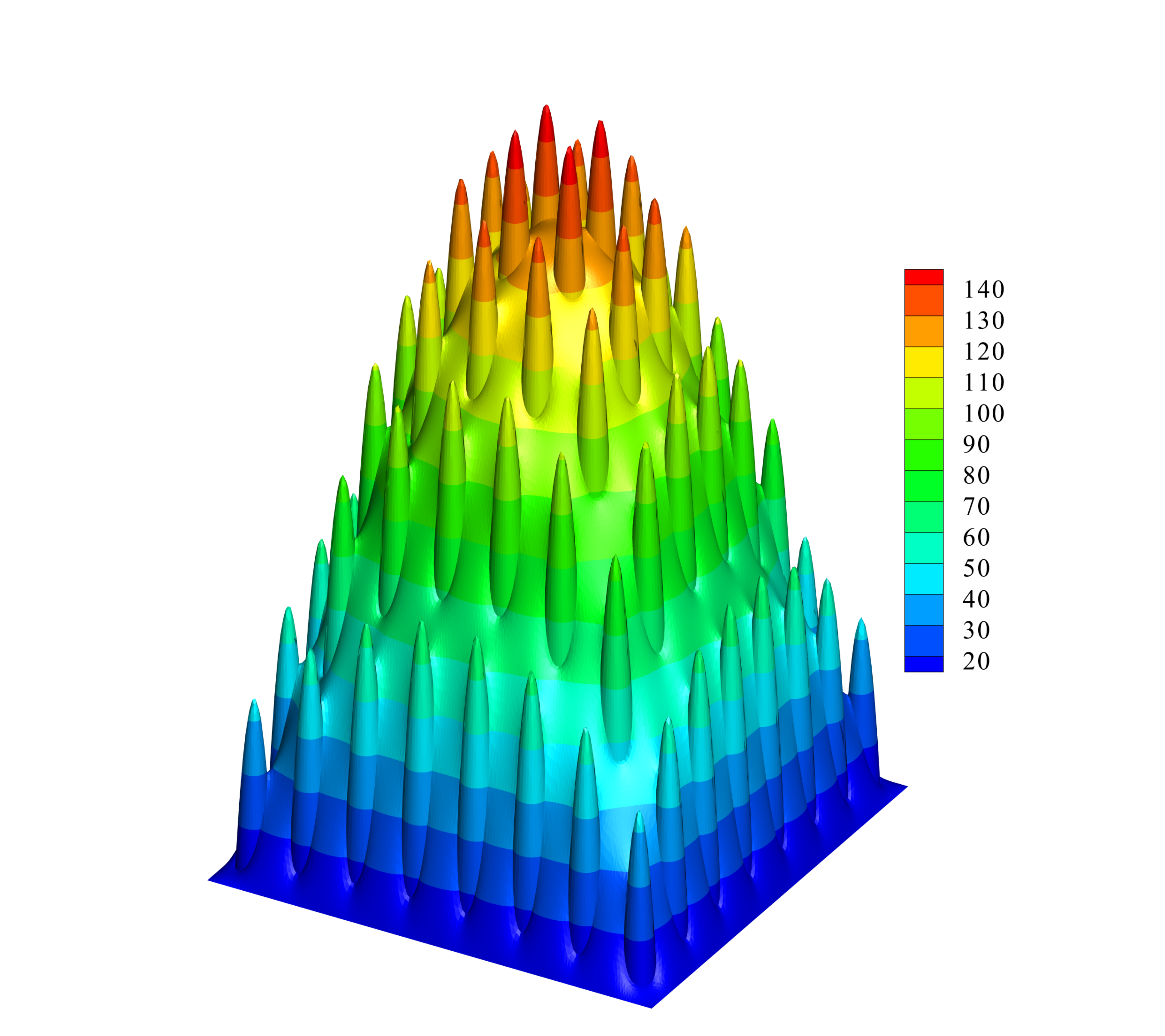}}
\caption{The pressure field at $t = 1.0 $ of first-continuum media.}
\label{fig:yuan2du1}
\end{figure}
\begin{figure}[pos=htbp]
\centering
\subcaptionbox{$u_{\scriptscriptstyle 2}^{\scriptscriptstyle(0)}$\label{fig:y2d_u2A}}
{\includegraphics[width=0.235\textwidth]{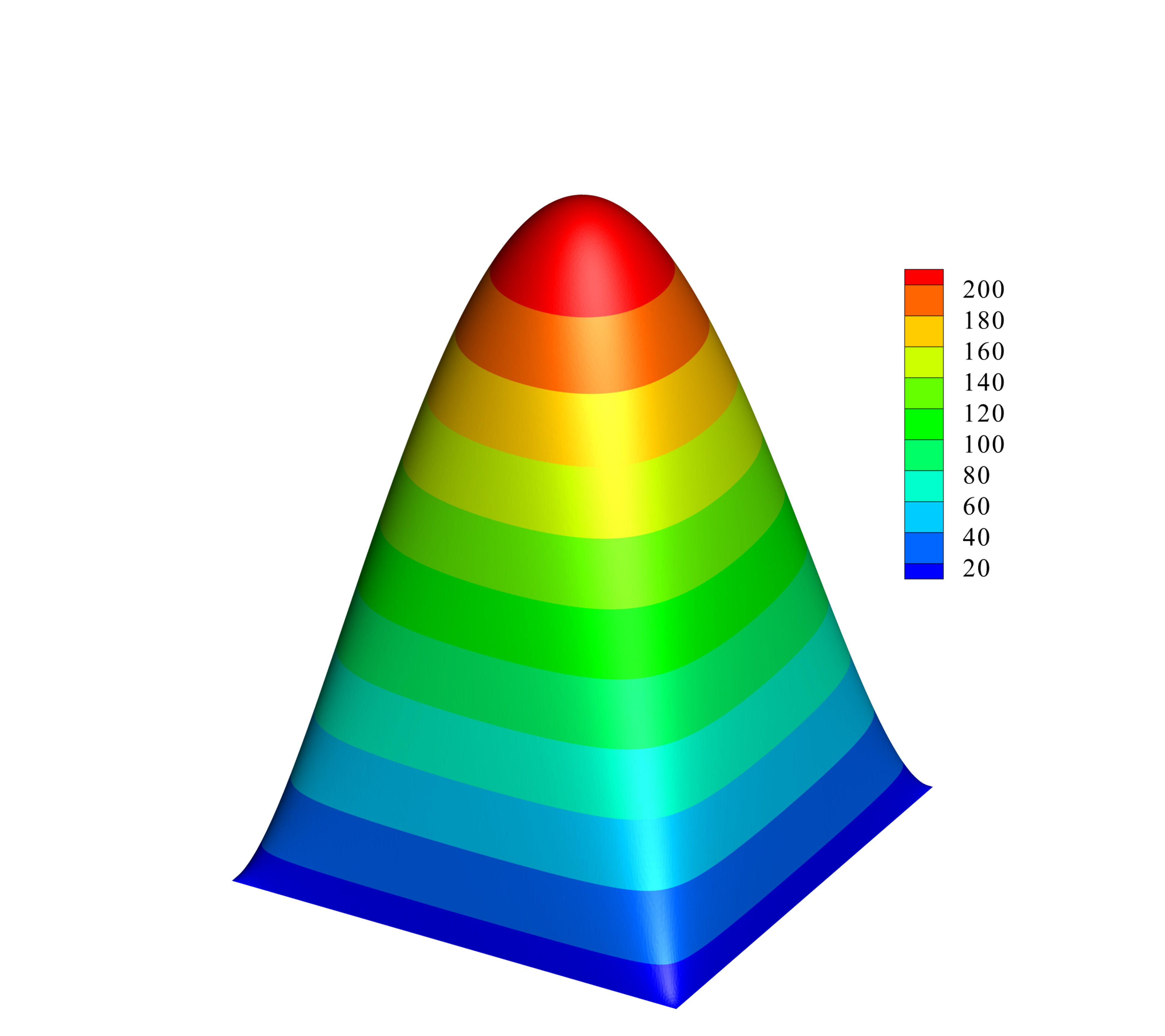}}
\subcaptionbox{$u_{\scriptscriptstyle 2}^{\scriptscriptstyle(1\varepsilon)}$\label{fig:y2d_u2B}}
{\includegraphics[width=0.235\textwidth]{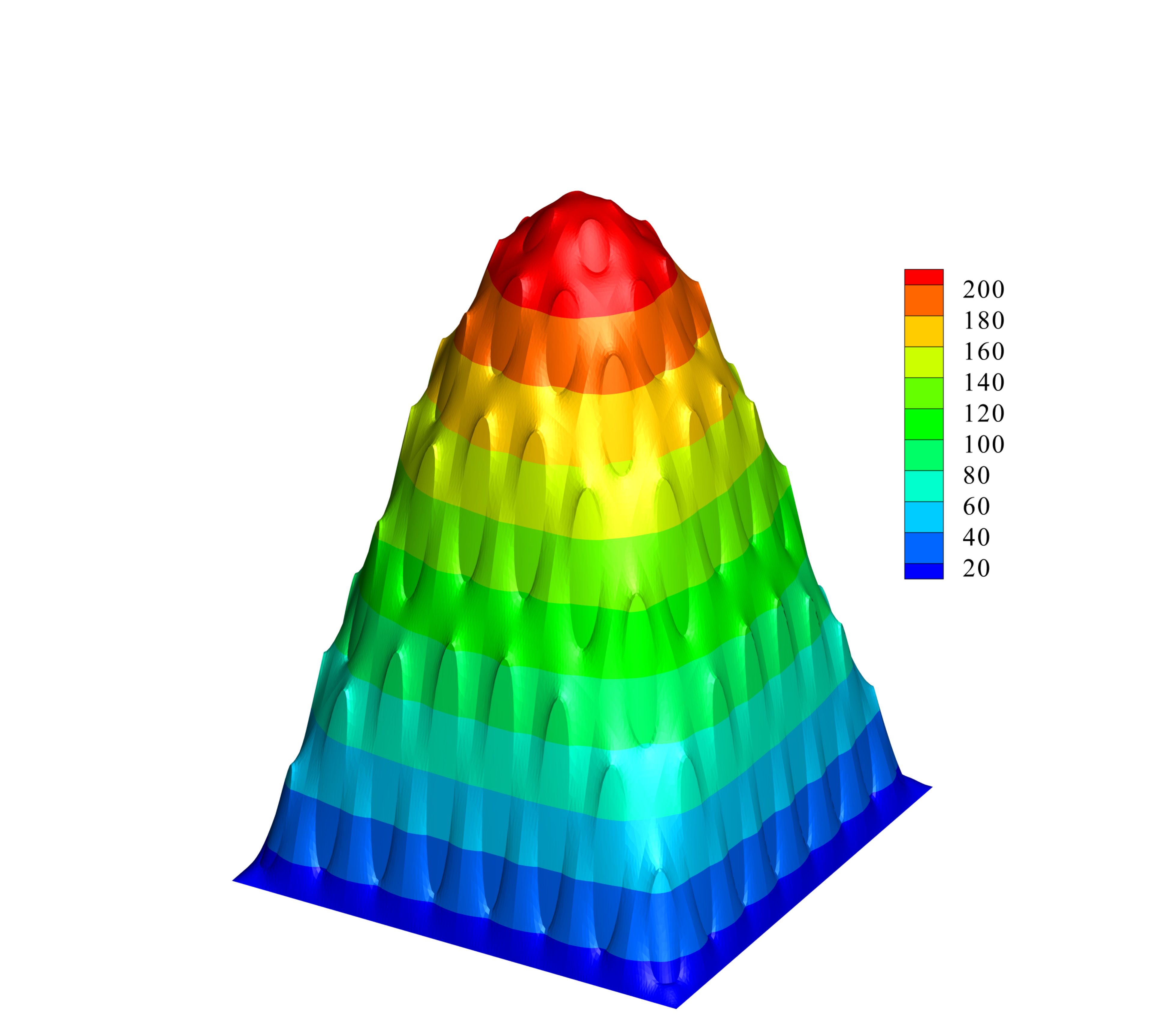}}
\subcaptionbox{$u_{\scriptscriptstyle 2}^{\scriptscriptstyle(2\varepsilon)}$\label{fig:y2d_u2C}}
{\includegraphics[width=0.235\textwidth]{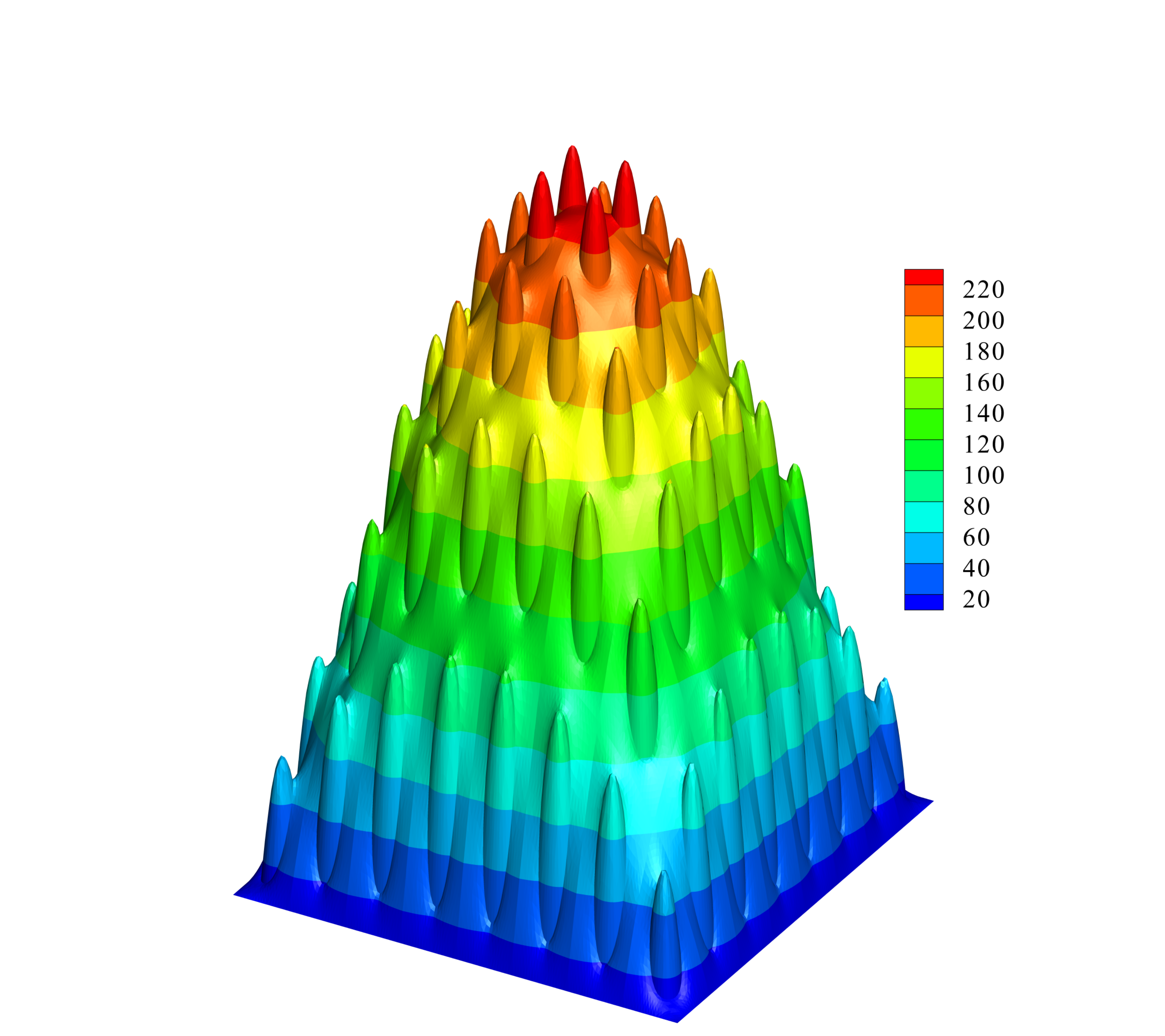}}
\subcaptionbox{$u_{\scriptscriptstyle 2}^{\scriptscriptstyle\varepsilon}$\label{fig:y2d_u2D}}
{\includegraphics[width=0.235\textwidth]{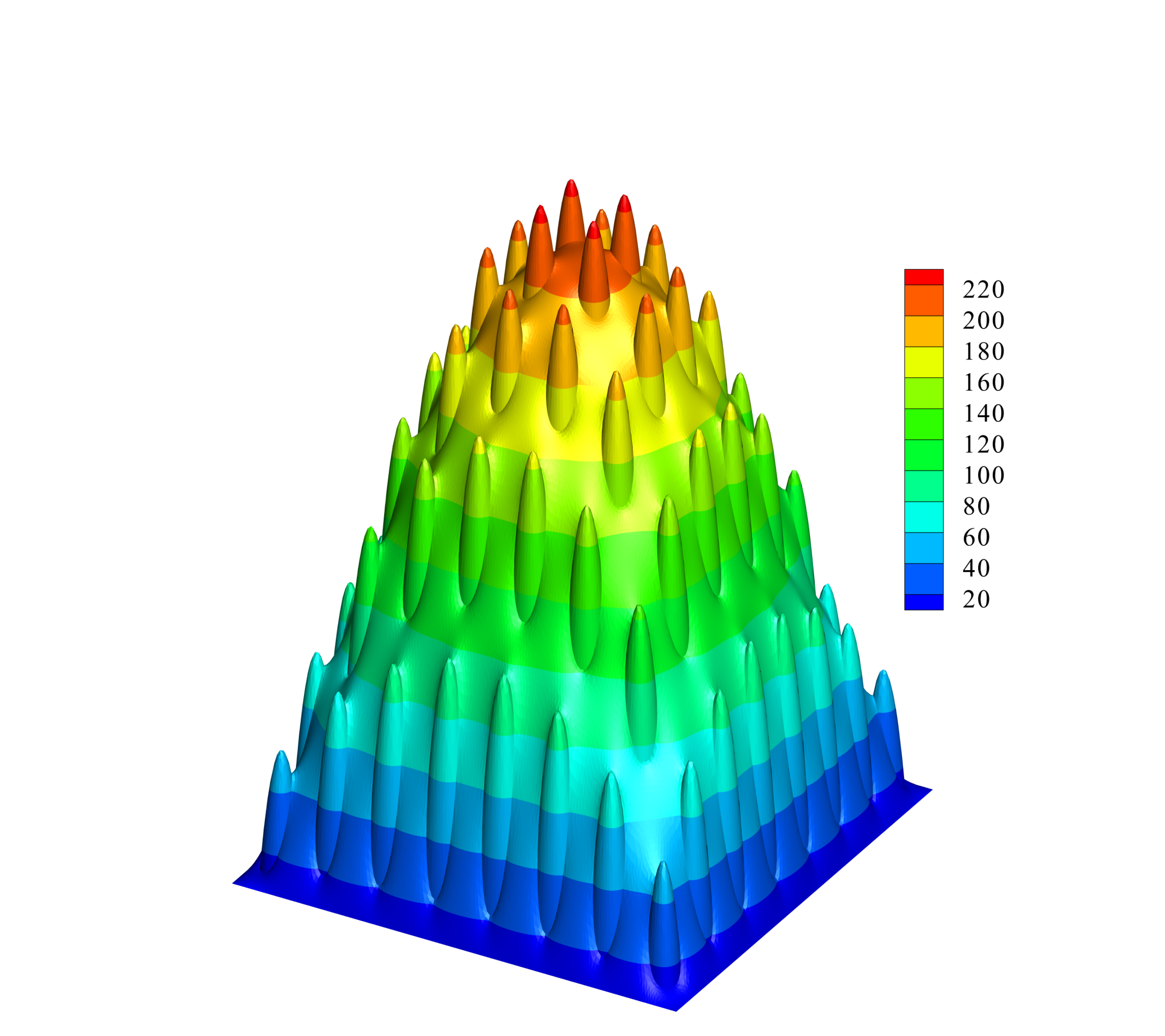}}
\caption{The pressure field at $t = 1.0 $ of second-continuum media.}
\label{fig:yuan2du2}
\end{figure}

\begin{figure}[pos=htbp]
\centering
\subcaptionbox{\label{fig:y2derA}}
{\includegraphics[width=0.233\textwidth]{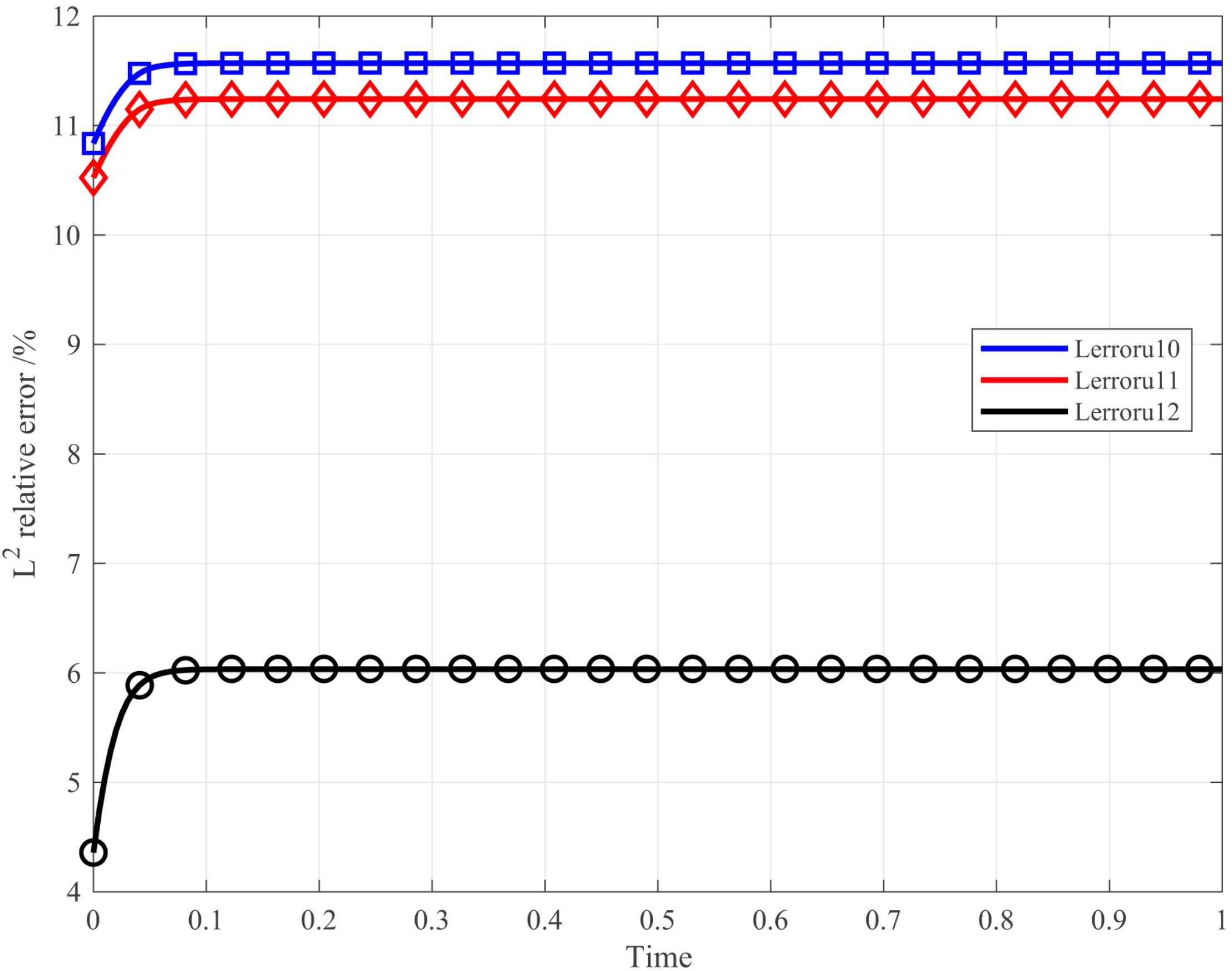}}
\subcaptionbox{\label{fig:y2derB}}
{\includegraphics[width=0.233\textwidth]{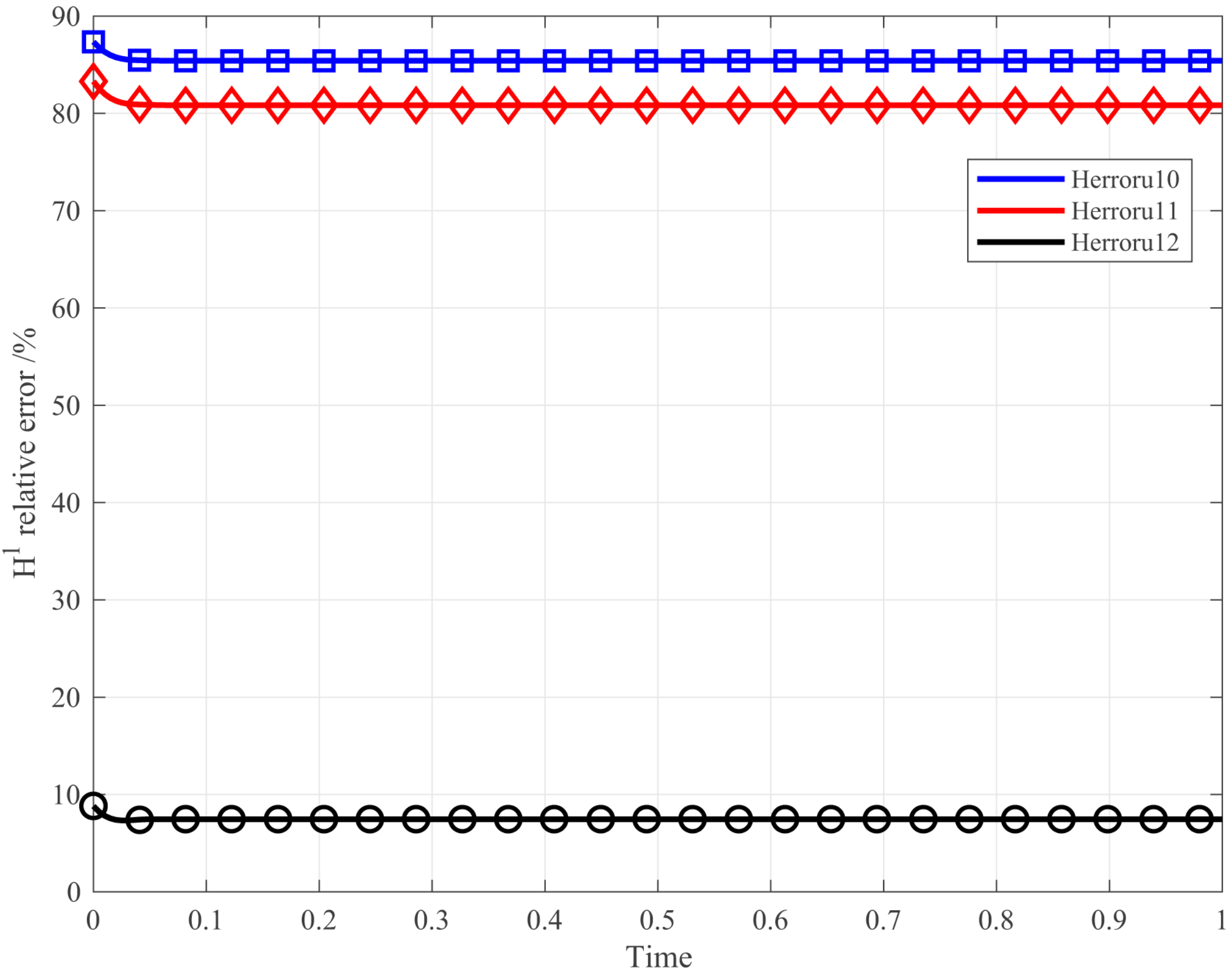}}
\subcaptionbox{\label{fig:y2derC}}
{\includegraphics[width=0.233\textwidth]{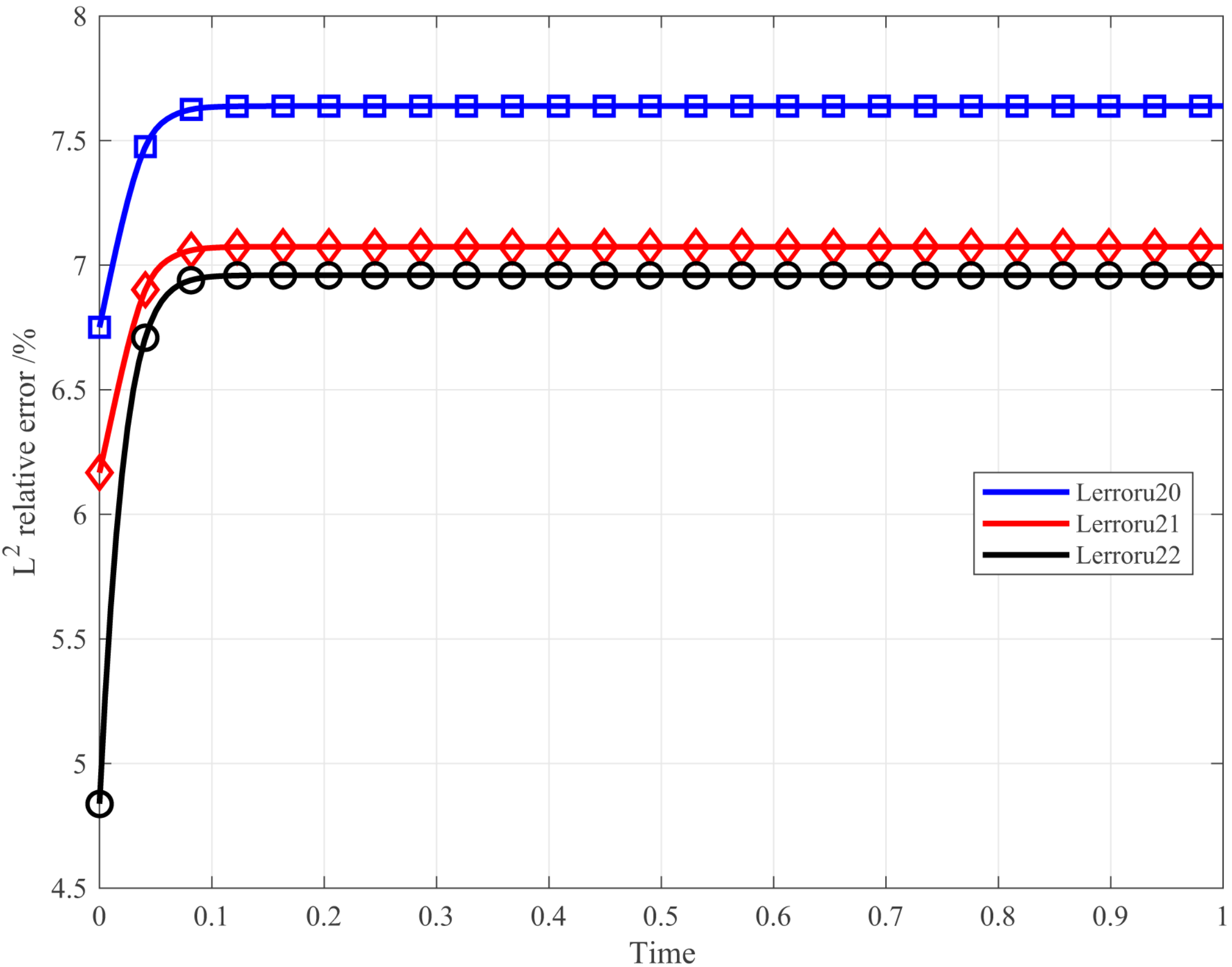}}
\subcaptionbox{\label{fig:y2derD}}
{\includegraphics[width=0.233\textwidth]{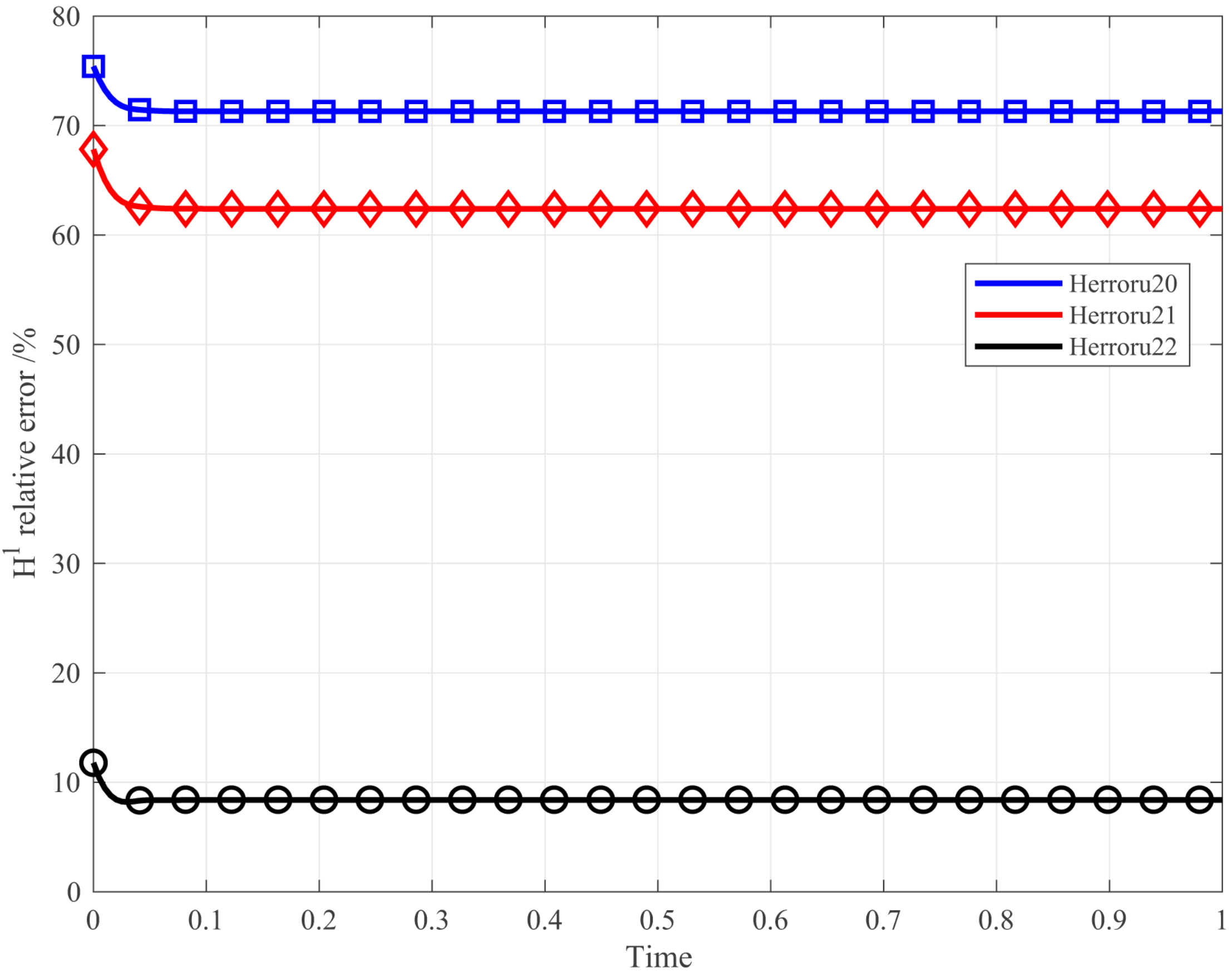}}
\caption{Evolution of the relative errors of the $L^{\scriptscriptstyle 2}$ norm and $H^{\scriptscriptstyle 1}$ seminorm over time: (a) $\text{Lerr1}$; (b) $\text{Herr1}$; (c) $\text{Lerr2}$; (d) $\text{Herr2}$.}\label{fig:Ery2d}
\end{figure}

From the mesh information in Table~\ref{tab:y2d}, it can be observed that the computational grid resource overhead required by the second-order two-scale method is significantly smaller than that of the direct FEM, and its efficiency is also markedly superior to the FEM. As shown in Fig.~\ref{fig:yuan2du1} and Fig.~\ref{fig:yuan2du2}, the approximation accuracy of the HOMS solution is significantly better than that of both the homogenized solution and the FOMS solution. Only the HOMS solution effectively captures the local oscillatory behavior of the highly heterogeneous media. From Fig.~\ref{fig:Ery2d}, it can be observed that under the \( L^{\scriptscriptstyle 2} \) norm, the relative errors of \( u_{\scriptscriptstyle 1}^{\scriptscriptstyle(2\varepsilon)} \) and \( u_{\scriptscriptstyle 2}^{\scriptscriptstyle(2\varepsilon)} \) are only about 6\% and 6\%, respectively; under the \( H^{\scriptscriptstyle 1} \) seminorm, the relative errors of \( u_{\scriptscriptstyle 1}^{\scriptscriptstyle(2\varepsilon)} \) and \( u_{\scriptscriptstyle 2}^{\scriptscriptstyle(2\varepsilon)} \) are only approximately 8\% and 9\%, significantly lower than those of the homogenized solution and the FOMS solution, thus meeting the requirements for engineering computations. Furthermore, as shown in Fig.~\ref{fig:Ery2d}, the HOMS method exhibits excellent numerical stability over time and can be effectively applied to the computation of dynamic problem \eqref{eq:3} that evolves with time.

\subsection{Example 2. 3D porous media}
For the three-dimensional case of problem \eqref{eq:3}, consider the macroscopic domain \(\Omega = (x_1, x_2, x_3) = [0, 1]^3\) and the microscopic unit cell \(Y = (y_1, y_2, y_3) = [0, 1]^3\) as shown in Fig.~\ref{fig:yuan3d}, where \(\varepsilon = 1/5\).

\begin{figure}[pos=htbp]
\centering
{\includegraphics[width=0.482\textwidth]{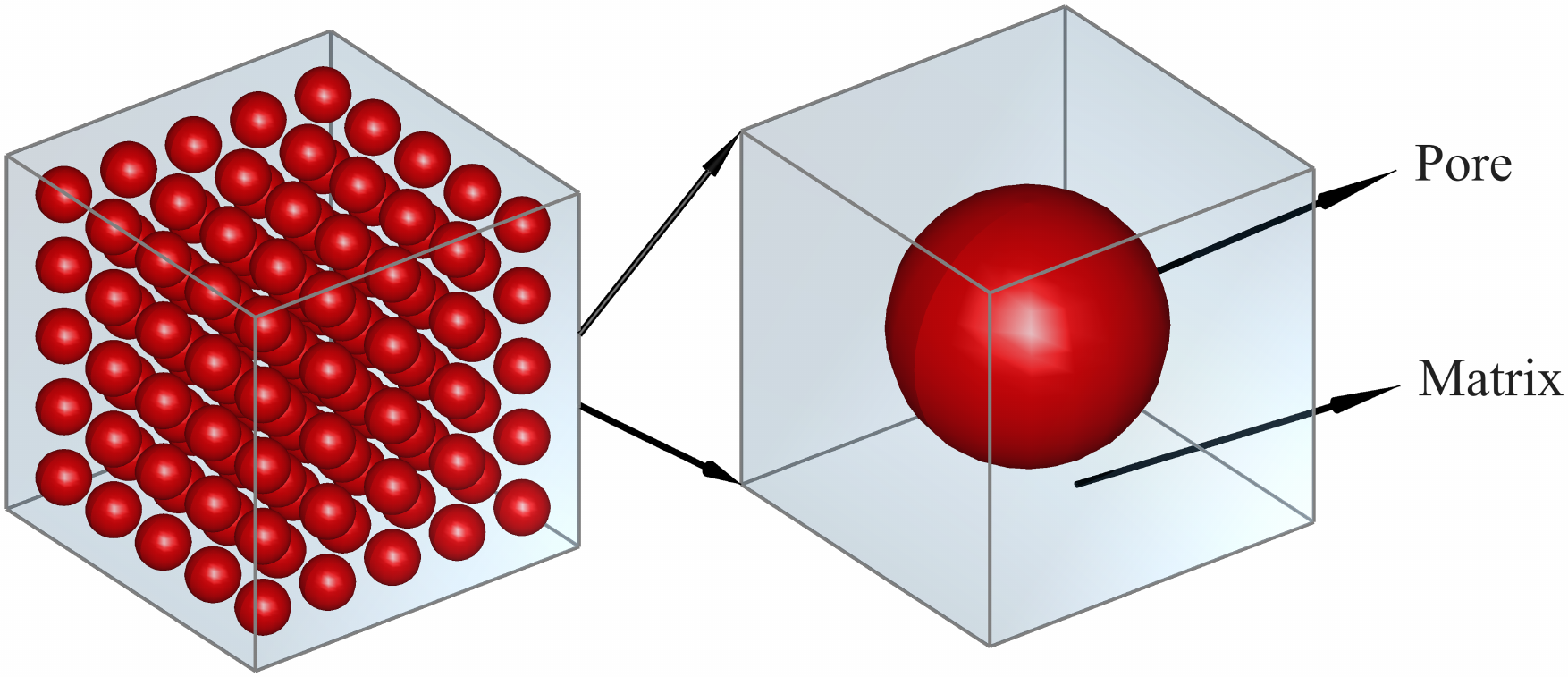}}
\caption{The schematic of 3D periodic media}
\label{fig:yuan3d}
\end{figure}
The input parameters for the validation example are given in Table~\ref{tab:materialy3d}. The source term, initial pressure, and boundary conditions for this problem are as follows:
\[
\begin{gathered}
q(\boldsymbol{x},t)=1\times10^{\scriptscriptstyle 5},\ g_{\scriptscriptstyle 1}(\boldsymbol{x})=60,\ g_{\scriptscriptstyle 2}(\boldsymbol{x})=60, \\
u_{\scriptscriptstyle 1}^{\scriptscriptstyle \epsilon}(\boldsymbol{x},t)=60,\ u_{\scriptscriptstyle 2}^{\scriptscriptstyle \epsilon}(\boldsymbol{x},t)=60.
\end{gathered}
\]

\begin{table}[htbp]
\centering
\caption{Input parameters}
\label{tab:materialy3d}
\begin{tabular}{@{} >{\centering\arraybackslash}p{3.9cm} >{\centering\arraybackslash}p{3.9cm} @{}}
\toprule
Material parameters    & Matrix / Channel  \\
\midrule
\( c_{\scriptscriptstyle 1} \)    
    & \( 50 \ / \ 25 \) \\
\( c_{\scriptscriptstyle 2} \)    
    & \( 125 \ / \ 15 \) \\
\( \kappa_{\scriptscriptstyle 1,ij} \)    
    & \( 800 \ / \ 30 \) \\
\( \kappa_{\scriptscriptstyle 2,ij} \)    
    & \( 1000 \ /\ 50 \) \\
\( Q_{\scriptscriptstyle 1} \)    
    & \( 250 \ /\ 25 \) \\
\( Q_{\scriptscriptstyle 2} \)    
    & \( 250 \ /\ 25 \) \\
\bottomrule
\end{tabular}
\end{table}

Given that the analytical solution \( u_{\scriptscriptstyle l}^{\varepsilon} \) for problem \eqref{eq:3} is difficult to obtain directly, this paper uses the finite element reference solution \( u_{\scriptscriptstyle l,Fe}^{\scriptscriptstyle\varepsilon} \) computed on an extremely fine mesh as a substitute. By comparing and analyzing this finite element reference solution with the asymptotic solutions of various orders obtained through the two-scale method, the accuracy of the HOMS solution is validated. To this end, tetrahedral mesh partitions are performed for the macroscopic domain, the unit cell domain, and the homogenized domain. Table~\ref{tab:y3d} lists the element and node information for the three sets of meshes, as well as a comparison of computation times between the refined FEM and the HOMS method. 

\begin{table}[htbp]
\centering
\setlength{\tabcolsep}{2pt}
\caption{Summary of computational cost}
\label{tab:y3d}
\begin{tabular}{lccc}
\toprule
 & \textbf{Multiscale eqs.} & \textbf{Cell eqs.} & \textbf{Homogenized eqs.} \\
\midrule
\textbf{FEM elements}   & 565,341    & 83,536     & 93,750 \\
\textbf{FEM nodes}      & 93,714     & 16,914     & 17,576 \\
\midrule
 & \textbf{FEM}    & \multicolumn{2}{c}{\textbf{FOMS}} \\
\midrule
\textbf{time}   & 8970.49\,s    & \multicolumn{2}{c}{2629.17\,s} \\
\bottomrule
\end{tabular}
\end{table}
Set the time step as $\Delta t = 0.002$, then compute the solutions $u_{\scriptscriptstyle l}^{\scriptscriptstyle\varepsilon}$, $u_{\scriptscriptstyle l}^{\scriptscriptstyle(0)}$, $u_{\scriptscriptstyle l}^{\scriptscriptstyle(1\varepsilon)}$ and $u_{\scriptscriptstyle l}^{\scriptscriptstyle(2\varepsilon)}$ of problem~\eqref{eq:3} over the time interval $[0,1]$. Denote the $L^{\scriptscriptstyle 2}$ norm and $H^{\scriptscriptstyle 1}$ seminorm as $\|\cdot\|_{L^{\scriptscriptstyle 2}(\Omega)}$ and $|\cdot|_{H^{\scriptscriptstyle 1}(\Omega)}$, respectively.

Fig.~\ref{fig:yuan3du1} and Fig.~\ref{fig:yuan3du2} display the distribution profiles of $u_{\scriptscriptstyle l}^{\scriptscriptstyle(0)}$, $u_{\scriptscriptstyle l}^{\scriptscriptstyle(1\varepsilon)}$, $u_{\scriptscriptstyle l}^{\scriptscriptstyle(2\varepsilon)}$ and $u_{\scriptscriptstyle l}^{\scriptscriptstyle\varepsilon}$ at time $t = 1.0$. 
Fig.~\ref{fig:Ery2d} shows the evolution of the relative errors in the $L^{\scriptscriptstyle 2}$ norm and $H^{\scriptscriptstyle 1}$ seminorm for $u_{\scriptscriptstyle l}^{\scriptscriptstyle(0)}$, $u_{\scriptscriptstyle l}^{\scriptscriptstyle(1\varepsilon)}$, and $u_{\scriptscriptstyle l}^{\scriptscriptstyle(2\varepsilon)}$ over the time interval $[0,1]$.

\begin{figure}[pos=htbp]
\centering
\subcaptionbox{$u_{\scriptscriptstyle 1}^{\scriptscriptstyle(0)}$\label{fig:y3d_u1A}}
{\includegraphics[width=0.23\textwidth]{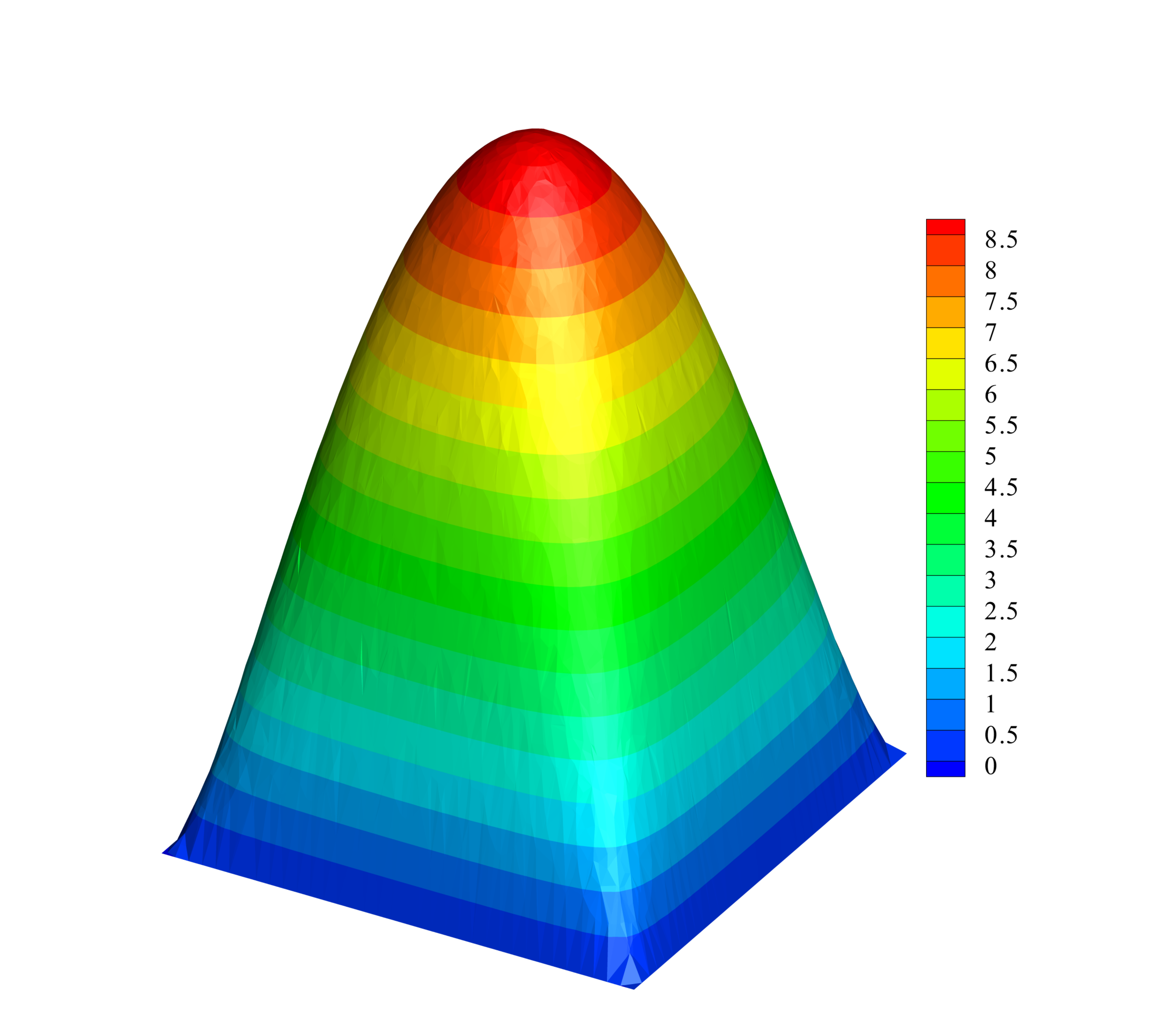}}
\subcaptionbox{$u_{\scriptscriptstyle 1}^{\scriptscriptstyle(1\varepsilon)}$\label{fig:y3d_u1B}}
{\includegraphics[width=0.23\textwidth]{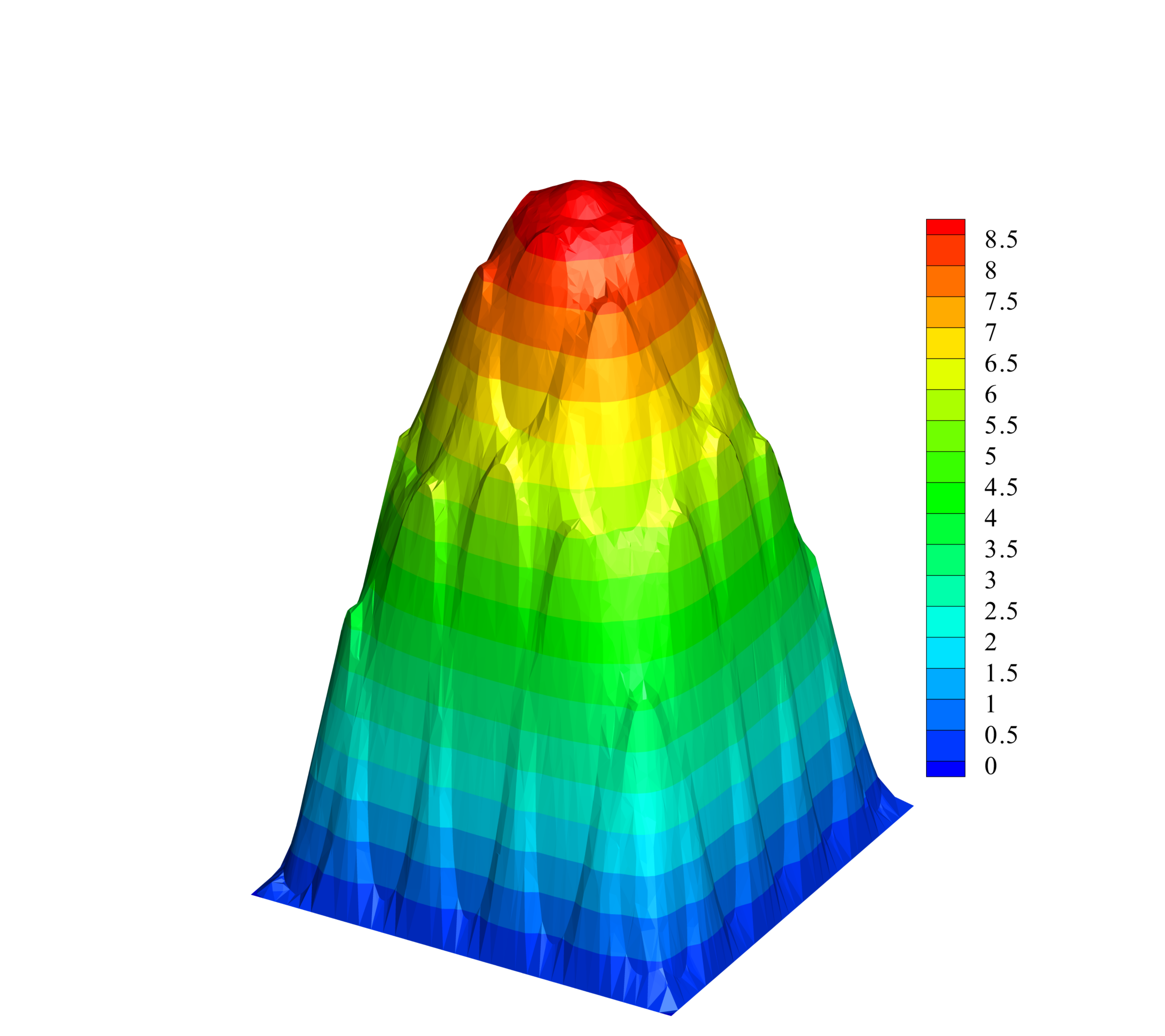}}
\subcaptionbox{$u_{\scriptscriptstyle 1}^{\scriptscriptstyle(2\varepsilon)}$\label{fig:y3d_u1C}}
{\includegraphics[width=0.23\textwidth]{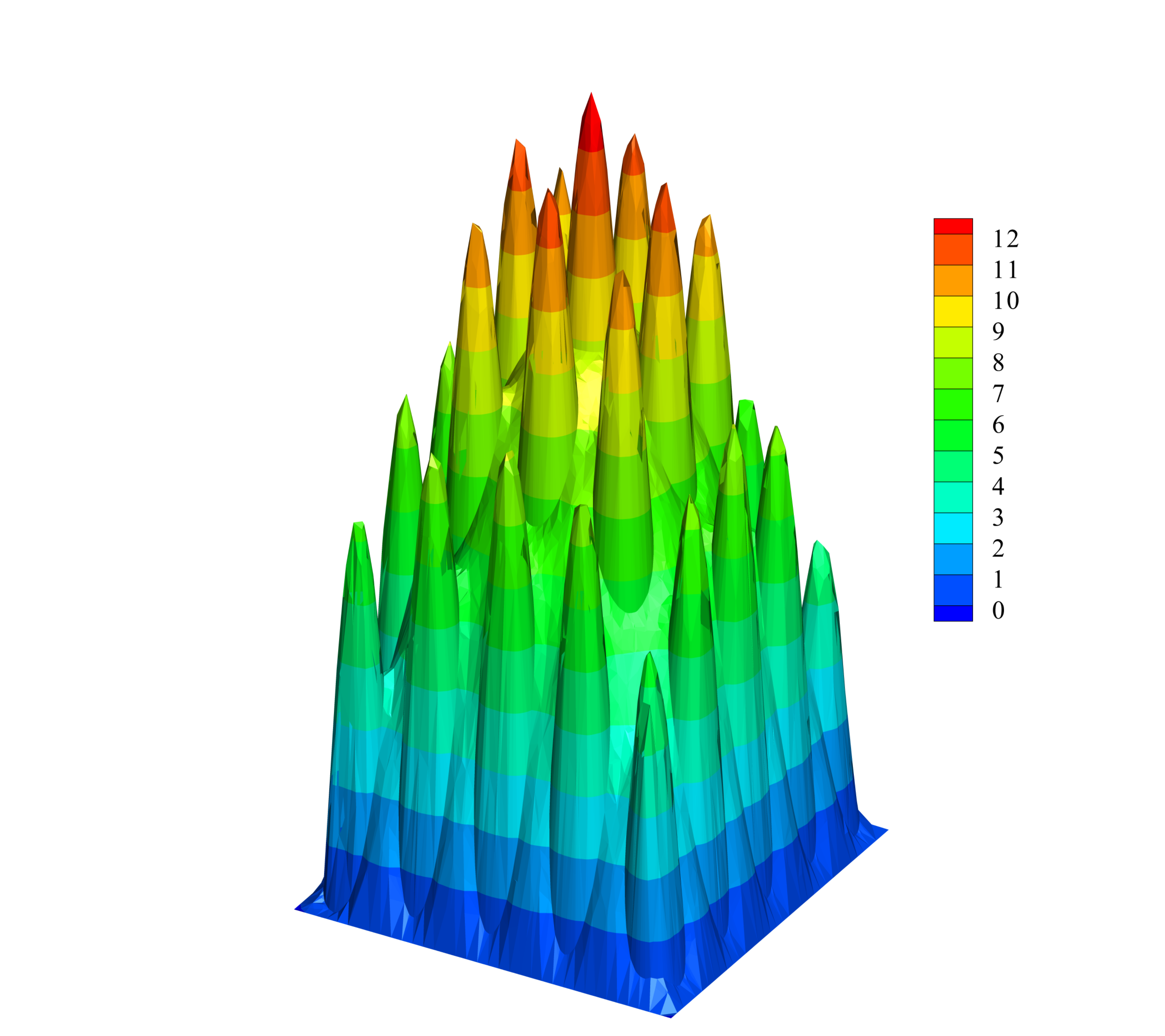}}
\subcaptionbox{$u_{\scriptscriptstyle 1}^{\scriptscriptstyle\varepsilon}$\label{fig:y3d_u1D}}
{\includegraphics[width=0.23\textwidth]{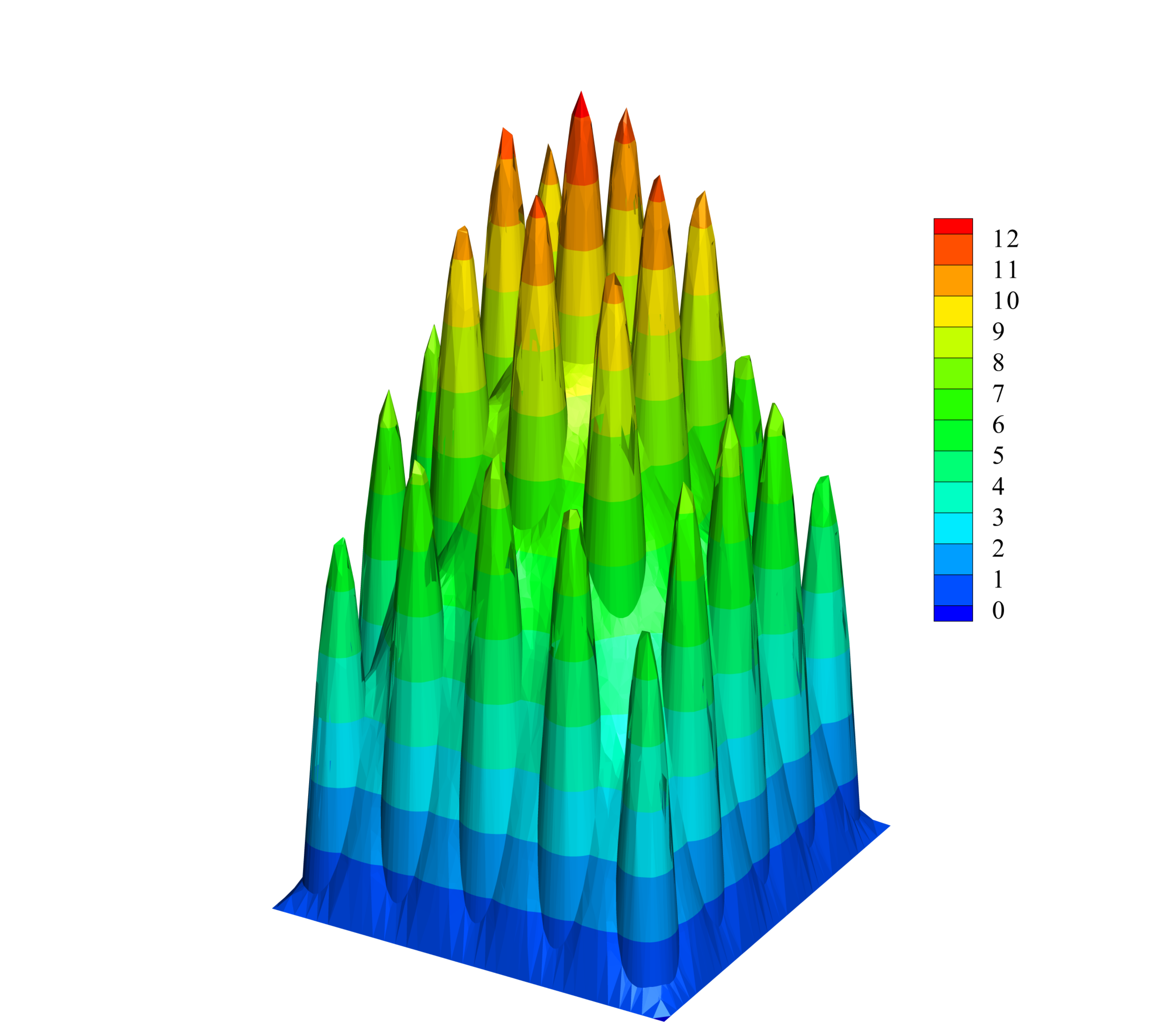}}
\caption{The pressure field in $x_{\scriptscriptstyle 3} = 0.5 $ at $t = 1.0 $ of first-continuum media.}\label{fig:yuan3du1}
\end{figure}
\begin{figure}[pos=htbp]
\centering
\subcaptionbox{$u_{\scriptscriptstyle 2}^{\scriptscriptstyle(0)}$\label{fig:y3d_u2A}}
{\includegraphics[width=0.23\textwidth]{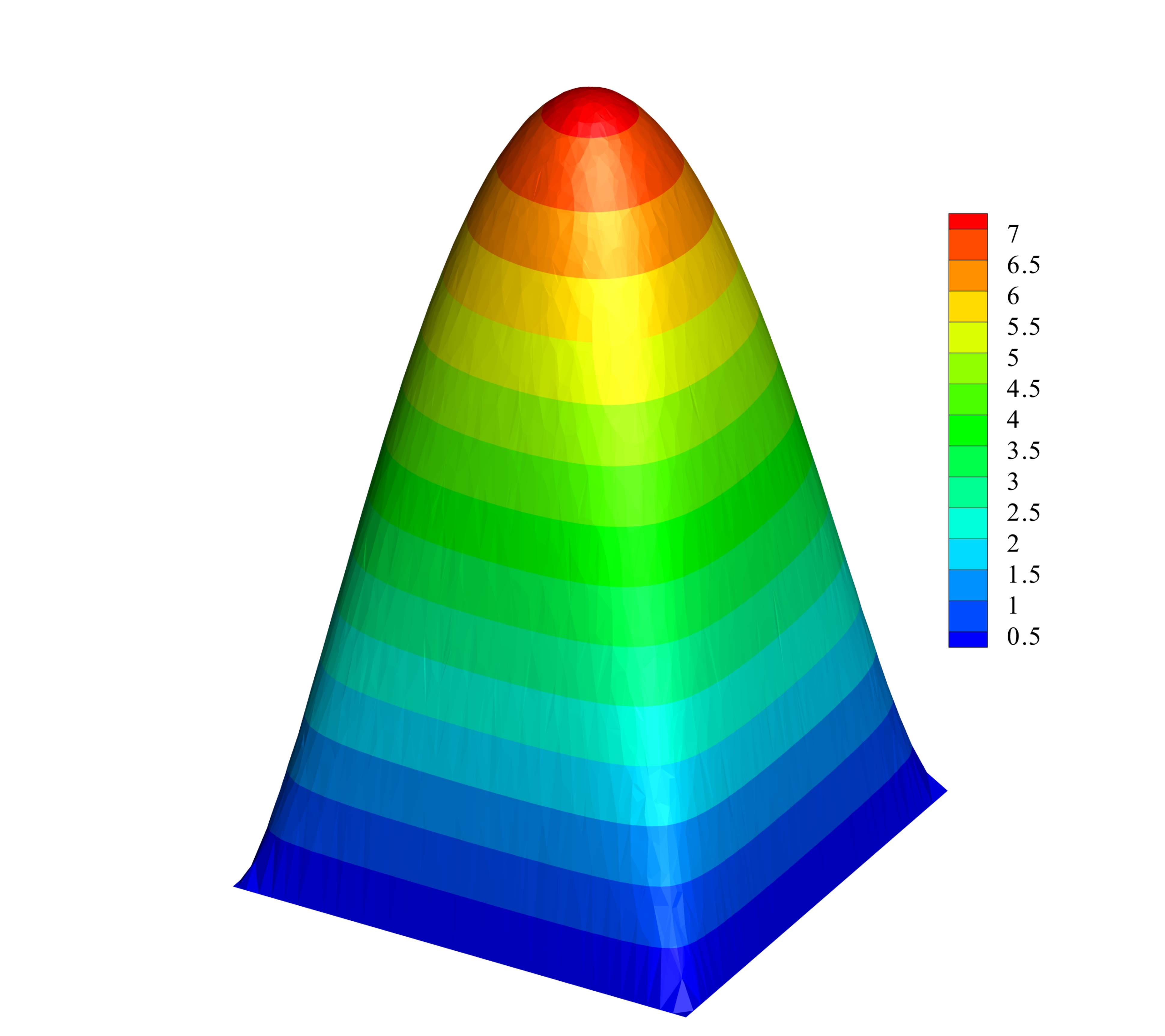}}
\subcaptionbox{$u_{\scriptscriptstyle 2}^{\scriptscriptstyle(1\varepsilon)}$\label{fig:y3d_u2B}}
{\includegraphics[width=0.23\textwidth]{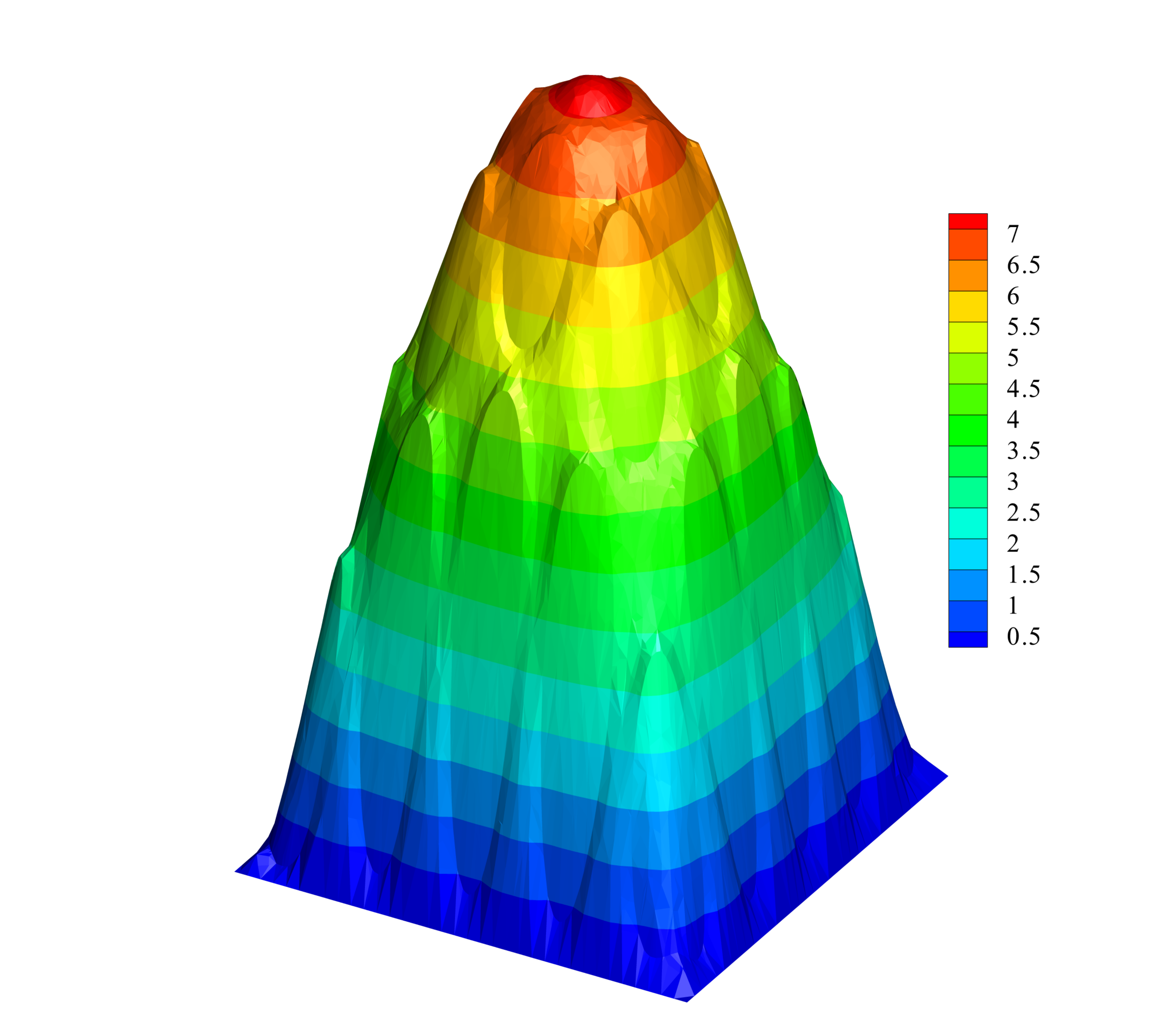}}
\subcaptionbox{$u_{\scriptscriptstyle 2}^{\scriptscriptstyle(2\varepsilon)}$\label{fig:y3d_u2C}}
{\includegraphics[width=0.23\textwidth]{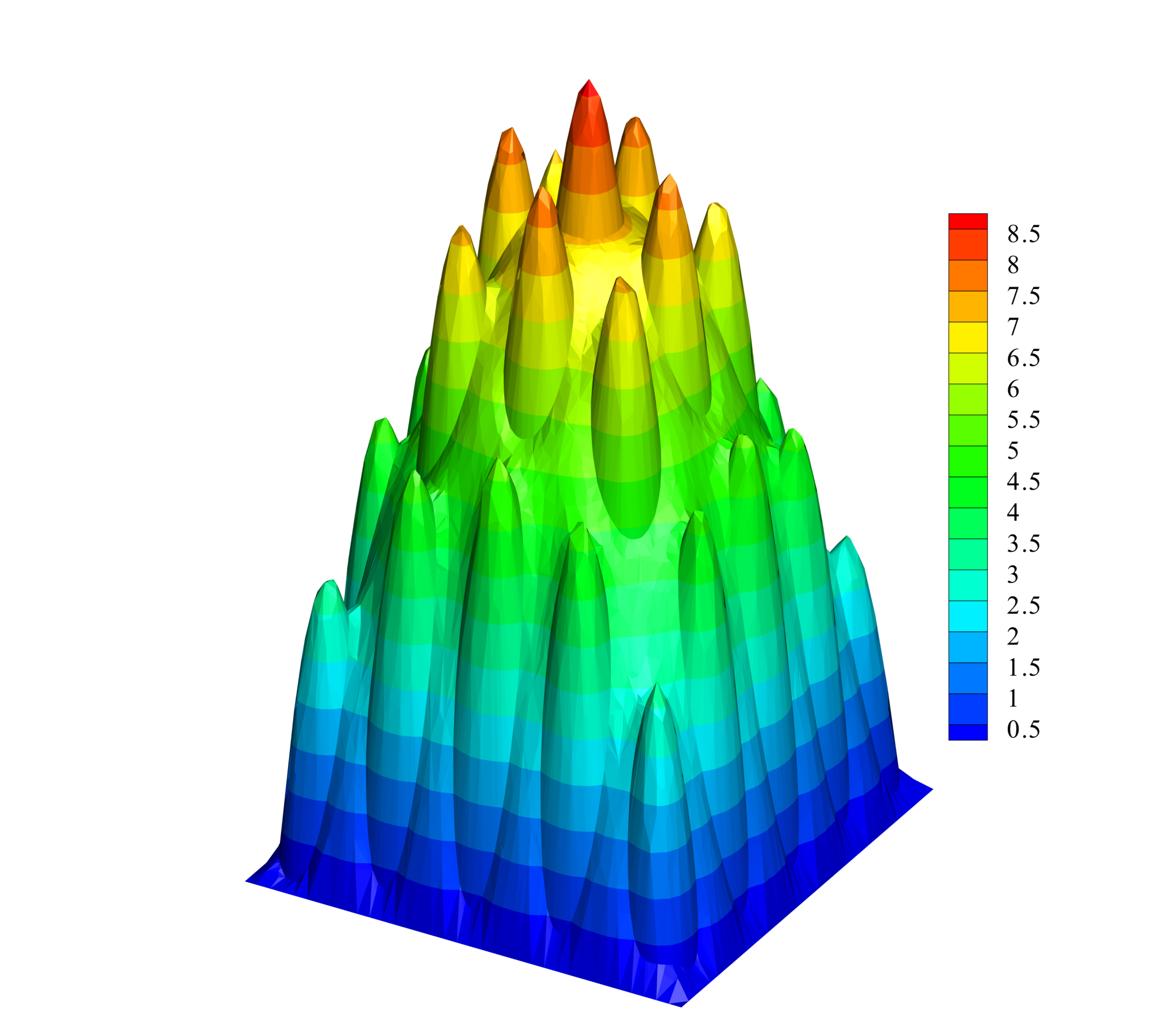}}
\subcaptionbox{$u_{\scriptscriptstyle 2}^{\scriptscriptstyle\varepsilon}$\label{fig:y3d_u2D}}
{\includegraphics[width=0.23\textwidth]{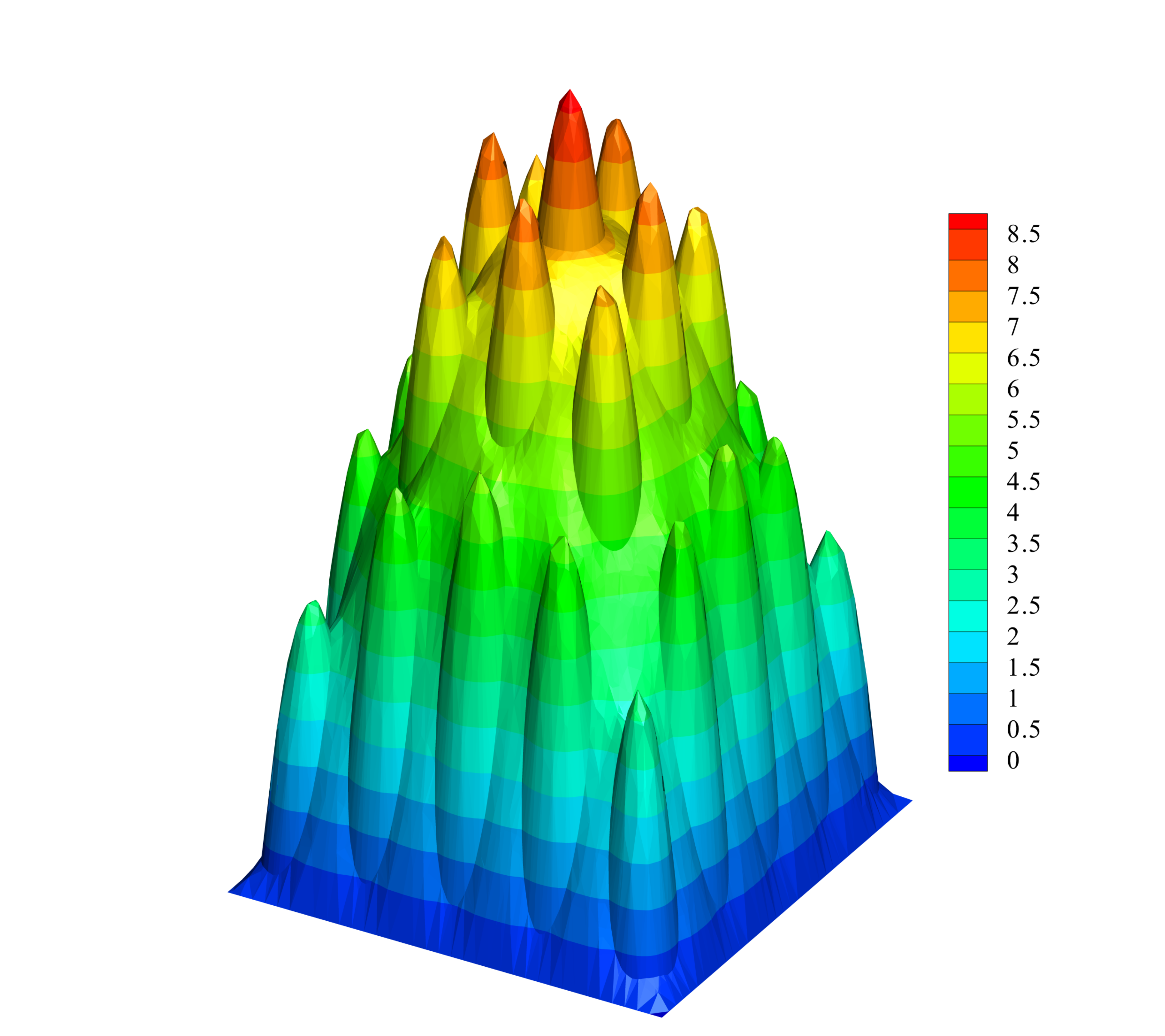}}
\caption{The pressure field in $x_{\scriptscriptstyle 3} = 0.5 $ at $t = 1.0 $ of second-continuum media.}\label{fig:yuan3du2}
\end{figure}

\begin{figure}[pos=htbp]
\centering
\subcaptionbox{\label{fig:y3derA}}
{\includegraphics[width=0.23\textwidth]{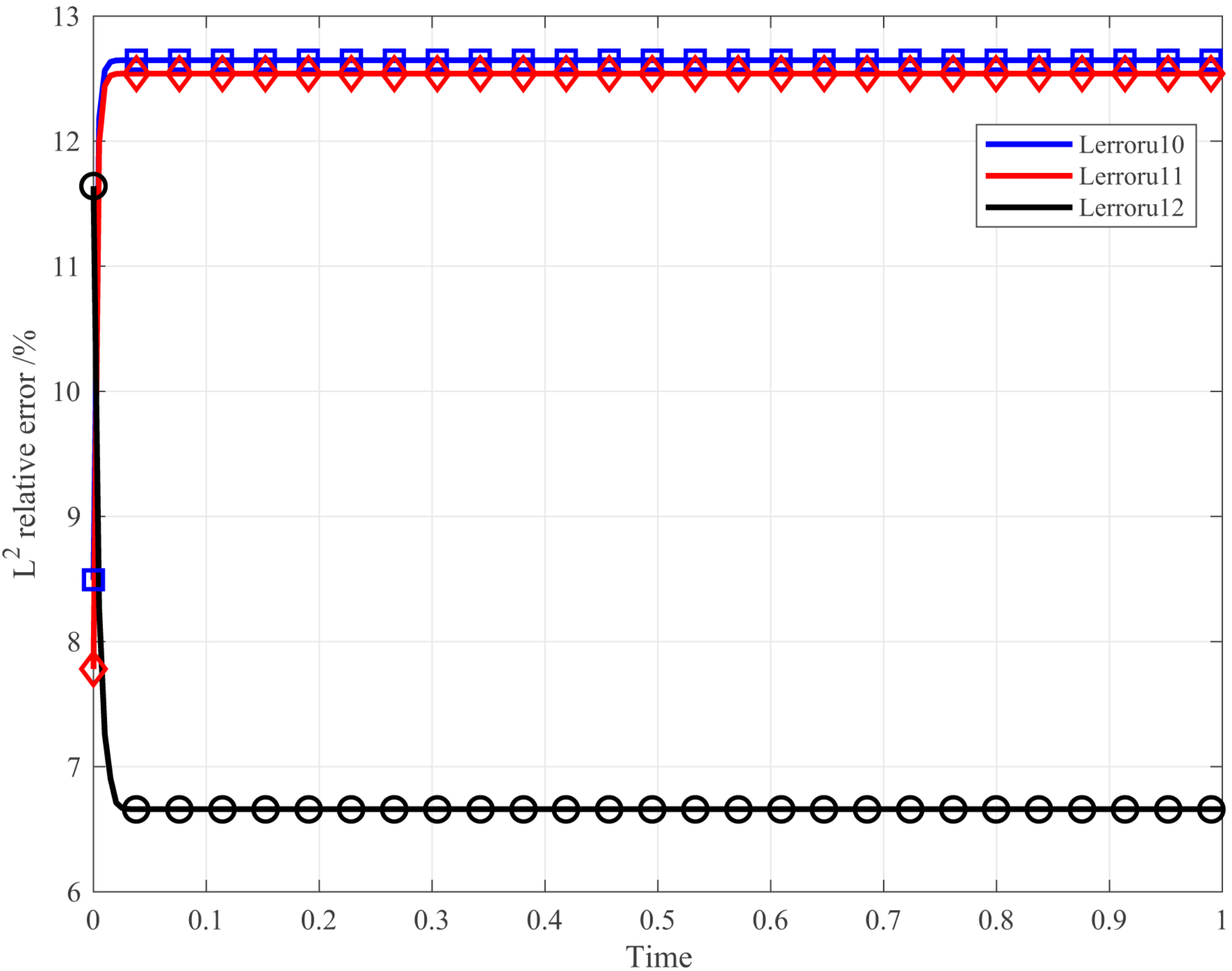}}
\subcaptionbox{\label{fig:y3derB}}
{\includegraphics[width=0.23\textwidth]{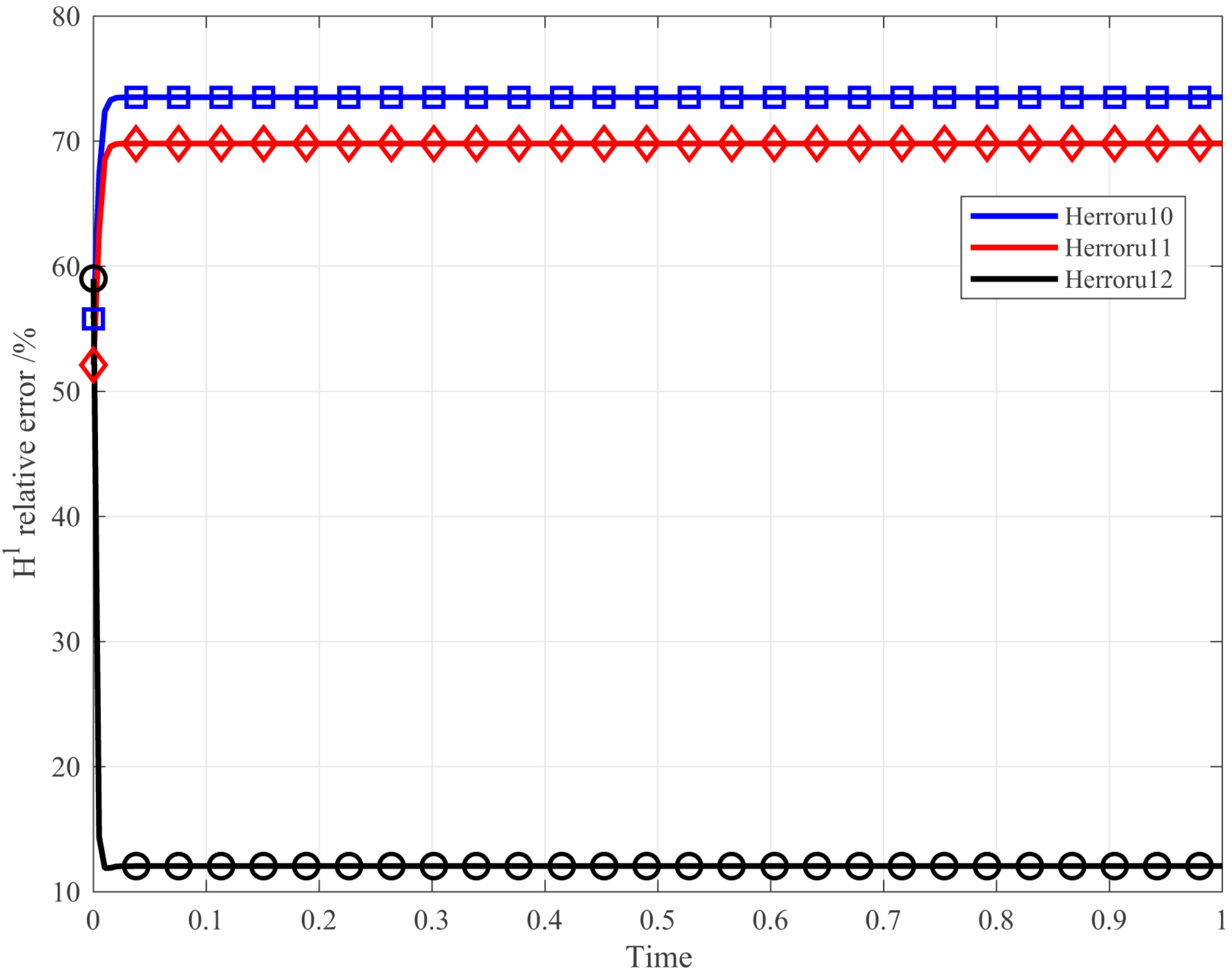}}
\subcaptionbox{\label{fig:y3derC}}
{\includegraphics[width=0.23\textwidth]{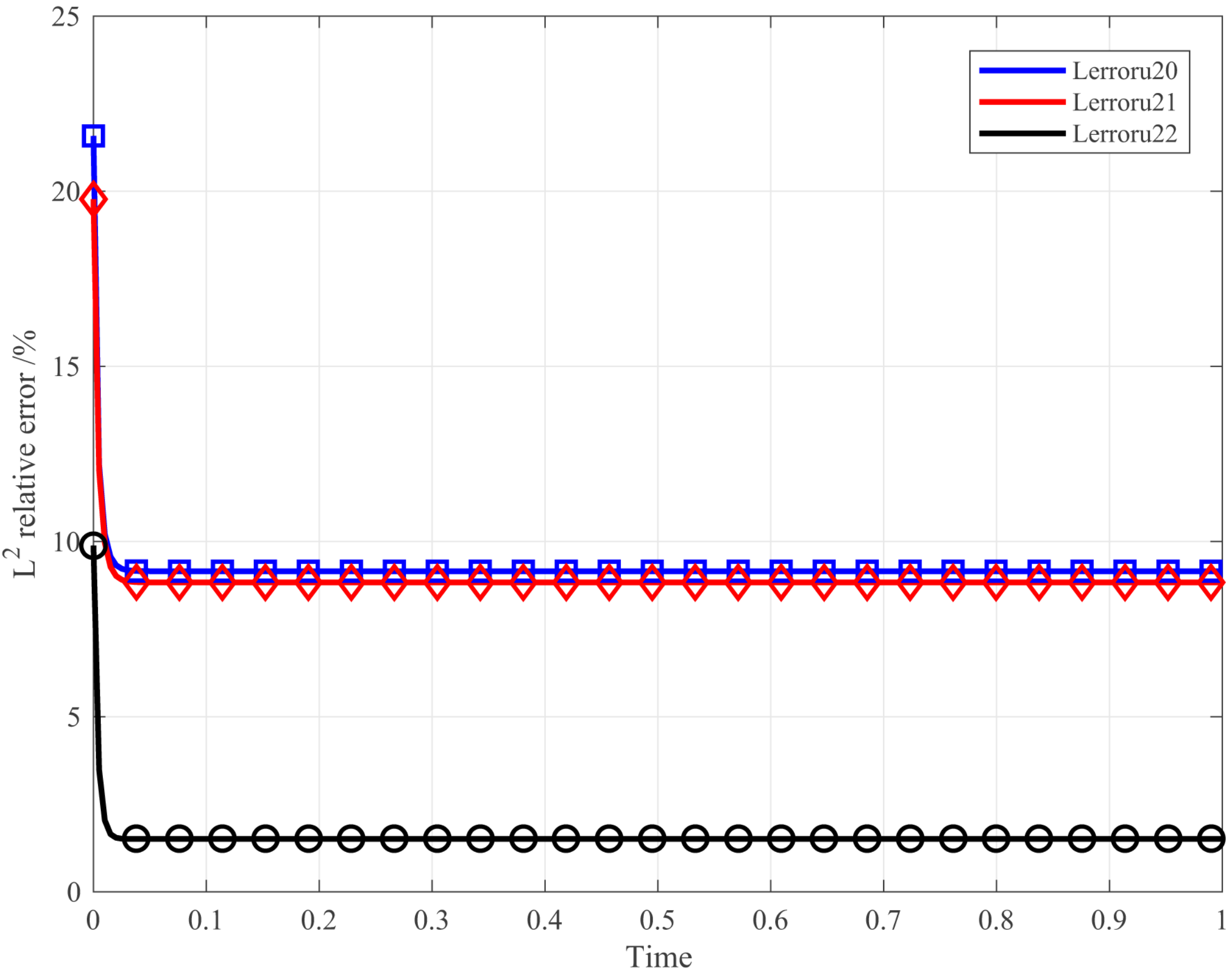}}
\subcaptionbox{\label{fig:y3derD}}
{\includegraphics[width=0.23\textwidth]{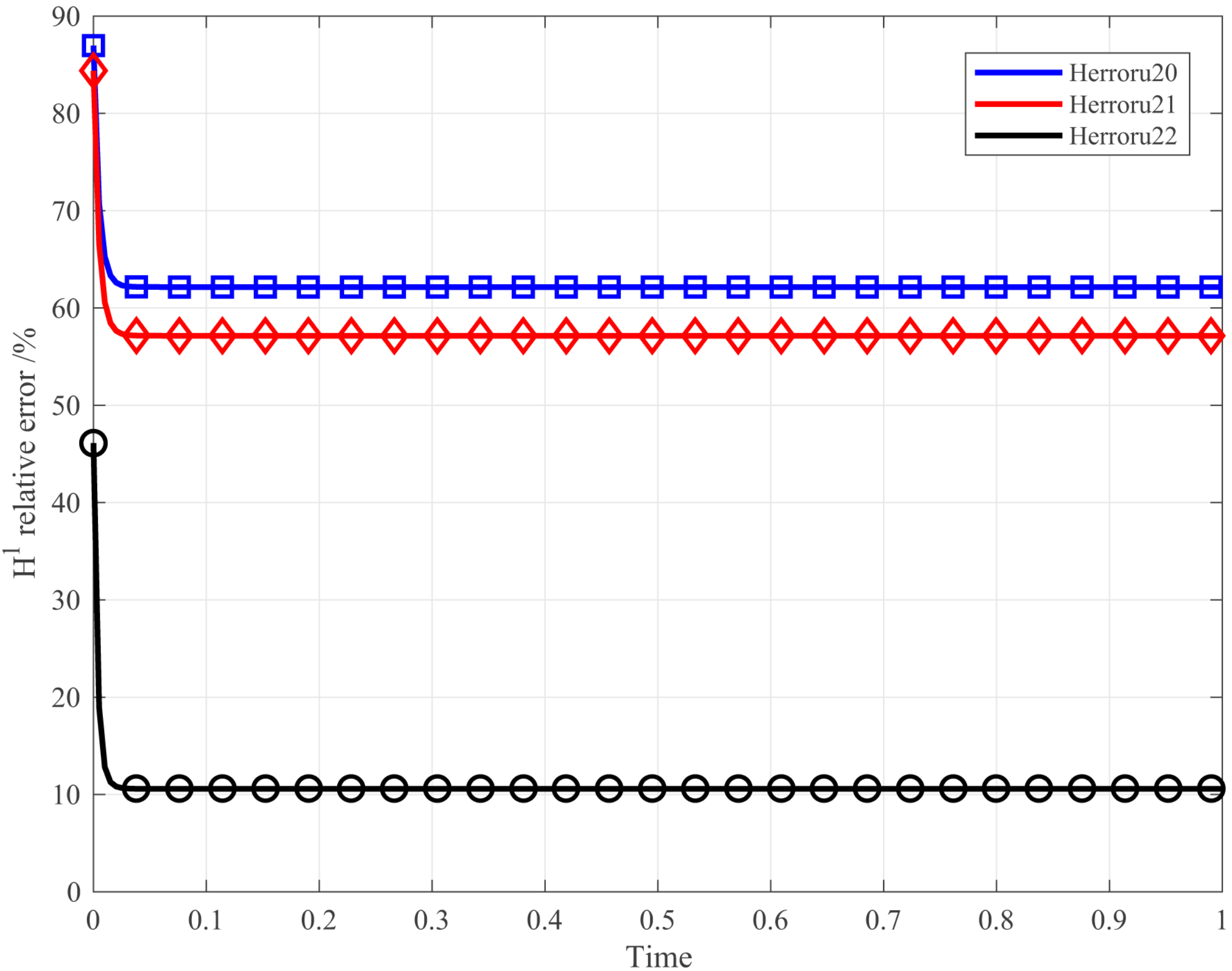}}
\caption{Evolution of the relative errors of the $L^{\scriptscriptstyle 2}$ norm and $H^{\scriptscriptstyle 1}$ seminorm over time: (a) $\text{Lerr1}$; (b) $\text{Herr1}$; (c) $\text{Lerr2}$; (d) $\text{Herr2}$.}\label{fig:Ery3d}
\end{figure}
From the mesh information in Table~\ref{tab:y3d}, it can be observed that the computational grid resource overhead required by the HOMS method is significantly smaller than that of the direct FEM, and its efficiency is also markedly superior to the FEM. As shown in Fig.~\ref{fig:yuan3du1} and Fig.~\ref{fig:yuan3du2}, the approximation accuracy of the HOMS solution is significantly better than that of both the homogenized solution and the FOMS solution. Only the HOMS solution effectively captures the local oscillatory behavior of the highly heterogeneous media. From Fig.~\ref{fig:Ery3d}, it can be observed that under the \( L^{\scriptscriptstyle 2} \) norm, the relative errors of \( u_{\scriptscriptstyle 1}^{\scriptscriptstyle(2\varepsilon)} \) and \( u_{\scriptscriptstyle 2}^{\scriptscriptstyle(2\varepsilon)} \) are only about 6\% and 1\%, respectively; under the \( H^{\scriptscriptstyle 1} \) seminorm, the relative errors of \( u_{\scriptscriptstyle 1}^{\scriptscriptstyle(2\varepsilon)} \) and \( u_{\scriptscriptstyle 2}^{\scriptscriptstyle(2\varepsilon)} \) are only approximately 12\% and 10\%, significantly lower than those of the homogenized solution and the FOMS solution, thus meeting the requirements for engineering computations. Furthermore, as shown in Fig.~\ref{fig:Ery3d}, the HOMS method exhibits excellent numerical stability over time and can be effectively applied to the computation of dynamic problem \eqref{eq:3} that evolves with time.

\subsection{Example 3. 2D channel media}
For the two-dimensional case of problem \eqref{eq:3}, consider the macroscopic domain \(\Omega = (x_1, x_2) = [0, 1]^{2}\) and the microscopic unit cell \(Y = (y_1, y_2) = [0, 1]^2\) as shown in Fig.~\ref{fig:cgd2d}, where \(\varepsilon = 1/8\).

\begin{figure}[pos=htbp]
\centering
{\includegraphics[width=0.483\textwidth]{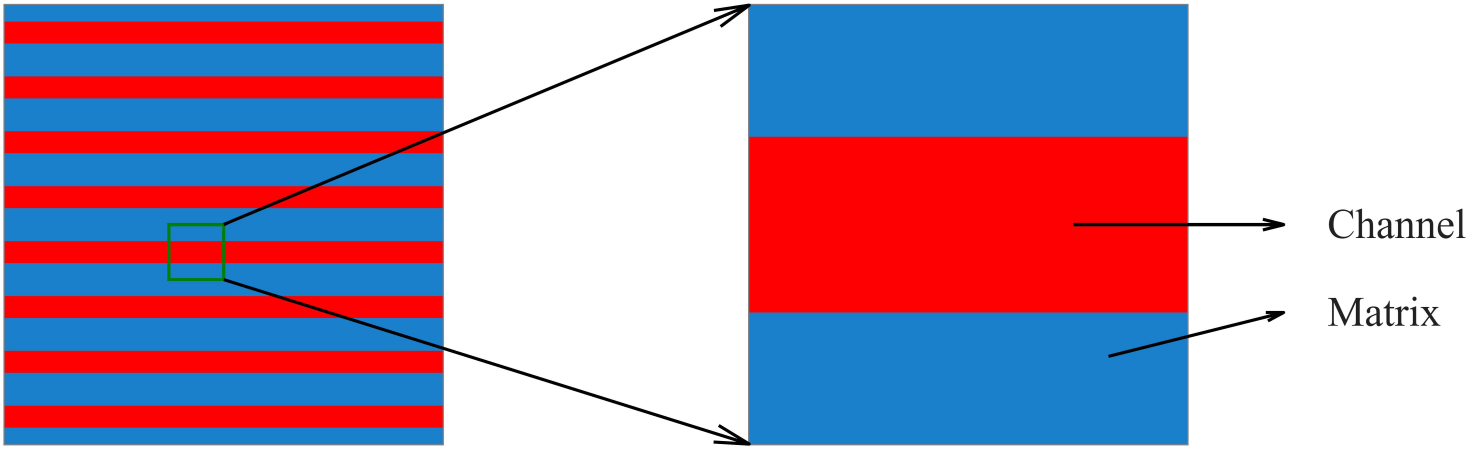}}
\caption{The schematic of 2D periodic media}\label{fig:cgd2d}
\end{figure}

The input parameters for the validation example are given in Table~\ref{tab:materialcgd2d}. The source term, initial pressure, and boundary conditions for this problem are as follows:
\[
\begin{gathered}
q(\boldsymbol{x},t)=1\times10^{\scriptscriptstyle 5},\ g_{\scriptscriptstyle 1}(\boldsymbol{x})=60,\ g_{\scriptscriptstyle 2}(\boldsymbol{x})=60, \\
u_{\scriptscriptstyle 1}^{\scriptscriptstyle \epsilon}(\boldsymbol{x},t)=60,\ u_{\scriptscriptstyle 2}^{\scriptscriptstyle \epsilon}(\boldsymbol{x},t)=60.
\end{gathered}
\]
\begin{table}[htbp]
\centering
\caption{Input parameters}
\label{tab:materialcgd2d}
\begin{tabular}{@{} >{\centering\arraybackslash}p{3.9cm} >{\centering\arraybackslash}p{3.9cm} @{}}
\toprule
Material parameters    & Matrix / Porous  \\
\midrule
\( c_{\scriptscriptstyle 1} \)    
    & \( 50 \ / \ 25 \) \\
\( c_{\scriptscriptstyle 2} \)    
    & \( 30 \ / \ 10 \) \\
\( \kappa_{\scriptscriptstyle 1,ij} \)    
    & \( 100 \ / \ 4 \) \\
\( \kappa_{\scriptscriptstyle 2,ij} \)    
    & \( 50 \ /\ 5 \) \\
\( Q_{\scriptscriptstyle1} \)    
    & \( 30 \ /\ 30 \) \\
\( Q_{\scriptscriptstyle 2} \)    
    & \( 30 \ /\ 30 \) \\
\bottomrule
\end{tabular}
\end{table}

Given that the analytical solution \( u_{\scriptscriptstyle l}^{\scriptscriptstyle \varepsilon} \) for problem \eqref{eq:3} is difficult to obtain directly, this paper uses the finite element reference solution \( u_{\scriptscriptstyle l,Fe}^{\scriptscriptstyle\varepsilon} \) computed on an extremely fine mesh as a substitute. By comparing and analyzing this finite element reference solution with the asymptotic solutions of various orders obtained through the HOMS method, the accuracy of the HOMS solution is validated. To this end, tetrahedral mesh partitions are performed for the macroscopic domain, the unit cell domain, and the homogenized domain. Table~\ref{tab:cgd2d} lists the element and node information for the three sets of meshes, as well as a comparison of computation times between the refined FEM and the HOMS method. 
\begin{table}[htbp]
\centering
\setlength{\tabcolsep}{2pt}
\caption{Summary of computational cost}
\label{tab:cgd2d}
\begin{tabular}{lccc}
\toprule
 & \textbf{Multiscale eqs.} & \textbf{Cell eqs.} & \textbf{Homogenized eqs.} \\
\midrule
\textbf{FEM elements}   & 56,832    & 1,946     & 15,080 \\
\textbf{FEM nodes}      & 28,737    & 1,034     & 7,701 \\
\midrule
 & \textbf{FEM}    & \multicolumn{2}{c}{\textbf{HOMS}} \\
\midrule
\textbf{time}   & 232.122\,s    & \multicolumn{2}{c}{77.028\,s} \\
\bottomrule
\end{tabular}
\end{table}

Set the time step as $\Delta t = 0.02$, then compute the solutions $u_{\scriptscriptstyle l}^{\scriptscriptstyle\varepsilon}$, $u_{\scriptscriptstyle l}^{\scriptscriptstyle(0)}$, $u_{\scriptscriptstyle l}^{\scriptscriptstyle(1\varepsilon)}$ and $u_{\scriptscriptstyle l}^{\scriptscriptstyle(2\varepsilon)}$ of problem~\eqref{eq:3} over the time interval $[0,1]$. Denote the $L^{\scriptscriptstyle 2}$ norm and $H^{\scriptscriptstyle 1}$ seminorm as $\|\cdot\|_{L^{\scriptscriptstyle 2}(\Omega)}$ and $|\cdot|_{H^{\scriptscriptstyle 1}(\Omega)}$, respectively.

Fig.~\ref{fig:cgd2du1} and Fig.~\ref{fig:cgd2du2} display the distribution profiles of $u_{\scriptscriptstyle l}^{\scriptscriptstyle(0)}$, $u_{\scriptscriptstyle l}^{\scriptscriptstyle(1\varepsilon)}$, $u_{\scriptscriptstyle l}^{\scriptscriptstyle(2\varepsilon)}$ and $u_{\scriptscriptstyle l}^{\scriptscriptstyle\varepsilon}$ at time $t = 1.0$. Fig.~\ref{fig:Ercgd2d} shows the evolution of the relative errors in the $L^{\scriptscriptstyle 2}$ norm and $H^{\scriptscriptstyle 1}$ seminorm for $u_{\scriptscriptstyle l}^{\scriptscriptstyle(0)}$, $u_{\scriptscriptstyle l}^{\scriptscriptstyle(1\varepsilon)}$, and $u_{\scriptscriptstyle l}^{\scriptscriptstyle(2\varepsilon)}$ over the time interval $[0,1]$.

\begin{figure}[pos=htbp]
\centering
\subcaptionbox{$u_{\scriptscriptstyle 1}^{\scriptscriptstyle(0)}$\label{fig:cgd2d_u1A}}
{\includegraphics[width=0.235\textwidth]{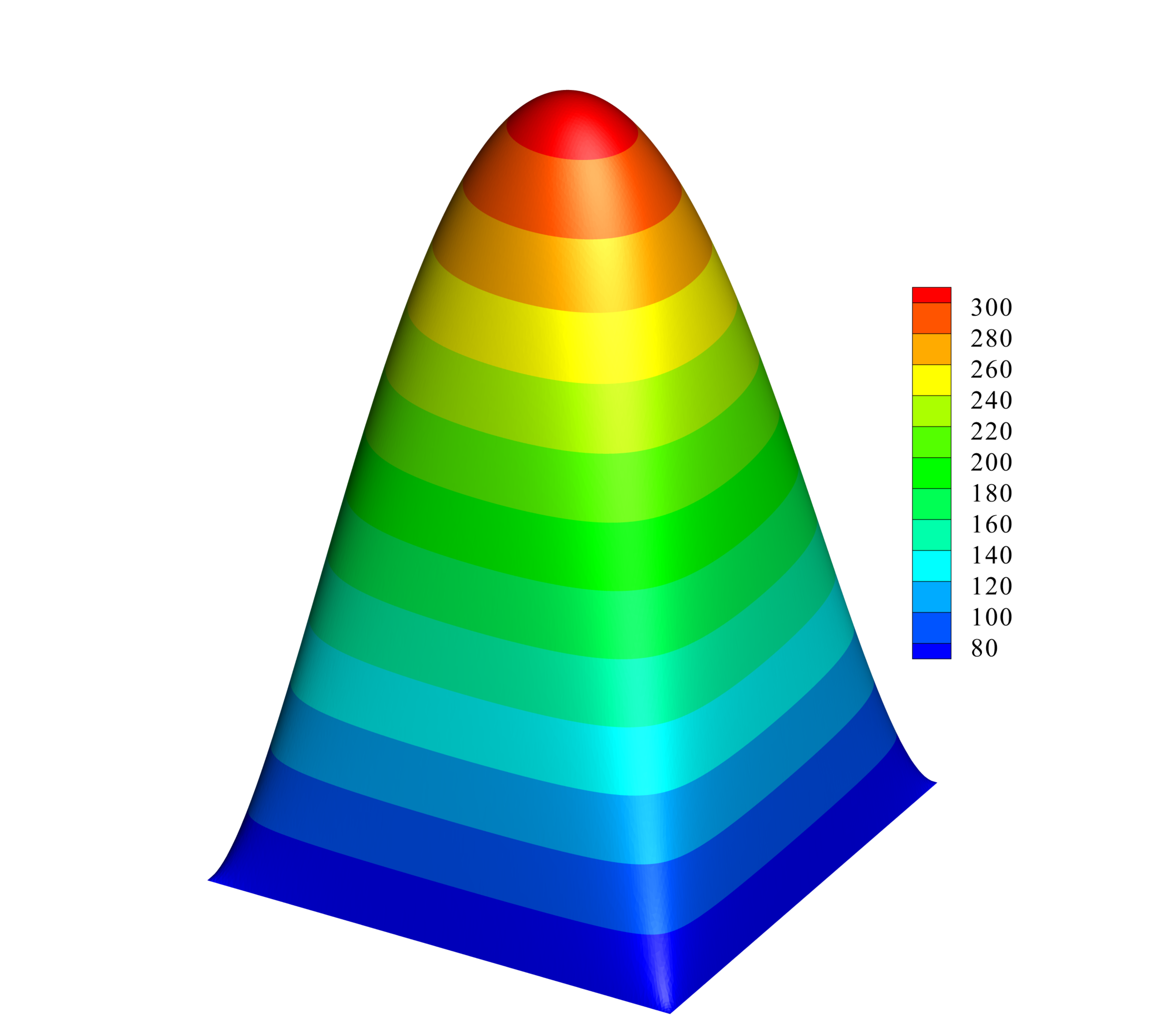}}
\subcaptionbox{$u_{\scriptscriptstyle 1}^{\scriptscriptstyle(1\varepsilon)}$\label{fig:cgd2d_u1B}}
{\includegraphics[width=0.235\textwidth]{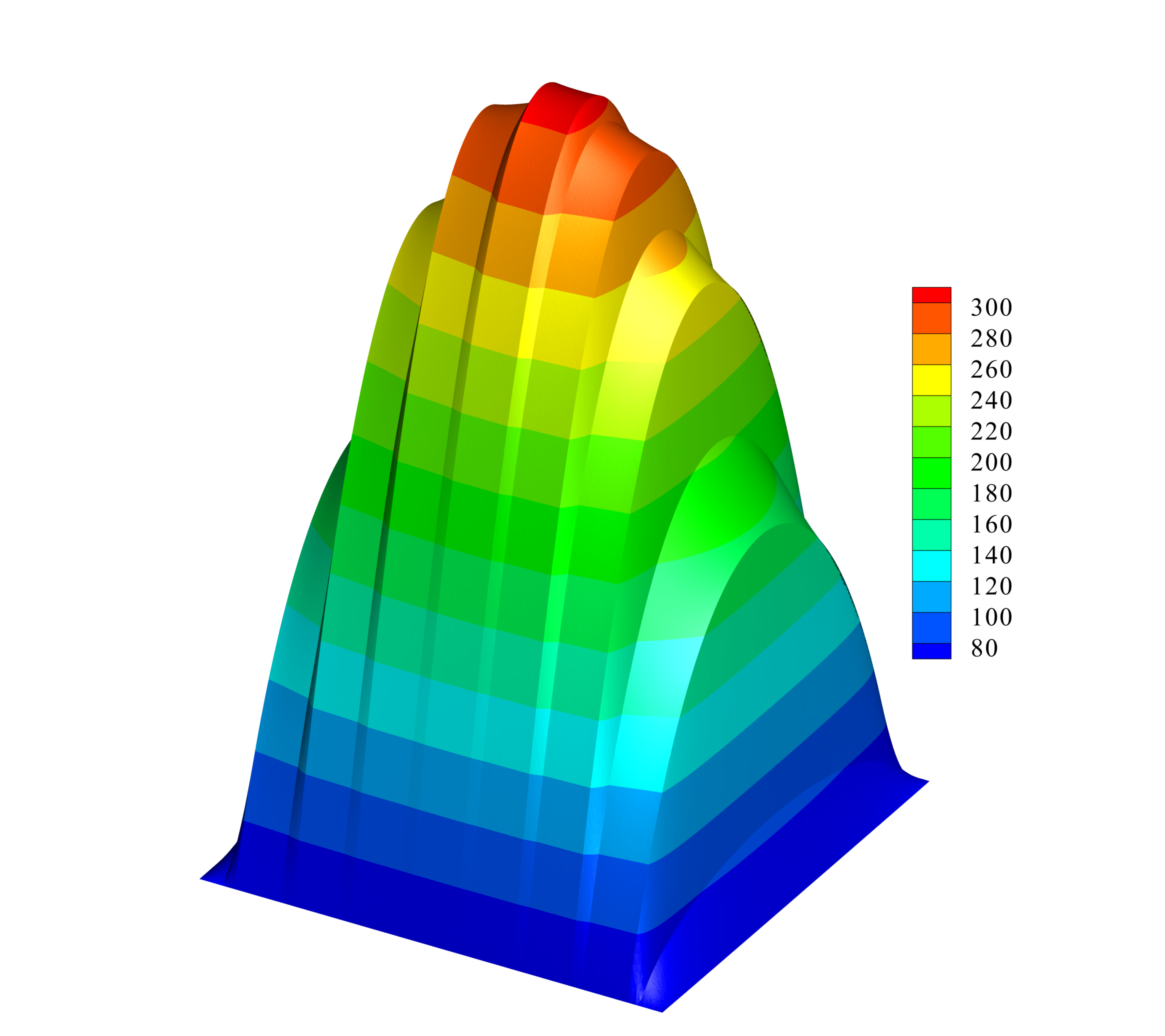}}
\subcaptionbox{$u_{\scriptscriptstyle 1}^{\scriptscriptstyle(2\varepsilon)}$\label{fig:cgd2d_u1C}}
{\includegraphics[width=0.235\textwidth]{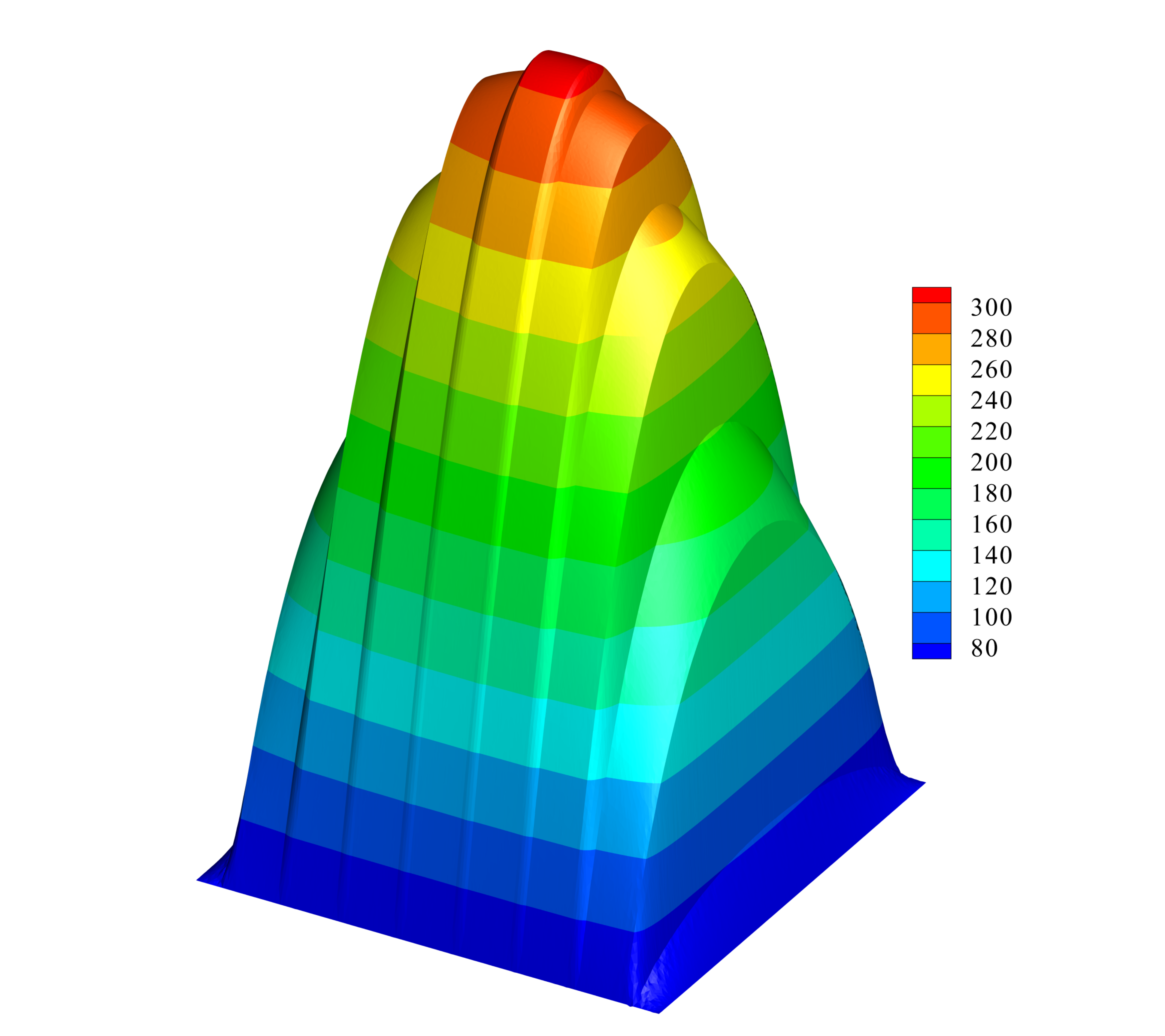}}
\subcaptionbox{$u_{\scriptscriptstyle 1}^{\scriptscriptstyle\varepsilon}$\label{fig:cgd2d_u1D}}
{\includegraphics[width=0.235\textwidth]{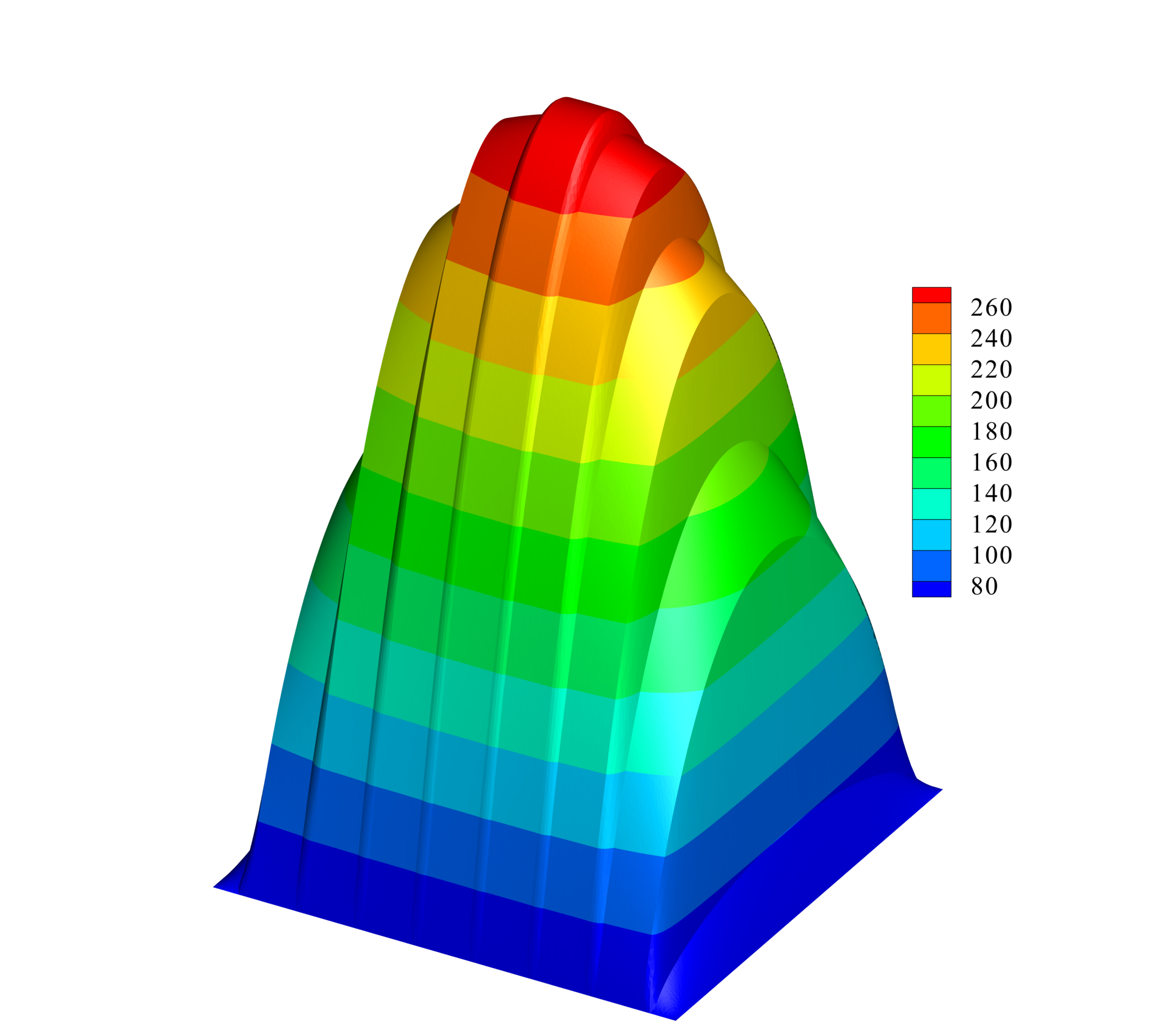}}
\caption{The pressure field at $t = 1.0 $ of first-continuum media.}\label{fig:cgd2du1}
\end{figure}
\begin{figure}[pos=htbp]
\centering
\subcaptionbox{$u_{\scriptscriptstyle 2}^{\scriptscriptstyle(0)}$\label{fig:cgd2d_u2A}}
{\includegraphics[width=0.235\textwidth]{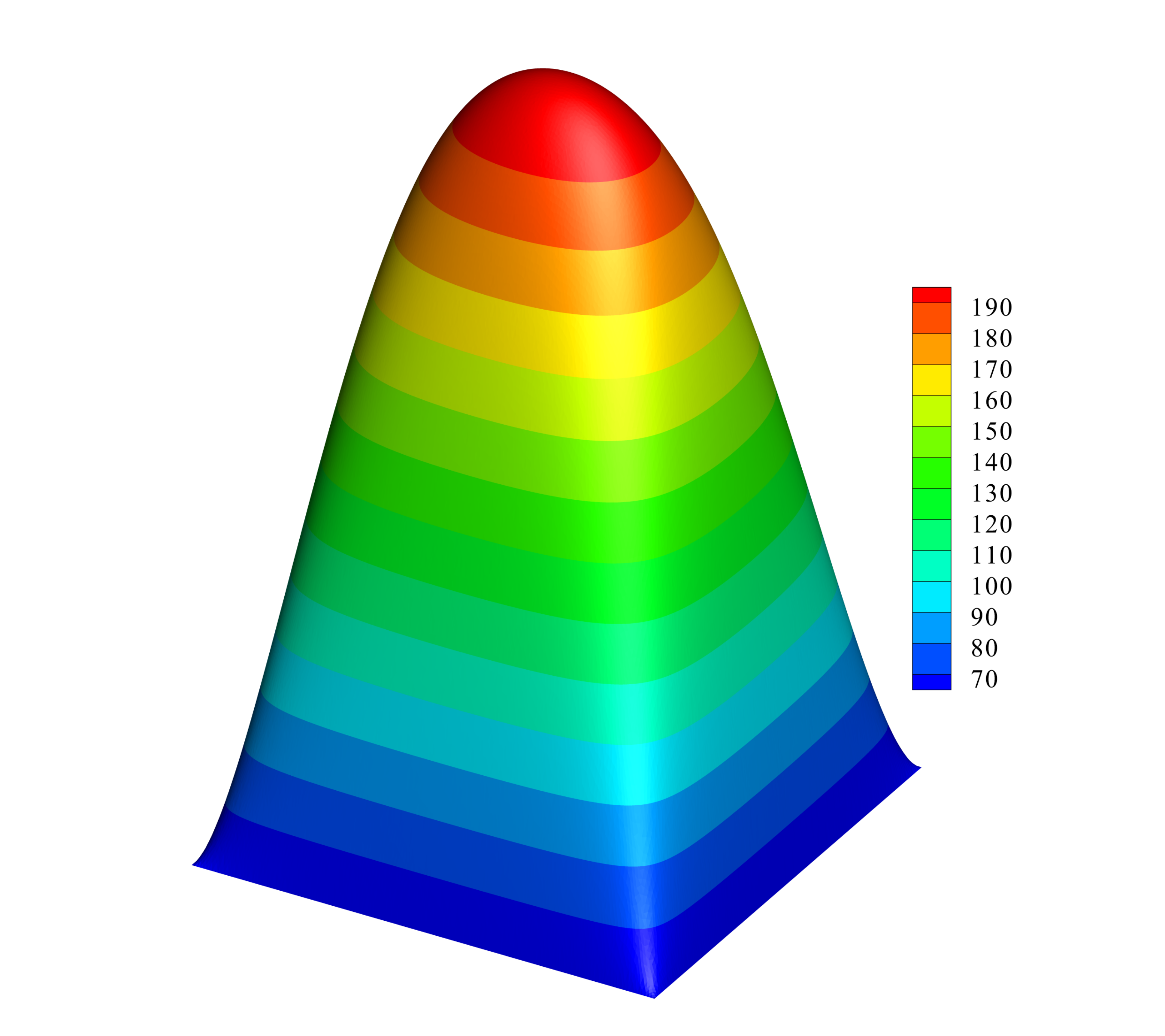}}
\subcaptionbox{$u_{\scriptscriptstyle 2}^{\scriptscriptstyle(1\varepsilon)}$\label{fig:cgd2d_u2B}}
{\includegraphics[width=0.235\textwidth]{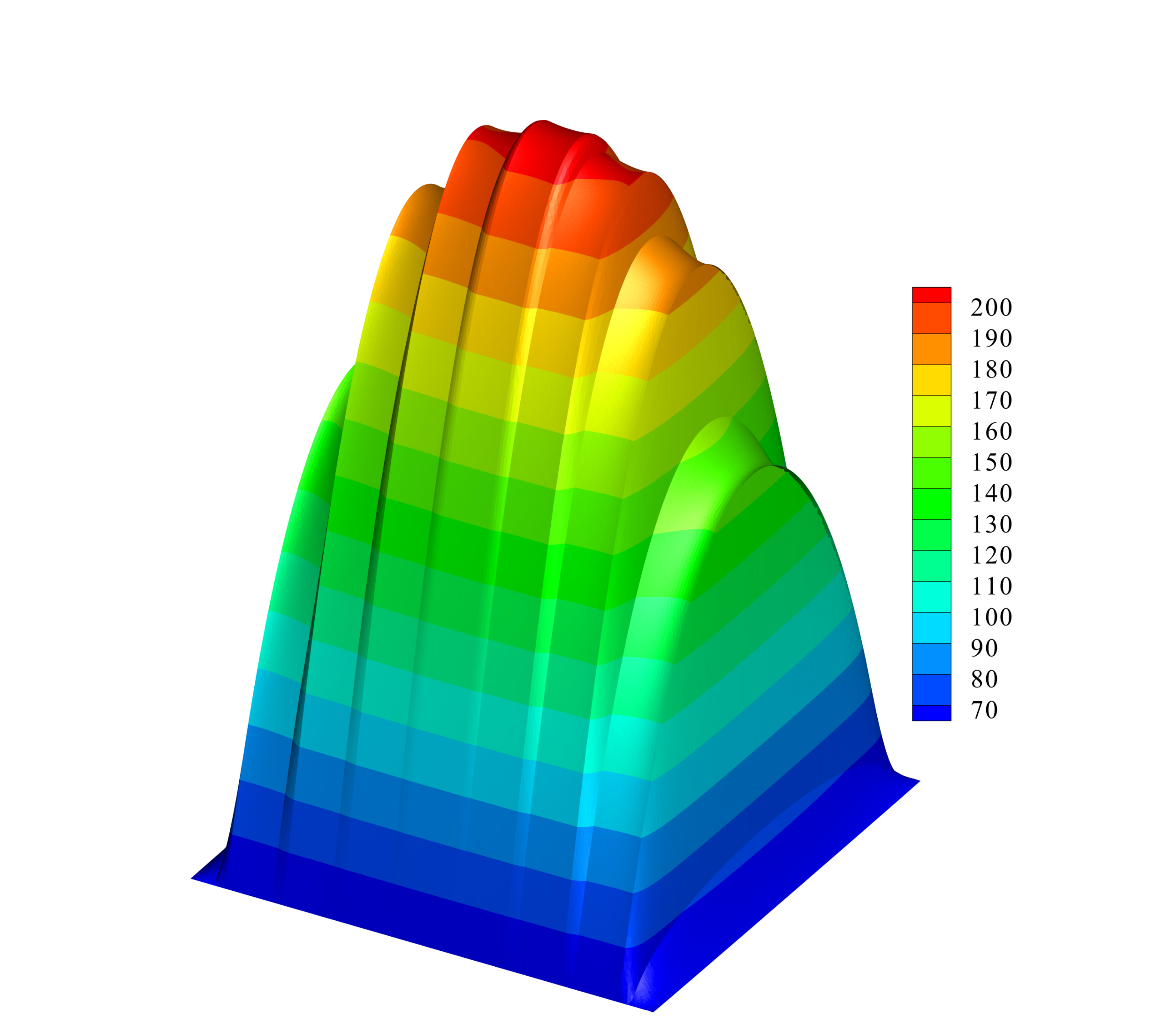}}
\subcaptionbox{$u_{\scriptscriptstyle 2}^{\scriptscriptstyle(2\varepsilon)}$\label{fig:cgd2d_u2C}}
{\includegraphics[width=0.235\textwidth]{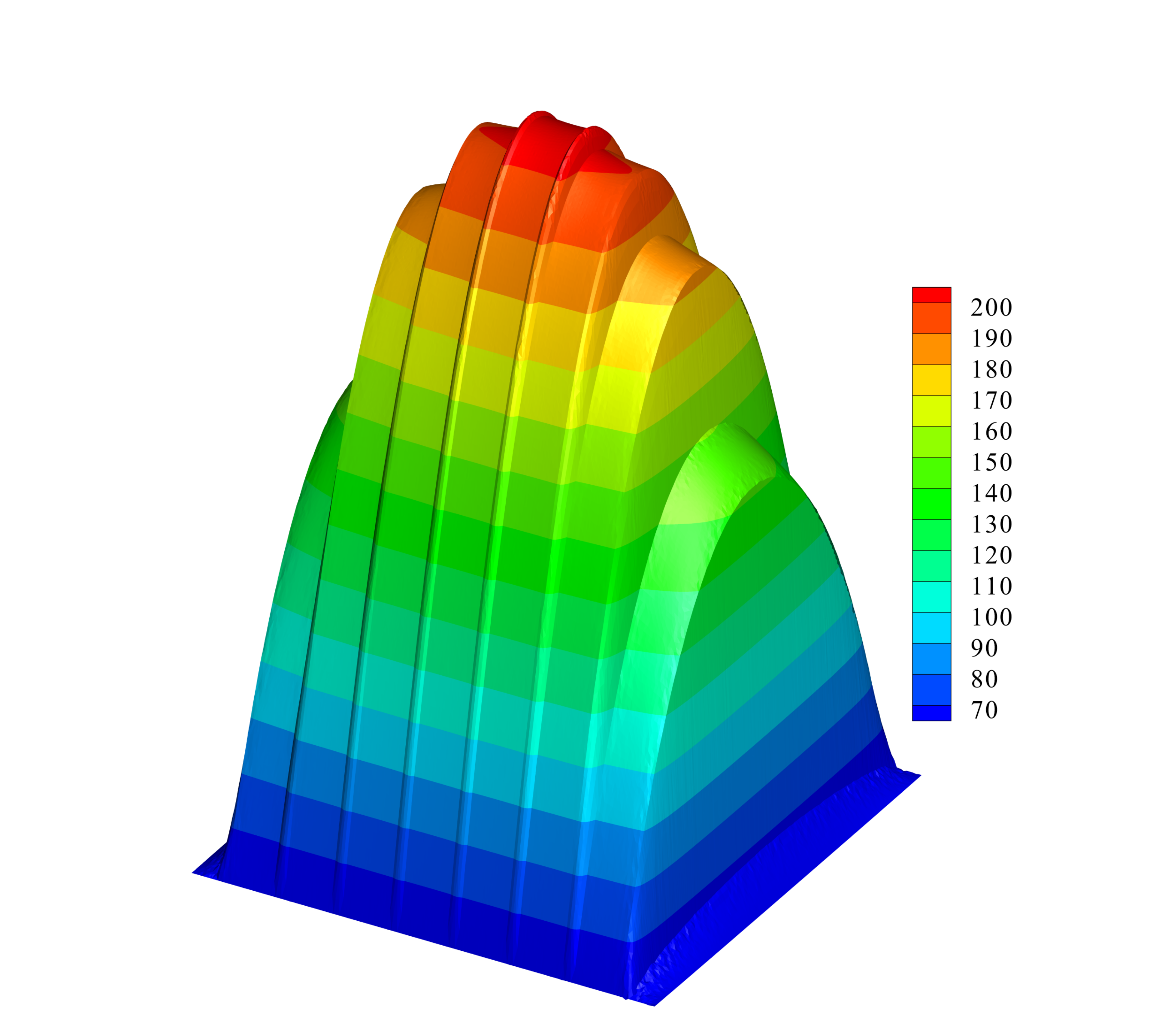}}
\subcaptionbox{$u_{\scriptscriptstyle 2}^{\scriptscriptstyle\varepsilon}$\label{fig:cgd2d_u2D}}
{\includegraphics[width=0.235\textwidth]{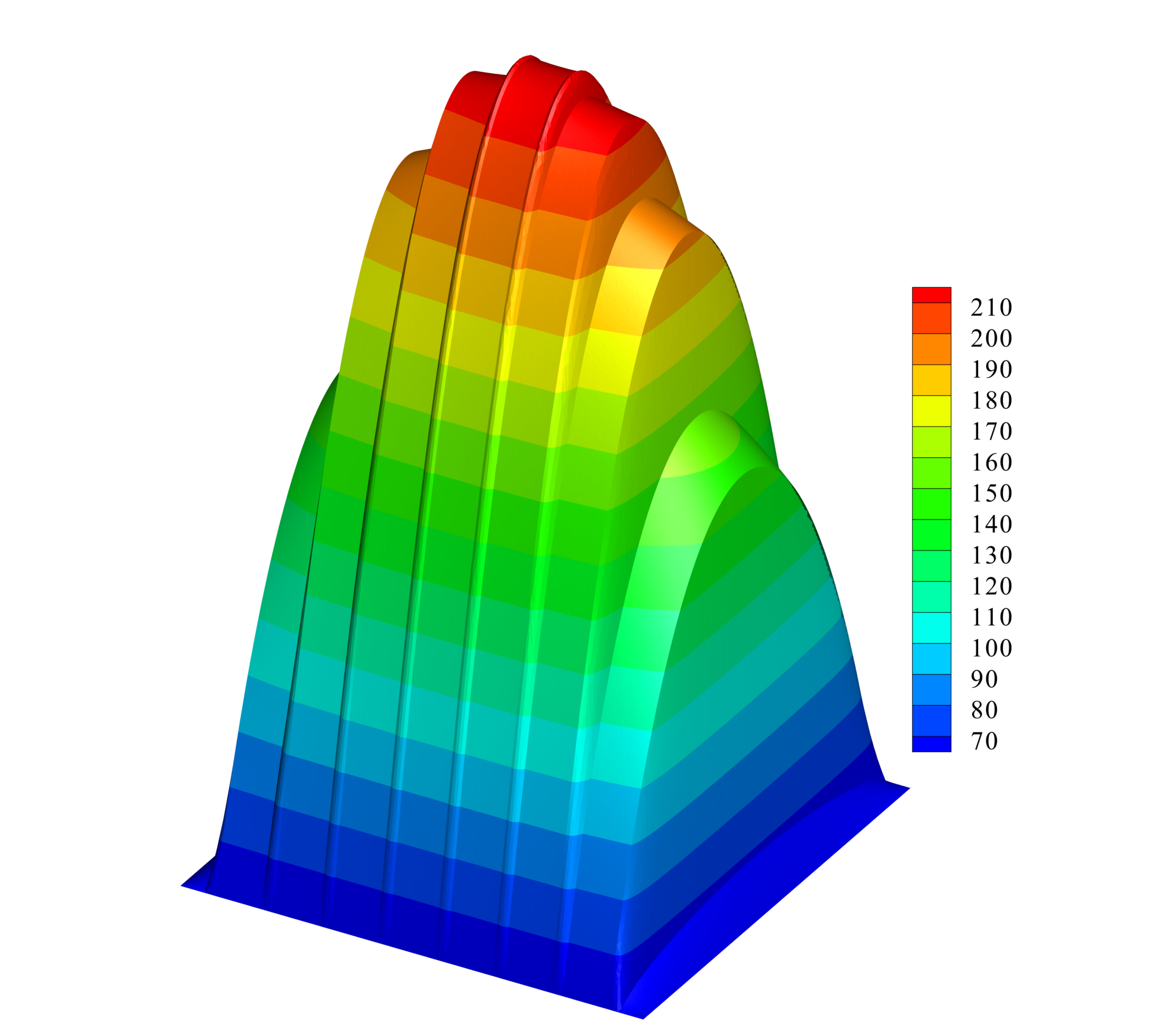}}
\caption{The pressure field at $t = 1.0 $ of second-continuum media.}\label{fig:cgd2du2}
\end{figure}

\begin{figure}[pos=htbp]
\centering
\subcaptionbox{\label{fig:cgd2derA}}
{\includegraphics[width=0.235\textwidth]{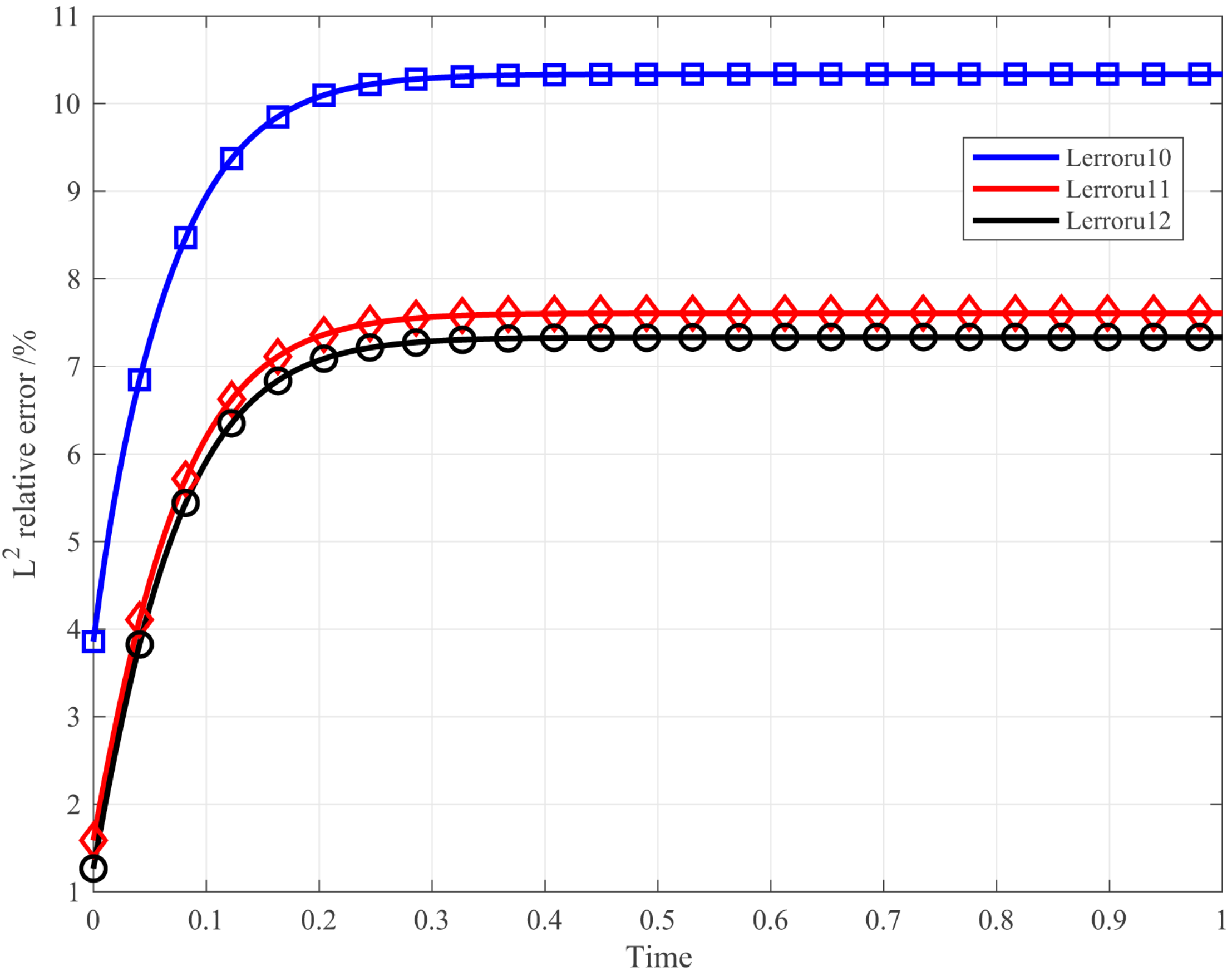}}
\subcaptionbox{\label{fig:cgd2derB}}
{\includegraphics[width=0.235\textwidth]{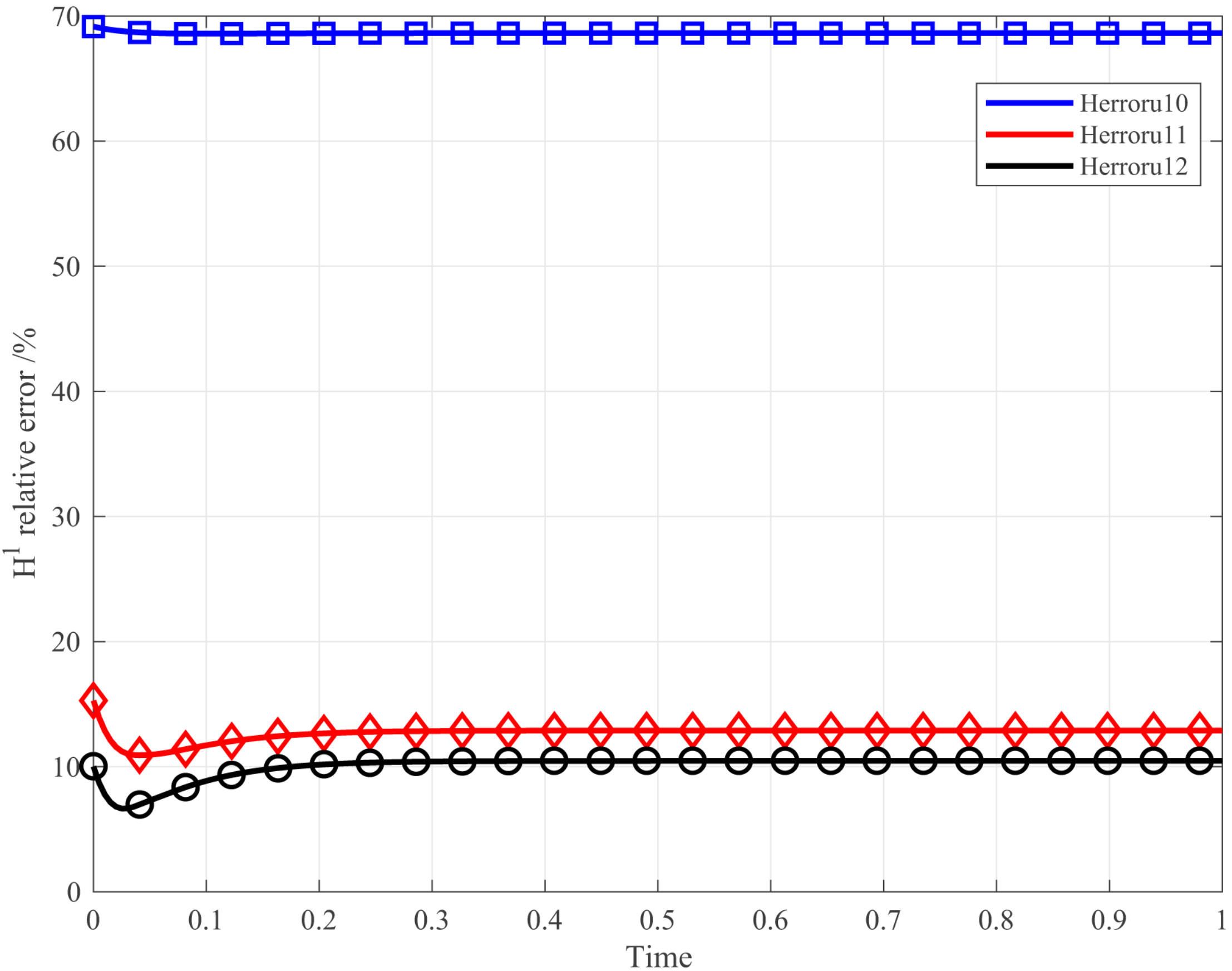}}
\subcaptionbox{\label{fig:cgd2derC}}
{\includegraphics[width=0.235\textwidth]{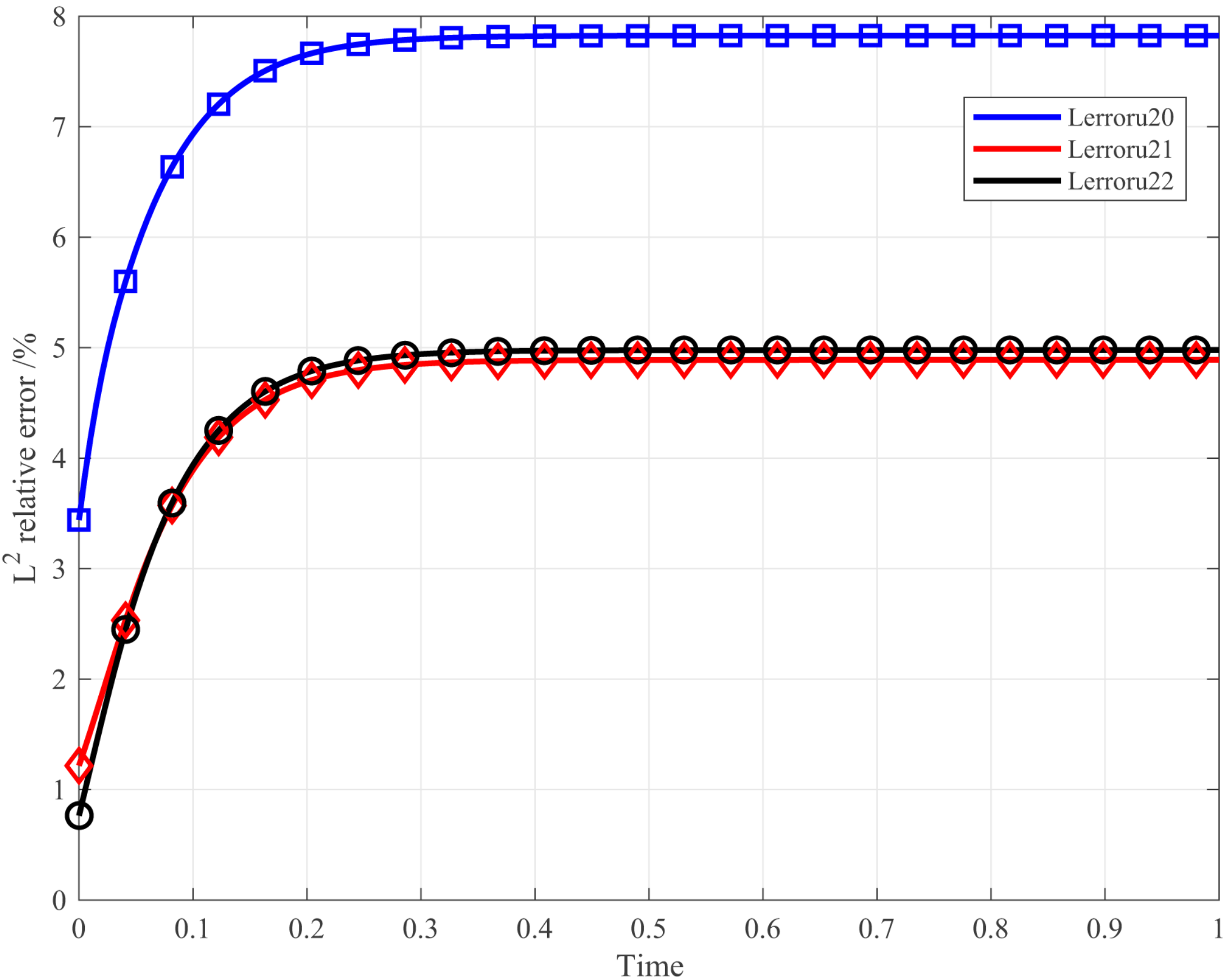}}
\subcaptionbox{\label{fig:cgd2derD}}
{\includegraphics[width=0.235\textwidth]{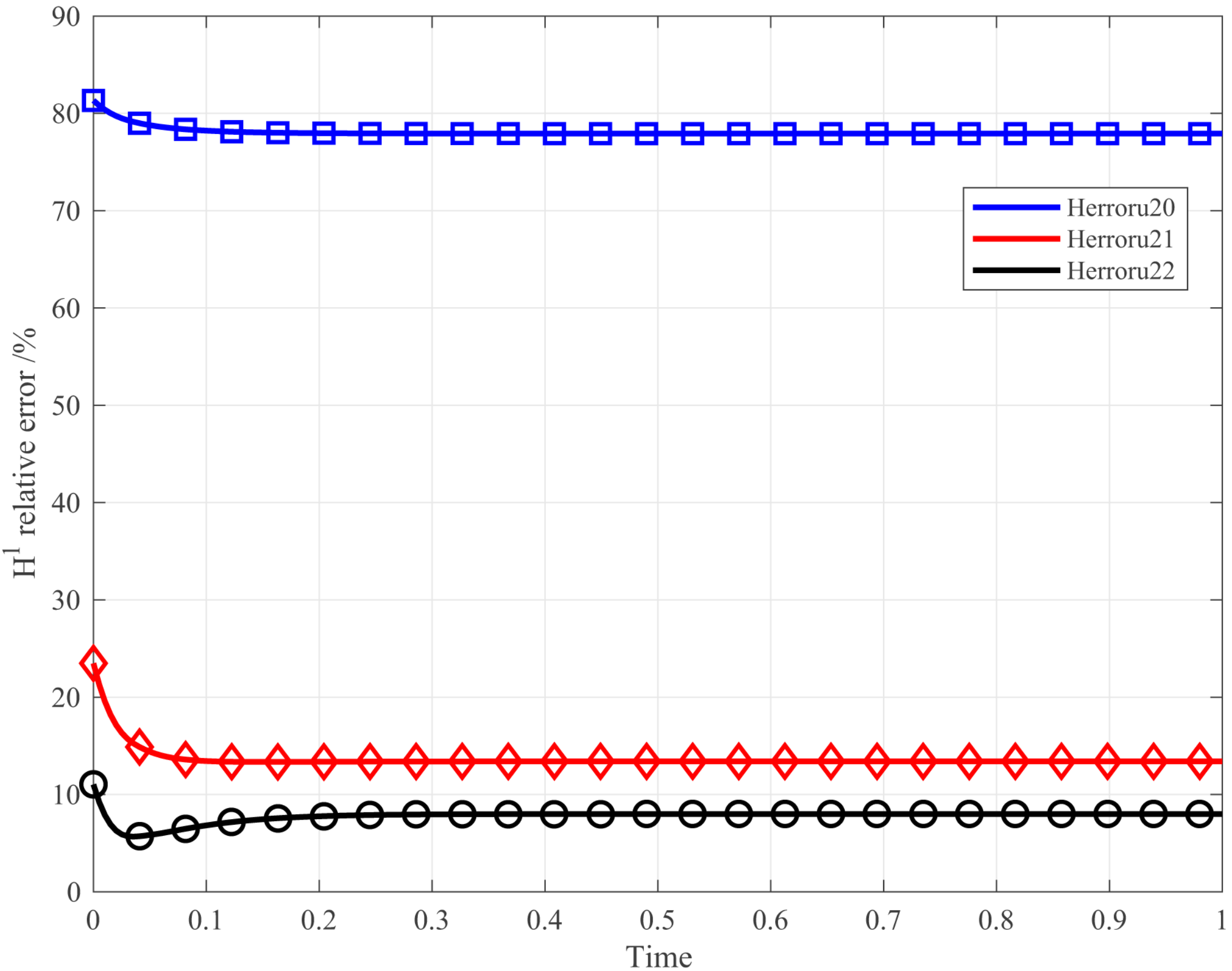}}
\caption{Evolution of the relative errors of the $L^{\scriptscriptstyle 2}$ norm and $H^{\scriptscriptstyle 1}$ seminorm over time: (a) $\text{Lerr1}$; (b) $\text{Herr1}$; (c) $\text{Lerr2}$; (d) $\text{Herr2}$.}\label{fig:Ercgd2d}
\end{figure}

From the mesh information in Table~\ref{tab:cgd2d}, it can be observed that the computational grid resource overhead required by the HOMS method is significantly smaller than that of the direct FEM, and its efficiency is also markedly superior to the FEM. As shown in Fig.~\ref{fig:cgd2du1} and Fig.~\ref{fig:cgd2du2}, the approximation accuracy of the HOMS solution is significantly better than that of both the homogenized solution and the FOMS solution. Only the HOMS solution effectively captures the local oscillatory behavior of the highly heterogeneous media. From Fig.~\ref{fig:Ercgd2d}, it can be observed that under the \( L^{\scriptscriptstyle 2} \) norm, the relative errors of \( u_{\scriptscriptstyle 1}^{(\scriptscriptstyle2\varepsilon)} \) and \( u_{\scriptscriptstyle 2}^{\scriptscriptstyle(2\varepsilon)} \) are only about 7\% and 5\%, respectively; under the \( H^{\scriptscriptstyle 1} \) seminorm, the relative errors of \( u_{\scriptscriptstyle 1}^{\scriptscriptstyle(2\varepsilon)} \) and \( u_{\scriptscriptstyle 2}^{\scriptscriptstyle(2\varepsilon)} \) are only approximately 10\% and 8\%, significantly lower than those of the homogenized solution and the FOMS solution, thus meeting the requirements for engineering computations. Furthermore, as shown in Fig.~\ref{fig:Ercgd2d}, the HOMS method exhibits excellent numerical stability over time and can be effectively applied to the computation of dynamic problem \eqref{eq:3} that evolves with time.

\subsection{Example 4. 3D channel media}
For the three-dimensional case of problem \eqref{eq:3}, consider the macroscopic domain \(\Omega = (x_1, x_2, x_3) = [0, 1]^3\) and the microscopic unit cell \(Y = (y_1, y_2, y_3) = [0, 1]^3\) as shown in Fig.~\ref{fig:cgd3d}, where \(\varepsilon = 1/4\).
\begin{figure}[pos=htbp]
\centering
{\includegraphics[width=0.483\textwidth]{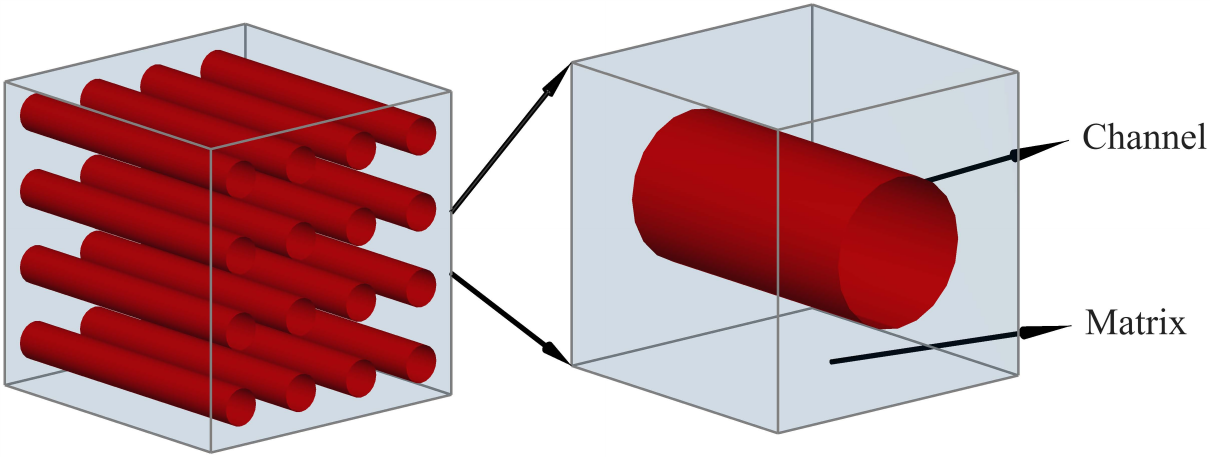}}
\caption{The schematic of 3D periodic media}\label{fig:cgd3d}
\end{figure}

The input parameters for the validation example are given in Table~\ref{tab:materialcgd3d}. The source term, initial pressure, and boundary conditions for this problem are as follows:
\[
\begin{gathered}
q(\boldsymbol{x},t)=1\times10^{\scriptscriptstyle 5},\ g_{\scriptscriptstyle 1}(\boldsymbol{x})=0,\ g_{\scriptscriptstyle 2}(\boldsymbol{x})=0, \\
u_{\scriptscriptstyle 1}^{\scriptscriptstyle \epsilon}(\boldsymbol{x},t)=0,\ u_{\scriptscriptstyle 2}^{\scriptscriptstyle \epsilon}(\boldsymbol{x},t)=0.
\end{gathered}
\]

\begin{table}[htbp]
\centering
\caption{Input parameters}
\label{tab:materialcgd3d}
\begin{tabular}{@{} >{\centering\arraybackslash}p{3.9cm} >{\centering\arraybackslash}p{3.9cm} @{}}
\toprule
Material parameters    & Matrix / Channel  \\
\midrule
\( c_{\scriptscriptstyle 1} \)    
    & \( 5 \ / \ 100 \) \\
\( c_{\scriptscriptstyle 2} \)    
    & \( 125 \ / \ 15 \) \\
\( \kappa_{\scriptscriptstyle 1,ij} \)    
    & \( 800 \ / \ 30 \) \\
\( \kappa_{\scriptscriptstyle 2,ij} \)    
    & \( 1000 \ /\ 50 \) \\
\( Q_{\scriptscriptstyle 1} \)    
    & \( 250 \ /\ 25 \) \\
\( Q_{\scriptscriptstyle 2} \)    
    & \( 250 \ /\ 25 \) \\
\bottomrule
\end{tabular}
\end{table}

Given that the analytical solution \( u_{\scriptscriptstyle l}^{\scriptscriptstyle\varepsilon} \) for problem \eqref{eq:3} is difficult to obtain directly, this paper uses the finite element reference solution \( u_{\scriptscriptstyle l,Fe}^{\scriptscriptstyle\varepsilon} \) computed on an extremely fine mesh as a substitute. By comparing and analyzing this finite element reference solution with the asymptotic solutions of various orders obtained through the HOMS method, the accuracy of the HOMS solution is validated. To this end, tetrahedral mesh partitions are performed for the macroscopic domain, the unit cell domain, and the homogenized domain. Table~\ref{tab:cgd3d} lists the element and node information for the three sets of meshes, as well as a comparison of computation times between the refined FEM and the HOMS method. 
\begin{table}[htbp]
\centering
\setlength{\tabcolsep}{2pt}
\caption{Summary of computational cost}
\label{tab:cgd3d}
\begin{tabular}{lccc}
\toprule
 & \textbf{Multiscale eqs.} & \textbf{Cell eqs.} & \textbf{Homogenized eqs.} \\
\midrule
\textbf{FEM elements}   & 941,440    & 161,850     & 93,750 \\
\textbf{FEM nodes}      & 164,945    & 29,380      & 17,576 \\
\midrule
 & \textbf{FEM}    & \multicolumn{2}{c}{\textbf{HOMS}} \\
\midrule
\textbf{time}   & 14892.6\,s    & \multicolumn{2}{c}{4217.79\,s} \\
\bottomrule
\end{tabular}
\end{table}

Set the time step as $\Delta t = 0.002$, then compute the solutions $u_{\scriptscriptstyle l}^{\scriptscriptstyle\varepsilon}$, $u_{\scriptscriptstyle l}^{\scriptscriptstyle(0)}$, $u_{\scriptscriptstyle l}^{\scriptscriptstyle(1\varepsilon)}$ and $u_{\scriptscriptstyle l}^{\scriptscriptstyle(2\varepsilon)}$ of problem~\eqref{eq:3} over the time interval $[0,1]$. Denote the $L^{\scriptscriptstyle 2}$ norm and $H^{\scriptscriptstyle 1}$ seminorm as $\|\cdot\|_{L^2(\Omega)}$ and $|\cdot|_{H^{\scriptscriptstyle 1}(\Omega)}$, respectively.

Fig.~\ref{fig:cgd3du1} and Fig.~\ref{fig:cgd3du2} display the distribution profiles of $u_{\scriptscriptstyle i}^{\scriptscriptstyle(0)}$, $u_{\scriptscriptstyle l}^{\scriptscriptstyle(1\varepsilon)}$, $u_{\scriptscriptstyle l}^{\scriptscriptstyle(2\varepsilon)}$ and $u_{\scriptscriptstyle l}^{\scriptscriptstyle\varepsilon}$ at time $t = 1.0$. Fig.~\ref{fig:Ercgd3d} shows the evolution of the relative errors in the $L^{\scriptscriptstyle 2}$ norm and $H^{\scriptscriptstyle 1}$ seminorm for $u_{\scriptscriptstyle l}^{\scriptscriptstyle(0)}$, $u_{\scriptscriptstyle l}^{\scriptscriptstyle(1\varepsilon)}$, and $u_{\scriptscriptstyle l}^{\scriptscriptstyle(2\varepsilon)}$ over the time interval $[0,1]$.

\begin{figure}[pos=htbp]
\centering
\subcaptionbox{$u_{\scriptscriptstyle 1}^{\scriptscriptstyle(0)}$\label{fig:cgd3d_u1A}}
{\includegraphics[width=0.235\textwidth]{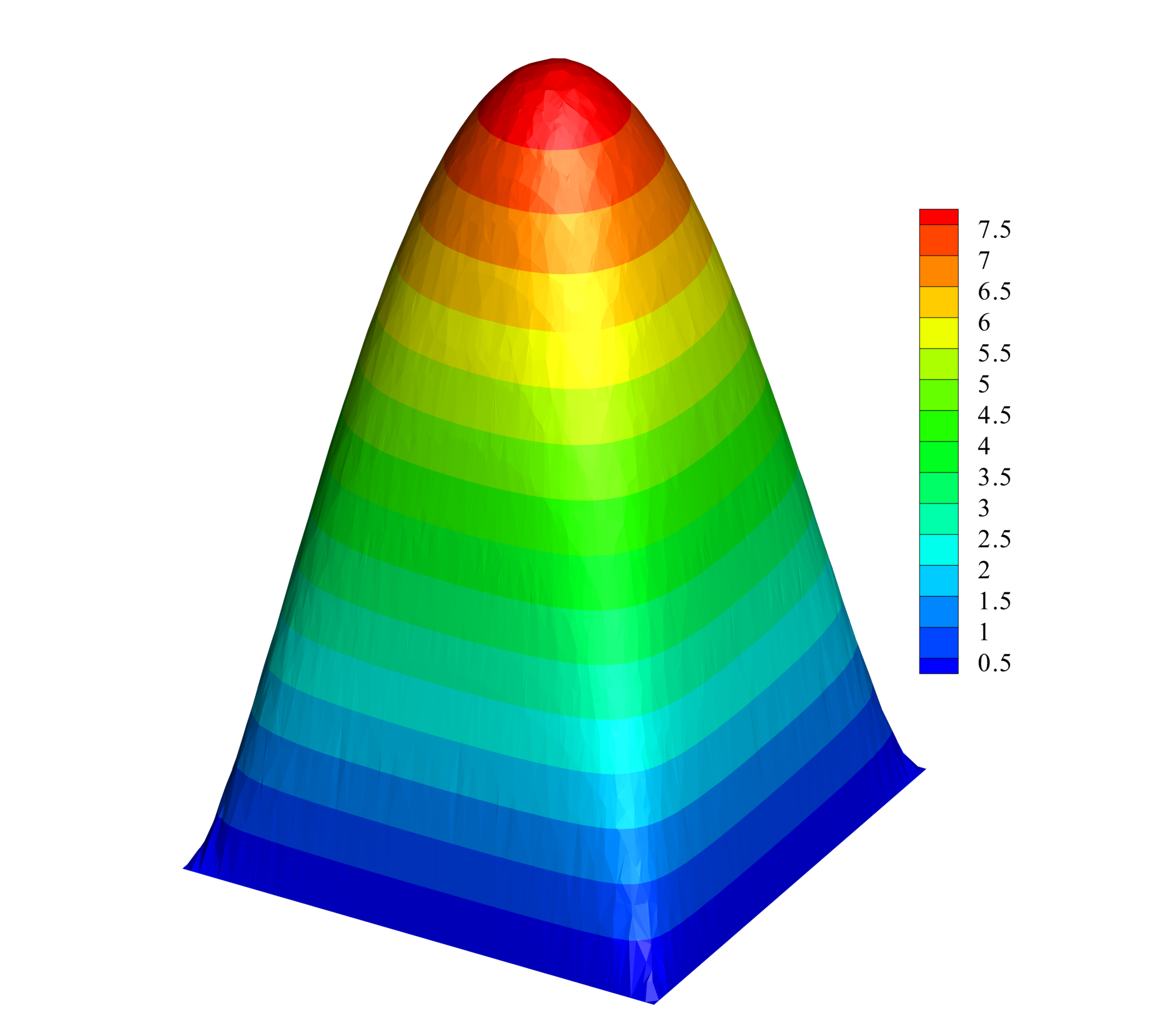}}
\subcaptionbox{$u_{\scriptscriptstyle 1}^{\scriptscriptstyle(1\varepsilon)}$\label{fig:cgd3d_u1B}}
{\includegraphics[width=0.235\textwidth]{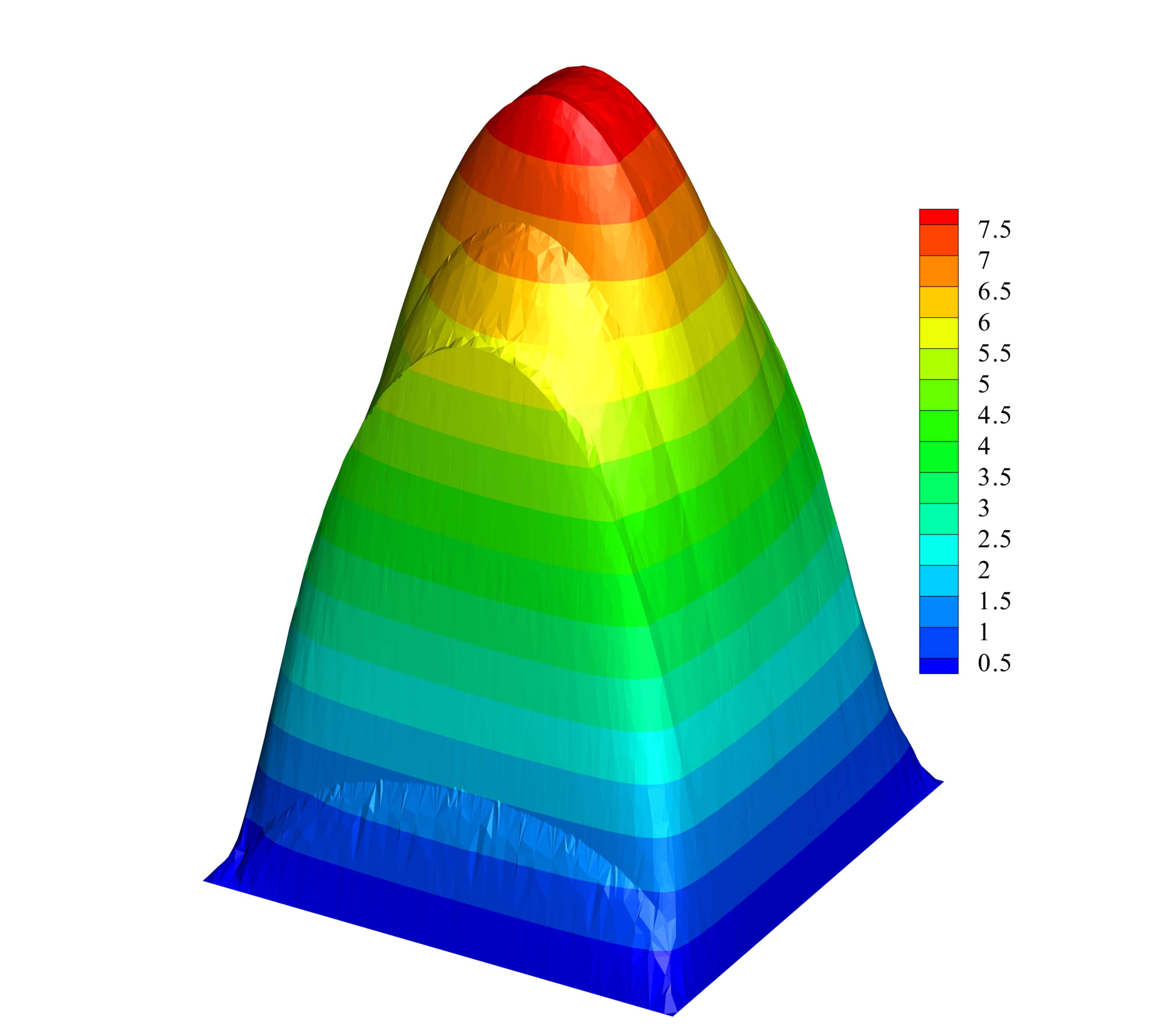}}
\subcaptionbox{$u_{\scriptscriptstyle 1}^{\scriptscriptstyle(2\varepsilon)}$\label{fig:cgd3d_u1C}}
{\includegraphics[width=0.235\textwidth]{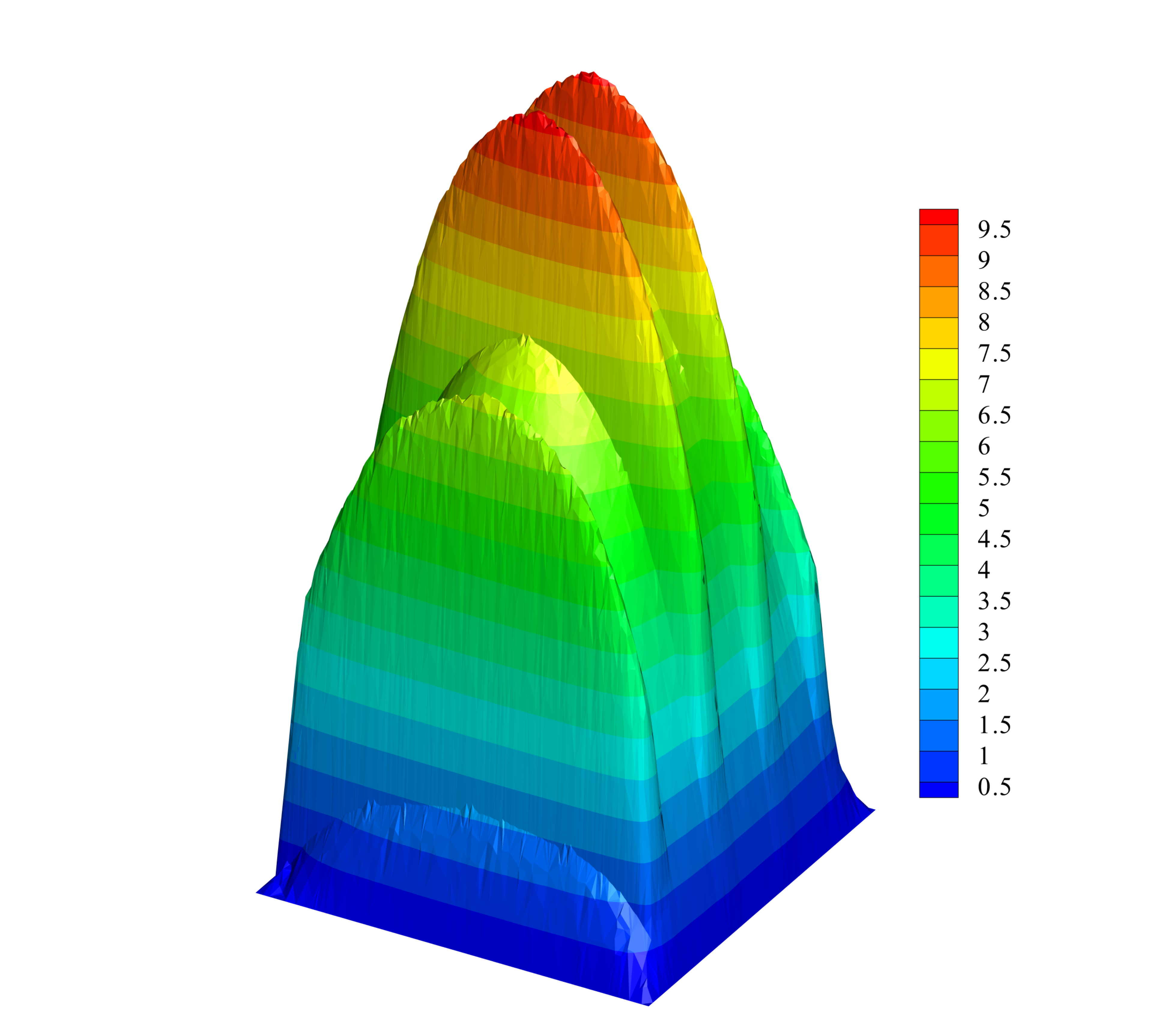}}
\subcaptionbox{$u_{\scriptscriptstyle 1}^{\scriptscriptstyle\varepsilon}$\label{fig:cgd3d_u1D}}
{\includegraphics[width=0.235\textwidth]{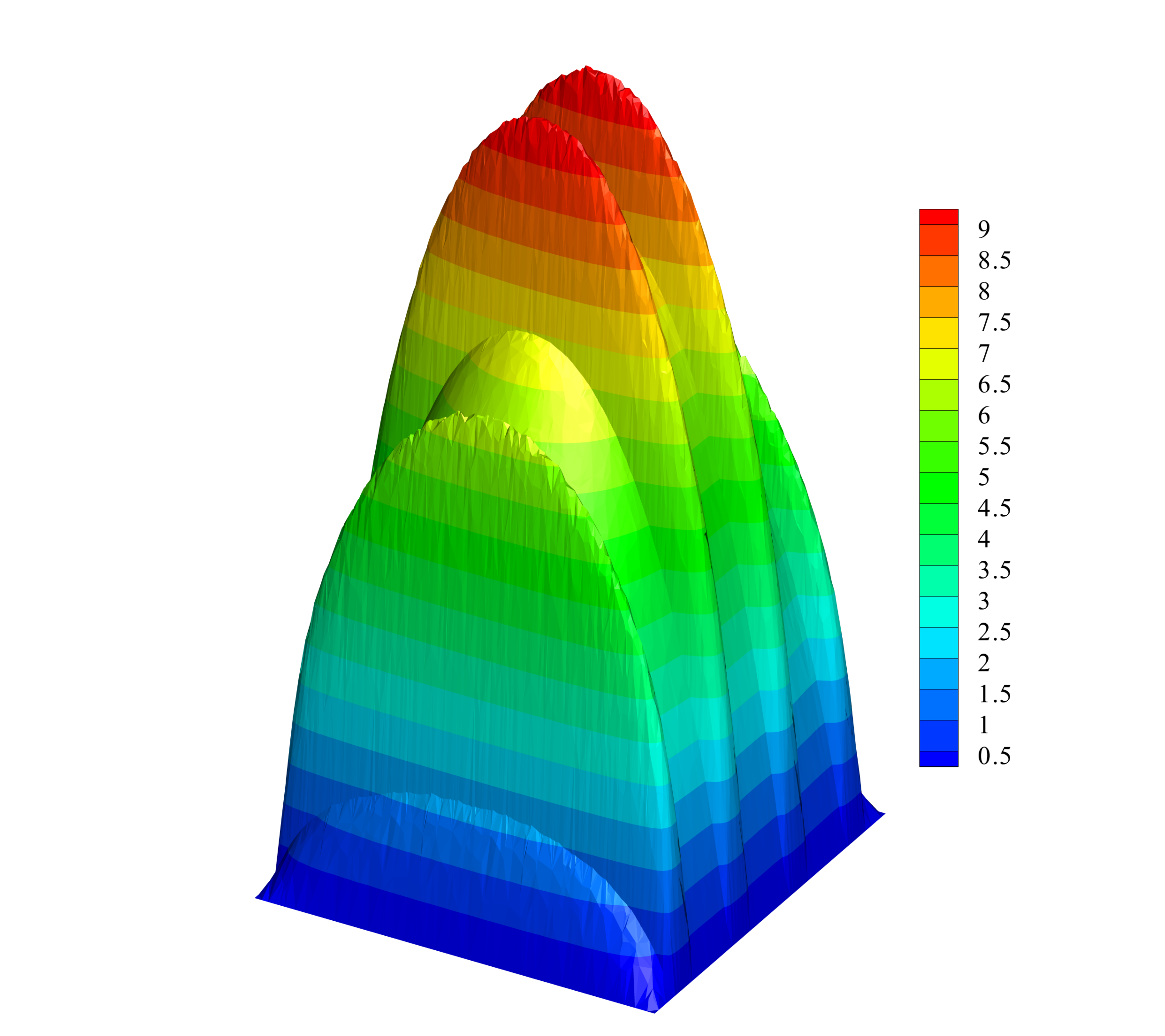}}
\caption{The pressure field in $x_3 = 0.375 $ at $t = 1.0 $ of first-continuum media.}
\label{fig:cgd3du1}
\end{figure}
\begin{figure}[pos=htbp]
\centering
\subcaptionbox{$u_{\scriptscriptstyle 2}^{\scriptscriptstyle(0)}$\label{fig:cgd3d_u2A}}
{\includegraphics[width=0.235\textwidth]{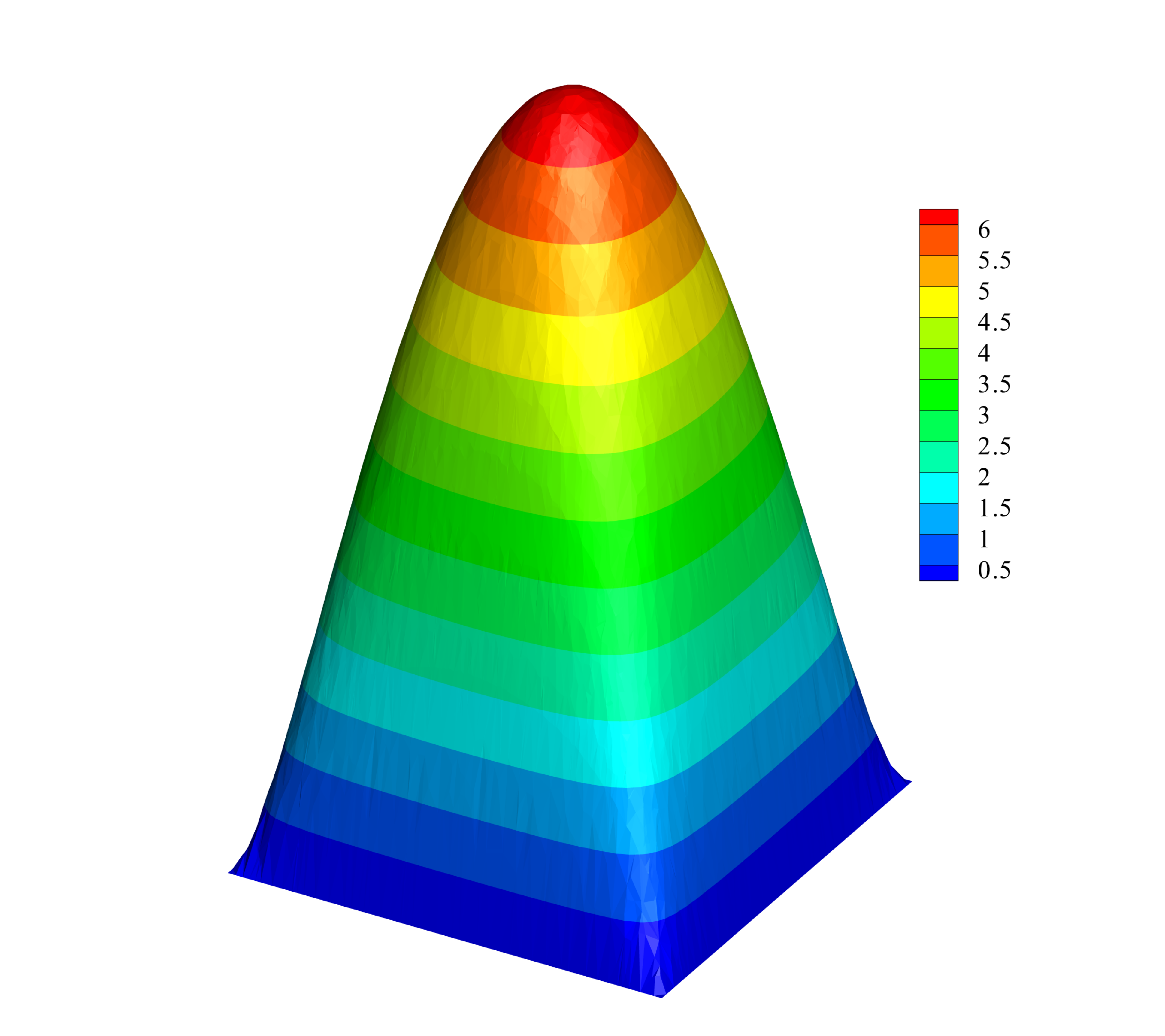}}
\subcaptionbox{$u_{\scriptscriptstyle 2}^{\scriptscriptstyle(1\varepsilon)}$\label{fig:cgd3d_u2B}}
{\includegraphics[width=0.235\textwidth]{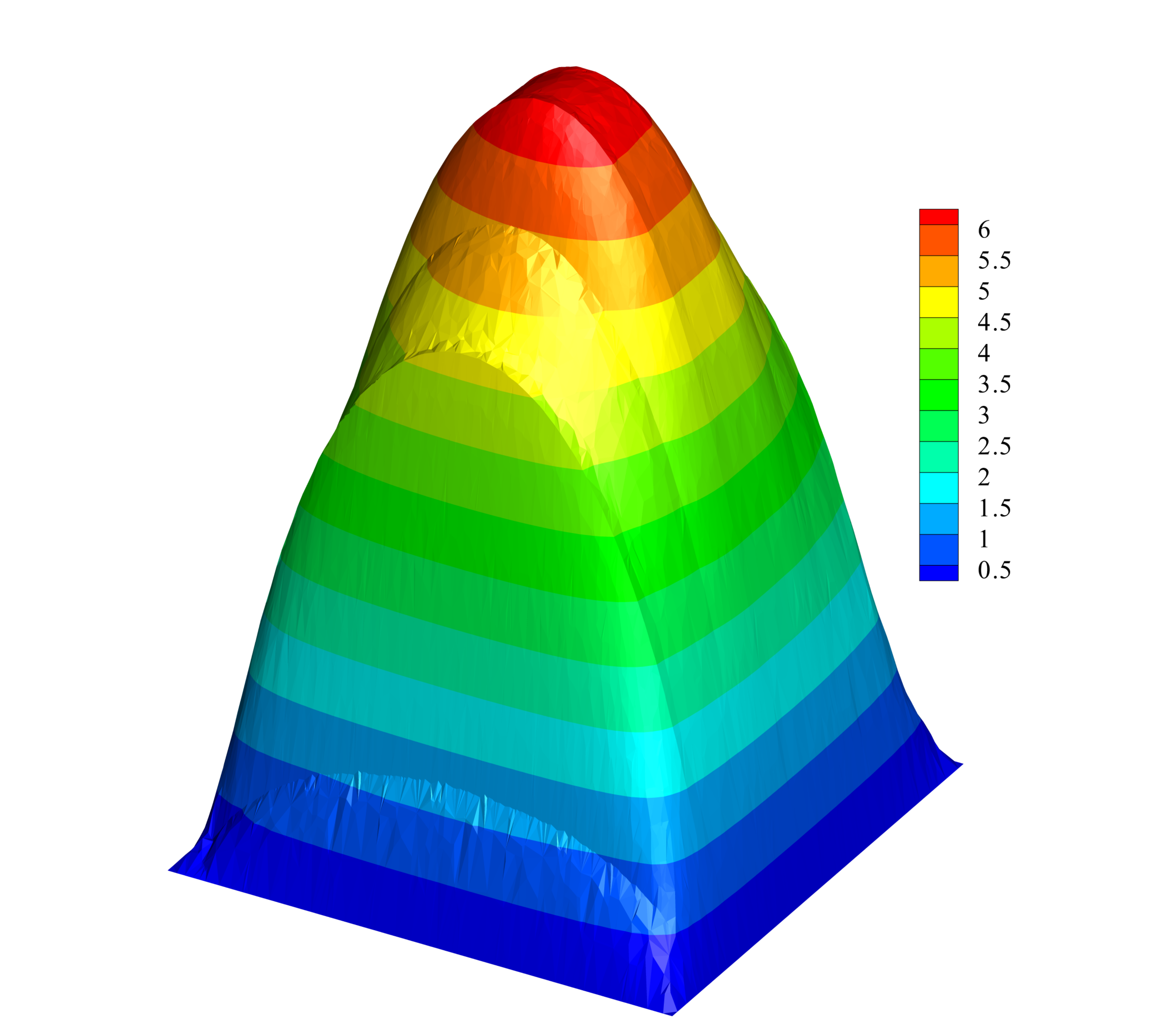}}
\subcaptionbox{$u_{\scriptscriptstyle 2}^{\scriptscriptstyle(2\varepsilon)}$\label{fig:cgd3d_u2C}}
{\includegraphics[width=0.235\textwidth]{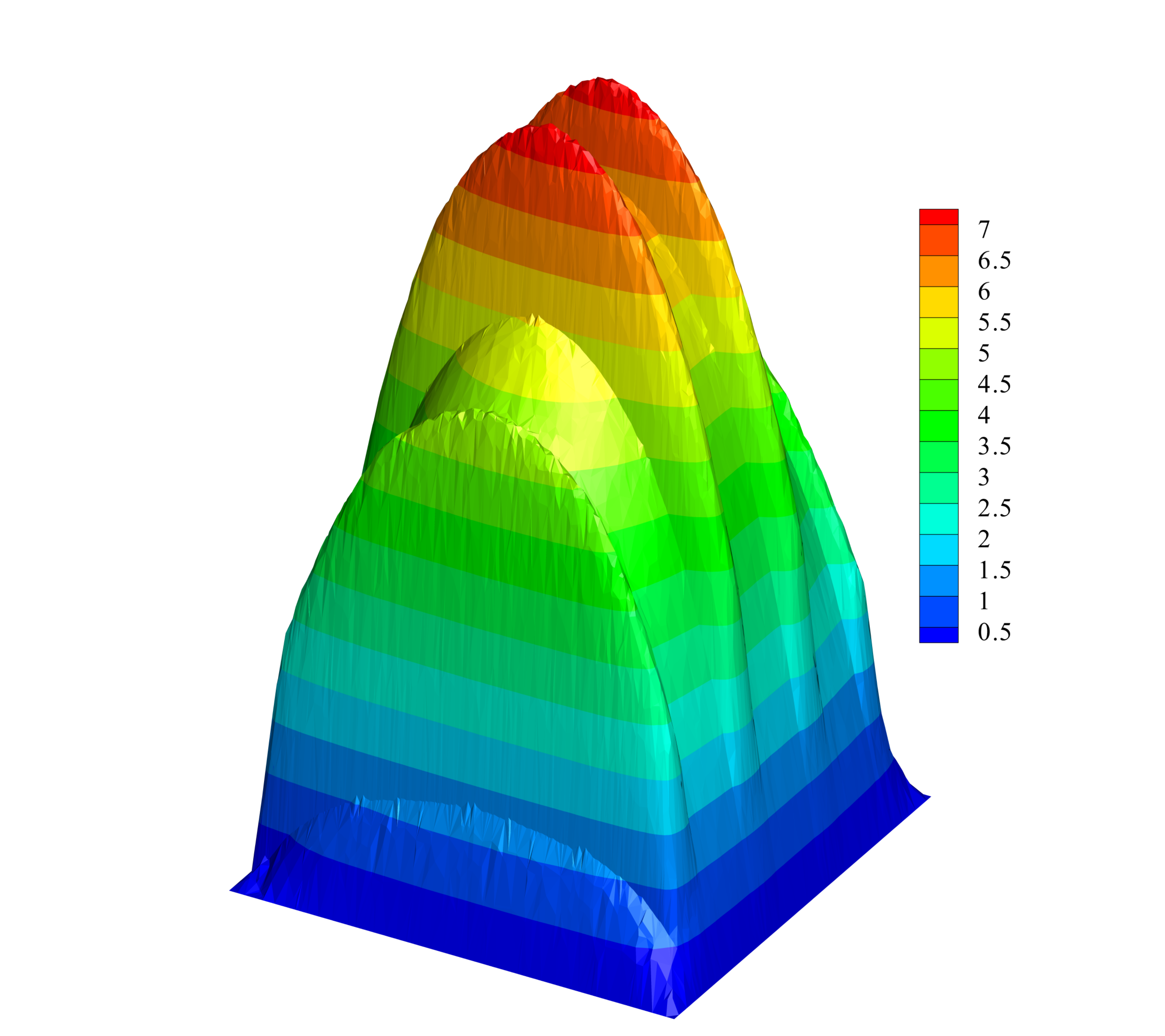}}
\subcaptionbox{$u_{\scriptscriptstyle 2}^{\scriptscriptstyle\varepsilon}$\label{fig:cgd3d_u2D}}
{\includegraphics[width=0.235\textwidth]{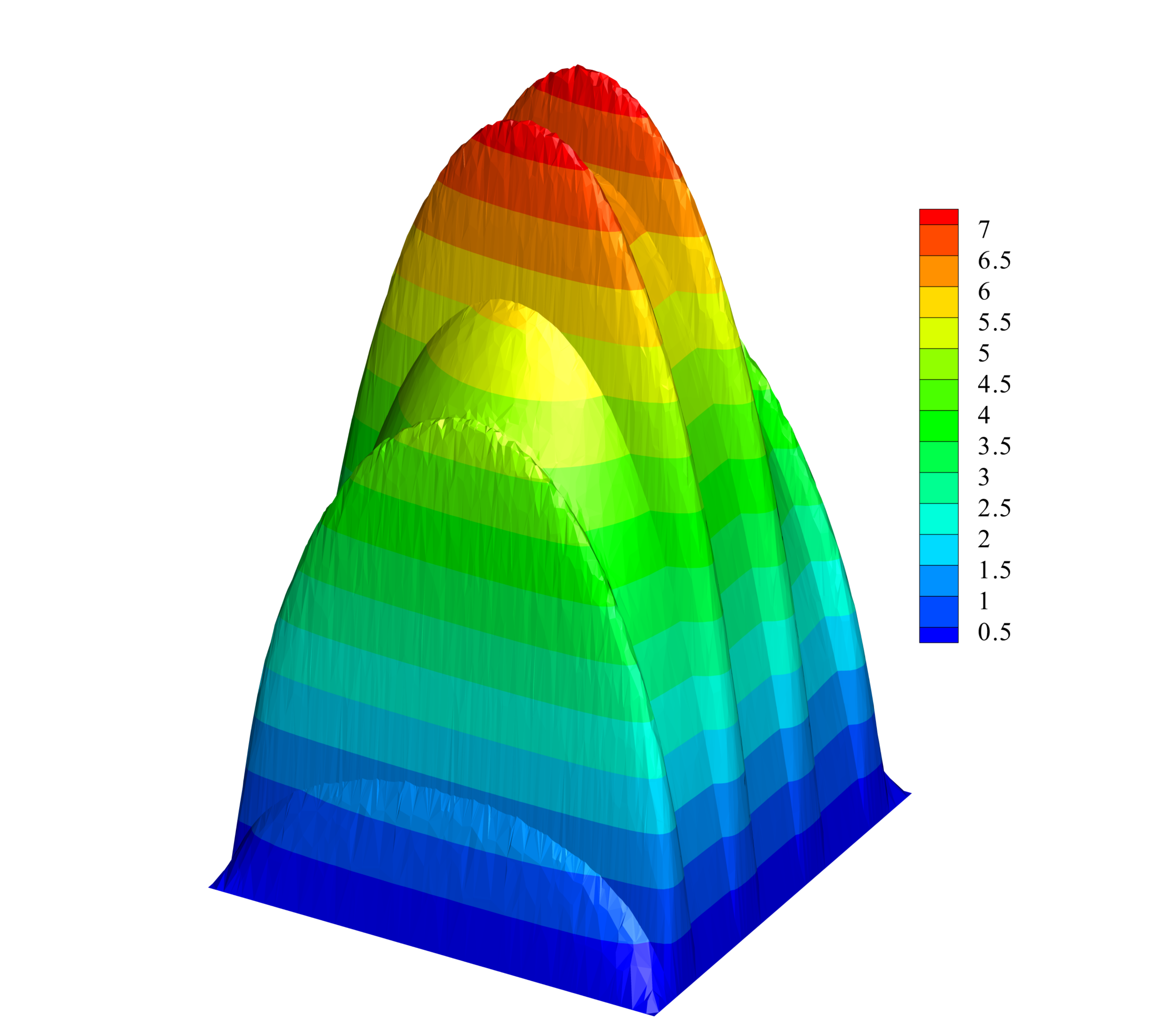}}
\caption{The pressure field in $x_3 = 0.375 $ at $t = 1.0 $ of second-continuum media.}
\label{fig:cgd3du2}
\end{figure}

\begin{figure}[pos=htbp]
\centering
\subcaptionbox{\label{fig:cgd3derA}}
{\includegraphics[width=0.235\textwidth]{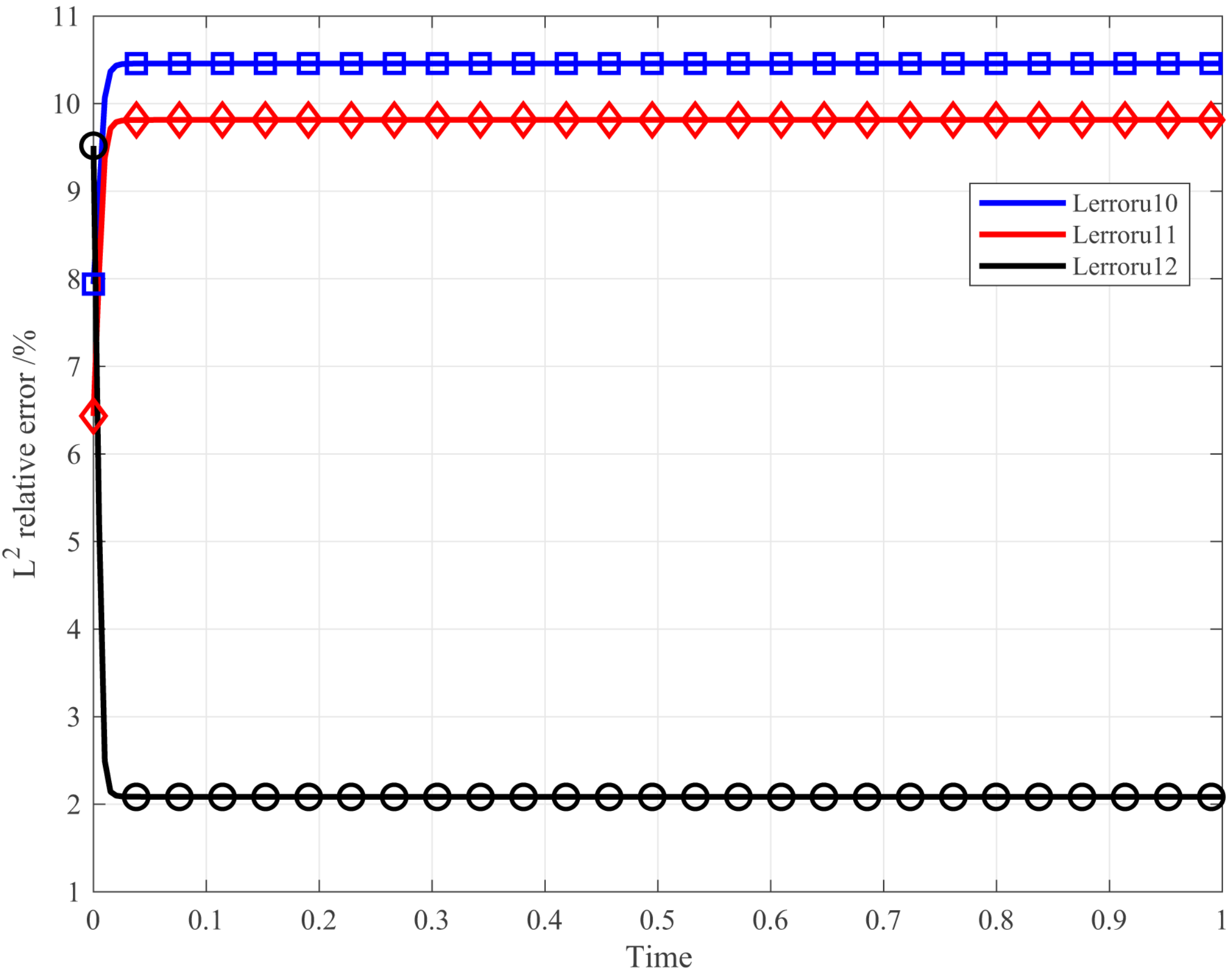}}
\subcaptionbox{\label{fig:cgd3derB}}
{\includegraphics[width=0.235\textwidth]{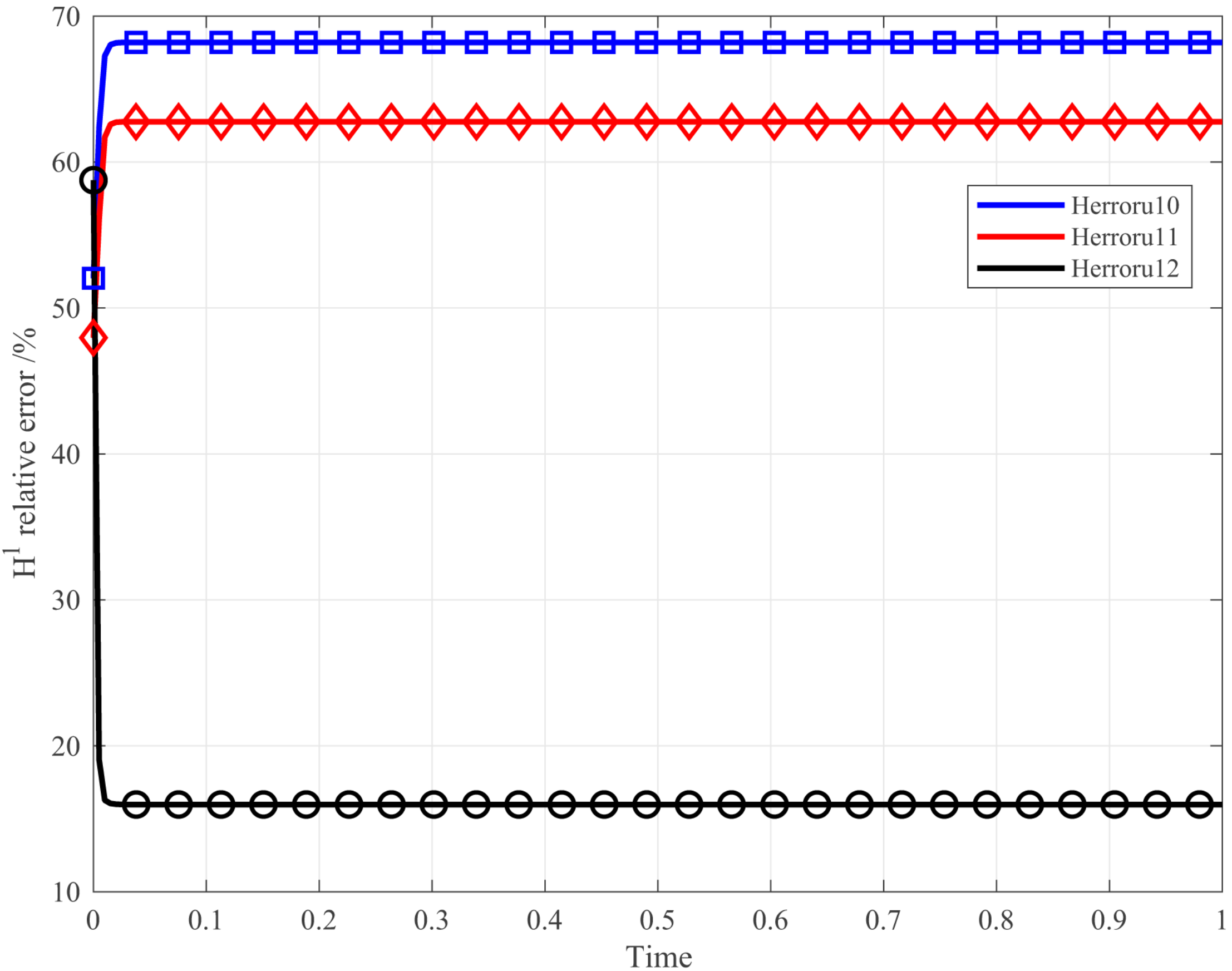}}
\subcaptionbox{\label{fig:cgd3derC}}
{\includegraphics[width=0.235\textwidth]{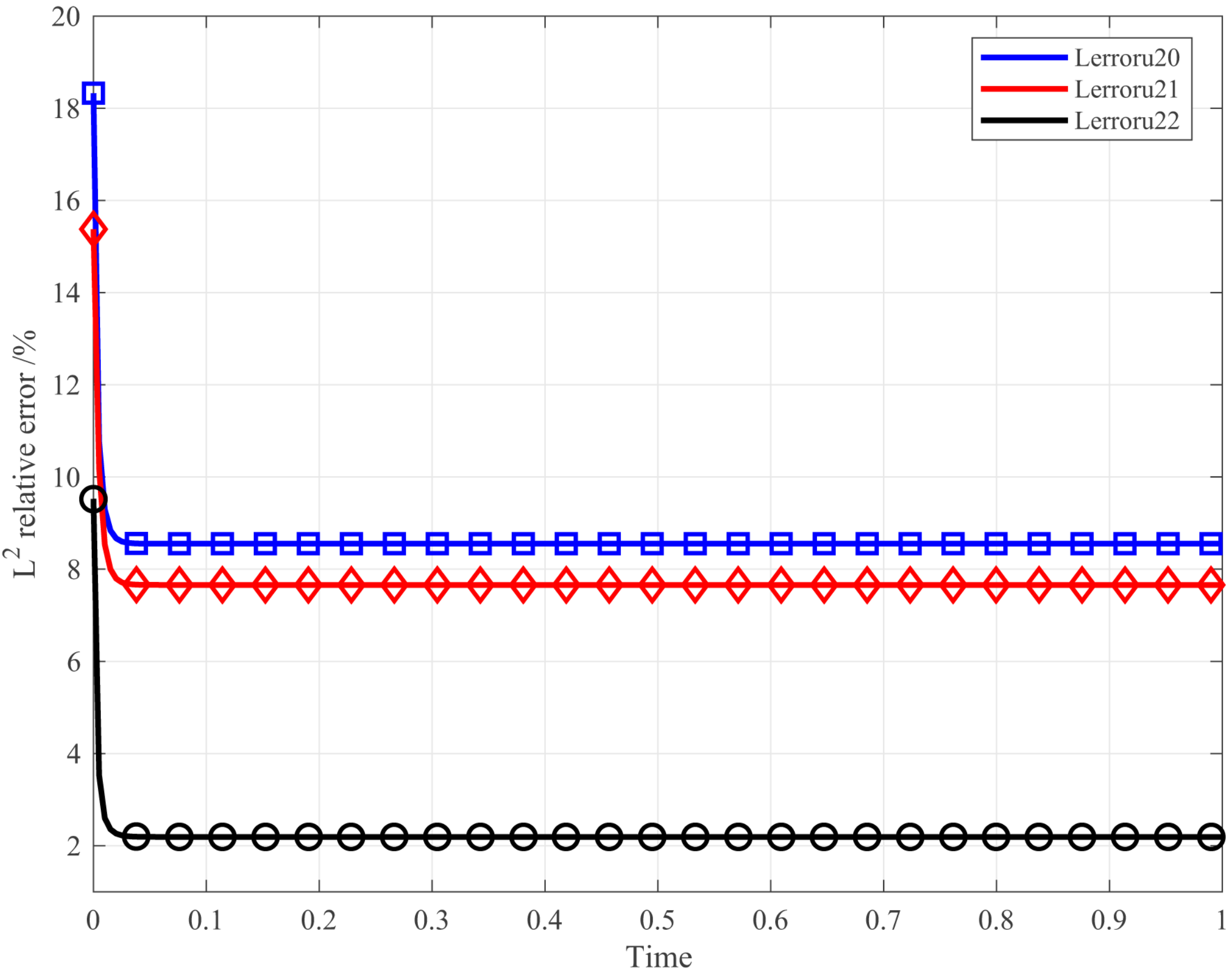}}
\subcaptionbox{\label{fig:cgd3derD}}
{\includegraphics[width=0.235\textwidth]{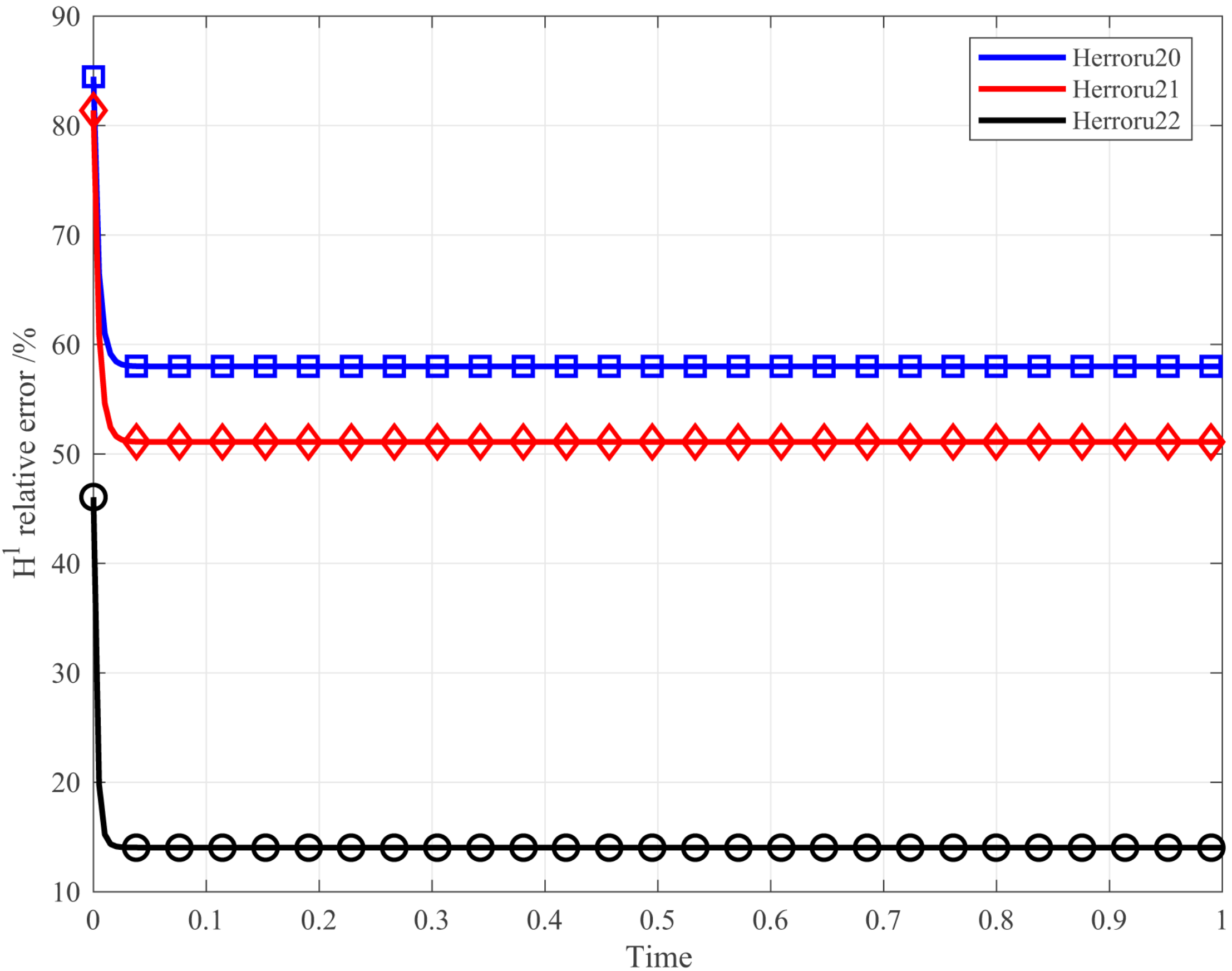}}
\caption{Evolution of the relative errors of the $L^{\scriptscriptstyle 2}$ norm and $H^{\scriptscriptstyle 1}$ seminorm over time: (a) $\text{Lerr1}$; (b) $\text{Herr1}$; (c) $\text{Lerr2}$; (d) $\text{Herr2}$.}\label{fig:Ercgd3d}
\end{figure}

From the mesh information in Table~\ref{tab:cgd3d}, it can be observed that the computational grid resource overhead required by the HOMS method is significantly smaller than that of the direct FEM, and its efficiency is also markedly superior to the FEM. As shown in Fig.~\ref{fig:cgd3du1} and Fig.~\ref{fig:cgd3du2}, the approximation accuracy of the HOMS solution is significantly better than that of both the homogenized solution and the FOMS solution. Only the HOMS solution effectively captures the local oscillatory behavior of the highly heterogeneous media. From Fig.~\ref{fig:Ercgd3d}, it can be observed that under the \( L^{\scriptscriptstyle 2} \) norm, the relative errors of \( u_{\scriptscriptstyle 1}^{\scriptscriptstyle(2\varepsilon)} \) and \( u_{\scriptscriptstyle 2}^{\scriptscriptstyle(2\varepsilon)} \) are only about 2\% and 2\%, respectively; under the \( H^{\scriptscriptstyle 1} \) seminorm, the relative errors of \( u_{\scriptscriptstyle 1}^{\scriptscriptstyle(2\varepsilon)} \) and \( u_{\scriptscriptstyle 2}^{\scriptscriptstyle(2\varepsilon)} \) are only approximately 17\% and 15\%, significantly lower than those of the homogenized solution and the FOMS solution, thus meeting the requirements for engineering computations. Furthermore, as shown in Fig.~\ref{fig:Ercgd3d}, the HOMS method exhibits excellent numerical stability over time and can be effectively applied to the computation of dynamic problem \eqref{eq:3} that evolves with time.

\section{\label{sec:6}Summary and Outlook}
\subsection{Summary of Research Findings}
This paper investigates multi-continuum problems in highly heterogeneous media, and develops corresponding multiscale computational models, multiscale algorithms, and convergence theories. The main achievements are as follows:
\begin{enumerate}
\item A high-precision HOMS framework has been established, defining auxiliary cell functions, deriving macroscopic homogenized equations and equivalent parameter calculation formulas, and further obtaining the HOMS asymptotic solution.  
\item Under the assumption of local equilibrium among the continua, the approximation properties of the HOMS solution to the original equation in the pointwise sense are analyzed, and the convergence order of the HOMS asymptotic solution in the integral sense is proved.
\item By integrating the finite element method, finite difference method, and interpolation method, a corresponding HOMS numerical algorithm has been developed. Numerical experiments were conducted, verifying the effectiveness of the established HOMS method and the necessity of introducing the second-order correction term.  
\end{enumerate}

In summary, the work in this paper provides a powerful numerical simulation tool for studying the physical and mechanical properties of highly heterogeneous media. Furthermore, the proposed method has broad applicability for the numerical solution of other multi-continuum problems.

\subsection{Prospects for Future Research Directions}
Multiscale analysis, as an interdisciplinary field, deeply integrates fundamental theories from materials science, mathematics, and mechanics, and is closely linked with engineering practice, exhibiting significant foundational, cross-disciplinary, and cutting-edge characteristics. Based on the research presented in this paper, the author believes the following directions warrant further in-depth exploration and refinement:
\begin{enumerate}
\item This paper mainly analyzes the coupling interaction of two-fluid continuous media. However, in practical scenarios such as hydrogeology and fractured rock masses in engineering, multi-continuum systems exist widely in nature. Therefore, future research can extend the existing model from two continua to multi-continua.
\item This paper focuses on linear multi-continuum problems. In practical applications, however, factors such as non-Darcy effects induced by high flow velocity, coupling between medium deformation and stress, and interfacial effects between multiphase fluids will lead to significant nonlinear characteristics. Linear models are difficult to accurately describe the above real physical mechanisms. Hence, it is necessary to further study nonlinear multi-continuum coupling problems in the future, and establish nonlinear governing equations and solution methods that are more consistent with actual engineering and natural scenarios.
\item The media considered in this paper are periodic media. In natural hydrogeological fracture networks and engineered porous media systems, affected by formation conditions, evolution processes, external disturbances and other factors, the pore distribution, fracture morphology and structural parameters of the media usually exhibit irregular random distribution characteristics. Random media problems will be further investigated in future work.
\item This paper focuses on the analysis of two-scale fluid-medium systems. In practical engineering and natural scenarios, however, fluid flow often presents three-scale coupling characteristics involving micro-pores, meso-structures, and macro-reservoirs/aquifers. A two-scale framework cannot fully reflect the multi-scale seepage mechanism of fluids in real scenarios. In the future, the research scale can be extended to three-scale fluid-medium systems, and a three-scale coupled fluid pressure solution model covering micro-meso-macro scales can be established to better meet the practical demands of various fields such as hydrogeology and engineering fractures.
\end{enumerate}


\appendix
\section{\label{app}appendix}
The specific forms of \(f_{\scriptscriptstyle 10}(\boldsymbol{x},\boldsymbol{y},t)\) and \(f_{\scriptscriptstyle 20}(\boldsymbol{x},\boldsymbol{y},t)\) are given as follows:
\begin{equation}
\begin{aligned}
&f_{\scriptscriptstyle 10}(\boldsymbol{x},\boldsymbol{y},t) = \bigl(c_{\scriptscriptstyle 1}^{\scriptscriptstyle *} - c_{\scriptscriptstyle 1}\bigr)\frac{\partial u_{\scriptscriptstyle 1}^{\scriptscriptstyle (0)}(\boldsymbol{x},t)}{\partial t} \\
&- \bigl(Q^{\scriptscriptstyle*} + Q_{\scriptscriptstyle 1}\bigl(M_{\scriptscriptstyle 1} + M_{\scriptscriptstyle 2}\bigr)\bigr)\bigl(u_{\scriptscriptstyle 2}^{\scriptscriptstyle(0)}(\boldsymbol{x},t) - u_{\scriptscriptstyle 1}^{\scriptscriptstyle (0)}(\boldsymbol{x},t)\bigr) \\
& + \biggl(\kappa_{\scriptscriptstyle 1,ij} - \kappa_{\scriptscriptstyle 1,ij}^{\scriptscriptstyle *} + \kappa_{\scriptscriptstyle 1,i\alpha_1}\frac{\partial N_{\scriptscriptstyle 1}^{\scriptscriptstyle j}}{\partial y_{\alpha_1}} + \frac{\partial\bigl(\kappa_{\scriptscriptstyle 1,\alpha_1 i}N_{\scriptscriptstyle 1}^{\scriptscriptstyle j}\bigr)}{\partial y_{\alpha_1}}\biggr)\frac{\partial^{\scriptscriptstyle 2} u_{\scriptscriptstyle 1}^{\scriptscriptstyle (0)}(\boldsymbol{x},t)}{\partial x_i\partial x_j} \\
& + \biggl(\bar{K}_{\scriptscriptstyle 2i}^{\scriptscriptstyle 1*} - Q_{\scriptscriptstyle 1} N_{\scriptscriptstyle 1}^{\scriptscriptstyle i} - \kappa_{\scriptscriptstyle 1,i\alpha_1}\frac{\partial M_{\scriptscriptstyle 1}}{\partial y_j} - \frac{\partial\bigl(\kappa_{\scriptscriptstyle 1,ji} M_{\scriptscriptstyle 1}\bigr)}{\partial y_j}\biggr)\frac{\partial u_{\scriptscriptstyle1}^{\scriptscriptstyle(0)}(\boldsymbol{x},t)}{\partial x_i} \\
& + \biggl(Q_{\scriptscriptstyle 1} N_{\scriptscriptstyle 2}^{\scriptscriptstyle i} - \bar{K}_{\scriptscriptstyle 1i}^{\scriptscriptstyle 1*} + \kappa_{\scriptscriptstyle 1,i\alpha_1}\frac{\partial M_{\scriptscriptstyle 1}}{\partial y_j} + \frac{\partial\bigl(\kappa_{\scriptscriptstyle 1,ji}M_{\scriptscriptstyle 1}\bigr)}{\partial y_j}\biggr)\frac{\partial u_{\scriptscriptstyle 2}^{\scriptscriptstyle(0)}(\boldsymbol{x},t)}{\partial x_i}
\end{aligned}
\label{eq:f10}
\end{equation}

\begin{equation}
\begin{aligned}
&f_{\scriptscriptstyle 20}(\boldsymbol{x},\boldsymbol{y},t) = \bigl(c_{\scriptscriptstyle 2}^{\scriptscriptstyle *} - c_{\scriptscriptstyle 2} \bigr)\frac{\partial u_{\scriptscriptstyle 1}^{\scriptscriptstyle (0)}(\boldsymbol{x},t)}{\partial t} \\
& - \bigl(Q^* + Q_{\scriptscriptstyle 2}\bigl(M_{\scriptscriptstyle 1}+ M_{\scriptscriptstyle 2}\bigr)\bigl(u_{\scriptscriptstyle 1}^{\scriptscriptstyle(0)}(\boldsymbol{x},t) - u_{\scriptscriptstyle 2}^{\scriptscriptstyle(0)}(\boldsymbol{x},t)\bigr) \\
& + \biggl(\kappa_{\scriptscriptstyle 2,ij} - \kappa_{\scriptscriptstyle 2,ij}^{\scriptscriptstyle*} + \kappa_{\scriptscriptstyle 2,i\alpha_1} \frac{\partial N_{\scriptscriptstyle 2}^{\scriptscriptstyle j}}{\partial y_{\alpha_1}} + \frac{\partial\bigl(\kappa_{\scriptscriptstyle 2,\alpha_1 i} N_{\scriptscriptstyle 2}^{\scriptscriptstyle j}\bigr)}{\partial y_{\alpha_1}}\biggr)\frac{\partial^{\scriptscriptstyle 2} u_{\scriptscriptstyle 2}^{\scriptscriptstyle (0)}(\boldsymbol{x},t)}{\partial x_i\partial x_j} \\
& + \biggl(\bar{K}_{\scriptscriptstyle 2i}^{\scriptscriptstyle 2*} - Q_{\scriptscriptstyle 2} N_{\scriptscriptstyle 2}^{\scriptscriptstyle i} - \kappa_{\scriptscriptstyle 2,i\alpha_1} \frac{\partial M_{\scriptscriptstyle 2}}{\partial y_j} - \frac{\partial \bigl(\kappa_{\scriptscriptstyle 2,ji} M_{\scriptscriptstyle 2}\bigr)}{\partial y_j}\biggr)\frac{\partial u_{\scriptscriptstyle 2}^{\scriptscriptstyle (0)}(\boldsymbol{x},t)}{\partial x_i} \\
& + \biggl(Q_{\scriptscriptstyle 2} N_{\scriptscriptstyle 1}^{\scriptscriptstyle i} - \bar{K}_{\scriptscriptstyle 1i}^{\scriptscriptstyle 2*} + \kappa_{\scriptscriptstyle 2,i\alpha_1} \frac{\partial M_{\scriptscriptstyle 2}}{\partial y_j} + \frac{\partial\bigl(\kappa_{\scriptscriptstyle 2,ji} M_{\scriptscriptstyle 2}\bigr)}{\partial y_j}\biggr)\frac{\partial u_{\scriptscriptstyle 1}^{\scriptscriptstyle (0)}(\boldsymbol{x},t)}{\partial x_i}
\end{aligned}
\label{eq:f20}
\end{equation}

The specific forms of \(\tilde{f}_{\scriptscriptstyle 1}(\boldsymbol{x},\boldsymbol{y},t)\) and \(\tilde{f}_{\scriptscriptstyle 2}(\boldsymbol{x},\boldsymbol{y},t)\) are given as follows:
\begin{equation}
\begin{aligned}
&\tilde{f}_{\scriptscriptstyle 1}(\boldsymbol{x},\boldsymbol{y},t) = \bigg(
\kappa_{\scriptscriptstyle 1,ij}(\boldsymbol{y}) N_{\scriptscriptstyle 1}^{\scriptscriptstyle \alpha_1}(\boldsymbol{y}) \frac{\partial^{\scriptscriptstyle 3} u_{\scriptscriptstyle 1}^{\scriptscriptstyle(0)}(\boldsymbol{x},t)}{\partial x_i \partial x_j \partial x_{\scriptscriptstyle \alpha_1}} \\
& + M_{\scriptscriptstyle 1}(\boldsymbol{y}) \big( \frac{\partial^{\scriptscriptstyle 2} u_{\scriptscriptstyle 2}^{\scriptscriptstyle(0)}(\boldsymbol{x},t)}{\partial x_i \partial x_j} - \frac{\partial^{\scriptscriptstyle 2} u_{\scriptscriptstyle 1}^{\scriptscriptstyle(0)}(\boldsymbol{x},t)}{\partial x_i \partial x_j} \big) \\
& - c_{\scriptscriptstyle 1}(\boldsymbol{y}) M_{\scriptscriptstyle 1}(\boldsymbol{y}) \big( \frac{\partial u_{\scriptscriptstyle 2}^{\scriptscriptstyle(0)}(\boldsymbol{x},t)}{\partial t} - \frac{\partial u_{\scriptscriptstyle 1}^{\scriptscriptstyle(0)}(\boldsymbol{x},t)}{\partial t} \big) \\
& - c_{\scriptscriptstyle 1}(\boldsymbol{y}) N_{\scriptscriptstyle 1}^{\scriptscriptstyle \alpha_1}(\boldsymbol{y}) \frac{\partial^{\scriptscriptstyle 2} u_{\scriptscriptstyle 1}^{\scriptscriptstyle(0)}(\boldsymbol{x},t)}{\partial t \partial x_{\scriptscriptstyle \alpha_1}} + Q_{\scriptscriptstyle 1}(\boldsymbol{y}) O(1)
\bigg)
\end{aligned}
\label{eq:f1}
\end{equation}

\begin{equation}
\begin{aligned}
&\tilde{f}_{\scriptscriptstyle 2}(\boldsymbol{x},\boldsymbol{y},t) = \bigg(
\kappa_{\scriptscriptstyle 2,ij}(\boldsymbol{y}) N_{\scriptscriptstyle 2}^{\scriptscriptstyle \alpha_1}(\boldsymbol{y}) \frac{\partial^{\scriptscriptstyle 3} u_{\scriptscriptstyle 2}^{\scriptscriptstyle (0)}(\boldsymbol{x},t)}{\partial x_i \partial x_j \partial x_{\alpha_1}} \\
& + M_{\scriptscriptstyle 2}(\boldsymbol{y}) \big( \frac{\partial^{\scriptscriptstyle 2} u_{\scriptscriptstyle 1}^{\scriptscriptstyle(0)}(\boldsymbol{x},t)}{\partial x_i \partial x_j} - \frac{\partial^{\scriptscriptstyle2} u_{\scriptscriptstyle 2}^{\scriptscriptstyle(0)}(\boldsymbol{x},t)}{\partial x_i \partial x_j} \big) \\
& - c_{\scriptscriptstyle 2}(\boldsymbol{y}) M_{\scriptscriptstyle 2}(\boldsymbol{y}) \big( \frac{\partial u_{\scriptscriptstyle 1}^{\scriptscriptstyle(0)}(\boldsymbol{x},t)}{\partial t} - \frac{\partial u_{\scriptscriptstyle 2}^{\scriptscriptstyle(0)}(\boldsymbol{x},t)}{\partial t} \big) \\
& - c_{\scriptscriptstyle 2}(\boldsymbol{y}) N_{\scriptscriptstyle 2}^{\alpha_{\scriptscriptstyle 1}}(\boldsymbol{y}) \frac{\partial^{\scriptscriptstyle2} u_{\scriptscriptstyle 2}^{\scriptscriptstyle(0)}(\boldsymbol{x},t)}{\partial t \partial x_{\alpha_{\scriptscriptstyle 1}}} + Q_{\scriptscriptstyle 2}(\boldsymbol{y}) O(1)
\bigg)
\end{aligned}
\label{eq:f2_tilde}
\end{equation}

The specific forms of \(\tilde{\psi}_{\scriptscriptstyle 1}(\boldsymbol{x})\) and \(\tilde{\psi}_{\scriptscriptstyle 2}(\boldsymbol{x})\) are given as follows:
\begin{equation}
\! \tilde{\psi}_{\scriptscriptstyle 1}(\boldsymbol{x}) \!= \!-N_{\scriptscriptstyle 1}^{\scriptscriptstyle\! \alpha_{\scriptscriptstyle 1}}\!(\boldsymbol{y}) \dfrac{\partial g_{\scriptscriptstyle 1}\!(\boldsymbol{x})}{\partial x_{\scriptscriptstyle \alpha_1}} \!-\! M_{\scriptscriptstyle 1}\!(\boldsymbol{y})\!\bigl(g_{\scriptscriptstyle 2}(\boldsymbol{x}) \!-\! g_{\scriptscriptstyle 1}(\boldsymbol{x})\bigr)
\label{eq:psi_tilde_1}
\end{equation}
\begin{equation}
\! \tilde{\psi}_{\scriptscriptstyle 2}(\boldsymbol{x})\! =\! -\!N_{\scriptscriptstyle 2}^{\scriptscriptstyle\! \alpha_{\scriptscriptstyle 1}}\!
(\boldsymbol{y}) \dfrac{\partial g_{\scriptscriptstyle 2}(\boldsymbol{x})}{\partial x_{\scriptscriptstyle \alpha_1}} \!-\! M_{\scriptscriptstyle 2}\!(\boldsymbol{y})\!\bigl(g_{\scriptscriptstyle 1}(\boldsymbol{x})\! -\! g_{\scriptscriptstyle 2}(\boldsymbol{x})\bigr)
\label{eq:psi_tilde_2}
\end{equation}

The specific forms of \({f}_{\scriptscriptstyle 1}(\boldsymbol{x},\boldsymbol{y},t)\) and \({f}_{\scriptscriptstyle 2}(\boldsymbol{x},\boldsymbol{y},t)\) are given as follows:
\begin{equation}
\begin{aligned}
&f_{\scriptscriptstyle 1} = \bigg(
Q_{\scriptscriptstyle 1} \Big(G_{\scriptscriptstyle 1} \frac{\partial u_{\scriptscriptstyle 1}^{\scriptscriptstyle(0)}}{\partial t} \!+\! N_{\scriptscriptstyle 1}^{\scriptscriptstyle \alpha_1\alpha_2} \frac{\partial^{\scriptscriptstyle 2} u_{\scriptscriptstyle 1}^{\scriptscriptstyle (0)}}{\partial x_{\scriptscriptstyle \alpha_1}\partial x_{\scriptscriptstyle \alpha_2}} + \left(C_{\scriptscriptstyle 1}^{\scriptscriptstyle \alpha_1} \! - \!F_{\scriptscriptstyle 2}^{\scriptscriptstyle \alpha_1}\right) \frac{\partial u_{\scriptscriptstyle 1}^{\scriptscriptstyle (0)}}{\partial x_{\scriptscriptstyle \alpha_1}} \\
&  + \left(F_{\scriptscriptstyle 1}^{\scriptscriptstyle \alpha_1} - C_{\scriptscriptstyle 2}^{\scriptscriptstyle \alpha_1}\right) \frac{\partial u_{\scriptscriptstyle 2}^{\scriptscriptstyle (0)}}{\partial x_{\scriptscriptstyle \alpha_1}} + (K_{\scriptscriptstyle 1} + K_{\scriptscriptstyle 2})(u_{\scriptscriptstyle 2}^{\scriptscriptstyle (0)} - u_{\scriptscriptstyle 1}^{\scriptscriptstyle (0)}) - G_{\scriptscriptstyle 2} \frac{\partial u_{\scriptscriptstyle 2}^{\scriptscriptstyle (0)}}{\partial t}  \\
& - N_{\scriptscriptstyle 2}^{\scriptscriptstyle \alpha_1\alpha_2} \frac{\partial^{\scriptscriptstyle 2} u_{\scriptscriptstyle 2}^{\scriptscriptstyle (0)}}{\partial x_{\scriptscriptstyle \alpha_1}\partial x_{\scriptscriptstyle \alpha_2}}\Big) 
\! -\!  c_{\scriptscriptstyle 1} \Big(N_{\scriptscriptstyle 1}^{\scriptscriptstyle \alpha_1} \frac{\partial^{\scriptscriptstyle 2} u_{\scriptscriptstyle 1}^{\scriptscriptstyle (0)}}{\partial t\partial x_{\scriptscriptstyle \alpha_1}} + M_{\scriptscriptstyle 1} \big( \frac{\partial u_{\scriptscriptstyle 2}^{\scriptscriptstyle (0)}}{\partial t} \!- \! \frac{\partial u_{\scriptscriptstyle 1}^{\scriptscriptstyle (0)}}{\partial t} \big) \! \Big) \\
& + \kappa_{\scriptscriptstyle 1,ij} \Big(N_{\scriptscriptstyle 1}^{\scriptscriptstyle \alpha_1} \frac{\partial^{\scriptscriptstyle 3} u_{\scriptscriptstyle 1}^{\scriptscriptstyle (0)}}{\partial x_{\scriptscriptstyle i}\partial x_{\scriptscriptstyle j}\partial x_{\scriptscriptstyle \alpha_1}} + M_{\scriptscriptstyle 1} \big( \frac{\partial^{\scriptscriptstyle 2} u_{\scriptscriptstyle 2}^{\scriptscriptstyle (0)}}{\partial x_{\scriptscriptstyle i}\partial x_{\scriptscriptstyle j}} \!- \! \frac{\partial^{\scriptscriptstyle 2} u_{\scriptscriptstyle 1}^{\scriptscriptstyle (0)}}{\partial x_{\scriptscriptstyle i}\partial x_{\scriptscriptstyle j}} \big) \\
&+ \frac{\partial G_{\scriptscriptstyle 1}}{\partial y_{\scriptscriptstyle j}} \frac{\partial^{\scriptscriptstyle 2} u_{\scriptscriptstyle 1}^{\scriptscriptstyle (0)}}{\partial x_{\scriptscriptstyle i}\partial t}   + \frac{\partial N_{\scriptscriptstyle 1}^{\scriptscriptstyle \alpha_1\alpha_2}}{\partial y_{\scriptscriptstyle j}} \frac{\partial^{\scriptscriptstyle 3} u_{\scriptscriptstyle 1}^{\scriptscriptstyle (0)}}{\partial x_{\scriptscriptstyle i}\partial x_{\scriptscriptstyle \alpha_1}\partial x_{\scriptscriptstyle \alpha_2}} + \frac{\partial C_{\scriptscriptstyle 1}^{\scriptscriptstyle \alpha_1}}{\partial y_{\scriptscriptstyle j}} \frac{\partial^{\scriptscriptstyle 2} u_{\scriptscriptstyle 1}^{\scriptscriptstyle (0)}}{\partial x_{\scriptscriptstyle i}\partial x_{\scriptscriptstyle \alpha_1}} \\
&\!+ \! \frac{\partial F_{\scriptscriptstyle 1}^{\scriptscriptstyle \alpha_1}}{\partial y_{\scriptscriptstyle j}} \frac{\partial^{\scriptscriptstyle 2} u_{\scriptscriptstyle 2}^{\scriptscriptstyle (0)}}{\partial x_{\scriptscriptstyle i}\partial x_{\scriptscriptstyle \alpha_1}}  
\! + \!\frac{\partial K_{\scriptscriptstyle 1}}{\partial y_{\scriptscriptstyle j}} \! \big( \frac{\partial u_{\scriptscriptstyle 2}^{\scriptscriptstyle (0)}}{\partial x_{\scriptscriptstyle i}} \!-\!  \frac{\partial u_{\scriptscriptstyle 1}^{\scriptscriptstyle (0)}}{\partial x_{\scriptscriptstyle i}} \big) \! \Big) \! + \! \frac{\partial(\kappa_{\scriptscriptstyle 1,ij} G_{\scriptscriptstyle 1})}{\partial y_{\scriptscriptstyle i}} \frac{\partial^{\scriptscriptstyle 2} u_{\scriptscriptstyle 1}^{\scriptscriptstyle (0)}}{\partial x_{\scriptscriptstyle j}\partial t} \\ 
&+\frac{\partial(\kappa_{\scriptscriptstyle 1,ij} N_{\scriptscriptstyle 1}^{\scriptscriptstyle \alpha_1\alpha_2})}{\partial y_{\scriptscriptstyle i}} \frac{\partial^{\scriptscriptstyle 3} u_{\scriptscriptstyle 1}^{\scriptscriptstyle (0)}}{\partial x_{\scriptscriptstyle j}\partial x_{\scriptscriptstyle \alpha_1}\partial x_{\scriptscriptstyle \alpha_2}} + \frac{\partial(\kappa_{\scriptscriptstyle 1,ij} C_{\scriptscriptstyle 1}^{\scriptscriptstyle \alpha_1})}{\partial y_{\scriptscriptstyle i}} \frac{\partial^{\scriptscriptstyle 2} u_{\scriptscriptstyle 1}^{\scriptscriptstyle (0)}}{\partial x_{\scriptscriptstyle j}\partial x_{\scriptscriptstyle \alpha_1}} \\
& + \frac{\partial(\kappa_{\scriptscriptstyle 1,ij} F_{\scriptscriptstyle 1}^{\scriptscriptstyle \alpha_1})}{\partial y_{\scriptscriptstyle i}} \frac{\partial^{\scriptscriptstyle 2} u_{\scriptscriptstyle 2}^{\scriptscriptstyle (0)}}{\partial x_{\scriptscriptstyle j}\partial x_{\scriptscriptstyle \alpha_1}} + \frac{\partial(\kappa_{\scriptscriptstyle 1,ij} K_{\scriptscriptstyle 1})}{\partial y_{\scriptscriptstyle i}}\big(\frac{\partial u_{\scriptscriptstyle 2}^{\scriptscriptstyle (0)}}{\partial x_{\scriptscriptstyle j}} - \frac{\partial u_{\scriptscriptstyle 1}^{\scriptscriptstyle (0)}}{\partial x_{\scriptscriptstyle j}}\big) \\
& - \varepsilon c_{\scriptscriptstyle 1}\!\Big(G_{\scriptscriptstyle 1} \frac{\partial^{\scriptscriptstyle 2} u_{\scriptscriptstyle 1}^{\scriptscriptstyle (0)}}{\partial^{\scriptscriptstyle 2} t} + N_{\scriptscriptstyle 1}^{\scriptscriptstyle \alpha_1\alpha_2} \frac{\partial^{\scriptscriptstyle 3} u_{\scriptscriptstyle 1}^{\scriptscriptstyle (0)}}{\partial t\partial x_{\scriptscriptstyle \alpha_1}\partial x_{\scriptscriptstyle \alpha_2}} + C_{\scriptscriptstyle 1}^{\scriptscriptstyle \alpha_1} \frac{\partial^{\scriptscriptstyle 2} u_{\scriptscriptstyle 1}^{\scriptscriptstyle (0)}}{\partial t\partial x_{\scriptscriptstyle \alpha_1}}\\
&+\! F_{\scriptscriptstyle 1}^{\scriptscriptstyle \alpha_{\scriptscriptstyle 1}} \frac{\partial^{\scriptscriptstyle 2} u_{\scriptscriptstyle 2}^{\scriptscriptstyle (0)}}{\partial t\partial x_{\scriptscriptstyle \alpha_{\scriptscriptstyle 1}}} \!+\! K_{\scriptscriptstyle 1} \big(\frac{\partial u_{\scriptscriptstyle 2}^{\scriptscriptstyle (0)}}{\partial t} \!-\! \frac{\partial u_{\scriptscriptstyle 1}^{\scriptscriptstyle (0)}}{\partial t}\big) \! \Big)  \!+\! \varepsilon \kappa_{\scriptscriptstyle 1,ij}\!\Big(G_{\scriptscriptstyle 1} \frac{\partial^{\scriptscriptstyle 3} u_{\scriptscriptstyle 1}^{\scriptscriptstyle (0)}}{\partial x_{\scriptscriptstyle i}\partial x_{\scriptscriptstyle j}\partial t} \\
&\!+\!  C_{\scriptscriptstyle 1}^{\scriptscriptstyle \alpha_{\scriptscriptstyle 1}} \frac{\partial^{\scriptscriptstyle 3} u_{\scriptscriptstyle 1}^{\scriptscriptstyle (0)}}{\partial x_{\scriptscriptstyle i}\partial x_{\scriptscriptstyle j}\partial x_{\scriptscriptstyle \alpha_{\scriptscriptstyle 1}}}
\!+\! F_{\scriptscriptstyle 1}^{\scriptscriptstyle \alpha_{\scriptscriptstyle 1}} \frac{\partial^{\scriptscriptstyle 3} u_{\scriptscriptstyle 2}^{\scriptscriptstyle (0)}}{\partial x_{\scriptscriptstyle i}\partial x_{\scriptscriptstyle j}\partial x_{\scriptscriptstyle \alpha_{\scriptscriptstyle 1}}} \!+\! Q_{\scriptscriptstyle 1} O(1)\\
&\!+\!N_{\scriptscriptstyle 1}^{\scriptscriptstyle \alpha_{\scriptscriptstyle 1}\alpha_{\scriptscriptstyle 2}} \frac{\partial^{\scriptscriptstyle 4} u_{\scriptscriptstyle 1}^{\scriptscriptstyle (0)}}{\partial x_{\scriptscriptstyle i}\partial x_{\scriptscriptstyle j}\partial x_{\scriptscriptstyle \alpha_{\scriptscriptstyle 1}}\partial x_{\scriptscriptstyle \alpha_{\scriptscriptstyle 2}}} 
+\! K_{\scriptscriptstyle 1} \big( \frac{\partial^{\scriptscriptstyle 2} u_{\scriptscriptstyle 2}^{\scriptscriptstyle (0)}}{\partial x_{\scriptscriptstyle i}\partial x_{\scriptscriptstyle j}} -  \frac{\partial^{\scriptscriptstyle 2} u_{\scriptscriptstyle 1}^{\scriptscriptstyle (0)}}{\partial x_{\scriptscriptstyle i}\partial x_{\scriptscriptstyle j}}\big) \Big)
\bigg)
\label{eq:f1_full}
\end{aligned}
\end{equation}

\begin{equation}
\begin{aligned}
&f_{\scriptscriptstyle 2} = \bigg(
Q_{\scriptscriptstyle 2} \Big(G_{\scriptscriptstyle 2} \frac{\partial u_{\scriptscriptstyle 2}^{\scriptscriptstyle(0)}}{\partial t} \!+\!N_{\scriptscriptstyle 2}^{\scriptscriptstyle \alpha_1\alpha_2} \frac{\partial^{\scriptscriptstyle 2} u_{\scriptscriptstyle 2}^{\scriptscriptstyle (0)}}{\partial x_{\scriptscriptstyle \alpha_1}\partial x_{\scriptscriptstyle \alpha_2}} + \left(C_{\scriptscriptstyle 2}^{\scriptscriptstyle \alpha_1} \! - \!F_{\scriptscriptstyle 1}^{\scriptscriptstyle \alpha_1}\right) \frac{\partial u_{\scriptscriptstyle 2}^{\scriptscriptstyle (0)}}{\partial x_{\scriptscriptstyle \alpha_1}} \\
&  + \left(F_{\scriptscriptstyle 2}^{\scriptscriptstyle \alpha_1} - C_{\scriptscriptstyle 1}^{\scriptscriptstyle \alpha_1}\right) \frac{\partial u_{\scriptscriptstyle 2}^{\scriptscriptstyle (0)}}{\partial x_{\scriptscriptstyle \alpha_1}} + (K_{\scriptscriptstyle 1} + K_{\scriptscriptstyle 2})(u_{\scriptscriptstyle 1}^{\scriptscriptstyle (0)} - u_{\scriptscriptstyle 2}^{\scriptscriptstyle (0)}) - G_{\scriptscriptstyle 1} \frac{\partial u_{\scriptscriptstyle 1}^{\scriptscriptstyle (0)}}{\partial t}  \\
& - N_{\scriptscriptstyle 1}^{\scriptscriptstyle \alpha_1\alpha_2} \frac{\partial^{\scriptscriptstyle 2} u_{\scriptscriptstyle 1}^{\scriptscriptstyle (0)}}{\partial x_{\scriptscriptstyle \alpha_1}\partial x_{\scriptscriptstyle \alpha_2}}\Big) 
\! -\!  c_{\scriptscriptstyle 2} \Big(N_{\scriptscriptstyle 2}^{\scriptscriptstyle \alpha_1} \frac{\partial^{\scriptscriptstyle 2} u_{\scriptscriptstyle 2}^{\scriptscriptstyle (0)}}{\partial t\partial x_{\scriptscriptstyle \alpha_1}} + M_{\scriptscriptstyle 2} \big( \frac{\partial u_{\scriptscriptstyle 1}^{\scriptscriptstyle (0)}}{\partial t} \!- \! \frac{\partial u_{\scriptscriptstyle 2}^{\scriptscriptstyle (0)}}{\partial t} \big) \! \Big) \\
& + \kappa_{\scriptscriptstyle 2,ij} \Big(N_{\scriptscriptstyle 2}^{\scriptscriptstyle \alpha_1} \frac{\partial^{\scriptscriptstyle 3} u_{\scriptscriptstyle 2}^{\scriptscriptstyle (0)}}{\partial x_{\scriptscriptstyle i}\partial x_{\scriptscriptstyle j}\partial x_{\scriptscriptstyle \alpha_1}} + M_{\scriptscriptstyle 2} \big( \frac{\partial^{\scriptscriptstyle 2} u_{\scriptscriptstyle 1}^{\scriptscriptstyle (0)}}{\partial x_{\scriptscriptstyle i}\partial x_{\scriptscriptstyle j}} \!- \! \frac{\partial^{\scriptscriptstyle 2} u_{\scriptscriptstyle 2}^{\scriptscriptstyle (0)}}{\partial x_{\scriptscriptstyle i}\partial x_{\scriptscriptstyle j}} \big) \\
&+ \frac{\partial G_{\scriptscriptstyle 2}}{\partial y_{\scriptscriptstyle j}} \frac{\partial^{\scriptscriptstyle 2} u_{\scriptscriptstyle 2}^{\scriptscriptstyle (0)}}{\partial x_{\scriptscriptstyle i}\partial t}   + \frac{\partial N_{\scriptscriptstyle 2}^{\scriptscriptstyle \alpha_1\alpha_2}}{\partial y_{\scriptscriptstyle j}} \frac{\partial^{\scriptscriptstyle 3} u_{\scriptscriptstyle 2}^{\scriptscriptstyle (0)}}{\partial x_{\scriptscriptstyle i}\partial x_{\scriptscriptstyle \alpha_1}\partial x_{\scriptscriptstyle \alpha_2}} + \frac{\partial C_{\scriptscriptstyle 2}^{\scriptscriptstyle \alpha_1}}{\partial y_{\scriptscriptstyle j}} \frac{\partial^{\scriptscriptstyle 2} u_{\scriptscriptstyle 2}^{\scriptscriptstyle (0)}}{\partial x_{\scriptscriptstyle i}\partial x_{\scriptscriptstyle \alpha_1}} \\
&+ \! \frac{\partial F_{\scriptscriptstyle 2}^{\scriptscriptstyle \alpha_1}}{\partial y_{\scriptscriptstyle j}} \frac{\partial^{\scriptscriptstyle 2} u_{\scriptscriptstyle 1}^{\scriptscriptstyle (0)}}{\partial x_{\scriptscriptstyle i}\partial x_{\scriptscriptstyle \alpha_1}}  
\! + \!\frac{\partial K_{\scriptscriptstyle 2}}{\partial y_{\scriptscriptstyle j}} \! \big( \frac{\partial u_{\scriptscriptstyle 1}^{\scriptscriptstyle (0)}}{\partial x_{\scriptscriptstyle i}} \!-\!  \frac{\partial u_{\scriptscriptstyle 2}^{\scriptscriptstyle (0)}}{\partial x_{\scriptscriptstyle i}} \big) \! \Big) \! + \! \frac{\partial(\kappa_{\scriptscriptstyle 2,ij} G_{\scriptscriptstyle 2})}{\partial y_{\scriptscriptstyle i}} \frac{\partial^{\scriptscriptstyle 2} u_{\scriptscriptstyle 2}^{\scriptscriptstyle (0)}}{\partial x_{\scriptscriptstyle j}\partial t} \\ &+\frac{\partial(\kappa_{\scriptscriptstyle 2,ij} N_{\scriptscriptstyle 2}^{\scriptscriptstyle \alpha_1\alpha_2})}{\partial y_{\scriptscriptstyle i}} \frac{\partial^{\scriptscriptstyle 3} u_{\scriptscriptstyle 2}^{\scriptscriptstyle (0)}}{\partial x_{\scriptscriptstyle j}\partial x_{\scriptscriptstyle \alpha_1}\partial x_{\scriptscriptstyle \alpha_2}} + \frac{\partial(\kappa_{\scriptscriptstyle 2,ij} C_{\scriptscriptstyle 2}^{\scriptscriptstyle \alpha_1})}{\partial y_{\scriptscriptstyle i}} \frac{\partial^{\scriptscriptstyle 2} u_{\scriptscriptstyle 2}^{\scriptscriptstyle (0)}}{\partial x_{\scriptscriptstyle j}\partial x_{\scriptscriptstyle \alpha_1}} \\
& + \frac{\partial(\kappa_{\scriptscriptstyle 2,ij} F_{\scriptscriptstyle 2}^{\scriptscriptstyle \alpha_1})}{\partial y_{\scriptscriptstyle i}} \frac{\partial^{\scriptscriptstyle 2} u_{\scriptscriptstyle 1}^{\scriptscriptstyle (0)}}{\partial x_{\scriptscriptstyle j}\partial x_{\scriptscriptstyle \alpha_1}} + \frac{\partial(\kappa_{\scriptscriptstyle 2,ij} K_{\scriptscriptstyle 2})}{\partial y_{\scriptscriptstyle i}}\big(\frac{\partial u_{\scriptscriptstyle 1}^{\scriptscriptstyle (0)}}{\partial x_{\scriptscriptstyle j}} - \frac{\partial u_{\scriptscriptstyle 2}^{\scriptscriptstyle (0)}}{\partial x_{\scriptscriptstyle j}}\big) \\
& - \varepsilon c_{\scriptscriptstyle 2}\!\Big(G_{\scriptscriptstyle 2} \frac{\partial^{\scriptscriptstyle 2} u_{\scriptscriptstyle 2}^{\scriptscriptstyle (0)}}{\partial^{\scriptscriptstyle 2} t} + N_{\scriptscriptstyle 2}^{\scriptscriptstyle \alpha_1\alpha_2} \frac{\partial^{\scriptscriptstyle 3} u_{\scriptscriptstyle 2}^{\scriptscriptstyle (0)}}{\partial t\partial x_{\scriptscriptstyle \alpha_1}\partial x_{\scriptscriptstyle \alpha_2}} + C_{\scriptscriptstyle 2}^{\scriptscriptstyle \alpha_1} \frac{\partial^{\scriptscriptstyle 2} u_{\scriptscriptstyle 2}^{\scriptscriptstyle (0)}}{\partial t\partial x_{\scriptscriptstyle \alpha_1}}\\
&+ F_{\scriptscriptstyle 2}^{\scriptscriptstyle \alpha_{\scriptscriptstyle 1}} \frac{\partial^{\scriptscriptstyle 2} u_{\scriptscriptstyle 1}^{\scriptscriptstyle (0)}}{\partial t\partial x_{\scriptscriptstyle \alpha_{\scriptscriptstyle 1}}} + K_{\scriptscriptstyle 2} \frac{\partial u_{\scriptscriptstyle 1}^{\scriptscriptstyle (0)}}{\partial t} - K_{\scriptscriptstyle 2} \frac{\partial u_{\scriptscriptstyle 2}^{\scriptscriptstyle (0)}}{\partial t} + Q_{\scriptscriptstyle 1} O(1)\Big) \\
& + \varepsilon \kappa_{\scriptscriptstyle 2,ij}\!\Big(G_{\scriptscriptstyle 2} \frac{\partial^{\scriptscriptstyle 3} u_{\scriptscriptstyle 2}^{\scriptscriptstyle (0)}}{\partial x_{\scriptscriptstyle i}\partial x_{\scriptscriptstyle j}\partial t} \!+\!  C_{\scriptscriptstyle 2}^{\scriptscriptstyle \alpha_{\scriptscriptstyle 1}} \frac{\partial^{\scriptscriptstyle 3} u_{\scriptscriptstyle 2}^{\scriptscriptstyle (0)}}{\partial x_{\scriptscriptstyle i}\partial x_{\scriptscriptstyle j}\partial x_{\scriptscriptstyle \alpha_{\scriptscriptstyle 1}}}
\!+\! F_{\scriptscriptstyle 2}^{\scriptscriptstyle \alpha_{\scriptscriptstyle 1}} \frac{\partial^{\scriptscriptstyle 3} u_{\scriptscriptstyle 1}^{\scriptscriptstyle (0)}}{\partial x_{\scriptscriptstyle i}\partial x_{\scriptscriptstyle j}\partial x_{\scriptscriptstyle \alpha_{\scriptscriptstyle 1}}}\\ 
&+N_{\scriptscriptstyle 2}^{\scriptscriptstyle \alpha_{\scriptscriptstyle 1}\alpha_{\scriptscriptstyle 2}} \frac{\partial^{\scriptscriptstyle 4} u_{\scriptscriptstyle 2}^{\scriptscriptstyle (0)}}{\partial x_{\scriptscriptstyle i}\partial x_{\scriptscriptstyle j}\partial x_{\scriptscriptstyle \alpha_{\scriptscriptstyle 1}}\partial x_{\scriptscriptstyle \alpha_{\scriptscriptstyle 2}}} +\! K_{\scriptscriptstyle 2} \frac{\partial^{\scriptscriptstyle 2} u_{\scriptscriptstyle 1}^{\scriptscriptstyle (0)}}{\partial x_{\scriptscriptstyle i}\partial x_{\scriptscriptstyle j}} - K_{\scriptscriptstyle 2} \frac{\partial^{\scriptscriptstyle 2} u_{\scriptscriptstyle 2}^{\scriptscriptstyle (0)}}{\partial x_{\scriptscriptstyle i}\partial x_{\scriptscriptstyle j}}\Big)
\bigg)
\label{eq:f2_full}
\end{aligned}
\end{equation}

The specific forms of \({\psi}_{\scriptscriptstyle 1}(\boldsymbol{x})\) and \({\psi}_{\scriptscriptstyle 2}(\boldsymbol{x})\) are given as follows:
\begin{equation}
\begin{aligned}
&\psi_{\scriptscriptstyle 1}(\boldsymbol{x}) = 
- N_{\scriptscriptstyle 1}^{\scriptscriptstyle \alpha_{1}}(\boldsymbol{y}) \frac{\partial g_{\scriptscriptstyle 1}(\boldsymbol{x})}{\partial x_{\scriptscriptstyle \alpha_{1}}} \\
&- M_{\scriptscriptstyle 1}(\boldsymbol{y}) \big( g_{\scriptscriptstyle 2}(\boldsymbol{x})\! - g_{\scriptscriptstyle 1}(\boldsymbol{x})  \big) \!
- \varepsilon G_{\scriptscriptstyle 1}(\boldsymbol{y}) \frac{\partial u_{\scriptscriptstyle 1}^{\scriptscriptstyle (0)}(\boldsymbol{x}, t)}{\partial t} \bigg|_{t=0} \\
&- \varepsilon N_{\scriptscriptstyle 1}^{\scriptscriptstyle \alpha_1\alpha_{2}}(\boldsymbol{y}) \frac{\partial^{2} g_{\scriptscriptstyle 1}(\boldsymbol{x})}{\partial x_{\scriptscriptstyle \alpha_{1}} \partial x_{\scriptscriptstyle \alpha_{2}}} 
- \varepsilon C_{\scriptscriptstyle 1}^{\scriptscriptstyle \alpha_{1}}(\boldsymbol{y}) \frac{\partial g_{\scriptscriptstyle 1}(\boldsymbol{x})}{\partial x_{\scriptscriptstyle \alpha_{1}}} \\
&- \varepsilon F_{\scriptscriptstyle 1}^{\scriptscriptstyle \alpha_{1}}(\boldsymbol{y}) \frac{\partial g_{\scriptscriptstyle 2}(\boldsymbol{x})}{\partial x_{\scriptscriptstyle \alpha_{1}}} 
- \varepsilon K_{\scriptscriptstyle 1}(\boldsymbol{y}) \big(  g_{\scriptscriptstyle 2}(\boldsymbol{x}) \!- g_{\scriptscriptstyle 1}(\boldsymbol{x})  \big),
\end{aligned}
\label{eq:psi_expressions_1}
\end{equation}
\begin{equation}
\begin{aligned}
&\psi_{\scriptscriptstyle 2}(\boldsymbol{x}) = 
- N_{\scriptscriptstyle 2}^{\scriptscriptstyle \alpha_{1}}(\boldsymbol{y}) \frac{\partial g_{\scriptscriptstyle 2}(\boldsymbol{x})}{\partial x_{\scriptscriptstyle \alpha_{1}}} \\
 &- M_{\scriptscriptstyle 2}(\boldsymbol{y}) \big( g_{\scriptscriptstyle 1}(\boldsymbol{x})\! - g_{\scriptscriptstyle 2}(\boldsymbol{x}) \big) 
\! - \varepsilon G_{\scriptscriptstyle 2}(\boldsymbol{y}) \frac{\partial u_{\scriptscriptstyle 2}^{\scriptscriptstyle (0)}(\boldsymbol{x}, t)}{\partial t} \bigg|_{t=0} \\
&- \varepsilon N_{\scriptscriptstyle 2}^{\scriptscriptstyle \alpha_{\scriptscriptstyle 1}\alpha_{\scriptscriptstyle 2}}(\boldsymbol{y}) \frac{\partial^{2} g_{\scriptscriptstyle 2}(\boldsymbol{x})}{\partial x_{\scriptscriptstyle \alpha_{1}} \partial x_{\scriptscriptstyle \alpha_{2}}} 
- \varepsilon C_{\scriptscriptstyle 2}^{\scriptscriptstyle \alpha_{1}}(\boldsymbol{y}) \frac{\partial g_{\scriptscriptstyle 2}^(\boldsymbol{x})}{\partial x_{\scriptscriptstyle \alpha_{1}}} \\
&- \varepsilon F_{\scriptscriptstyle 2}^{\scriptscriptstyle \alpha_{1}}(\boldsymbol{y}) \frac{\partial g_{\scriptscriptstyle 1}(\boldsymbol{x})}{\partial x_{\scriptscriptstyle \alpha_{1}}} 
- \varepsilon K_{\scriptscriptstyle 2}(\boldsymbol{y}) \big( g_{\scriptscriptstyle 1}(\boldsymbol{x}) - g_{\scriptscriptstyle 2}(\boldsymbol{x}) \big),
\end{aligned}
\label{eq:psi_expressions_2}
\end{equation}


%


\begin{thebibliography}{41}%
\makeatletter
\providecommand \@ifxundefined [1]{%
 \@ifx{#1\undefined}
}%
\providecommand \@ifnum [1]{%
 \ifnum #1\expandafter \@firstoftwo
 \else \expandafter \@secondoftwo
 \fi
}%
\providecommand \@ifx [1]{%
 \ifx #1\expandafter \@firstoftwo
 \else \expandafter \@secondoftwo
 \fi
}%
\providecommand \natexlab [1]{#1}%
\providecommand \enquote  [1]{``#1''}%
\providecommand \bibnamefont  [1]{#1}%
\providecommand \bibfnamefont [1]{#1}%
\providecommand \citenamefont [1]{#1}%
\providecommand \href@noop [0]{\@secondoftwo}%
\providecommand \href [0]{\begingroup \@sanitize@url \@href}%
\providecommand \@href[1]{\@@startlink{#1}\@@href}%
\providecommand \@@href[1]{\endgroup#1\@@endlink}%
\providecommand \@sanitize@url [0]{\catcode `\\12\catcode `\$12\catcode `\&12\catcode `\#12\catcode `\^12\catcode `\_12\catcode `\%12\relax}%
\providecommand \@@startlink[1]{}%
\providecommand \@@endlink[0]{}%
\providecommand \url  [0]{\begingroup\@sanitize@url \@url }%
\providecommand \@url [1]{\endgroup\@href {#1}{\urlprefix }}%
\providecommand \urlprefix  [0]{URL }%
\providecommand \Eprint [0]{\href }%
\providecommand \doibase [0]{http://dx.doi.org/}%
\providecommand \selectlanguage [0]{\@gobble}%
\providecommand \bibinfo  [0]{\@secondoftwo}%
\providecommand \bibfield  [0]{\@secondoftwo}%
\providecommand \translation [1]{[#1]}%
\providecommand \BibitemOpen [0]{}%
\providecommand \bibitemStop [0]{}%
\providecommand \bibitemNoStop [0]{.\EOS\space}%
\providecommand \EOS [0]{\spacefactor3000\relax}%
\providecommand \BibitemShut  [1]{\csname bibitem#1\endcsname}%
\let\auto@bib@innerbib\@empty
\bibitem [{\citenamefont {Zhao}\ \emph {et~al.}(2022)\citenamefont {Zhao}, \citenamefont {Wang}, \citenamefont {Sima}, \citenamefont {Guo},\ and\ \citenamefont {Zhang}}]{ZHAO2022110987}%
  \BibitemOpen
  \bibfield  {author} {\bibinfo {author} {\bibfnamefont {N.}~\bibnamefont {Zhao}}, \bibinfo {author} {\bibfnamefont {L.}~\bibnamefont {Wang}}, \bibinfo {author} {\bibfnamefont {L.}~\bibnamefont {Sima}}, \bibinfo {author} {\bibfnamefont {Y.}~\bibnamefont {Guo}}, \ and\ \bibinfo {author} {\bibfnamefont {H.}~\bibnamefont {Zhang}},\ }\bibfield  {title} {\enquote {\bibinfo {title} {Understanding stress-sensitive behavior of pore structure in tight sandstone reservoirs under cyclic compression using mineral, morphology, and stress analyses},}\ }\href@noop {} {\bibfield  {journal} {\bibinfo  {journal} {Journal of Petroleum Science and Engineering}\ }\textbf {\bibinfo {volume} {218}},\ \bibinfo {pages} {110987} (\bibinfo {year} {2022})}\BibitemShut {NoStop}%
\bibitem [{\citenamefont {Henning}\ and\ \citenamefont {Ohlberger}(2009)}]{henning_heterogeneous_2009}%
  \BibitemOpen
  \bibfield  {author} {\bibinfo {author} {\bibfnamefont {P.}~\bibnamefont {Henning}}\ and\ \bibinfo {author} {\bibfnamefont {M.}~\bibnamefont {Ohlberger}},\ }\bibfield  {title} {\enquote {\bibinfo {title} {The heterogeneous multiscale finite element method for elliptic homogenization problems in perforated domains},}\ }\href@noop {} {\bibfield  {journal} {\bibinfo  {journal} {NUMERISCHE MATHEMATIK}\ }\textbf {\bibinfo {volume} {113}},\ \bibinfo {pages} {601--629} (\bibinfo {year} {2009})}\BibitemShut {NoStop}%
\bibitem [{\citenamefont {Hillairet}(2018)}]{hillairet_homogenization_2018}%
  \BibitemOpen
  \bibfield  {author} {\bibinfo {author} {\bibfnamefont {M.}~\bibnamefont {Hillairet}},\ }\bibfield  {title} {\enquote {\bibinfo {title} {On the {Homogenization} of the {Stokes} {Problem} in a {Perforated} {Domain}},}\ }\href@noop {} {\bibfield  {journal} {\bibinfo  {journal} {ARCHIVE FOR RATIONAL MECHANICS AND ANALYSIS}\ }\textbf {\bibinfo {volume} {230}},\ \bibinfo {pages} {1179--1228} (\bibinfo {year} {2018})}\BibitemShut {NoStop}%
\bibitem [{\citenamefont {Allaire}(1991{\natexlab{a}})}]{Gr1991Homogenization1}%
  \BibitemOpen
  \bibfield  {author} {\bibinfo {author} {\bibfnamefont {G.}~\bibnamefont {Allaire}},\ }\bibfield  {title} {\enquote {\bibinfo {title} {Homogenization of the navier-stokes equations in open sets perforated with tiny holes i. abstract framework, a volume distribution of holes},}\ }\href@noop {} {\bibfield  {journal} {\bibinfo  {journal} {Archive for Rational Mechanics \& Analysis}\ } (\bibinfo {year} {1991}{\natexlab{a}})}\BibitemShut {NoStop}%
\bibitem [{\citenamefont {Allaire}(1991{\natexlab{b}})}]{Gr1991Homogenization2}%
  \BibitemOpen
  \bibfield  {author} {\bibinfo {author} {\bibfnamefont {G.}~\bibnamefont {Allaire}},\ }\bibfield  {title} {\enquote {\bibinfo {title} {Homogenization of the navier-stokes equations in open sets perforated with tiny holes ii: Non-critical sizes of the holes for a volume distribution and a surface distribution of holes},}\ }\href@noop {} {\bibfield  {journal} {\bibinfo  {journal} {Archive for Rational Mechanics \& Analysis}\ } (\bibinfo {year} {1991}{\natexlab{b}})}\BibitemShut {NoStop}%
\bibitem [{\citenamefont {Allaire}(1991{\natexlab{c}})}]{Gr1991Homogenization3}%
  \BibitemOpen
  \bibfield  {author} {\bibinfo {author} {\bibfnamefont {G.}~\bibnamefont {Allaire}},\ }\bibfield  {title} {\enquote {\bibinfo {title} {Homogenization of the navier-stokes equations with a slip boundary condition},}\ }\href@noop {} {\bibfield  {journal} {\bibinfo  {journal} {Communications on Pure and Applied Mathematics}\ } (\bibinfo {year} {1991}{\natexlab{c}})}\BibitemShut {NoStop}%
\bibitem [{\citenamefont {Efendiev}, \citenamefont {Galvis},\ and\ \citenamefont {Hou}(2013)}]{2013Generalized}%
  \BibitemOpen
  \bibfield  {author} {\bibinfo {author} {\bibfnamefont {Y.}~\bibnamefont {Efendiev}}, \bibinfo {author} {\bibfnamefont {J.}~\bibnamefont {Galvis}}, \ and\ \bibinfo {author} {\bibfnamefont {T.~Y.}\ \bibnamefont {Hou}},\ }\bibfield  {title} {\enquote {\bibinfo {title} {Generalized multiscale finite element methods (gmsfem)},}\ }\href@noop {} {\bibfield  {journal} {\bibinfo  {journal} {Journal of Computational Physics}\ }\textbf {\bibinfo {volume} {251}},\ \bibinfo {pages} {116--135} (\bibinfo {year} {2013})}\BibitemShut {NoStop}%
\bibitem [{\citenamefont {Chung}, \citenamefont {Vasilyeva},\ and\ \citenamefont {Wang}(2017)}]{chung_conservative_2017}%
  \BibitemOpen
  \bibfield  {author} {\bibinfo {author} {\bibfnamefont {E.~T.}\ \bibnamefont {Chung}}, \bibinfo {author} {\bibfnamefont {M.}~\bibnamefont {Vasilyeva}}, \ and\ \bibinfo {author} {\bibfnamefont {Y.}~\bibnamefont {Wang}},\ }\bibfield  {title} {\enquote {\bibinfo {title} {A conservative local multiscale model reduction technique for {Stokes} flows in heterogeneous perforated domains},}\ }\href@noop {} {\bibfield  {journal} {\bibinfo  {journal} {JOURNAL OF COMPUTATIONAL AND APPLIED MATHEMATICS}\ }\textbf {\bibinfo {volume} {321}},\ \bibinfo {pages} {389--405} (\bibinfo {year} {2017})}\BibitemShut {NoStop}%
\bibitem [{\citenamefont {Chung}\ \emph {et~al.}(2016)\citenamefont {Chung}, \citenamefont {Efendiev}, \citenamefont {Li},\ and\ \citenamefont {Vasilyeva}}]{chung_generalized_2016}%
  \BibitemOpen
  \bibfield  {author} {\bibinfo {author} {\bibfnamefont {E.~T.}\ \bibnamefont {Chung}}, \bibinfo {author} {\bibfnamefont {Y.}~\bibnamefont {Efendiev}}, \bibinfo {author} {\bibfnamefont {G.}~\bibnamefont {Li}}, \ and\ \bibinfo {author} {\bibfnamefont {M.}~\bibnamefont {Vasilyeva}},\ }\bibfield  {title} {\enquote {\bibinfo {title} {Generalized multiscale finite element methods for problems in perforated heterogeneous domains},}\ }\href@noop {} {\bibfield  {journal} {\bibinfo  {journal} {APPLICABLE ANALYSIS}\ }\textbf {\bibinfo {volume} {95}},\ \bibinfo {pages} {2254--2279} (\bibinfo {year} {2016})}\BibitemShut {NoStop}%
\bibitem [{\citenamefont {Chung}, \citenamefont {Leung},\ and\ \citenamefont {Vasilyeva}(2016)}]{chung_mixed_2016}%
  \BibitemOpen
  \bibfield  {author} {\bibinfo {author} {\bibfnamefont {E.~T.}\ \bibnamefont {Chung}}, \bibinfo {author} {\bibfnamefont {W.~T.}\ \bibnamefont {Leung}}, \ and\ \bibinfo {author} {\bibfnamefont {M.}~\bibnamefont {Vasilyeva}},\ }\bibfield  {title} {\enquote {\bibinfo {title} {Mixed {GMsFEM} for second order elliptic problem in perforated domains},}\ }\href@noop {} {\bibfield  {journal} {\bibinfo  {journal} {JOURNAL OF COMPUTATIONAL AND APPLIED MATHEMATICS}\ }\textbf {\bibinfo {volume} {304}},\ \bibinfo {pages} {84--99} (\bibinfo {year} {2016})}\BibitemShut {NoStop}%
\bibitem [{\citenamefont {Chung}\ \emph {et~al.}(2017{\natexlab{a}})\citenamefont {Chung}, \citenamefont {Efendiev}, \citenamefont {Leung}, \citenamefont {Vasilyeva},\ and\ \citenamefont {Wang}}]{chung_online_2017}%
  \BibitemOpen
  \bibfield  {author} {\bibinfo {author} {\bibfnamefont {E.~T.}\ \bibnamefont {Chung}}, \bibinfo {author} {\bibfnamefont {Y.}~\bibnamefont {Efendiev}}, \bibinfo {author} {\bibfnamefont {W.~T.}\ \bibnamefont {Leung}}, \bibinfo {author} {\bibfnamefont {M.}~\bibnamefont {Vasilyeva}}, \ and\ \bibinfo {author} {\bibfnamefont {Y.}~\bibnamefont {Wang}},\ }\bibfield  {title} {\enquote {\bibinfo {title} {Online adaptive local multiscale model reduction for heterogeneous problems in perforated domains},}\ }\href@noop {} {\bibfield  {journal} {\bibinfo  {journal} {APPLICABLE ANALYSIS}\ }\textbf {\bibinfo {volume} {96}},\ \bibinfo {pages} {2002--2031} (\bibinfo {year} {2017}{\natexlab{a}})}\BibitemShut {NoStop}%
\bibitem [{\citenamefont {Xie}\ \emph {et~al.}(2025{\natexlab{a}})\citenamefont {Xie}, \citenamefont {Galvis}, \citenamefont {Yang},\ and\ \citenamefont {Huang}}]{xie_time_2025}%
  \BibitemOpen
  \bibfield  {author} {\bibinfo {author} {\bibfnamefont {W.}~\bibnamefont {Xie}}, \bibinfo {author} {\bibfnamefont {J.}~\bibnamefont {Galvis}}, \bibinfo {author} {\bibfnamefont {Y.}~\bibnamefont {Yang}}, \ and\ \bibinfo {author} {\bibfnamefont {Y.}~\bibnamefont {Huang}},\ }\bibfield  {title} {\enquote {\bibinfo {title} {On time integrators for {Generalized} {Multiscale} {Finite} {Element} {Methods} applied to advection-diffusion in high-contrast multiscale media},}\ }\href@noop {} {\bibfield  {journal} {\bibinfo  {journal} {JOURNAL OF COMPUTATIONAL AND APPLIED MATHEMATICS}\ }\textbf {\bibinfo {volume} {460}} (\bibinfo {year} {2025}{\natexlab{a}})}\BibitemShut {NoStop}%
\bibitem [{\citenamefont {Park}\ and\ \citenamefont {Hoang}(2020)}]{park_hierarchical_2020}%
  \BibitemOpen
  \bibfield  {author} {\bibinfo {author} {\bibfnamefont {J.~S.~R.}\ \bibnamefont {Park}}\ and\ \bibinfo {author} {\bibfnamefont {V.~H.}\ \bibnamefont {Hoang}},\ }\bibfield  {title} {\enquote {\bibinfo {title} {Hierarchical multiscale finite element method for multi-continuum media},}\ }\href@noop {} {\bibfield  {journal} {\bibinfo  {journal} {JOURNAL OF COMPUTATIONAL AND APPLIED MATHEMATICS}\ }\textbf {\bibinfo {volume} {369}} (\bibinfo {year} {2020})}\BibitemShut {NoStop}%
\bibitem [{\citenamefont {Park}\ and\ \citenamefont {Hoang}(2022)}]{park_homogenization_2022}%
  \BibitemOpen
  \bibfield  {author} {\bibinfo {author} {\bibfnamefont {J.~S.~R.}\ \bibnamefont {Park}}\ and\ \bibinfo {author} {\bibfnamefont {V.~H.}\ \bibnamefont {Hoang}},\ }\bibfield  {title} {\enquote {\bibinfo {title} {Homogenization of a multiscale multi-continuum system},}\ }\href@noop {} {\bibfield  {journal} {\bibinfo  {journal} {APPLICABLE ANALYSIS}\ }\textbf {\bibinfo {volume} {101}},\ \bibinfo {pages} {1271--1298} (\bibinfo {year} {2022})}\BibitemShut {NoStop}%
\bibitem [{\citenamefont {Park}\ \emph {et~al.}(2020)\citenamefont {Park}, \citenamefont {Cheung}, \citenamefont {Mai},\ and\ \citenamefont {Hoang}}]{park_multiscale_2020}%
  \BibitemOpen
  \bibfield  {author} {\bibinfo {author} {\bibfnamefont {J.~S.~R.}\ \bibnamefont {Park}}, \bibinfo {author} {\bibfnamefont {S.~W.}\ \bibnamefont {Cheung}}, \bibinfo {author} {\bibfnamefont {T.}~\bibnamefont {Mai}}, \ and\ \bibinfo {author} {\bibfnamefont {V.~H.}\ \bibnamefont {Hoang}},\ }\bibfield  {title} {\enquote {\bibinfo {title} {Multiscale simulations for upscaled multi-continuum flows},}\ }\href@noop {} {\bibfield  {journal} {\bibinfo  {journal} {JOURNAL OF COMPUTATIONAL AND APPLIED MATHEMATICS}\ }\textbf {\bibinfo {volume} {374}} (\bibinfo {year} {2020})}\BibitemShut {NoStop}%
\bibitem [{\citenamefont {Park}, \citenamefont {Cheung},\ and\ \citenamefont {Mai}(2021)}]{park_multiscale_2021}%
  \BibitemOpen
  \bibfield  {author} {\bibinfo {author} {\bibfnamefont {J.~S.~R.}\ \bibnamefont {Park}}, \bibinfo {author} {\bibfnamefont {S.~W.}\ \bibnamefont {Cheung}}, \ and\ \bibinfo {author} {\bibfnamefont {T.}~\bibnamefont {Mai}},\ }\bibfield  {title} {\enquote {\bibinfo {title} {Multiscale simulations for multi-continuum {Richards} equations},}\ }\href@noop {} {\bibfield  {journal} {\bibinfo  {journal} {JOURNAL OF COMPUTATIONAL AND APPLIED MATHEMATICS}\ }\textbf {\bibinfo {volume} {397}} (\bibinfo {year} {2021})}\BibitemShut {NoStop}%
\bibitem [{\citenamefont {Barenblatt}, \citenamefont {Zheltov},\ and\ \citenamefont {Kochina}(1960)}]{1960Basic}%
  \BibitemOpen
  \bibfield  {author} {\bibinfo {author} {\bibfnamefont {G.~I.}\ \bibnamefont {Barenblatt}}, \bibinfo {author} {\bibfnamefont {Y.~P.}\ \bibnamefont {Zheltov}}, \ and\ \bibinfo {author} {\bibfnamefont {I.~N.}\ \bibnamefont {Kochina}},\ }\bibfield  {title} {\enquote {\bibinfo {title} {Basic concepts in the theory of seepage ofhomogeneous liquids in fissured rocks},}\ }\href@noop {} {\bibfield  {journal} {\bibinfo  {journal} {Journal of Applied Mathematics}\ }\textbf {\bibinfo {volume} {24}},\ \bibinfo {pages} {5--1303} (\bibinfo {year} {1960})}\BibitemShut {NoStop}%
\bibitem [{\citenamefont {Chung}\ \emph {et~al.}(2017{\natexlab{b}})\citenamefont {Chung}, \citenamefont {Efendiev}, \citenamefont {Leung},\ and\ \citenamefont {Vasilyeva}}]{2017Coupling}%
  \BibitemOpen
  \bibfield  {author} {\bibinfo {author} {\bibfnamefont {E.~T.}\ \bibnamefont {Chung}}, \bibinfo {author} {\bibfnamefont {Y.}~\bibnamefont {Efendiev}}, \bibinfo {author} {\bibfnamefont {T.}~\bibnamefont {Leung}}, \ and\ \bibinfo {author} {\bibfnamefont {M.}~\bibnamefont {Vasilyeva}},\ }\bibfield  {title} {\enquote {\bibinfo {title} {Coupling of multiscale and multi-continuum approaches},}\ }\href@noop {} {\bibfield  {journal} {\bibinfo  {journal} {GEM - International Journal on Geomathematics}\ }\textbf {\bibinfo {volume} {8}},\ \bibinfo {pages} {9--41} (\bibinfo {year} {2017}{\natexlab{b}})}\BibitemShut {NoStop}%
\bibitem [{\citenamefont {Pruess}\ and\ \citenamefont {Narasimhan}(1982)}]{K1982On}%
  \BibitemOpen
  \bibfield  {author} {\bibinfo {author} {\bibfnamefont {K.}~\bibnamefont {Pruess}}\ and\ \bibinfo {author} {\bibfnamefont {T.~N.}\ \bibnamefont {Narasimhan}},\ }\bibfield  {title} {\enquote {\bibinfo {title} {On fluid reserves and the production of superheated steam from fractured, vapor-dominated geothermal reservoirs},}\ }\href@noop {} {\bibfield  {journal} {\bibinfo  {journal} {Journal of Geophysical Research}\ }\textbf {\bibinfo {volume} {87}},\ \bibinfo {pages} {9329} (\bibinfo {year} {1982})}\BibitemShut {NoStop}%
\bibitem [{\citenamefont {Pruess}\ and\ \citenamefont {Narasimhan}(1985)}]{1985Practical}%
  \BibitemOpen
  \bibfield  {author} {\bibinfo {author} {\bibfnamefont {K.}~\bibnamefont {Pruess}}\ and\ \bibinfo {author} {\bibfnamefont {T.~N.}\ \bibnamefont {Narasimhan}},\ }\bibfield  {title} {\enquote {\bibinfo {title} {Practical method for modeling fluid and heat flow in fractured porous media},}\ }\href@noop {} {\bibfield  {journal} {\bibinfo  {journal} {Society of Petroleum Engineers Journal}\ }\textbf {\bibinfo {volume} {25}},\ \bibinfo {pages} {14--26} (\bibinfo {year} {1985})}\BibitemShut {NoStop}%
\bibitem [{\citenamefont {Wu}\ and\ \citenamefont {Pruess}(1988)}]{1988A}%
  \BibitemOpen
  \bibfield  {author} {\bibinfo {author} {\bibfnamefont {Y.~S.}\ \bibnamefont {Wu}}\ and\ \bibinfo {author} {\bibfnamefont {K.}~\bibnamefont {Pruess}},\ }\bibfield  {title} {\enquote {\bibinfo {title} {A multiple-porosity method for simulation of naturally fractured petroleum reservoirs},}\ }\href@noop {} {\bibfield  {journal} {\bibinfo  {journal} {Spe Reservoir Engineering}\ }\textbf {\bibinfo {volume} {3}},\ \bibinfo {pages} {327--336} (\bibinfo {year} {1988})}\BibitemShut {NoStop}%
\bibitem [{\citenamefont {Xie}\ \emph {et~al.}(2025{\natexlab{b}})\citenamefont {Xie}, \citenamefont {Efendiev}, \citenamefont {Huang}, \citenamefont {Leung},\ and\ \citenamefont {Yang}}]{xie_multicontinuum_2025}%
  \BibitemOpen
  \bibfield  {author} {\bibinfo {author} {\bibfnamefont {W.}~\bibnamefont {Xie}}, \bibinfo {author} {\bibfnamefont {Y.}~\bibnamefont {Efendiev}}, \bibinfo {author} {\bibfnamefont {Y.}~\bibnamefont {Huang}}, \bibinfo {author} {\bibfnamefont {W.~T.}\ \bibnamefont {Leung}}, \ and\ \bibinfo {author} {\bibfnamefont {Y.}~\bibnamefont {Yang}},\ }\bibfield  {title} {\enquote {\bibinfo {title} {Multicontinuum homogenization in perforated domains},}\ }\href@noop {} {\bibfield  {journal} {\bibinfo  {journal} {JOURNAL OF COMPUTATIONAL PHYSICS}\ }\textbf {\bibinfo {volume} {530}} (\bibinfo {year} {2025}{\natexlab{b}})}\BibitemShut {NoStop}%
\bibitem [{\citenamefont {Ammosov}\ \emph {et~al.}(2026)\citenamefont {Ammosov}, \citenamefont {Huang}, \citenamefont {Leung},\ and\ \citenamefont {Shan}}]{ammosov_multicontinuum_2026}%
  \BibitemOpen
  \bibfield  {author} {\bibinfo {author} {\bibfnamefont {D.}~\bibnamefont {Ammosov}}, \bibinfo {author} {\bibfnamefont {J.}~\bibnamefont {Huang}}, \bibinfo {author} {\bibfnamefont {W.~T.}\ \bibnamefont {Leung}}, \ and\ \bibinfo {author} {\bibfnamefont {B.}~\bibnamefont {Shan}},\ }\bibfield  {title} {\enquote {\bibinfo {title} {Multicontinuum homogenization for coupled flow and transport equations},}\ }\href@noop {} {\bibfield  {journal} {\bibinfo  {journal} {JOURNAL OF COMPUTATIONAL AND APPLIED MATHEMATICS}\ }\textbf {\bibinfo {volume} {471}} (\bibinfo {year} {2026})}\BibitemShut {NoStop}%
\bibitem [{\citenamefont {D.~A.}\ \emph {et~al.}(2024)\citenamefont {D.~A.}, \citenamefont {J.}, \citenamefont {W.~T.},\ and\ \citenamefont {B.}}]{2024Generalized}%
  \BibitemOpen
  \bibfield  {author} {\bibinfo {author} {\bibfnamefont {A.}~\bibnamefont {D.~A.}}, \bibinfo {author} {\bibfnamefont {H.}~\bibnamefont {J.}}, \bibinfo {author} {\bibfnamefont {L.}~\bibnamefont {W.~T.}}, \ and\ \bibinfo {author} {\bibfnamefont {S.}~\bibnamefont {B.}},\ }\bibfield  {title} {\enquote {\bibinfo {title} {Generalized multiscale finite element method for multicontinuum coupled flow and transport model},}\ }\href@noop {} {\bibfield  {journal} {\bibinfo  {journal} {Lobachevskii journal of mathematics}\ ,\ \bibinfo {pages} {45}} (\bibinfo {year} {2024})}\BibitemShut {NoStop}%
\bibitem [{\citenamefont {Wu}\ \emph {et~al.}(2011)\citenamefont {Wu}, \citenamefont {Di}, \citenamefont {Kang},\ and\ \citenamefont {Fakcharoenphol}}]{wu_multiple-continuum_2011}%
  \BibitemOpen
  \bibfield  {author} {\bibinfo {author} {\bibfnamefont {Y.-S.}\ \bibnamefont {Wu}}, \bibinfo {author} {\bibfnamefont {Y.}~\bibnamefont {Di}}, \bibinfo {author} {\bibfnamefont {Z.}~\bibnamefont {Kang}}, \ and\ \bibinfo {author} {\bibfnamefont {P.}~\bibnamefont {Fakcharoenphol}},\ }\bibfield  {title} {\enquote {\bibinfo {title} {A multiple-continuum model for simulating single-phase and multiphase flow in naturally fractured vuggy reservoirs},}\ }\href@noop {} {\bibfield  {journal} {\bibinfo  {journal} {Journal of Petroleum Science and Engineering}\ }\textbf {\bibinfo {volume} {78}},\ \bibinfo {pages} {13--22} (\bibinfo {year} {2011})}\BibitemShut {NoStop}%
\bibitem [{\citenamefont {Li}, \citenamefont {Zhang},\ and\ \citenamefont {Li}(2015)}]{li_multi-continuum_2015}%
  \BibitemOpen
  \bibfield  {author} {\bibinfo {author} {\bibfnamefont {X.}~\bibnamefont {Li}}, \bibinfo {author} {\bibfnamefont {D.}~\bibnamefont {Zhang}}, \ and\ \bibinfo {author} {\bibfnamefont {S.}~\bibnamefont {Li}},\ }\bibfield  {title} {\enquote {\bibinfo {title} {A multi-continuum multiple flow mechanism simulator for unconventional oil and gas recovery},}\ }\href@noop {} {\bibfield  {journal} {\bibinfo  {journal} {Journal of Natural Gas Science and Engineering}\ }\textbf {\bibinfo {volume} {26}},\ \bibinfo {pages} {652--669} (\bibinfo {year} {2015})}\BibitemShut {NoStop}%
\bibitem [{\citenamefont {Kazemi}\ \emph {et~al.}(1976)\citenamefont {Kazemi}, \citenamefont {Merrill}, \citenamefont {Porterfield},\ and\ \citenamefont {Zeman}}]{1976Numerical}%
  \BibitemOpen
  \bibfield  {author} {\bibinfo {author} {\bibfnamefont {H.}~\bibnamefont {Kazemi}}, \bibinfo {author} {\bibfnamefont {L.~S.}\ \bibnamefont {Merrill}}, \bibinfo {author} {\bibfnamefont {K.~L.}\ \bibnamefont {Porterfield}}, \ and\ \bibinfo {author} {\bibfnamefont {P.~R.}\ \bibnamefont {Zeman}},\ }\bibfield  {title} {\enquote {\bibinfo {title} {Numerical simulation of water-oil flow in naturally fractured reservoirs},}\ }\href@noop {} {\bibfield  {journal} {\bibinfo  {journal} {Society of Petroleum Engineers Journal}\ }\textbf {\bibinfo {volume} {16}},\ \bibinfo {pages} {317--326} (\bibinfo {year} {1976})}\BibitemShut {NoStop}%
\bibitem [{\citenamefont {Warren}\ and\ \citenamefont {Root}(1963)}]{1963The}%
  \BibitemOpen
  \bibfield  {author} {\bibinfo {author} {\bibfnamefont {J.~E.}\ \bibnamefont {Warren}}\ and\ \bibinfo {author} {\bibfnamefont {P.~J.}\ \bibnamefont {Root}},\ }\bibfield  {title} {\enquote {\bibinfo {title} {The behavior of naturally fractured reservoirs},}\ }\href@noop {} {\bibfield  {journal} {\bibinfo  {journal} {Society of Petroleum Engineers Journal}\ }\textbf {\bibinfo {volume} {3}},\ \bibinfo {pages} {245--255} (\bibinfo {year} {1963})}\BibitemShut {NoStop}%
\bibitem [{\citenamefont {Bensoussan}\ and\ \citenamefont {Alain}(1978)}]{Bensoussan1978Asymptotic}%
  \BibitemOpen
  \bibfield  {author} {\bibinfo {author} {\bibnamefont {Bensoussan}}\ and\ \bibinfo {author} {\bibnamefont {Alain}},\ }\href@noop {} {\emph {\bibinfo {title} {Asymptotic analysis for periodic structures /}}}\ (\bibinfo  {publisher} {Asymptotic analysis for periodic structures /},\ \bibinfo {year} {1978})\BibitemShut {NoStop}%
\bibitem [{\citenamefont {Oleinik}, \citenamefont {Shamaev},\ and\ \citenamefont {Yosifian}(2012)}]{2012Mathematical}%
  \BibitemOpen
  \bibfield  {author} {\bibinfo {author} {\bibfnamefont {O.}~\bibnamefont {Oleinik}}, \bibinfo {author} {\bibfnamefont {A.}~\bibnamefont {Shamaev}}, \ and\ \bibinfo {author} {\bibfnamefont {G.}~\bibnamefont {Yosifian}},\ }\bibfield  {title} {\enquote {\bibinfo {title} {Mathematical problems in elasticity and homogenization},}\ }\href@noop {} {\  (\bibinfo {year} {2012})}\BibitemShut {NoStop}%
\bibitem [{\citenamefont {Cioranescu}\ and\ \citenamefont {Donato}(1999)}]{1999An}%
  \BibitemOpen
  \bibfield  {author} {\bibinfo {author} {\bibfnamefont {D.}~\bibnamefont {Cioranescu}}\ and\ \bibinfo {author} {\bibfnamefont {P.}~\bibnamefont {Donato}},\ }\bibfield  {title} {\enquote {\bibinfo {title} {An introduction to homogenization},}\ }\href@noop {} {\bibfield  {journal} {\bibinfo  {journal} {Oxford University Press,}\ } (\bibinfo {year} {1999})}\BibitemShut {NoStop}%
\bibitem [{\citenamefont {Wang}, \citenamefont {Cao},\ and\ \citenamefont {Wong}(2015)}]{wang_multiscale_2015}%
  \BibitemOpen
  \bibfield  {author} {\bibinfo {author} {\bibfnamefont {X.}~\bibnamefont {Wang}}, \bibinfo {author} {\bibfnamefont {L.}~\bibnamefont {Cao}}, \ and\ \bibinfo {author} {\bibfnamefont {Y.}~\bibnamefont {Wong}},\ }\bibfield  {title} {\enquote {\bibinfo {title} {{MULTISCALE} {COMPUTATION} {AND} {CONVERGENCE} {FOR} {COUPLED} {THERMOELASTIC} {SYSTEM} {IN} {COMPOSITE} {MATERIALS}},}\ }\href@noop {} {\bibfield  {journal} {\bibinfo  {journal} {MULTISCALE MODELING \& SIMULATION}\ }\textbf {\bibinfo {volume} {13}},\ \bibinfo {pages} {661--690} (\bibinfo {year} {2015})}\BibitemShut {NoStop}%
\bibitem [{\citenamefont {Dong}\ \emph {et~al.}(2018{\natexlab{a}})\citenamefont {Dong}, \citenamefont {Cui}, \citenamefont {Nie}, \citenamefont {Yang},\ and\ \citenamefont {Yang}}]{dong_multiscale_2018}%
  \BibitemOpen
  \bibfield  {author} {\bibinfo {author} {\bibfnamefont {H.}~\bibnamefont {Dong}}, \bibinfo {author} {\bibfnamefont {J.}~\bibnamefont {Cui}}, \bibinfo {author} {\bibfnamefont {Y.}~\bibnamefont {Nie}}, \bibinfo {author} {\bibfnamefont {Z.}~\bibnamefont {Yang}}, \ and\ \bibinfo {author} {\bibfnamefont {Z.}~\bibnamefont {Yang}},\ }\bibfield  {title} {\enquote {\bibinfo {title} {Multiscale computational method for heat conduction problems of composite structures with diverse periodic configurations in different subdomains},}\ }\href@noop {} {\bibfield  {journal} {\bibinfo  {journal} {COMPUTERS \& MATHEMATICS WITH APPLICATIONS}\ }\textbf {\bibinfo {volume} {76}},\ \bibinfo {pages} {2549--2565} (\bibinfo {year} {2018}{\natexlab{a}})}\BibitemShut {NoStop}%
\bibitem [{\citenamefont {Cao}(2006)}]{cao_multiscale_2006}%
  \BibitemOpen
  \bibfield  {author} {\bibinfo {author} {\bibfnamefont {L.}~\bibnamefont {Cao}},\ }\bibfield  {title} {\enquote {\bibinfo {title} {Multiscale asymptotic expansion and finite element methods for the mixed boundary value problems of second order elliptic equation in perforated domains},}\ }\href@noop {} {\bibfield  {journal} {\bibinfo  {journal} {NUMERISCHE MATHEMATIK}\ }\textbf {\bibinfo {volume} {103}},\ \bibinfo {pages} {11--45} (\bibinfo {year} {2006})}\BibitemShut {NoStop}%
\bibitem [{\citenamefont {Dong}\ \emph {et~al.}(2018{\natexlab{b}})\citenamefont {Dong}, \citenamefont {Cui}, \citenamefont {Nie},\ and\ \citenamefont {Yang}}]{dong_second-order_2018}%
  \BibitemOpen
  \bibfield  {author} {\bibinfo {author} {\bibfnamefont {H.}~\bibnamefont {Dong}}, \bibinfo {author} {\bibfnamefont {J.}~\bibnamefont {Cui}}, \bibinfo {author} {\bibfnamefont {Y.}~\bibnamefont {Nie}}, \ and\ \bibinfo {author} {\bibfnamefont {Z.}~\bibnamefont {Yang}},\ }\bibfield  {title} {\enquote {\bibinfo {title} {Second-order two-scale computational method for damped dynamic thermo-mechanical problems of quasi-periodic composite materials},}\ }\href@noop {} {\bibfield  {journal} {\bibinfo  {journal} {JOURNAL OF COMPUTATIONAL AND APPLIED MATHEMATICS}\ }\textbf {\bibinfo {volume} {343}},\ \bibinfo {pages} {575--601} (\bibinfo {year} {2018}{\natexlab{b}})}\BibitemShut {NoStop}%
\bibitem [{\citenamefont {Dong}\ \emph {et~al.}(2018{\natexlab{c}})\citenamefont {Dong}, \citenamefont {Nie}, \citenamefont {Cui}, \citenamefont {Yang},\ and\ \citenamefont {Wang}}]{dong_second-order_2018-1}%
  \BibitemOpen
  \bibfield  {author} {\bibinfo {author} {\bibfnamefont {H.}~\bibnamefont {Dong}}, \bibinfo {author} {\bibfnamefont {Y.}~\bibnamefont {Nie}}, \bibinfo {author} {\bibfnamefont {J.}~\bibnamefont {Cui}}, \bibinfo {author} {\bibfnamefont {Z.}~\bibnamefont {Yang}}, \ and\ \bibinfo {author} {\bibfnamefont {Z.}~\bibnamefont {Wang}},\ }\bibfield  {title} {\enquote {\bibinfo {title} {{SECOND}-{ORDER} {TWO}-{SCALE} {ANALYSIS} {METHOD} {FOR} {DYNAMIC} {THERMO}-{MECHANICAL} {PROBLEMS} {OF} {COMPOSITE} {STRUCTURES} {WITH} {CYLINDRICAL} {PERIODICITY}},}\ }\href@noop {} {\bibfield  {journal} {\bibinfo  {journal} {INTERNATIONAL JOURNAL OF NUMERICAL ANALYSIS AND MODELING}\ }\textbf {\bibinfo {volume} {15}},\ \bibinfo {pages} {834--863} (\bibinfo {year} {2018}{\natexlab{c}})}\BibitemShut {NoStop}%
\bibitem [{\citenamefont {Dong}\ \emph {et~al.}(2018{\natexlab{d}})\citenamefont {Dong}, \citenamefont {Cao}, \citenamefont {Wang},\ and\ \citenamefont {Huang}}]{dong_multiscale_2018-1}%
  \BibitemOpen
  \bibfield  {author} {\bibinfo {author} {\bibfnamefont {Q.-l.}\ \bibnamefont {Dong}}, \bibinfo {author} {\bibfnamefont {L.-q.}\ \bibnamefont {Cao}}, \bibinfo {author} {\bibfnamefont {X.}~\bibnamefont {Wang}}, \ and\ \bibinfo {author} {\bibfnamefont {J.-z.}\ \bibnamefont {Huang}},\ }\bibfield  {title} {\enquote {\bibinfo {title} {Multiscale numerical algorithms for elastic wave equations with rapidly oscillating coefficients},}\ }\href@noop {} {\bibfield  {journal} {\bibinfo  {journal} {APPLIED MATHEMATICS AND COMPUTATION}\ }\textbf {\bibinfo {volume} {336}},\ \bibinfo {pages} {16--35} (\bibinfo {year} {2018}{\natexlab{d}})}\BibitemShut {NoStop}%
\bibitem [{\citenamefont {Cao}\ and\ \citenamefont {Cui}(2004)}]{cao_asymptotic_2004}%
  \BibitemOpen
  \bibfield  {author} {\bibinfo {author} {\bibfnamefont {L.}~\bibnamefont {Cao}}\ and\ \bibinfo {author} {\bibfnamefont {J.}~\bibnamefont {Cui}},\ }\bibfield  {title} {\enquote {\bibinfo {title} {Asymptotic expansions and numerical algorithms of eigenvalues and eigenfunctions of the {Dirichlet} problem for second order elliptic equations in perforated domains},}\ }\href@noop {} {\bibfield  {journal} {\bibinfo  {journal} {NUMERISCHE MATHEMATIK}\ }\textbf {\bibinfo {volume} {96}},\ \bibinfo {pages} {525--581} (\bibinfo {year} {2004})}\BibitemShut {NoStop}%
\bibitem [{\citenamefont {Feng}\ and\ \citenamefont {Cui}(2004)}]{feng_multi-scale_2004}%
  \BibitemOpen
  \bibfield  {author} {\bibinfo {author} {\bibfnamefont {Y.}~\bibnamefont {Feng}}\ and\ \bibinfo {author} {\bibfnamefont {J.}~\bibnamefont {Cui}},\ }\bibfield  {title} {\enquote {\bibinfo {title} {Multi-scale analysis and {FE} computation for the structure of composite materials with small periodic configuration under condition of coupled thermoelasticity},}\ }\href@noop {} {\bibfield  {journal} {\bibinfo  {journal} {INTERNATIONAL JOURNAL FOR NUMERICAL METHODS IN ENGINEERING}\ }\textbf {\bibinfo {volume} {60}},\ \bibinfo {pages} {1879--1910} (\bibinfo {year} {2004})}\BibitemShut {NoStop}%
\bibitem [{\citenamefont {Dong}\ and\ \citenamefont {Cao}(2009)}]{dong_multiscale_2009}%
  \BibitemOpen
  \bibfield  {author} {\bibinfo {author} {\bibfnamefont {Q.-L.}\ \bibnamefont {Dong}}\ and\ \bibinfo {author} {\bibfnamefont {L.-Q.}\ \bibnamefont {Cao}},\ }\bibfield  {title} {\enquote {\bibinfo {title} {Multiscale asymptotic expansions and numerical algorithms for the wave equations of second order with rapidly oscillating coefficients},}\ }\href@noop {} {\bibfield  {journal} {\bibinfo  {journal} {APPLIED NUMERICAL MATHEMATICS}\ }\textbf {\bibinfo {volume} {59}},\ \bibinfo {pages} {3008--3032} (\bibinfo {year} {2009})}\BibitemShut {NoStop}%
\bibitem [{\citenamefont {Dong}\ and\ \citenamefont {Cao}(2014)}]{dong_multiscale_2014}%
  \BibitemOpen
  \bibfield  {author} {\bibinfo {author} {\bibfnamefont {Q.-L.}\ \bibnamefont {Dong}}\ and\ \bibinfo {author} {\bibfnamefont {L.-Q.}\ \bibnamefont {Cao}},\ }\bibfield  {title} {\enquote {\bibinfo {title} {Multiscale asymptotic expansions methods and numerical algorithms for the wave equations in perforated domains},}\ }\href@noop {} {\bibfield  {journal} {\bibinfo  {journal} {APPLIED MATHEMATICS AND COMPUTATION}\ }\textbf {\bibinfo {volume} {232}},\ \bibinfo {pages} {872--887} (\bibinfo {year} {2014})}\BibitemShut {NoStop}%
\end{thebibliography}
\end{document}